\documentclass{book}

\usepackage{geometry}                % See geometry.pdf to learn the layout options. There are lots.
\usepackage{graphicx}

\usepackage{epstopdf}

\usepackage{enumerate}

\usepackage{amsmath}
\usepackage{amssymb}
\usepackage{amsthm}
\usepackage{url}
\usepackage{color}

\usepackage{tikz-cd}

\usepackage{pifont}% http://ctan.org/pkg/pifont
\usepackage{stmaryrd}
\usepackage{dsfont}
\usepackage{textcomp}
\usepackage{verbatim}

\newtheorem{thm}{Theorem}[section]
\newtheorem{prop}[thm]{Proposition}
\newtheorem{lemma}[thm]{Lemma}
\newtheorem{cor}[thm]{Corollary}

\theoremstyle{definition}

\newtheorem{dfn}[thm]{Definition}
\newtheorem{ex}[thm]{Example}

\newtheorem{constr}[thm]{Construction}
\newtheorem{ax}[thm]{Axioms}
\newtheorem{ass}[thm]{Assumption}
\newtheorem{ntn}[thm]{Notation}
\newtheorem{que}[thm]{Question}
\newtheorem{rem}[thm]{Remark}

\newcommand{\pow}{\mathcal P}

\newcommand{\xs}{x_1, \dots, x_{n}}

\newcommand{\as}{a_1, \dots, a_{n}}

\newcommand{\rnk}{\mathrm{rank}}

\newcommand{\id}{\mathrm{id}}

\newcommand{\dom}{\mathrm{dom}}

\newcommand{\setmany}{\text{\reflectbox{$\mathsf{S}$}}}

\newcommand{\df}{\mathrm{df}}
\newcommand{\epsin}{\mathrel{\varepsilon}}

\title{Self-similarity in the Foundations}
\author{
Paul K. Gorbow \\
\\
\\
\\
\\
Thesis submitted for the degree of Ph.D. in Logic, defended on June 14, 2018. \\
\\
Supervisors: \\
Ali Enayat (primary)\\ 
Peter LeFanu Lumsdaine (secondary)\\
Zachiri McKenzie (secondary)
\\
\\
	\small University of Gothenburg \\
	\small Department of Philosophy, Linguistics, and Theory of Science \\
	\small Box 200, 405 30 G\"OTEBORG, Sweden \\
}
\date{}                                           % 

\begin{document}

\maketitle

\tableofcontents

\chapter{Introduction}

\section{Introduction to a general audience}

Zooming in on a fern, one may be struck by how its parts resemble the whole, and zooming in further, how the parts of the parts again exhibit a similar structure. Likewise, when learning about atoms, solar systems and galaxies, one may naturally be struck by the analogy between a typical spiral galaxy and the accretion disc of an infant solar system, as well as the analogy between how the planets stand to the sun and how the electrons stand to the atomic nucleus. These analogies are far from perfect, but nonetheless inevitable. In the mathematical realm, idealized forms of this phenomenon have become well known to the broad public through the theory of fractals, a field from which a rich variety of beautiful images has spun off, to the delight of the human mind.

The term {\em self-similarity} has come to be used quite broadly for this phenomenon. More formal mathematical umbrella-terms are {\em endomorphism}, the slightly narrower {\em self-embedding} and the even more narrow {\em automorphism}. In either case, the words {\em morphism} and {\em embedding} are largely left open to be defined as appropriate for the domain of study. Generally, an endomorphism exhibits a part with similar structure as the whole, but where some details of the structure may be lost; with a self-embedding the part and whole have the same structure; and with an automorphism no proper part is involved, but the structure of the whole emerges in several different ways on the same whole. Since the whole is in a trivial sense a part of itself, any automorphism is also a self-embedding, and moreover any self-embedding is an endomorphism.  In some fields, e.g. the field of the present monograph, so much structure needs to be preserved that the notions of endomorphism and self-embedding coincide. On the other hand, for finite structures the notions of self-embedding and automorphism end up coinciding.

Let us consider a circular clock from $0$ to $12$ as an example. The clock has only an hour-dial, and this dial takes just one step forward every hour, never passing in between two numbers on the clock: We have $0$ (considered equal to $12$) at the top of the clock. If one does addition and subtraction on this clock,  then we have $2 +3 = 5$, $5+8 = 1$ and $5 - 8 = 9$, for example. It turns out that this mathematical structure (a finite group called $C_{12}$) has a few self-embeddings. For example, notice that the even numbers of the clock exhibit a similar structure to that of the whole clock (this {\em substructure} is a {\em subgroup} of $C_{12}$ called $C_6$). Let us define a function $f$ from the clock to itself by $f\hspace{2pt}(x) = x + x$, with the addition done in $C_{12}$, of course. So $f\hspace{2pt}(0) = 0$, $f\hspace{2pt}(3) = 6$ and $f\hspace{2pt}(7) = 2$, for example. This is an endomorphism of $C_{12}$, which reveals that $C_{12}$ has $C_6$ as a similar part of itself. The essential criterion that $f$ passes to attain the status of endomorphism, is that $f\hspace{2pt}(x + y) = f\hspace{2pt}(x) + f\hspace{2pt}(y)$, for all $x, y$ in $C_{12}$. The function $f$ is neither a self-embedding nor an automorphism, because $C_6$ is not the same as $C_{12}$ (half of the details of the structure were left out). However, it turns out the function $g(x) = x+x+x+x+x$ is an automorphism. For if you renumber your clock by this function $g$, then you will go around the clock with the numbers $0$ at the top, and then $5$, $10$, $3$, $8$, $1$, $6$, $11$, $4$, $9$, $2$, $7$, and finally come back again to $0$. If you then just change the mechanics so that the hour-dial takes $5$ steps each hour instead of just $1$, then this clock will also show the time correctly, stepping first from $0$ to $1$, then from $1$ to $2$, and so on. Thus, the same structure as the original structure, is present on the whole in a very different way.

In this monograph, structures for the foundations of mathematics are studied, particularly structures satisfying axioms of set theory. By a key theorem of Tennenbaum from the 1960:s, these structures are so complex that it is impossible to devise a theoretical algorithm or computer program to describe them, even if it were allowed to run for infinite time. Remarkably, the existence of these complex structures follows from rather innocent looking axioms of set theory. In particular, this monograph studies structures of set theory and structures of category theory of such extraordinary complexity. To give a hunch of the kind of structures studied, here follows a sketch of a construction of a non-standard model of arithmetic:

In arithmetic we have a first order predicate language with numerals $0,1,2,3, \dots$ denoting the {\em standard} natural numbers, and we have the operations of $+$ and $\times$ as well as the relation $<$ for comparing size. We now add a name $c$ to this language. There is no way of showing that $c < t$ for any numeral $t = 0, 1, 2, \dots$ in the original language, because this name may be assigned to denote any numeral. This means that it is consistent that $0 < c$, $1 < c$, $2 < c$, and so on. Taking all these sentences together, one finds that it is consistent with arithmetic that $c$ is larger than every standard natural number. G\"odel's Completeness Theorem says that for any consistent set of sentences, there is a structure that satisfies those sentences. So there is a structure satisfying all the sentences of arithmetic, which still has numbers greater than all of the standard natural numbers $0, 1, 2, \dots$. By Tennenbaum's Theorem, this non-standard structure is not algorithmically describable. Thus, we find ourselves in a situation where we can show the existence of structures which are so complex that our prospects for describing them are severely limited. 

Nonetheless, a rich mathematical theory has emerged from the study of these structures, especially for those structures of smallest possible infinite size (countably infinite structures). Chapters \ref{ch prel emb} and \ref{ch emb set theory} of this monograph are concerned with such structures of set theory. A theorem of Friedman from the 1970:s states roughly that {\em every} countable non-standard structure of the conventional theory of arithmetic, known as Peano Arithmetic ($\mathrm{PA}$), has a proper self-embedding. Friedman also proved this for a fragment of Zermelo-Fraenkel set theory ($\mathrm{ZF}$), the conventional set theory. So the phenomenon of self-similarity is abundant in the theory of non-standard foundational structures. In Section \ref{Existence of embeddings between models of set theory}, theorems along these lines are refined and generalized for the setting of a fairly weak fragment of $\mathrm{ZF}$. 

$\mathrm{ZF}$ axiomatizes a hierarchical conception of sets. Each set can be assigned a rank, such that if $x$ is a member of $y$, then the rank of $x$ is less than the rank of $y$. Thus it makes sense to ask whether a structure of $\mathrm{ZF}$ can be extended to a larger structure, such that new sets are only added at higher level of this rank-hierarchy. Such an extension is called a {\em rank-end-extension}. In Section \ref{Existence of models of set theory with automorphisms} it is shown that each of a certain class of structures of $\mathrm{ZF}$ can be rank-end-extended to a structure with a non-trivial automorphism such that the sets that are not moved by the automorphism are precisely the sets in the original structure. This constitutes a set-theoretic generalization of an arithmetic result by Gaifman.

In Section \ref{Characterizations}, the Friedman- and Gaifman-style theorems are combined in various ways to obtain several new theorems about non-standard countable structures of set theory.

In Chapters \ref{ch prel NF cat} and \ref{ch equicon set cat NF} we turn to category theoretic foundations of mathematics. While set theory is concerned with the relation of {\em membership}, category theory is concerned with {\em transformations}. Since these languages are so different, leading to different branches of the foundations of mathematics, it is of interest to relate the two. Chapter \ref{ch equicon set cat NF} develops a new category theoretic system and establishes a bridge between this system and an alternative set theory called New Foundations ($\mathrm{NF}$). An interesting feature of $\mathrm{NF}$ is that it provides another perspective on the self-similarity phenomena studied in Chapter \ref{ch emb set theory}. From a structure of $\mathrm{ZF}$ with a non-trivial automorphism one can actually obtain a structure of a version of $\mathrm{NF}$ called $\mathrm{NFU}$ (a weaker theory allowing so called urelements or atoms). It is shown in Chapter \ref{ch equicon set cat NF} that $\mathrm{NF}$ and $\mathrm{NFU}$ can be expressed in category theory (in the technical sense that the category theoretic version is equiconsistent to and interprets the set theoretic version). Thus a bridge is built between these two branches of the foundations of mathematics.

\section{Introduction for logicians}

This monograph is a study of self-similarity in foundational structures of set theory and category theory. Chapters \ref{ch prel emb} and \ref{ch emb set theory} concern the former and Chapters \ref{ch prel NF cat} and \ref{ch equicon set cat NF} concern the latter. In Chapter \ref{ch tour}, we take a tour of the theories considered, and in Chapter \ref{ch mot} we look at the motivation behind the research. Chapter \ref{ch where to go} looks ahead to further research possibilities.

Chapter \ref{ch prel emb} is a detailed presentation of basic definitions and results relevant to the study of non-standard models of set theory and embeddings between such models. Chapter \ref{ch emb set theory} contains the main contributions of the author to this field. It will help to state these contributions in the context of previous results that it builds upon. In \cite{EM56}, it is shown that any first-order theory with an infinite model has a model with a non-trivial automorphism. This theorem can be used to show that there are models of $\mathrm{PA}$, $\mathrm{ZFC}$, etc. with non-trivial automorphisms. Later on, Gaifman refined this technique considerably in the domain of models of arithmetic, showing that any countable model $\mathcal{M}$ of $\mathrm{PA}$ can be elementarily end-extended to a model $\mathcal{N}$ with an automorphism $j : \mathcal{N} \rightarrow \mathcal{N}$ whose set of fixed points is precisely $\mathcal{M}$ \cite{Gai76}. This was facilitated by the technical break-through of iterated ultrapowers, introduced by Gaifman and later adapted by Kunen to a set theoretical setting. In \cite{Ena04}, Gaifman's results along these lines were partly generalized to models of set theory. They are further refined and generalized in Section \ref{Existence of models of set theory with automorphisms} of the present monograph. The gist is that Gaifman's result also holds for models of the theory $\mathrm{GBC} + \text{``$\mathrm{Ord}$ is weakly compact''}$, where $\mathrm{GBC}$ is the G\"odel-Bernays theory of classes with the axiom of choice. (The result in arithmetic is actually most naturally stated for models of $\mathrm{ACA}_0$, a theory which essentially stands to $\mathrm{PA}$ as $\mathrm{GBC}$ stands to $\mathrm{ZFC}$.)

Only a few years prior to Gaifman's result, Friedman showed that {\em every} non-standard countable model of a certain fragment of $\mathrm{ZF}$ (or $\mathrm{PA}$) has a proper self-embedding \cite{Fri73}. He actually proved a sharper and more general result, and his discovery lead to several similar results, which are refined and generalized in the present monograph. These results require a few definitions: 

The {\em Takahashi hierarchy}, presented e.g. in \cite{Tak72} (and in Chapter \ref{ch prel emb} of the present monograph), is similar to the well-known L\'evy hierarchy, but any quantifiers of the forms $\exists x \in y, \exists x \subseteq y, \forall x \in y, \forall x \subseteq y$ are considered bounded. $\Delta_0^\mathcal{P}$ is the set of set-theoretic formulae with only bounded quantifiers in that sense, and $\Sigma_n^\mathcal{P}$ and $\Pi_n^\mathcal{P}$ are then defined recursively in the usual way for all $n \in \mathbb{N}$. $\mathrm{KP}^\mathcal{P}$ is the set theory axiomatized by Extensionality, Pair, Union, Powerset, Infinity, $\Delta_0^\mathcal{P} \textnormal{-Separation}$, $\Delta_0^\mathcal{P} \textnormal{-Collection}$ and $\Pi_1^\mathcal{P} \textnormal{-Foundation}$.

Let us now go through some notions of substructure relevant to set theory. A {\em rank-initial} substructure $\mathcal{S}$ of a model $\mathcal{M} \models \mathrm{KP}^\mathcal{P}$ is a submodel that is downwards closed in ranks (so if $s \in \mathcal{S}$ and $\mathcal{M} \models \mathrm{rank}(m) \leq \mathrm{rank}(s)$, then $m \in \mathcal{S}$). It is a {\em rank-cut} if, moreover, there is an infinite strictly descending downwards cofinal sequence of ordinals in $\mathcal{M} \setminus \mathcal{S}$. It is a {\em strong} rank-cut if, moreover, for every function $f : \mathrm{Ord}^\mathcal{S} \rightarrow \mathrm{Ord}^\mathcal{M}$ coded in $\mathcal{M}$ (in the sense that $\mathcal{M}$ believes there is a function $\hat f$ whose externalization restricted to $\mathrm{Ord}^\mathcal{S}$ equals $f$), there is an ordinal $\mu \in \mathcal{M} \setminus \mathcal{S}$ such that $f\hspace{2pt}(\xi) \not\in \mathcal{S} \Leftrightarrow f\hspace{2pt}(\xi) > \mu$. Note that these notions for substructures also make sense for embeddings. 

If (the interpretation of the element-relation in) $\mathcal{M}$ is well-founded, then we say that $\mathcal{M}$ is a {\em standard} model, and otherwise we say that it is non-standard. The largest well-founded rank-initial substructure of $\mathcal{M}$ exists. It is called the {\em well-founded part} of $\mathcal{M}$ and is denoted $\mathrm{WFP}(\mathcal{M})$. It turns out that $\mathrm{WFP}(\mathcal{M})$ is a rank-cut of $\mathcal{M}$. 

Lastly, let us go through the notion of standard system. Suppose that $\mathcal{M}$ is a model of $\mathrm{KP}^\mathcal{P}$ with a proper rank-cut $\mathcal{S}$. If $A \subseteq \mathcal{S}$, then $A$ is {\em coded} in $\mathcal{M}$ if there is $a \in \mathcal{M}$ such that $\{x \in \mathcal{S} \mid \mathcal{M} \models x \in a\} = A$. The {\em standard system} of $\mathcal{M}$ over $\mathcal{S}$, denoted $\mathrm{SSy}_\mathcal{S}(\mathcal{M})$ is the second order structure obtained by expanding $\mathcal{S}$ with all the subsets of $\mathcal{S}$ coded in $\mathcal{M}$. We define $\mathrm{SSy}(\mathcal{M}) = \mathrm{SSy}_{\mathrm{WFP}(\mathcal{M})}(\mathcal{M})$.

Friedman showed that for any countable non-standard models $\mathcal{M}$ and $\mathcal{N}$ of 
$$\mathrm{KP}^\mathcal{P} + \Sigma_1 \textnormal{-Separation} + \textnormal{Foundation},$$ 
and $\mathcal{S}$ such that 
$$\mathcal{S} = \mathrm{WFP}(\mathcal{M}) = \mathrm{WFP}(\mathcal{N}),$$ 
there is a proper rank-initial embedding of $\mathcal{M}$ into $\mathcal{N}$ iff the $\Sigma_1^\mathcal{P}$-theory of $\mathcal{M}$ with parameters in $\mathcal{S}$ is included in the corresponding theory of $\mathcal{N}$ and $\mathrm{SSy}_\mathcal{S}(\mathcal{M}) = \mathrm{SSy}_\mathcal{S}(\mathcal{N})$. In Section \ref{Existence of embeddings between models of set theory} Friedman's result is refined in multiple ways. Firstly, we show that it holds for any common rank-cut $\mathcal{S}$ of $\mathcal{M}$ and $\mathcal{N}$ (not just for the standard cut), secondly, we show that it holds for all countable non-standard models of $\mathrm{KP}^\mathcal{P} + \Sigma_1 \textnormal{-Separation}$, and thirdly, we show that the embedding can be constructed so as to yield a rank-cut of the co-domain.

Friedman's insight lead to further developments in this direction in the model-theory of arithmetic. In particular, it was established for countable non-standard models of $\mathrm{I}\Sigma_1$. 

Ressayre showed, conversely, that if $\mathcal{M} \models \mathrm{I}\Sigma_0 + \mathrm{exp}$, and for every $a \in \mathcal{M}$ there is a proper initial self-embedding of $\mathcal{M}$ which fixes every element $\mathcal{M}$-below $a$, then $\mathcal{M} \models \mathrm{I}\Sigma_1$ \cite{Res87b}. In Section \ref{Characterizations}, we prove a set theoretic version of this optimality result, to the effect that if $\mathcal{M} \models \mathrm{KP}^\mathcal{P}$, and for every $a \in \mathcal{M}$ there is a proper rank-initial self-embedding of $\mathcal{M}$ which fixes every element that is an $\mathcal{M}$-member of $a$, then $\mathcal{M} \models \mathrm{KP}^\mathcal{P} + \Sigma_1 \textnormal{-Separation}$.

Wilkie showed that for every countable non-standard model $\mathcal{M} \models\mathrm{PA}$ and for every element $a$ of $\mathcal{M}$, there is a proper initial self-embedding whose image includes $a$ \cite{Wil77}. In Section \ref{Existence of embeddings between models of set theory}, this result is generalized to set theory in refined form: For every countable non-standard model $\mathcal{M} \models \mathrm{KP}^\mathcal{P} + \Sigma_2^\mathcal{P}\textnormal{-Separation} + \Pi_2^\mathcal{P}\textnormal{-Foundation}$ and for every element $a$ of $\mathcal{M}$, there is a proper initial self-embedding whose image includes $a$. Moreover, Wilkie showed that every model of $\mathrm{PA}$ has continuum many self-embeddings \cite{Wil73}. In Section \ref{Existence of embeddings between models of set theory}, the analogous result is established for countable non-standard models of $\mathrm{KP}^\mathcal{P} + \Sigma_1 \textnormal{-Separation}$.

Yet another result in this vein is that the isomorphism types of countable recursively saturated models of $\mathrm{PA}$ only depends on their theory and standard system. In Section \ref{Existence of embeddings between models of set theory}, we provide a new proof of a generalization of this result for models of $\mathrm{ZF}$, that was first established in \cite{Res87a}.

Once these results have been established for set theory, we are able in Section \ref{Characterizations} to prove a number of new results about non-standard models of set theory. 

Kirby and Paris essentially showed in \cite{KP77} that any cut $\mathcal{S}$ of a model $\mathcal{M} \models \mathrm{I}\Delta_0$ is strong iff $\mathrm{SSy}_\mathcal{S}(\mathcal{M}) \models \mathrm{ACA}_0$. In Section \ref{Characterizations} this is generalized to set theory. It turns out that any rank-cut $\mathcal{S}$ including $\omega^\mathcal{M}$ of an ambient model $\mathcal{M} \models \mathrm{KP}^\mathcal{P} + \mathrm{Choice}$ is strong iff $\mathrm{SSy}_\mathcal{S}(\mathcal{M}) \models \mathrm{GBC} + \textnormal{``$\mathrm{Ord}$ is weakly compact''}$. This result is given a new proof relying on our refined and generalized versions of the Friedman and Gaifman theorems. In particular, the Gaifman theorem needed to be further generalized for this proof to work. A similar technique was used in \cite{Ena07} to reprove the result of Kirby and Paris in the context of arithmetic.

An interesting feature of \ref{ch emb set theory} is that the Friedman- and Gaifman-style theorems are combined in diverse new ways to arrive at several important theorems, thus establishing this as a viable technique. A proof using such a combination was pioneered by Bahrami and Enayat in \cite{BE18}. Generalizing their result to set theory, we show that for any countable model $\mathcal{M} \models \mathrm{KP}^\mathcal{P}$, and any rank-cut $\mathcal{S}$ of $\mathcal{M}$: there is a self-embedding of $\mathcal{M}$ whose set of fixed points is precisely $\mathcal{S}$ iff $\mathcal{S}$ is a $\Sigma_1^\mathcal{P}$-elementary  strong rank-cut of $\mathcal{M}$. In \cite{BE18} the analogous result is shown for models of $\mathrm{I}\Sigma_1$. 

The result of Bahrami and Enayat was inspired by an analogous result in the context of countable recursively saturated models of $\mathrm{PA}$ \cite{KKK91}, namely that any cut of such a model is the fixed point set of an automorphism iff it is an elementary strong cut. In Section \ref{Characterizations}, we generalize this result to set theory by means of a new proof, again relying on a combination of our Friedman- and Gaifman-style theorems. It is shown that for any rank-cut $\mathcal{S}$ of a countable recursively saturated model of $\mathrm{ZFC} + V = \mathrm{HOD}$: $\mathcal{S}$ is the fixed point set of an automorphism of $\mathcal{M}$ iff it is an elementary strong rank-cut.

Finally, in Section \ref{Characterizations} we are also able to combine the Friedman- and Gaifman-style theorems to show the new result that for any countable non-standard $\mathcal{M} \models \mathrm{KP}^\mathcal{P} + \Sigma_1^\mathcal{P} \textnormal{-Separation} + \textnormal{Choice}$: $\mathcal{M}$ has a strong rank-cut isomorphic to $\mathcal{M}$ iff $\mathcal{M}$ expands to a model of $\mathrm{GBC} + \textnormal{``}\mathrm{Ord}$ is weakly compact''. 

Chapters \ref{ch prel NF cat} and \ref{ch equicon set cat NF} provide a category theoretic characterization of the set theory New Foundations, and some of its weaker variants. The former chapter contains the necessary preliminaries, and the latter contains the main new results.

New Foundations ($\mathrm{NF}$) is a set theory obtained from naive set theory by putting a stratification constraint on the comprehension schema; for example, it proves that there is a universal set $V$, and the natural numbers are implemented in the Fregean way (i.e. $n$ is implemented as the set of all sets with $n$ many elements). $\mathrm{NFU}$ ($\mathrm{NF}$ with atoms) is known to be consistent through its remarkable connection with models of conventional set theory that admit automorphisms. This connection was discovered by Jensen, who established the equiconsistency of $\mathrm{NFU}$ with a weak fragment of $\mathrm{ZF}$, and its consistency with the axiom of choice \cite{Jen69}. (So in the $\mathrm{NF}$-setting atoms matter; Jensen's consistency proof for $\mathrm{NFU}$ does not work for $\mathrm{NF}$.) 

This part of the monograph aims to lay the ground for an algebraic approach to the study of $\mathrm{NF}$. A first-order theory, $\mathrm{ML}_\mathsf{Cat}$, in the language of categories is introduced and proved to be equiconsistent to $\mathrm{NF}$. $\mathrm{ML}_\mathsf{Cat}$ is intended to capture the categorical content of the predicative version of the class theory $\mathrm{ML}$ of $\mathrm{NF}$. The main result, for which this research is motivated, is that $\mathrm{NF}$ is interpreted in $\mathrm{ML}_\mathsf{Cat}$ through the categorical semantics. This enables application of category theoretic techniques to meta-mathematical problems about $\mathrm{NF}$-style set theory. Conversely, it is shown that the class theory $\mathrm{ML}$ interprets $\mathrm{ML}_\mathsf{Cat}$, and that a model of $\mathrm{ML}$ can be obtained constructively from a model of $\mathrm{NF}$. Each of the results in this paragraph is shown for the versions of the theories with and without atoms, both for intuitionistic and classical logic.\footnote{Due to the lack of knowledge about the consistency strength of $\mathrm{INF(U)}$, the non-triviality of the statement $\mathrm{Con(INF(U)) \Rightarrow Con(IML(U))}$ needs to be taken as conditional, see Remark \ref{remML}.} Therefore, we use the notation $\mathrm{(I)NF(U)}$ and $\mathrm{(I)ML(U)}$, where the $\mathrm{I}$ stands for the intuitionistic version and the $\mathrm{U}$ stands for the version with atoms. Thus four versions of the theories are considered in parallel. An immediate corollary of these results is that $\mathrm{(I)NF}$ is equiconsistent to $\mathrm{(I)NFU} + |V| = |\mathcal{P}(V)|$. For the classical case, this has already been proved in \cite{Cra00}, but the intuitionistic case appears to be new. Moreover, the result becomes quite transparent in the categorical setting. 

Just like a category of classes has a distinguished subcategory of small morphisms (cf. \cite{ABSS14}), a category modeling $\mathrm{(I)ML(U)}_\mathsf{Cat}$ has a distinguished subcategory of type-level morphisms. This corresponds to the distinction between sets and proper classes in $\mathrm{(I)NF(U)}_\mathsf{Set}$. With this in place, the axiom of power objects familiar from topos theory can be appropriately reformulated for $\mathrm{(I)ML(U)}_\mathsf{Cat}$. It turns out that the subcategory of type-level morphisms contains a topos as a natural subcategory.

Section \ref{Stratified set theory and class theory} introduces the set theories $\mathrm{(I)NF(U)}_\mathsf{Set}$ and the class theories $\mathrm{(I)ML(U)}_\mathsf{Class}$. Here we also establish that $\mathrm{NF(U)}_\mathsf{Set}$ is equiconsistent to $\mathrm{ML(U)}_\mathsf{Class}$, through classical model theory.

In Section \ref{CatAxioms}, categorical semantics is explained in the context of Heyting and Boolean categories. This semantics is applied to show generally that $\mathrm{(I)NF(U)}_\mathsf{Set}$ is equiconsistent to $\mathrm{(I)ML(U)}_\mathsf{Class}$.

The axioms of the novel categorical theory $\mathrm{(I)ML(U)}_\mathsf{Cat}$ are given in Section \ref{Formulation of ML_CAT}, along with an interpretation of $\mathrm{(I)ML(U)}_\mathsf{Cat}$ in $\mathrm{(I)ML(U)}_\mathsf{Class}$. 

It is only after this that the main original results are proved. Most importantly, in Section \ref{interpret set in cat}, category theoretic reasoning is used to validate the axioms of $\mathrm{(I)NF(U)}_\mathsf{Set}$ in the internal language of $\mathrm{(I)ML(U)}_\mathsf{Cat}$ through the categorical semantics. This means that $\mathrm{(I)NF(U)}_\mathsf{Set}$ is interpretable in $\mathrm{(I)ML(U)}_\mathsf{Cat}$. The equiconsistency of $\mathrm{(I)NF}_\mathsf{Set}$ and $\mathrm{(I)NFU} + |V| = |\mathcal{P}(V)|$ is obtained as a corollary. 

In Section \ref{subtopos}, it is shown that every $\mathrm{(I)ML(U)}_\mathsf{Cat}$-category contains a topos as a subcategory.

\chapter{Tour of the theories considered}\label{ch tour}

In this chapter we give overviews of the main theories studied in this monograph. More strictly mathematical introductions to these theories are given in Chapters \ref{ch prel emb} and \ref{ch prel NF cat}. Here we present them in a semi-formal style, attempting to explain some of their underlying intuitions and pointing to some prior research.

\section{Power Kripke-Platek set theory}\label{tour KP}

The set theory $\mathrm{KP}^\mathcal{P}$ may be viewed as the natural extension of Kripke-Platek set theory $\mathrm{KP}$ ``generated\hspace{1pt}'' by adding the Powerset axiom. A strictly mathematical introduction to these theories is given in Section \ref{Power Kripke-Platek set theory}.

\begin{ax}[{\em Power Kripke-Platek set theory}, $\mathrm{KP}^\mathcal{P}$] $\mathrm{KP}^\mathcal{P}$ is the $\mathcal{L}^0$-theory given by these axioms and axiom schemata:
	\[
	\begin{array}{ll}
	\textnormal{Extensionality} & \forall x . \forall y . ((\forall u . u \in x \leftrightarrow u \in y) \rightarrow x = y) \\
\textnormal{Pair} & \forall u . \forall v . \exists x . \forall w . (w \in x \leftrightarrow (w = u \vee w = v)) \\
\textnormal{Union} & \forall x . \exists u . \forall r . (r \in u \leftrightarrow \exists v \in x . r \in v) \\
\textnormal{Powerset} & \forall u . \exists x . \forall v . (v \in x \leftrightarrow v \subseteq u) \\
\textnormal{Infinity} & \exists x . (\varnothing \in x \wedge \forall u \in x . \{u\} \in x) \\
\Delta_0^\mathcal{P} \textnormal{-Separation} & \forall x . \exists y . \forall u . (u \in y \leftrightarrow (u \in x \wedge  \phi(u))) \\
\Delta_0^\mathcal{P} \textnormal{-Collection} & \forall x . (\forall u \in x . \exists v . \phi(u, v) \rightarrow \exists y . \forall u \in x . \exists v \in y . \phi(u, v)) \\
\Pi_1^\mathcal{P} \textnormal{-Foundation} & \exists x . \phi(x) \rightarrow \exists y . (\phi(y) \wedge \forall v \in y . \neg \phi(v)) \\
	\end{array}
	\]
\end{ax}

Apart from adding the Powerset axiom, $\mathrm{KP}^\mathcal{P}$ differs from $\mathrm{KP}$ in that the schemata of Separation, Collection and Foundation are extended to broader sets of formulae, using the Takahashi hierarchy instead of the L\'evy hierarchy. ($\mathrm{KP}$ has $\Delta_0 \textnormal{-Separation}$, $\Delta_0 \textnormal{-Collection}$ and $\Pi_1 \textnormal{-Foundation}$.) As explained in the introduction, and rigorously defined in Section \ref{Basic logic and model theory}, the difference lies in that not only are quantifiers of the forms $\exists x \in y$ and $\forall x \in y$ considered bounded, as in the L\'evy hierarchy, but quantifiers of the forms $\exists x \subseteq y$ and $\forall x \subseteq y$ are also considered bounded in the Takahashi hierarchy. Since 
\[
\mathcal{P}(u) = y \Leftrightarrow (\forall v \subseteq u . v \in y) \wedge \forall v \in y . \forall r \in v . r \in u),
\]
the Takahashi hierarchy may be viewed as the result of treating the powerset operation as a bounded operation. It is in this sense that $\mathrm{KP}^\mathcal{P}$ is ``generated\hspace{1pt}'' from $\mathrm{KP}$ by adding powersets.

The theory $\mathrm{KP}$ has received a great deal of attention, because of its importance to G\"odel's $L$ (the hierarchy of constructible sets), definability theory, recursion theory and infinitary logic. The ``bible'' on this subject is \cite{Bar75}. The main sources on $\mathrm{KP}^\mathcal{P}$ seem to be the papers by Friedman and Mathias that are discussed later on in this section.

There is of course much to say about what can and cannot be proved in these theories. Both theories enjoy a decent recursion theorem. In $\mathrm{KP}$ we have $\Sigma_1$-Recursion and in $\mathrm{KP}^\mathcal{P}$ we have $\Sigma_1^\mathcal{P}$-Recursion. This is quite important to the present monograph in that it enables $\mathrm{KP}$ to prove the totality of the rank-function, but (in the absence of Powerset) it is not sufficient to establish that the function $\alpha \mapsto V_\alpha$ is total on the ordinals. $\mathrm{KP}^\mathcal{P}$ does however prove the latter claim, and this is needed for certain arguments in Chapter \ref{ch emb set theory}, particularly in the proof of our Friedman-style embedding theorem. Also of interest, though not used in this monograph, is that neither of the theories proves the existence of an uncountable ordinal. (This may be seen from the short discussion about the Church-Kleene ordinal later on in the section.) However, $\mathrm{KP}^\mathcal{P}$ augmented with the axiom of Choice proves the existence of an uncountable ordinal, essentially because Choice gives us that $\mathcal{P}(\omega)$ can be well-ordered.

There is also a philosophical reason for considering $\mathrm{KP}$ and $\mathrm{KP}^\mathcal{P}$, in that they encapsulate a more parsimonious ontology of sets than $\mathrm{ZF}$. If $\mathcal{M}$ is a model of $\mathrm{KP}$, $a$ is an element of  $\mathcal{M}$, and $\phi(x)$ is a $\Delta_0$-formula of set theory, then it is fairly easy to see that 
\[
\mathcal{M} \models \phi(a) \Leftrightarrow \mathrm{TC}(\{a\})_\mathcal{M} \models \phi(a),
\]
where $\mathrm{TC}(\{a\})_\mathcal{M}$ is the substructure of $\mathcal{M}$ on the set 
\[\{x \in \mathcal{M} \mid \mathcal{M} \models \text{``$x$ is in the transitive closure of $\{a\}$''}\},\]
where the transitive closure $\mathrm{TC}(x)$ of a set $x$ is its closure under elements, i.e. the least superset of $x$ such that $\forall u \in \mathrm{TC}(x) . \forall r \in u . r \in \mathrm{TC}(x)$. The reason for this equivalence is that a $\Delta_0$-formula can only quantify over elements in the transitive closure of $\{a\}$. 

Similarly, if $\mathcal{M}$ is a model of $\mathrm{KP}^\mathcal{P}$, $a$ is an element of  $\mathcal{M}$, and $\phi(x)$ is a $\Delta_0^\mathcal{P}$-formula of set theory, then 
\[
\mathcal{M} \models \phi(a) \Leftrightarrow \mathrm{STC}(\{a\})_\mathcal{M} \models \phi(a),
\]
where $\mathrm{STC}(\{a\})_\mathcal{M}$ is the substructure of $\mathcal{M}$ on the set 
\[
\{x \in \mathcal{M} \mid \mathcal{M} \models \text{``$x$ is in the supertransitive closure of $\{a\}$''}\},
\]
where the supertransitive closure $\mathrm{STC}(x)$ of a set $x$ is its closure under elements and subsets of elements, i.e. the least superset of $x$ such that $\forall u \in \mathrm{STC}(x) . \forall r \in u . r \in \mathrm{STC}(x)$ and $\forall u \in \mathrm{STC}(x) . \forall r \subseteq u . r \in \mathrm{STC}(x)$.

Thus, the Separation and Collection schemata of $\mathrm{KP}$ and $\mathrm{KP}^\mathcal{P}$ only apply to formulae whose truth depends exclusively on the part of the model which is below the parameters and free variables appearing in the formula (in the respective senses specified above). 

Heuristically, it is often helpful to picture a model of set theory as a triangle $\triangledown$, with the empty set at the bottom and with each rank of the cumulative hierarchy as an imagined horizontal line through the $\triangledown$, with higher ranks higher up in the $\triangledown$. With this picture in mind, since $\mathrm{KP}$ does not include Powerset, it may be viewed as allowing ``thin'' models; on the other hand since it includes $\Delta_0$-Collection, its models are quite ``tall''. The models of $\mathrm{KP}^\mathcal{P}$ on the other hand, are all fairly ``thick'', since this theory includes Powerset.

Friedman's groundbreaking paper \cite{Fri73}, established several important results in the model theory of $\mathrm{KP}^\mathcal{P}$. Section \ref{Existence of embeddings between models of set theory} is concerned with generalizing and refining one of these results (Theorem 4.1 of that paper), as well as related results. In its simplest form, this result is that every countable non-standard model of $\mathrm{KP}^\mathcal{P} + \Sigma_1^\mathcal{P} \textnormal{-Separation}$ has a proper self-embedding. 

A second important result of Friedman's paper (its Theorem 2.3) is that every countable standard model of $\mathrm{KP}^\mathcal{P}$ is the well-founded part of a non-standard model of $\mathrm{KP}^\mathcal{P}$.

Thirdly, let us also consider Theorem 2.6 of Friedman's paper. This theorem says that any countable model of $\mathrm{KP}$ can be extended to a model of $\mathrm{KP}^\mathcal{P}$ with the same ordinals. The {\em ordinal height} of a standard model of $\mathrm{KP}$ is the ordinal representing the order type of the ordinals of the model. An ordinal is said to be {\em admissible} if it is the ordinal height of some model of $\mathrm{KP}$. This notion turns out to be closely connected with recursion theory. For example, the first admissible ordinal is the {\em Church-Kleene ordinal} $\omega_1^{\mathrm{CK}}$, which may also be characterized as the least ordinal which is not order-isomorphic to a recursive well-ordering of the natural numbers. So in particular, Friedman's theorem shows that every countable admissible ordinal is also the ordinal height of some model of $\mathrm{KP}^\mathcal{P}$. 

Another important paper on $\mathrm{KP}^\mathcal{P}$ is Mathias's \cite{Mat01}, which contains a large body of results on weak set theories. See its Section 6 for results on $\mathrm{KP}^\mathcal{P}$. One of many results established there is its Theorem 6.47, which shows that $\mathrm{KP}^\mathcal{P} + V = L$ proves the consistency of $\mathrm{KP}^\mathcal{P}$, where $V = L$ is the statement that every set is G\"odel constructible.

From the perspective of this monograph, the main $\mathrm{KP}^\mathcal{P}$-style set theory of interest is $\mathrm{KP}^\mathcal{P} + \Sigma_1^\mathcal{P} \textnormal{-Separation}$, because it is for non-standard countable models of this theory that Friedman's embedding theorem holds universally. As we saw above, the truth of a $\Delta_0^\mathcal{P}$-sentence with parameters only depends on sets which are in a sense ``below'' the parameters appearing in the sentence. On the other hand, for any $\mathcal{M} \models \mathrm{KP}$ and for any $\Sigma_1^\mathcal{P}$-sentence $\sigma$, written out as $\exists x . \delta(x)$ for some $\delta_0^\mathcal{P}$-formula $\delta(x)$, we have for each element $a$ in $\mathcal{M}$ such that $\mathcal{M} \models \delta(a)$, that there is an element $b$ in $\mathcal{M}$ such that 
\[
\mathcal{M} \models \delta(a) \Leftrightarrow b_\mathcal{M} \models \delta(a),
\]
where $b_\mathcal{M}$ is the substructure of $\mathcal{M}$ on the set
\[
\{x \in \mathcal{M} \mid \mathcal{M} \models x \in b\}.
\]
Thus, all we can say is that the truth of a $\Sigma_1^\mathcal{P}$-sentence only depends on sets appearing ``below'' {\em some} set. For $\Delta_0^\mathcal{P}$-sentences, ``some set'' may be replaced by ``the parameters appearing in the sentence''.

\section{Stratified set theory}\label{tour NF}

The set theory of {\em New Foundations} ($\mathrm{NF}$) evolved in the logicist tradition from the system of Frege in his {\em Grundlagen der Arithemtik} \cite{Fre84} and the system of Russell and Whitehead in their {\em Principia Mathematica} \cite{RW10}. New Foundations was introduced by Quine in his \cite{Qui37}. 

Russell had shown that Frege's system is inconsistent. In set theoretic terms, Russell's paradox arises from the Comprehension axiom schema, which says that for all formulae $\phi(u)$ (possibly with other free variables, but not $x$ free) in the language of set theory: 
\[
\begin{array}{ll}
\text{Comprehension schema} & \exists x . \forall u . (u \in x \leftrightarrow \phi(u))
\end{array}
\]
The instance of this schema for $\phi \equiv u \not\in u$ gives the existence of the set of all sets which are not self-membered. From this a contradiction follows. The argument is quite general in that it goes through intuitionistically and that it only uses this one axiom. In particular, extensionality is not used, so it may be interpreted at least as much as a paradox about properties and exemplification as a paradox about sets and membership:

\begin{prop}[Russell's paradox]
$\exists x . \forall u . (u \in x \leftrightarrow u \not\in u) \vdash \bot$
\end{prop}
\begin{proof}
Since this result is fundamental to the developments of mathematical logic, we prove it with a detailed informal natural deduction in {\em intuitionistic} logic. We take $\neg \phi$ to abbreviate $\phi \rightarrow \bot$. By universal elimination (substituting $x$ for $u$), we obtain $x \in x \leftrightarrow (x \in x \rightarrow \bot)$.

Firstly, we show that $x \in x \rightarrow \bot$: Assume that $x \in x$. By implication elimination (in the rightwards direction), we get $x \in x \rightarrow \bot$, whence by implication elimination again, we derive $\bot$. So by implication introduction, $x \in x \rightarrow \bot$, as claimed.

Secondly, now that we know that $x \in x \rightarrow \bot$, by implication elimination (in the leftwards direction) we obtain $x \in x$. Now by implication elimination we derive $\bot$, as required.
\end{proof}

Let us record the axioms of the set theory based on the comprehension schema.

\begin{ax}[{\em Naive set-theory}]
{\em Naive set-theory} is the theory in first-order logic axiomatized by the universal closures of the following formulae. For each well-formed formula $\phi(u)$ (possibly with other free variables, but not with $x$ free):
\[
\begin{array}{ll}
\text{Extensionality} & (u \in x \leftrightarrow u \in y) \rightarrow x = y \\
\text{Comprehension schema} & \exists x . \forall u . (u \in x \leftrightarrow \phi(u))
\end{array}
\]
\end{ax}

Russel and Whitehead aimed to provide a foundation for mathematics, which would be consistent, philosophically conservative and mathematically workable. It turned out, however, to be difficult to combine these three properties. Russel's paradox was avoided by organizing the system in a hierarchy of types. Philosophical conservativity was materialized in the form of predicativity: Roughly, a set of a certain type $t$, as the extension of a formula $\phi(u)$, can only be constructed if the parameters and quantifiers in the formula ranges over types below $t$ in the hierarchy. The resulting system is quite weak, and it is impractical to found common mathematics on it: If you want to do analysis on the real numbers, you have to move fairly high up in the hierarchy. So Russel and Whitehead introduced an axiom of reducibility which affirms the existence of common mathematical objects on the lowest level. It turns out, however, that the axiom of reducibility makes the system impredicative, in the end. 

From the standpoint of the mathematical logician, who studies logical systems, the resulting system is unworkably complicated. Most of its meta-mathematical features are embodied in a much simpler system, namely {\em The Simple Theory of Types} ($\mathrm{TST}$):

$\mathrm{TST}$ is formulated in $\omega$-sorted first order logic. I.e. we have a sort, $\mathrm{type}^n$, for each $n < \omega$. For each $\mathrm{type}^n$, where $n < \omega$, we have:
\begin{itemize}
\item A countable infinity of variables $v_0^n, v_1^n, \dots$ ranging over $\mathrm{type}^n$. (For simplicity we also use symbols such as $x^n, u^n, \dots$ to refer to these variables.)
\item A binary relation symbol $=^n$ on the derived sort $\mathrm{type}^n \times \mathrm{type}^n$, and atomic formulae $x_k^n =^n x_{k'}^n$, for all $k, k' < \omega$.
\item A binary relation symbol $\in^n$ on the derived sort $\mathrm{type}^n \times \mathrm{type}^{n+1}$, and atomic formulae $x_k^n \in^n x_{k'}^{n+1}$, for all $k, k' < \omega$.
\end{itemize}
The well-formed formulae are obtained as usual, by structural recursion, from these atomic formulae. 

\begin{ax}[{\em The simple theory of types}, $\mathrm{TST}$]
The axioms of $\mathrm{TST}$ are the universal closures of the following formulae, for each $n < \omega$ and for each well-formed formula $\phi(u^n)$ (possibly with other free variables of any sorts, but not with $x^{n+1}$ free):
\[
\begin{array}{ll}
\textnormal{Extensionality}^n & (u^n \in^n x^{n+1} \leftrightarrow u^n \in^n y^{n+1}) \rightarrow x^{n+1} =^{n+1} y^{n+1} \\
\textnormal{Comprehension}^n & \exists x^{n+1} . \forall u^n . (u^n \in^n x^{n+1} \leftrightarrow \phi(u^n))
\end{array}
\]
\end{ax}

Note that if $\phi$ is a formula in this language, then a well-formed formula $\phi^+$ is obtained by replacing each variable $v_k^n$, each symbol $=^n$, and each symbol $\in^n$, by $v_k^{n+1}$, $=^{n+1}$ and $\in^{n+1}$, respectively, for each $n \in \mathbb{N}$. Similarly, if $p$ is a proof in this system, then we obtain a proof $p^+$ by replacing each formula $\phi$ in $p$ by $\phi^+$. So any proof on one ``level'', can be shifted upwards to any higher ``level''. Thus, when actually writing proofs in $\mathrm{TST}$, one is tempted to leave out the type-superscripts, and simply take care that all the symbols appearing in the proof can be typed to yield a formal proof in the system. This motivated Quine to simply remove the sorts from the system! The resulting system has come to be called New Foundations ($\mathrm{NF}$), now to be presented. 

\begin{dfn}
A formula $\phi$ in the (one-sorted) language of set theory is {\em stratified}, if it can be turned into a well-formed formula of $\mathrm{TST}$ by putting type-superscripts on all instances of the variables and all instances of the symbols $=$ and $\in$, such that for each variable $x$, each instance of $x$ receives the same superscript.
\end{dfn}

An analogous definition could be made for {\em stratified proof}, such that any stratified proof from Extensionality and the Comprehension schema can be turned into a proof in $\mathrm{TST}$ by assigning type-superscripts to the variables and relation-symbols appearing in the proof. Obviously, the proof of Russell's paradox is not stratified. Indeed, it seems that there are proof-theoretic interpretations back-and-forth between $\mathrm{TST}$ and a certain restriction of Naive set theory to a logic with only stratified formulae and stratified proofs. This latter theory is where one would de facto be {\em formally} working if one were {\em informally} working without type-superscripts in $\mathrm{TST}$. (Of course, in order to be formal, one would also need to define that restricted logic.) In contrast, $\mathrm{NF}$ is obtained from Naive set-theory, not by restricting the whole logic to stratified formulae and stratified proofs, but only restricting the theory to stratified formulae:

\begin{ax}[{\em New Foundations}, $\mathrm{NF}$]
$\mathrm{NF}$ is the theory in first-order logic axiomatized by the universal closures of the following formulae. For each stratified formula $\phi(u)$ (possibly with other free variables, but not with $x$ free):
\[
\begin{array}{ll}
\text{Extensionality} & (u \in x \leftrightarrow u \in y) \rightarrow x = y \\
\text{Stratified Comprehension schema} & \exists x . \forall u . (u \in x \leftrightarrow \phi(u))
\end{array}
\]
\end{ax}

Connecting to the paragraph preceding the axiomatization, every stratified theorem of $\mathrm{NF}$ has a stratified proof \cite{Cra78}.

\begin{ex}
$\mathrm{NF}$ proves 
\begin{enumerate}
\item $\exists V . \forall u . u \in V$ (a universal set exists),
\item $\exists x . x\in x$ (a self-membered set exists),
\item $\forall x . \exists y . \forall u . (u \in y \leftrightarrow u \not\in x)$ (complements exist),
\item $\forall x . \exists y . \forall u . (u \in y \leftrightarrow u \subseteq x)$ (powersets exist),
\item $\exists x . \forall u . (u \in x \leftrightarrow (\exists p, q \in u . (p \neq q \wedge \forall r \in u . (r = p \vee r = q))))$ (the set of all sets of cardinality 2 exists).
\end{enumerate}
\end{ex}

Examples 1, 3, 4 and 5 above are instances of Stratified Comprehension, while 2 follows from 1. The set whose existence is affirmed by 5 above is the implementation of the number $2$ in $\mathrm{NF}$.

It is natural to axiomatize $\mathrm{NF}$ as above, with Comprehension for all stratified formulae $\phi$. But note that all the axioms of this theory are themselves stratified, so:
\[
\mathrm{NF} = \text{The set of stratified axioms of Naive set-theory}
\]

Therefore, $\mathrm{TST}$ can easily be proof-theoretically interpreted in $\mathrm{NF}$, simply by dropping the type-superscripts. But since $\mathrm{NF}$ is a theory in the usual first-order logic, we are free to make non-stratified proofs in $\mathrm{NF}$. For instance, as seen above, $\mathrm{NF} \vdash \exists x . x \in x$, which is not stratified.

That $\mathrm{NF}$ proves Infinity was shown in \cite{Spe53}. Specker even showed that $\mathrm{NF}$ proves the negation of Choice. So since it can be shown that every finite set satisfies choice, it follows that there must be an infinite set. Since $\mathrm{NF}$ proves Infinity, ordered pairs can be implemented in such a way that the formula ``$p$ is the ordered pair of $x$ and $y$'' is stratified with the same type assigned to both $p$, $x$ and $y$. (Note that for any stratification of the Kuratowski ordered pair $p = \{\{x\}, \{x, y\}\}$, the variable $p$ is assigned a type $2$ higher than that of $x$ and $y$. This is important for obtaining a workable implementation of functions in $\mathrm{NF}$.

An introduction to $\mathrm{NF}$ is given in \cite{For95}. For any basic claims about $\mathrm{NF}$, we implicitly refer to that monograph.

$\mathrm{NFU}$ is a version of $\mathrm{NF}$ that allows for atoms (`U' stands for Urelemente). In this monograph, $\mathrm{NFU}$ is expressed in a language containing a unary predicate symbol of sethood $S$ and a binary function symbol of ordered pair $\langle - , - \rangle$. The notion of stratification for a formula $\phi$ is extended to this language by adding the requirements that every term is assigned a type, and that for any subformula $p = \langle s, t \rangle$ of $\phi$, the same type is assigned to each of the terms $p, s, t$. This is spelled out in more detain in Section \ref{Stratified set theory and class theory}.

\begin{ax}[{\em New Foundations with urelements}, $\mathrm{NFU}$]
$\mathrm{NFU}$ is the theory in first-order logic axiomatized by the universal closures of the following formulae. For each stratified formula $\phi(u)$ (possibly with other free variables, but not with $x$ free):
\[
\begin{array}{ll}
\textnormal{Extensionality for Sets} & (S(x) \wedge S(y) \wedge \forall z . z \in x \leftrightarrow z \in y) \rightarrow x = y \\
\textnormal{Stratified Comprehension} & \exists x . (S(x) \wedge (\forall u . (u \in x \leftrightarrow \phi(u)))) \\
\textnormal{Ordered Pair} & \langle x, y \rangle = \langle x\hspace{1pt}', y\hspace{1pt}' \rangle \rightarrow ( x = x\hspace{1pt}' \wedge y = y\hspace{1pt}' ) \\
\textnormal{Sethood} & z \in x \rightarrow S(x) \\
\end{array}
\]
\end{ax}

Some authors have taken $\mathrm{NFU}$ to refer to a weaker system, in a language with only the relation symbol $\in$, axiomatized by Extensionality for non-empty sets and Stratified comprehension. Let us temporarily call that system $\mathrm{NFU}^*$. If an appropriate axiom of Infinity is added to that system, then we obtain a system $\mathrm{NFU}^* + \textnormal{Infinity}$ that interprets $\mathrm{NFU}$ (but it does not prove the general existence of type-level ordered pairs). Conversely, $\mathrm{NFU}$ proves all axioms of $\mathrm{NFU}^* +  \textnormal{Infinity}$. The two formulations are convenient in different circumstances. When using a theory as a foundation for mathematics, in particular for implementing relations and functions, it is convenient to have the type-level ordered pair of $\mathrm{NFU}$. However, when proving meta-mathematical results, it can be convenient to work with the simpler language of $\mathrm{NFU}^* (+ \textnormal{Infinity})$. The practice of axiomatizing $\mathrm{NFU}$ with an axiom of ordered pair, and extending the stratification requirements accordingly, originates in the work of Randall Holmes, see e.g. \cite{Hol98}. Jensen, who initiated the study of New Foundations with urelements (see below), worked with $\mathrm{NFU}^*$ and extensions thereof, and does not appear to have been aware of the issue of type-level ordered pairs.

Although the problem of proving the consistency of $\mathrm{NF}$ in terms of a traditional $\mathrm{ZF}$-style set theory turned out to be difficult, Jensen proved the consistency of the system $\mathrm{NFU}^* + \textnormal{Infinity} + \textnormal{Choice}$ in \cite{Jen69}. Jensen used Ramsey\hspace{1pt}'s theorem to obtain a particular  model of Mac Lane set theory with an automorphism, and it is relatively straightforward to obtain a model of $\mathrm{NFU}^* + \textnormal{Infinity} + \textnormal{Choice}$ from that model. There are various interesting axioms that can be added to $\mathrm{NFU}$ to increase its consistency strength. As the understanding of automorphisms of non-standard models of $\mathrm{ZF}$-style set theories has increased, several results on the consistency strength of such extensions of $\mathrm{NFU}$ have been obtained in the work of Solovay \cite{Sol97}, Enayat \cite{Ena04} and McKenzie \cite{McK15}. $\mathrm{NFU}$ proves Infinity and is equiconsistent with Mac Lane set theory; Ext$_S$ + SC$_S$ is weaker and does not prove Infinity. From now on we define $\mathrm{NF}$ as $\mathrm{NFU}$ + ``everything is a set'', which (in classical logic) is equivalent to the axiomatization given above. An introduction to $\mathrm{NFU}$ and extended systems is given in \cite{Hol98}. For any basic claims about $\mathrm{NFU}$, we implicitly refer to that monograph.

The theories $\mathrm{NF}$ and $\mathrm{NFU}$ in intuitionistic logic will be referred to as $\mathrm{INF}$ and $\mathrm{INFU}$, respectively. Note that the way $\mathrm{NFU}$ and $\mathrm{NF}$ are axiomatized in this monograph, the intuitionistic versions $\mathrm{INFU}$ and $\mathrm{INF}$ also satisfy e.g. the axiom of ordered pair. But if $\mathrm{INF}$ were axiomatized as $\mathrm{Ext} + \mathrm{SC}$ with intuitionistic logic, as done e.g. in \cite{Dzi95}, it is not clear that the resulting intuitionsitic theory would be as strong.

As shown in \cite{Hai44}, $\mathrm{NF}$ and $\mathrm{NFU}$ also have finite axiomatizations, which clarify that their ``categories of sets and functions'' are Boolean categories. In this monograph certain extentions of the theories of (Heyting) Boolean categories (in the language of category theory) are proved equiconsistent to $\mathrm{(I)NF(U)}$, respectively.

\section{Categorical semantics and algebraic set theory}

Recall that the first-order theories of Heyting algebras and Boolean algebras are precisely what we need to obtain semantics for propositional intuitionistic and classical logic, respectively. Analogously, the theories of Heyting and Boolean categories are first-order theories in the language of category theory which give us semantics for first-order intuitionistic and classical logic, respectively. Actually, any Heyting or Boolean algebra may be considered as partial order, and any partial order may be considered as a category, both steps without loss of information, so category theory provides a convenient framework for the semantics of both propositional and first-order logic. 

Any classical model of set theory, considered as a category with sets as objects and functions as morphisms, is a Boolean category. Thus, since the theory of topoi may be viewed as a ``categorification'' of intuitionistic set theory, it is not surprising that every topos is a Heyting category. In fact, we may view the theory of Heyting categories as the fragment of the the theory of topoi needed for first-order semantics. 

Any object $A$ of a Heyting category $\mathbf{C}$ may be viewed as the domain of a model of first-order intuitionistic logic. For example, if $m : R \rightarrowtail A \times A$ is a monic morphism in $\mathbf{C}$, then $R$ may be viewed as a binary relation on $A$. Thus, $A$ and $m : R \rightarrowtail A \times A$ are sufficient to specify a first-order structure in the language of a single binary relation. Of course, this may not be a structure in the traditional sense where $A$ is a set and $R$ is a set of ordered pairs. In the categorical setting there are generally no elements to talk about, only morphisms between objects, whose behavior may be axiomatized in the language of category theory. In the field of algebraic set theory, one typically extends the theory of Heyting categories with additional axioms to ensure that it has an object which is model of some set theory. In Section \ref{CatAxioms} such axioms are given with stratified set theory in mind.

Algebraic set theory, categorical semantics and categorical logic more generally, have been developed by a large number of researchers. An early pioneering paper of categorical logic is \cite{Law63}; Joyal and Moerdijk started out algebraic set theory with \cite{JM91} and wrote a short book on the subject \cite{JM95}. Chapter \ref{ch equicon set cat NF} is influenced by the comprehensive work of Awodey, Butz, Simpson and Streicher embodied in \cite{ABSS14}. It parallels their approach to an algebraic set theory of categories of classes. The most important difference is that the $\mathrm{NF}$ context leads to a different reformulation of the power object axiom. 

A category of classes $\mathbf{C}$ is a (Heyting) Boolean category with a subcategory $\mathbf{S}$, satisfying various axioms capturing the notion of `smallness', and with a universal object $U$, such that every object is a subobject of $U$. While the axiomatization of categories of classes naturally focuses on the notion of {\em smallness}, the axiomatization in Chapter \ref{ch equicon set cat NF} focuses on the notion of {\em type-level stratification}. Like in \cite{ABSS14}, a restricted notion of power object is needed, which facilitates interpretation of set theory in the categorical semantics. However, to get a ``categorification'' of $\mathrm{NF}$, the power object axiom needs to be restricted in quite a different way, involving an endofunctor.

\chapter{Motivation}\label{ch mot}

Here we shall go through some of the motivation behind the research of the present monograph. The first section concerns the research in Chapter \ref{ch emb set theory}, the second concerns the research in Chapter \ref{ch equicon set cat NF}, and the third concerns how the results of these two chapters connect with each other.

\section{Motivation behind research on embeddings between models of set theory}

It is a common theme throughout mathematics to study structures (in a wide sense) and how these structures relate to each other. Usually structures are related to each other by functions from one structure (the domain) to another (the co-domain), which preserve some of the structure involved. Usually, such a function exhibits that some of the structure of the domain is present in the co-domain as well. Since the study of such functions has turned out to be very fruitful in many branches of mathematics, it makes sense to apply this methodology to models of set theory as well.

When we consider models of such expressive theories as set theories, it is natural to compare structures by means of embeddings. Any embedding exhibits the domain as a substructure of the co-domain, but we can ask various questions about ``how nicely\hspace{1pt}'' the domain can be embedded in the co-domain: Firstly, for any first-order structure we can ask if the embedding is elementary, i.e. whether the truth of every first-order sentence with parameters in the domain is preserved by the embedding. Secondly, for structures of set theory we can ask whether the domain is embedded ``initially\hspace{1pt}'' in the co-domain. For set theory, the intuition of ``initiality\hspace{1pt}'' may be captured by various different formal notions, of different strengths (see Section \ref{Models of set theory}). The weakest notion of this form is called {\em initiality} and requires simply that the image of the embedding is downwards closed under $\in$, i.e. if $b$ is in the image, and the co-domain satisfies that $c \in b$, then $c$ is also in the image. The research in Chapter \ref{ch emb set theory} is concerned with {\em rank-initial} embeddings, defined by the stronger property that for every value of the embedding, every element of the co-domain of rank less than or equal to the rank of that value is also a value of the embedding. As noted in Section \ref{tour KP}, $\mathrm{KP}^\mathcal{P}$ proves that the function $(\alpha \mapsto V_\alpha)$ is total on the ordinals, so it makes sense to consider a notion of embedding which preserves this structure $(\alpha \mapsto V_\alpha)$. And indeed, an embedding $i : \mathcal{M} \rightarrow \mathcal{N}$ is rank-initial iff $i$ is initial and $i(V_\alpha^{\hspace{2pt}\mathcal{M}}) = V_{i(\alpha)}^{\hspace{2pt}\mathcal{N}}$, for every ordinal $\alpha$ in $\mathcal{M}$ (see Corollary \ref{rank-init equivalences in KPP}), so the choice to study the notion of rank-initial embedding is quite a natural in the setting of $\mathrm{KP}^\mathcal{P}$.

Between well-founded structures, all initial embeddings are trivial: This follows from the Mostowski collapse theorem (Theorem \ref{Mostowski collapse} and Proposition \ref{unique embedding well-founded}). In particular, if $i : \mathcal{M} \rightarrow \mathcal{N}$ is an initial embedding between well-founded extensional structures, then $\mathcal{M}$ and $\mathcal{N}$ are isomorphic to transitive sets (with the inherited $\in$-structure) $\mathcal{M}'$ and $\mathcal{N}'$, respectively, and $i$ is induced by the inclusion function of $\mathcal{M}'$ into $\mathcal{N}'$. So for a study of initial embeddings of models of set theory to yield any insight, we must turn our attention to non-standard models. As explained in the introduction, for non-standard models of arithmetic, several interesting results have been obtained that are either directly about initial embeddings between such models, or are proved by means of considering such embeddings. Thus a major motivation for the work in Chapter \ref{ch emb set theory} is to determine whether these generalize to the set theoretic setting, and if so, for which particular set theory. For example, while the results in \cite{BE18} are largely concerned with the theory $\mathrm{I}\Sigma_1$ (the fragment of Peano Arithmetic that restricts induction to $\Sigma_1$-formulae), it is established in Chapter \ref{ch emb set theory} that the corresponding set theory (for this context) is $\mathrm{KP}^\mathcal{P} + \Sigma_1^\mathcal{P}\textnormal{-Separation}$ (see Theorems \ref{Friedman thm} and \ref{Ressayre characterization Sigma_1-Separation}). Thus, the $\Sigma_1$-Induction schema of arithmetic corresponds in a natural way to the combination of the schemata of $\Sigma_1^\mathcal{P}$-Separation, $\Delta_0^\mathcal{P}$-Collection and $\Pi_1$-Foundation in set theory.

Several results and proofs of Section \ref{Characterizations} testify to the importance of studying embeddings. As noted in more detail in the introduction, there are interesting relationships between the $\mathrm{GBC}$ class theory with weakly compact class of ordinals, the notion of strong cut, rank-initial embeddings, and the fixed point set of rank-initial embeddings. Moreover, as indicated in Section \ref{tour NF}, self-embeddings of models of set theory are strongly connected to models of $\mathrm{NFU}$. Thus, all in all, it is intriguing to think about what more results might be obtained from research on embeddings between non-standard models.

\section{Motivation behind stratified algebraic set theory}

$\mathrm{NF}$ corresponds closely with the simple theory of types, $\mathrm{TST}$, an extensional version of higher order logic which Chwistek and Ramsey independently formulated as a simplification of Russell and Whitehead\hspace{1pt}'s system in Principia Mathematica. It was from contemplation of $\mathrm{TST}$ that Quine introduced $\mathrm{NF}$ \cite{Qui37}. Essentially, $\mathrm{NF}$ is obtained from $\mathrm{TST}$ by forgetting the typing of the relations while retaining the restriction on comprehension induced by the typing (thus avoiding Russell's paradox). This results in the notion of stratification, see Definition \ref{DefStrat} below. Thus $\mathrm{NF}$ and $\mathrm{NFU}$ resolve an aspect of type theory which may be considered philosophically dissatisfying: Ontologically, it is quite reasonable to suppose that there are relations which can take both individuals and relations as relata. The simplest example is probably the relation of identity. But in type theory, it is not possible to relate entities of different types. We cannot even say that they are unequal. Since the universe of $\mathrm{NF}$ or $\mathrm{NFU}$ is untyped, such issues disappear. It is therefore not surprising that stratified set theory has attracted attention from philosophers. For example, Cocchiarella applied these ideas to repair Frege's system \cite{Coc85} (a similar result is obtained in \cite{Hol15}), and Cantini applied them to obtain an interesting type-free theory of truth \cite{Can15}. Along a similar line of thought, the categorical version of $\mathrm{NF}$ and $\mathrm{NFU}$ brought forth in this monograph may well be helpful for transferring the ideas of stratified set theory to research in formal ontology. In a formal ontology, one may account for what individuals, properties, relations and tropes exist, where properties and relations are considered in an intensional rather than an extensional sense. Roughly, in category theory the objects are non-extensional, but the morphisms are extensional, and this is arguably fitting to the needs of formal ontology.

$\mathrm{NFU}$ is also intimately connected with the field of non-standard models of arithmetic and set theory. Out of this connection, Feferman proposed a version of $\mathrm{NFU}$ as a foundation for category theory, allowing for such unlimited categories as the category of all sets, the category of all groups, the category of all topological spaces, the category of all categories, etc \cite{Fef06}. This line of research was further pursued by Enayat, McKenzie and the author in \cite{EGM17}. In short, conventional category theory works perfectly fine in a subdomain of the $\mathrm{NFU}$-universe, but the unlimited categories live outside of this subdomain, and their category theoretic properties are unconventional. Even though they are unconventional (usually failing to be Cartesian closed), one might argue that nothing is lost by including them in our mathematical universe. These categories remain to be systematically studied.

The need for a categorical understanding of stratified set theory is especially pressing since very little work has been done in this direction. It is shown in \cite{McL92} that the ``category of sets and functions'' in $\mathrm{NF}$ is not Cartesian closed. However, several positive results concerning this category were proved in an unpublished paper by Forster, Lewicki and Vidrine \cite{FLV14}: In particular, they showed that it has a property they call ``pseudo-Cartesian closedness''. Similarly, Thomas showed in \cite{Tho17} that it has a property he calls ``stratified Cartesian closedness''. The moral is that it is straightforward to show in $\mathrm{INFU}$, that if $A$ and $B$ are sets, which are respectively isomorphic to sets of singletons $A\hspace{1pt}'$ and $B\hspace{1pt}'$, then the set of functions from $\bigcup A\hspace{1pt}'$ to $\bigcup B\hspace{1pt}'$ is an exponential object of $A$ and $B$. ($V$ is not isomorphic to any set of singletons.) In \cite{FLV14} a generalization of the notion of topos was proposed, with ``the category of sets and functions'' of $\mathrm{NF}$ as an instance. It has however not been proved that the appropriate extension $T$ of this theory (which $\mathrm{NF}$ interprets) satisfies $\mathrm{Con}(T) \Rightarrow \mathrm{Con}(\mathrm{NF})$. Using the results of Section \ref{interpret set in cat} of this monograph, it seems within reach to obtain that result by canonically extending a model of $T$ to a model of the categorical theory $\mathrm{ML}_\mathsf{Cat}$ introduced here. That line of research would also help carve out exactly what axioms of $T$ are necessary for that result. Moreover, in \cite{FLV14} it was conjectured that any model of $T$ has a subcategory which is a topos. In Section \ref{subtopos} of this monograph, it is proved that every model of $\mathrm{(I)ML(U)}_\mathsf{Cat}$ has a subcategory which is a topos.

A related direction of research opened up by the present monograph is to generalize the techniques of automorphisms of models of conventional set theory, in order to study automorphisms of topoi. The author expects that a rich landscape of models of $\mathrm{(I)ML(U)}_\mathsf{Cat}$ would be uncovered from such an enterprise. For example, just like there is a topos in which every function on the reals is continuous, a similar result may be obtainable for $\mathrm{IMLU}_\mathsf{Cat}$ by finding such a topos with an appropriate automorphism. Given the intriguing prospects for founding category theory in stratified set theory, this would open up interesting possibilities for stratified category theoretic foundation of mathematics.

The categorical approach to $\mathrm{NF}$ is also promising for helping the metamathematical study of $\mathrm{NF}$. As stated in the introduction, the main result of this research has the immediate corollary that $\mathrm{(I)NF}$ is equiconsistent to $\mathrm{(I)NFU} + |V| = |\mathcal{P}(V)|$. A major open question in the metamathematics of $\mathrm{NF}$ is whether $\mathrm{NF}$ (or even its intuitionistic counterpart, which has not been shown to be equiconsistent to $\mathrm{NF}$) is consistent relative to a system of conventional set theory (independent proof attempts by Gabbay and Holmes have recently been put forth). So yet a motivation for introducing $\mathrm{ML}_\mathsf{Cat}$ is simply that the flexibility of category theory may make it easier to construct models of $\mathrm{ML}_\mathsf{Cat}$, than to construct models of $\mathrm{NF}$, thus aiding efforts to prove and/or simplify proofs of $\mathrm{Con(NF)}$.

Since categorical model theory tends to be richer in the intuitionistic setting, an intriguing line of research is to investigate the possibilities for stratified dependent type theory. Dependent type theory is commonly formulated with a hierarchy of universes. In a sense, this hierarchy is inelegant and seemingly redundant, since any proof on one level of the hierarchy can be shifted to a proof on other levels of the hierarchy. Model-theoretically, this can be captured in a model with an automorphism. Since the semantics of type theory tends to be naturally cast in category theory, the understanding arising from the present paper would be helpful in such an effort.

In conclusion, ``categorification'' tends to open up new possibilities, as forcefully shown by the fruitfulness of topos theory as a generalization of set theory. In the present paper it has already resulted in a simple intuitive proof of the old result of Crabb\'e stated above. So given the relevance of $\mathrm{NF}$ and $\mathrm{NFU}$ to type theory, philosophy, non-standard models of conventional set theory and the foundations of category theory, it is important to investigate how $\mathrm{NF}$ and $\mathrm{NFU}$ can be expressed as theories in the language of category theory.

%------------------------------------------------------------

\chapter{Logic, set theory and non-standard models}\label{ch prel emb}

\section{Basic logic and model theory}\label{Basic logic and model theory}

We work with the usual first-order logic. A {\em signature} is a set of constant, function and relation symbols. The {\em language} of a signature is the set of well-formed formulas in the signature. The arity of function symbols, $f$, and relation symbols, $R$, are denoted $\mathrm{arity}(f\hspace{2pt})$ and $\mathrm{arity}(R)$, respectively. Models in a language are written as $\mathcal{M}$, $\mathcal{N}$, etc. They consist of interpretations of the symbols in the signature; for each symbol $S$ in the signature, its interpretation in $\mathcal{M}$ is denoted $S^\mathcal{M}$. If $X$ is a term, relation or function definable in the language over some theory under consideration, then $X^\mathcal{M}$ denotes its interpretation in $\mathcal{M}$. 

The domain of $\mathcal{M}$ is also denoted $\mathcal{M}$, so $a \in \mathcal{M}$ means that $a$ is an element of the domain of $\mathcal{M}$. Finite tuples are written as $\vec{a}$, and the tuple $\vec{a}$ considered as a set (forgetting the ordering of the coordinates) is also denoted $\vec{a}$. Moreover, $\vec{a} \in \mathcal{M}$ means that each coordinate of $\vec{a}$ is an element of the domain of $\mathcal{M}$. $\mathrm{length}(\vec{a})$ denotes the number of coordinates in $\vec{a}$. For each natural number $k \in \{ 1, \dots, \mathrm{length}(\vec{a}) \}$, $\pi_k(\vec{a})$ is the $k$-th coordinate of $\vec{a}$. When a function $f : A \rightarrow B$ is applied as $f\hspace{2pt}(\vec{a})$ to a tuple $\vec{a} \in A^n$, where $n \in \mathbb{N}$, then it is evaluated coordinate-wise, so $f\hspace{2pt}(\as) = (f\hspace{2pt}(a_1), \dots, f\hspace{2pt}(a_n))$.
If $\Gamma$ is a set of formulae in a language and $n \in \mathbb{N}$, then $\Gamma[x_1, \dots, x_n]$ denotes the subset of $\Gamma$ of formulae all of whose free variables are in $\{x_1, \dots, x_n\}$.

The {\em theory} of a model $\mathcal{M}$, denoted $\mathrm{Th}(\mathcal{M})$, is the set of formulae in the language satisfied by $\mathcal{M}$. If $\Gamma$ is a subset of the language and $S \subseteq \mathcal{M}$, then 
\[\mathrm{Th}_{\Gamma, S}(\mathcal{M}) =_\df \{ \phi(\vec{s}) \mid \phi \in \Gamma \wedge \vec{s} \in S \wedge (\mathcal{M}, \vec{s}) \models \phi(\vec{s}) \}.\]

The standard model of arithmetic is denoted $\mathbb{N}$.

$\mathcal{L}^0$ is the language of set theory,  i.e. the set of all well-formed formulae generated by $\{\in\}$.

$\mathcal{L}^1$ is defined as a two-sorted language in the single binary relation symbol $\{\in\}$; we have a sort $\mathsf{Class}$ of classes (which covers the whole domain and whose variables and parameters are written in uppercase $X, Y, Z, A, B, C,$ etc.) and a sort $\mathsf{Set}$ of sets (which is a subsort of $\mathsf{Class}$ and whose variables and parameters are written in lowercase $x, y, z, a, b, c,$ etc.). The relation $\in$ is a predicate on the derived sort $\mathsf{Set} \times \mathsf{Class}$. 

Models in $\mathcal{L}^1$ are usually written in the form $(\mathcal{M}, \mathcal{A})$, where $\mathcal{M}$ is an $\mathcal{L}^0$-structure on the domain of sets, and $\mathcal{A}$ is a set of classes. An $\mathcal{L}^1$-structure may reductively be viewed as an $\mathcal{L}^0$-structure. If $(\mathcal{M}, \mathcal{A})$ is an $\mathcal{L}^1$-structure, then (unless otherwise stated), by {\em an element of} $(\mathcal{M}, \mathcal{A})$, is meant an element of sort $\mathsf{Set}$.

Let $\mathcal{M}, \mathcal{N}$ be $\mathcal{L}$-structures. $\mathcal{M}$ is a {\em substructure} of $\mathcal{N}$ if $\mathcal{M} \subseteq \mathcal{N}$ and for every constant symbol $c$, relation symbol $R$ and function symbol $f$ of $\mathcal{L}$, we have 
\begin{align*}
c^\mathcal{M} &= c^\mathcal{N}, \\
R^\mathcal{M} &= R^\mathcal{N} \cap \mathcal{M}^{\mathrm{arity}(R)}, \\
f^\mathcal{M} &= f^\mathcal{N} \cap \mathcal{M}^{\mathrm{arity}(f\hspace{2pt}) + 1}.
\end{align*}
Note that since $\mathcal{M}$ is an $\mathcal{L}$-structure, the condition on function symbols implies that $\mathcal{M}$, as a subset of $\mathcal{N}$, is closed under $f^\mathcal{N}$. We also say that $\mathcal{N}$ is an {\em extension} of $\mathcal{M}$. The substructure is {\em proper} if its domain is a proper subset of the domain of the extension, in which case we write $\mathcal{M} < \mathcal{N}$. Note that $\mathcal{M} \leq \mathcal{N} \Leftrightarrow \mathcal{M} < \mathcal{N} \vee \mathcal{M} = \mathcal{N}$, whence this defines a partial order.

If a subset $S$ of $\mathcal{N}$ is closed under the interpretation of all constant and function symbols of $\mathcal{N}$, then we have a substructure $\mathcal{N}\restriction_S$ of $\mathcal{N}$, called {\em the restriction of $\mathcal{N}$ to $S$}, on the domain $S$ defined by $c^\mathcal{N}\restriction_S = c^\mathcal{N}$, $R^\mathcal{N}\restriction_S = R^\mathcal{N} \cap S^{\mathrm{arity}(R)}$ and $f^\mathcal{N}\restriction_S = f^\mathcal{N} \cap S^{\mathrm{arity}(f\hspace{2pt}) + 1}$, for all constant, relation and function symbols, $c$, $R$, $f$, respectively.

An {\em embedding} $f : \mathcal{M} \rightarrow \mathcal{N}$ from an $\mathcal{L}$-structure $\mathcal{M}$ to an $\mathcal{L}$-structure $\mathcal{N}$ is a function $f : \mathcal{M} \rightarrow \mathcal{N}$, such that for each atomic $\mathcal{L}$-formula $\phi(\vec{x})$ and for each $\vec{m} \in \mathcal{M}$, we have 
$$\mathcal{M} \models \phi(\vec{m}) \Leftrightarrow \mathcal{N} \models \phi(f\hspace{2pt}(\vec{m})).$$ 
It follows that the same equivalence holds for quantifier free $\phi$. Since embeddings preserve the formula $x \neq y$, they are injective. So category theoretically, embeddings tend to be monic in the categories where they appear. In particular, if $T$ is a theory, then the category of models of $T$ with embeddings as morphisms has only monic morphisms. In category theory monics are representatives of subobjects, and it is sometimes convenient to talk about embeddings as if they are the actual subobjects. Indeed, note that the domain $\mathcal{M}$ of the embedding $f$ is isomorphic to its image $f\hspace{2pt}(\mathcal{M})$, which is a substructure of its co-domain $\mathcal{N}$; and conversely, any substructure can be thought of as an embedding by considering the inclusion function. Hence, most notions of embeddings also make sense for substructures, and vice versa. Definitions below pertaining to embeddings are thus implicitly extended to substructures, by applying them to the inclusion function.

In accordance with this category theoretic viewpoint, we write $\mathcal{M} \leq \mathcal{N}$, if there is an embedding from $\mathcal{M}$ to $\mathcal{N}$. The set of embeddings from $\mathcal{M}$ to $\mathcal{N}$ is denoted $\llbracket\mathcal{M} \leq \mathcal{N}\rrbracket$. If $h \in \llbracket \mathcal{S} \leq \mathcal{N}\rrbracket$, and there is $f \in \llbracket \mathcal{M} \leq \mathcal{N} \rrbracket$, such that for some $g \in \llbracket \mathcal{S} \leq \mathcal{M} \rrbracket$, we have $h = f \circ g$, then we say that {\em $f$ is an embedding over $g$}  (or that {\em $f$ is an embedding over $\mathcal{S}$}), and we write $\mathcal{M} \leq_h \mathcal{N}$ (or $\mathcal{M} \leq_S \mathcal{N}$). We denote the set of such embeddings by $\llbracket \mathcal{M} \leq_h \mathcal{N}\rrbracket$ (or by $\llbracket \mathcal{M} \leq_\mathcal{S} \mathcal{N}\rrbracket$). Note that if $\mathcal{S} \subseteq \mathcal{M} \cap \mathcal{N}$, then $f \in \llbracket \mathcal{M} \leq \mathcal{N} \rrbracket$ is in $\llbracket \mathcal{M} \leq_\mathcal{S} \mathcal{N} \rrbracket$ iff $f\hspace{2pt}(s) = s$ for each $s \in \mathcal{S}$. Moreover, $\mathcal{M} \leq_\mathcal{S} \mathcal{N} \Rightarrow \mathcal{S} \leq \mathcal{M} \wedge \mathcal{S} \leq \mathcal{N}$. As we progress to define various types of embeddings, the denotation of all these variations of the notation will not always be specified explicitly. The ambition is that, when the denotation cannot be easily inferred from the context, it will be given explicitly. An embedding is {\em proper} if it is not onto; for proper embeddings, we use the symbol `$<$' in all the contexts above.

It is also of interest to consider partial embeddings. For any cardinal $\kappa$, $\llbracket \mathcal{M} \leq_\mathcal{S} \mathcal{N} \rrbracket^{<\kappa}$ denotes the set of partial functions from $\mathcal{M}$ to $\mathcal{N}$ whose domain has cardinality less than $\kappa$, and such that for each $\vec{m} \in \mathcal{M}$ and for each atomic $\mathcal{L}$-formula $\phi(\vec{x})$,
$$\mathcal{M} \models \phi(\vec{m}) \Leftrightarrow \mathcal{N} \models \phi(f\hspace{2pt}(\vec{m})).$$ 

Let $\mathbb{P} = \llbracket \mathcal{M} \leq_\mathcal{S} \mathcal{N} \rrbracket^{<\kappa}$. We endow $\mathbb{P}$ with the following partial order. For any $f,g \in \mathbb{P}$,
\[
f \leq^\mathbb{P} g \Leftrightarrow f\restriction_{\dom(g)} = g.
\]
In particular, we will be concerned with subposets of $\llbracket \mathcal{M} \leq \mathcal{N} \rrbracket^{<\omega}$, consisting of finite partial embeddings.

An {\em isomorphism} is an embedding that has an inverse embedding, or equivalently an embedding that is onto. We write $\mathcal{M} \cong \mathcal{N}$ if $\mathcal{M}$ and $\mathcal{N}$ are isomorphic, and $\llbracket \mathcal{M} \cong \mathcal{N} \rrbracket$ denotes the set of isomorphisms between $\mathcal{M}$ and $\mathcal{N}$. $\mathcal{M}$ is isomorphic to $\mathcal{N}$ {\em over} $\mathcal{S}$, if there is $f \in \llbracket \mathcal{M} \cong \mathcal{N} \rrbracket$ such that $f \in \llbracket \mathcal{M} \leq_\mathcal{S} \mathcal{N} \rrbracket$. We decorate the symbol `$\cong$' with subscripts, just as we do for embeddings.

Let $f : \mathcal{M} \rightarrow \mathcal{N}$ and $f\hspace{2pt}' : \mathcal{M'} \rightarrow \mathcal{N}'$ be embeddings. For each symbol $\vartriangleleft_\mathcal{S}$ among $\leq_\mathcal{S}, <_\mathcal{S}, \cong_\mathcal{S}$, etc., used to compare structures, we write $f\hspace{2pt}' \vartriangleleft f$ if $f\hspace{2pt}'(\mathcal{M}) \vartriangleleft_\mathcal{S} f\hspace{2pt}(\mathcal{M})$. Note that $f\hspace{2pt}' \leq f \wedge f \leq f\hspace{2pt}' \Leftrightarrow f \cong f\hspace{2pt}'$. If $\mathcal{N} = \mathcal{N}'$, then $f\hspace{2pt}' \leq f \wedge f \leq f\hspace{2pt}' \Leftrightarrow f\hspace{2pt}(\mathcal{M}) = f\hspace{2pt}'(\mathcal{M})$, but in general $f\hspace{2pt}' \leq f \wedge f \leq f\hspace{2pt}' \not \Rightarrow f = f\hspace{2pt}'$ even if they have the same domain.

An embedding $f : \mathcal{M} \rightarrow \mathcal{N}$ of $\mathcal{L}$-structures is $\Gamma${\em -elementary}, for some $\Gamma \subseteq \mathcal{L}$, if for each formula $\phi(\vec{x})$ in $\Gamma$, and for each $\vec{m} \in \mathcal{M}$, 
\[
\mathcal{M} \models \phi(\vec{m}) \Leftrightarrow \mathcal{N} \models \phi(f\hspace{2pt}(\vec{m})) .
\]
If there is such an embedding we write $\mathcal{M} \preceq_\Gamma \mathcal{N}$. $f$ is {\em elementary} if it $\mathcal{L}$-elementary. As above, $\mathcal{M} \preceq_{\Gamma, \mathcal{S}} \mathcal{N}$ if there is a $\Gamma$-elementary embedding over $\mathcal{S}$, and $\llbracket \mathcal{M} \preceq_{\Gamma, \mathcal{S}} \mathcal{N} \rrbracket$ denotes the set of witnesses. Like carpenter to hammer, so model theorist to:

\begin{lemma}[The Tarski Test]\label{Tarski test}
	Let $f : \mathcal{M} \rightarrow \mathcal{N}$ be an embedding of $\mathcal{L}$-structures and suppose that $\Gamma \subseteq \mathcal{L}$ is closed under subformulae. If for all $\vec{m} \in \mathcal{M}$ and for all $\psi(\vec{y}) \in \Gamma$ of the form $\psi(\vec{y}) \equiv \exists x . \phi(x, \vec{y})$, we have
	\[
	\mathcal{N} \models \exists x . \phi(x, f\hspace{2pt}(\vec{m})) \Rightarrow \exists m' \in \mathcal{M} \text{, such that } \mathcal{N} \models \phi(f\hspace{2pt}(m'), f\hspace{2pt}(\vec{m})),
	\]
	then $f$ is $\Gamma$-elementary.
\end{lemma}
\begin{proof}
	This is proved by structural induction on the formulae in $\Gamma$. For atomic formulae it follows from that $f$ is an embedding. Moreover, the inductive cases for the propositional connectives follow from that these commute with the satisfaction relation $\models$. So we concentrate on the inductive case for the existential quantifier:
	
	Let $\psi(\vec{y}) \in \Gamma$ be of the form $\psi(\vec{y}) \equiv \exists x . \phi(x, \vec{y})$ and inductively assume that $f$ is $\phi$-elementary. Let $\vec{m} \in \mathcal{M}$ be of the same length as $\vec{y}$. Note that by the condition of the Lemma
	\[
	\begin{array}{rcl}
	&&\mathcal{N} \models \exists x . \phi(x, f\hspace{2pt}(\vec{m})) \\ 
	\Leftrightarrow& \exists m' \in \mathcal{M} \textnormal{, such that} &\mathcal{N} \models \phi(f\hspace{2pt}(m'), f\hspace{2pt}(\vec{m})) \\
	\Leftrightarrow& \exists m' \in \mathcal{M} \textnormal{, such that} &\mathcal{M} \models \phi(m', \vec{m}) \\
	\Leftrightarrow& &\mathcal{M} \models \exists x . \phi(x, \vec{m}), 
	\end{array}
	\]
	as desired.
\end{proof}

If $\kappa$ is a cardinal, then $\llbracket \mathcal{M} \preceq_{\Gamma, \mathcal{S}} \mathcal{N} \rrbracket^{< \kappa}$ denotes the set of partial functions $f$ from $\mathcal{M}$ to $\mathcal{N}$, with domain of cardinality $< \kappa$, such that for all $\phi(\vec{x}) \in \Gamma$ and for all $\vec{m} \in \mathcal{M}$,
\[
\mathcal{M} \models \phi(\vec{m}) \Leftrightarrow \mathcal{N} \models \phi(f\hspace{2pt}(\vec{m})) .
\]

In Section \ref{Models of set theory} we will introduce definitions for more types of embeddings that are relevant to the study of models of set theory.

The {\em uniquely existential quantifier} $\exists! x . \phi(x)$ is defined as 
$$\exists x . (\phi(x) \wedge (\forall y . \phi(y) \rightarrow y = x)).$$
The {\em bounded quantifiers}, $\forall u \in y . \phi(u, y)$ and $\exists u \in y . \phi(u, y)$, are defined as $\forall u . (u \in y \rightarrow \phi(u, y))$ and $\exists u . (u \in y \wedge \phi(u, y))$, respectively. Suppose a background $\mathcal{L}^0$-theory $T$ is given. $\bar\Delta_0 \subseteq \mathcal{L}^0$ is the set of formulae all of whose quantifiers are bounded; $\Delta_0 \subseteq \mathcal{L}^0$ is the set of formulae provably equivalent in $T$ to a formula in $\Delta_0$. $\bar\Sigma_0$ and $\bar\Pi_0$ are defined as equal to $\bar\Delta_0$; $\Sigma_0$ and $\Pi_0$ are defined as equal to $\Delta_0$. Recursively, for every $n \in \mathbb{N}$: $\bar\Sigma_{n+1} \subseteq \mathcal{L}$ is the set of formulae of the form $\exists x . \phi$, where $\phi$ is in $\bar\Pi_n$; and dually, $\bar\Pi_{n+1} \subseteq \mathcal{L}$ is the set of formulae of the form $\forall x . \phi$, where $\phi$ is in $\bar\Sigma_n$. $\Sigma_{n+1} \subseteq \mathcal{L}$ is the set of formulae provably equivalent in $T$ to a formula of the form $\exists x . \phi$, where $\phi$ is in $\Pi_n$; and dually, $\Pi_{n+1} \subseteq \mathcal{L}$ is the set of formulae provably equivalent in $T$ to a formula of the form $\forall x . \phi$, where $\phi$ is in $\Sigma_n$.  Moreover, $\bar\Delta_{n+1} =_\df \bar\Sigma_{n+1} \cap \bar\Pi_{n+1}$ and $\Delta_{n+1} =_\df \Sigma_{n+1} \cap \Pi_{n+1}$. These sets of formulae are collectively called {\em the L\'evy hierarchy}, and we say that they measure a formula's {\em L\'evy complexity}.

If $\phi$ is an $\mathcal{L}^0$-formula and $t$ is an $\mathcal{L}^0$-term, such that none of the variables of $t$ occur in $\phi$, then $\phi^t$ denotes the formula obtained from $\phi$ by replacing each quantifier of the form `$\boxminus x$' by `$\boxminus x \in t$', where $\boxminus \in \{\exists, \forall\}$.

The $\mathcal{P}$-{\em bounded quantifiers} $\forall x \subseteq y . \phi(x, y)$ and $\exists x \subseteq y . \phi(x, y)$ are defined as $\forall x . (x \subseteq y \rightarrow \phi(x, y))$ and $\exists x . (x \subseteq y \wedge \phi(x, y))$, respectively. For each $n \in \mathbb{N}$, we define sets $\bar\Sigma^\mathcal{P}_n$, $\bar\Pi^\mathcal{P}_n$, $\bar\Delta^\mathcal{P}_n$,  $\Sigma^\mathcal{P}_n$, $\Pi^\mathcal{P}_n$ and $\Delta^\mathcal{P}_n$ analogously as above, but replacing ``bounded\hspace{1pt}'' by ``bounded or $\mathcal{P}$-bounded\hspace{1pt}''. These sets of formulae are called {\em the Takahashi hierarchy}, and we say that they measure a formula's {\em Takahashi complexity}. 

When a set of formulae is denoted with a name that includes free variables, for example $p(\vec{x})$, then it is assumed that each formula in the set has at most the free variables $\vec{x}$. Moreover, if $\vec{a}$ are terms or elements of a model, then $p(\vec{a}) = \{\phi(\vec{a}) \mid \phi(\vec{x}) \in p(\vec{x})\}$.  

A {\em type} $p(\vec{x})$ {\em over a theory} $T$ (in a language $\mathcal{L}$) is a set of formulae, such that $T \cup p(\vec{a})$ is a consistent theory in the language $\mathcal{L} \cup \vec{a}$, where $\vec{a}$ are new constant symbols. Given a subset $\Gamma \subseteq \mathcal{L}$, a $\Gamma${\em -type} is a type all of whose formulae are in $\Gamma$. 

Given a model $\mathcal{M}$ in a language $\mathcal{L}$, and $\vec{b} \in M$, a {\em type over} $\mathcal{M}$ is a set of formulae $p(\vec{x}, \vec{b})$, such that for every finite subset 
$$\{\phi_1(\vec{x}, \vec{b}), \dots, \phi_n(\vec{x}, \vec{b})\} \subseteq p(\vec{x}, \vec{b}),$$ 
there are $\vec{a} \in M$, for which $\mathcal{M} \models \phi_1(\vec{a}, \vec{b}) \wedge \dots \wedge \phi_n(\vec{a}, \vec{b})$. The type is {\em realized} in $\mathcal{M}$ if there are $\vec{a} \in M$, such that $\mathcal{M} \models \phi(\vec{a}, \vec{b})$, for every $\phi(\vec{x}, \vec{b}) \in p(\vec{x}, \vec{b})$. Given a fixed G\"odel numbering of the formulae in $\mathcal{L}$, a type $p(\vec{x}, \vec{b})$ over $\mathcal{M}$ is {\em recursive} if $\{ \ulcorner \phi(\vec{x}, \vec{y}) \urcorner  \mid \phi(\vec{x}, \vec{b}) \in p(\vec{x}, \vec{b})\}$ is a recursive set, where $\ulcorner \phi(\vec{x}, \vec{y}) \urcorner$ denotes the G\"odel code of $\phi(\vec{x}, \vec{y})$ (henceforth formulae will usually be identified with their G\"odel codes). $\mathcal{M}$ is {\em recursively $\Gamma$-saturated} if it realizes every recursive $\Gamma$-type over $\mathcal{M}$.

Given a model $\mathcal{M}$ in a language $\mathcal{L}$, a tuple $\vec{a} \in \mathcal{M}$, a subset $\Gamma \subseteq \mathcal{L}$ and a subset $S \subseteq \mathcal{M}$, {\em the $\Gamma$-type of $\vec{a}$ over $\mathcal{M}$ with parameters in $S$} is the set $\{\phi(\vec{x}, \vec{b}) \mid \phi \in \Gamma \wedge \vec{b} \in S \wedge \mathcal{M} \models \phi(\vec{a}, \vec{b})\}$, denoted $\mathrm{tp}_{\Gamma, S}(\vec{a})$.

\section{Order theory and category theory}

A {\em poset} (or {\em partial order}) is a structure $\mathbb{P}$ in the signature $\{\leq\}$ (i.e. a set endowed with a binary relation $\leq^\mathbb{P}$), which satisfies $\forall x . x \leq x$, $\forall x . \forall y . ((x \leq y \wedge y \leq x) \rightarrow x = y)$ and $\forall x . \forall y . \forall z . ((x \leq y \wedge y \leq z) \rightarrow x \leq z)$. A poset $\mathbb{P}$ is {\em linear} (or {\em total}) if it satisfies $\forall x . \forall y . (x \leq y \vee y \leq x)$. We introduce a defined relation-symbol by $x < y \leftrightarrow (x \leq y \wedge x \neq y)$

An {\em embedding $i : \mathbb{P} \rightarrow \mathbb{P}'$ of posets}, is just a special case of embeddings of structures, i.e. it is an embedding of $\{\leq\}$-structures. Let $i : \mathbb{P} \rightarrow \mathbb{P}'$ be an embedding of posets. $y \in \mathbb{P}'$ is an {\em upper bound} of $i$ if $\forall x \in \mathbb{P} . i(x) < y$. If such a $y$ exists then $i$ is {\em bounded above}. $i$ is {\em topless} if it is bounded above but does not have a $\mathbb{P}'$-least upper bound.

A self-embedding $i : \mathbb{P} \rightarrow \mathbb{P}$ is {\em proper} if it is not surjective. A self-embedding $i : \mathbb{P} \rightarrow \mathbb{P}$ is {\em contractive} if for all $x \in \mathbb{P}$, we have $i(x) <_\mathbb{P} x$. 

Let $\mathbb{P}$ be a poset. Given $x \in \mathbb{P}$, define $\mathbb{P}_{\leq x}$ as the substructure of $\mathbb{P}$ on $\{y \in \mathbb{P} \mid y \leq_\mathbb{P} x\}$; and similarly, if $X \in \mathbb{P}$, define $\mathbb{P}_{\leq X}$ as the substructure of $\mathbb{P}$ on $\{y \in \mathbb{P} \mid \exists x \in X . y \leq_\mathbb{P} x\}$. We have analogous definitions for when `$\leq$' is replaced by `$<$', `$\geq$' or `$>$'.

For any ordinal $\alpha$ and linearly ordered set $(\mathbb{L}, <^\mathbb{L})$, the set $\mathbb{L}^{<\alpha}$ of $\mathbb{L}$-valued sequences of length less than $\alpha$, can be {\em lexicographically} ordered by putting 
$$f <^\mathrm{lex} g \Leftrightarrow_\df \exists \gamma < \alpha . (f\hspace{2pt}(\gamma) <^\mathbb{L} g(\gamma) \wedge \forall \xi < \gamma . f\hspace{2pt}(\xi) = g(\xi)).$$
It is easily verified that the lexicographic order is a linear order.

Let $\mathbb{P}$ be a poset. A subset $\mathcal{D} \subseteq \mathbb{P}$ is {\em dense} if for any $x \in \mathbb{P}$ there is $y \in \mathcal{D}$ such that $y \leq x$. A {\em filter} $\mathcal{F}$ on $\mathbb{P}$ is a non-empty subset of $\mathbb{P}$, such that $\forall x, y \in \mathbb{P} . ((x \in \mathcal{F} \wedge x \leq y) \rightarrow y \in \mathcal{F})$ ({\em upwards closed}) and $\forall x, y \in \mathcal{F} . \exists z \in \mathcal{F} . (z \leq x \wedge z \leq y)$ ({\em downwards directed}). A filter $\mathcal{F}$ is an {\em ultrafilter} if it is {\em maximal}, i.e. if there is no filter $\mathcal{F}'$ on $\mathbb{P}$ such that $\mathcal{F} \subsetneq \mathcal{F}'$. Let $\mathbf{D}$ be a set of dense subsets of $\mathbb{P}$. A filter $\mathcal{F}$ is $\mathbf{D}${\em -generic}, if $\forall \mathcal{D} \in \mathbf{D} . \mathcal{D} \cap \mathcal{F} \neq \varnothing$.

\begin{lemma}\label{generic filter existence}
	Let $\mathbb{P}$ be a poset with an element $p$. If $\mathbf{D}$ is a countable set of dense subsets of $\mathbb{P}$, then there is a $\mathbf{D}$-generic filter $\mathcal{F}$ on $\mathbb{P}$ containing $p$.
\end{lemma}
\begin{proof}
	Let $\mathcal{D}_0, \mathcal{D}_1, \dots$ be an enumeration of $\mathbf{D}$. Recursively, and using choice and density, construct a sequence $d_k$ such that for each $k < \omega$,
	\begin{align*}
	d_k &\in \mathcal{D}_k,\\
	d_0 &\leq p,\\
	d_{k+1} &\leq d_k.
	\end{align*}
	Let $\mathcal{F} = \mathbb{P}_{\geq \{d_0, d_1, \dots\}}$. By construction $\mathcal{F}$ is upwards closed, contains $p$, and intersects every $\mathcal{D} \in \mathbf{D}$. If $x, y \in \mathcal{F}$, then we may assume that there are $k \leq l < \omega$ such that $x = d_k$ and $y = d_l$. So $d_l \leq x$ and $d_l \leq y$, whence $\mathcal{F}$ is downwards directed.
\end{proof}

\begin{lemma}\label{ultrafilter existence}
	Let $\mathbb{P}$ be a poset and let $\mathcal{F}$ be a filter on $\mathbb{P}$. There is an ultrafilter $\mathcal{U}$ such that $\mathcal{F} \subseteq \mathcal{U}$.
\end{lemma}
\begin{proof}
	By Zorn's lemma it suffices to show that for any ordinal $\alpha$ and for any $\subseteq$-increasing sequence $\mathcal{F} = (\mathcal{F}_\xi)_{\xi < \alpha}$ of filters on $\mathbb{P}$, the union $\mathcal{G} = \bigcup_{\xi < \alpha} \mathcal{F}_\xi$ is a filter. But this follows from that $\forall x, y \in \mathcal{G} . \exists \xi < \alpha . x \in \mathcal{F}_\xi \wedge y \in \mathcal{F}_\xi$, and from that each $\mathcal{F}_\xi$ is a filter.
\end{proof}

We shall also make use of a notion from category theory. A {\em category} is a set of {\em objects} and a set of {\em morphisms}, along with a partial binary operation of {\em composition} of morphisms, denoted $\circ$, satisfying the following requirements: Each morphism has a {\em domain} and {\em co-domain} which are objects. A morphism $f$ may be written $f : A \rightarrow B$, to indicate that its domain is $A$ and its co-domain is $B$. For any morphisms $f : A \rightarrow B$ and $g : B \rightarrow C$, $g \circ f$ exists and we have $g \circ f : A \rightarrow C$. For every object $A$ there is an {\em identity morphism} $\id_A$, such that for any $f : A \rightarrow B$ and any $g : B \rightarrow A$, we have $f \circ \id_A = f$ and $\id_A \circ g = g$. Finally, for any $f : A \rightarrow B$, $g : B \rightarrow C$ and $h : C \rightarrow D$, we have $(h \circ g) \circ f = h \circ (g \circ f\hspace{2pt})$.

An {\em equalizer } of a pair of morphisms $j, j\hspace{1pt}' : Y \rightarrow Z$ is a morphism $i : X \rightarrow Y$, such that $j \circ i = j\hspace{1pt}' \circ i$, and such that for any $i\hspace{1pt}' : X\hspace{1pt}' \rightarrow Y$ with this property, there is $u : X\hspace{1pt}' \rightarrow X$ such that $i\hspace{1pt}' = i \circ u$. The following examples are easily established:
\begin{enumerate}
\item The category of linear orders with embeddings as morphisms, has equalizers: Given embeddings $j, j\hspace{1pt}' : \mathbb{Y} \rightarrow \mathbb{Z}$, the linear suborder $\mathbb{X}$ of $\mathbb{Y}$ on $\{y \in \mathbb{Y} \mid j(y) = j\hspace{1pt}'(y)\}$, along with the inclusion function $i : X \hookrightarrow Y$, is an equalizer of $j, j\hspace{1pt}'$.
\item In the category of models of a complete theory $T$ with elementary embeddings as morphisms, given elementary embeddings $j, j\hspace{1pt}' : \mathcal{Y} \rightarrow \mathcal{Z}$, if the inclusion function $i : \mathcal{X} \rightarrow \mathcal{Y}$ of the submodel $\mathcal{X}$ of $\mathcal{Y}$ on $\{y \in \mathcal{Y} \mid j(y) = j\hspace{1pt}'(y)\}$ is an elementary embedding, then it is an equalizer of $j, j\hspace{1pt}'$. \end{enumerate}

\section{Power Kripke-Platek set theory}\label{Power Kripke-Platek set theory}

\begin{dfn}[Axioms of set theory]\label{axioms of set theory}
Some common axioms of set theory, in the language $\mathcal{L}^0$, are listed below. Let $\Gamma \subseteq \mathcal{L}^0$. The schemata of $\Gamma \textnormal{-Separation}$, $\Gamma \textnormal{-Collection}$, $\Gamma \textnormal{-Replacement}$ and $\Gamma \textnormal{-Set Induction}$ refer to the set of all instances, where $\phi$ ranges over $\Gamma$. In the former three schemata, $y$ is assumed to be not free in $\phi$.
\[
\begin{array}{ll}
\textnormal{Extensionality} & \forall x . \forall y . ((\forall u . u \in x \leftrightarrow u \in y) \rightarrow x = y) \\
\textnormal{Pair} & \forall u . \forall v . \exists x . \forall w . (w \in x \leftrightarrow (w = u \vee w = v)) \\
\textnormal{Union} & \forall x . \exists u . \forall r . (r \in u \leftrightarrow \exists v \in x . r \in v) \\
\textnormal{Powerset} & \forall u . \exists x . \forall v . (v \in x \leftrightarrow v \subseteq u) \\
\textnormal{Infinity} & \exists x . (\varnothing \in x \wedge \forall u \in x . \{u\} \in x) \\
\Gamma \textnormal{-Separation} & \forall x . \exists y . \forall u . (u \in y \leftrightarrow (u \in x \wedge  \phi(u))) \\
\Gamma \textnormal{-Collection} & \forall x . (\forall u \in x . \exists v . \phi(u, v) \rightarrow \exists y . \forall u \in x . \exists v \in y . \phi(u, v)) \\
\textnormal{Set Foundation} & \forall x . (x \neq \varnothing \rightarrow \exists u \in x . u \cap x = \varnothing) \\
\Gamma \textnormal{-Foundation} & \exists x . \phi(x) \rightarrow \exists y . (\phi(y) \wedge \forall v \in y . \neg \phi(v)) \\
\Gamma \textnormal{-Set Induction} & \big( \forall x . (\forall u \in x . \phi(u) \rightarrow \phi(x, p))\big) \rightarrow \forall x . \phi(x) \\
\end{array}
\]
We also consider $\textnormal{Strong }\Gamma\textnormal{-Collection}$,
$$\forall x . \exists y . \forall u \in x . (\exists v . \phi(u, v) \rightarrow \exists v\hspace{1pt}' \in y . \phi(u, v\hspace{1pt}')),$$
$\Gamma \textnormal{-Replacement}$,
$$\forall x . (\forall u \in x . \exists! v . \phi(u, v) \rightarrow \exists y . \forall v . (v \in y \leftrightarrow \exists u \in x . \phi(u, v))),$$
Transitive Containment,
$$\forall u . \exists x . (u \in x \wedge \forall v \in x . \forall r \in v . r \in x),$$
and Choice,
$$\forall x . ((\forall u \in x . u \neq \varnothing) \rightarrow \exists f : x \rightarrow \bigcup x . \forall u \in x . f\hspace{2pt}(u) \in u).$$

When the $\Gamma$ is omitted, it is assumed to be the whole language $\mathcal{L}^0$. A set $x$ is {\em transitive} if $\forall u \in x . u \subseteq x$.
\end{dfn}

Assuming Extensionality, we have for each $n \in \mathbb{N}$: 
\begin{itemize}
\item $\Sigma_n \textnormal{-Separation}$, $\Pi_n \textnormal{-Separation}$ and $\mathrm{B}(\Sigma_n) \textnormal{-Separation}$ are all equivalent, where $\mathrm{B}(\Sigma_n)$ is the Boolean closure of $\Sigma_n$: It follows from $\Delta_0$-Separation that the subsets of any set are closed under the boolean operations of intersection, union and relative complement.

\item $\Delta_n \textnormal{-Separation} + \Delta_n \textnormal{-Collection}$ implies $\Sigma_n \textnormal{-Replacement}$: If the $\Sigma_n$-formula $\phi(u, v)$ defines a function with domain $x$, then $\phi$ is actually $\Delta_n$, as seen by observing that the formula $\phi\hspace{1pt}'(u, v) \equiv_\df \forall v\hspace{1pt}' . (\phi(u, v\hspace{1pt}') \rightarrow v = v\hspace{1pt}')$ is equivalent to $\phi(u, v)$. By $\Delta_n \textnormal{-Collection}$ there is $y$ containing all values of this function. Now it easily follows from $\Delta_n \textnormal{-Separation}$ on $y$ that the image of the function is a set.

\item $\Sigma_n$-Set induction is equivalent to $\Pi_n$-Foundation.

\item The analogous claims for the Takahashi hierarchy, in place of the L\'evy hierarchy, are also true and proved with analogous arguments.
\end{itemize}

\begin{ax}[{\em Kripke-Platek set theory}, $\mathrm{KP}$] $\mathrm{KP}$ is the $\mathcal{L}^0$-theory given by these axioms and axiom schemata:
	\[
	\begin{array}{l}
	\textnormal{Extensionality} \\
	\textnormal{Pair} \\
	\textnormal{Union} \\
	\textnormal{Infinity} \\
	\Delta_0 \textnormal{-Separation} \\
	\Delta_0 \textnormal{-Collection} \\
	\Pi_1 \textnormal{-Foundation} \\	
	\end{array}
	\]
\end{ax}
$\mathrm{KP}$ proves $\Delta_1$-Separation, $\Sigma_1$-Collection, $\Sigma_1$-Replacement and Transitive containment. (Note the absence of Powerset!)

\begin{ax}[{\em Power Kripke-Platek set theory}, $\mathrm{KP}^\mathcal{P}$] $\mathrm{KP}^\mathcal{P}$ is the $\mathcal{L}^0$-theory given by these axioms and axiom schemata:
	\[
	\begin{array}{l}
	\textnormal{Extensionality} \\
	\textnormal{Pair} \\
	\textnormal{Union} \\
	\textnormal{Powerset} \\
	\textnormal{Infinity} \\
	\Delta_0^\mathcal{P} \textnormal{-Separation} \\
	\Delta_0^\mathcal{P} \textnormal{-Collection} \\
	\Pi_1^\mathcal{P} \textnormal{-Foundation} \\	
	\end{array}
	\]
\end{ax}
The bible on $\mathrm{KP}$ is \cite{Bar75}, which witnesses that a fair amount of mathematics can be conducted within this theory. Also see \cite{Mat01} for a detailed discussion of $\mathrm{KP}^\mathcal{P}$.

$\mathrm{KP}^\mathcal{P}$ proves $\Delta_1^\mathcal{P}$-Separation, $\Sigma_1^\mathcal{P}$-Collection and $\Sigma_1^\mathcal{P}$-Replacement. $\mathrm{KP}$ also proves that the usual arithmetic operations on $\omega$ make it a model of $\mathrm{PA}$. It is a rather weak set theory, in the sense that $L_{\omega_1^\mathrm{CK}} \models \mathrm{KP}$, where $\omega_1^\mathrm{CK}$, known as the {\em Church-Kleene ordinal}, is the least ordinal which is not order-isomorphic to a recursive well-ordering, and $L$ denotes the hieararchy of G\"odel's constructible sets. $\mathrm{KP}^\mathcal{P}$ proves the existence of $\beth_\alpha(A)$, for each set $A$ and ordinal $\alpha$. In particular, for each ordinal $\alpha$, a model of $\alpha$:th order arithmetic, $\mathrm{Z}_\alpha$, can be constructed in the natural way on $\beth_\alpha(\omega)$. 

$\mathrm{Trans}(x)$ is the formula $\forall u \in x . \forall r \in u . r \in x$. $\mathrm{Ord}(x)$ is the formula $\mathrm{Trans}(x) \wedge \forall u, v \in x . (u \in v \vee v \in u)$. Note that both are $\Delta_0$.

In this context, the ordered pair $p = \langle u, v \rangle$ is defined by $p = \{u, \{u, v\}\}$. Note that ordered pair, the projection functions on ordered pairs, and union are $\Delta_0$-notions. In particular, for each $\boxminus \in \{\forall, \exists\}$, we can define $\boxminus \langle u, v \rangle \in x . \phi$ by 
\[
\boxminus p \in x . \exists u, v \in p \cup (\bigcup p) . (p = \langle u, v \rangle \wedge \phi).
\]
So the L\'evy and Takahashi complexities of $\boxminus \langle u, v \rangle \in x . \phi$ are no greater than those of $\phi$. This turns out to be useful:

\begin{prop}
	$\mathrm{KP}^\mathcal{P} \vdash  \Sigma_1^\mathcal{P} \textnormal{-Collection}$
\end{prop}
\begin{proof}
	Suppose that for all $u \in x$, there is $v$ such that $\exists r . \delta(r, u, v)$, where $\delta \in \Delta_0^\mathcal{P}[r, u, v]$. By $\Delta_0^\mathcal{P}$-Collection, there is $y\hspace{1pt}'$ such that for all $u \in x$, there is $\langle r, v \rangle \in y\hspace{1pt}'$ such that $\delta(r, u, v)$. By $\Delta_0$-Collection, there is $y = \{v \mid \exists \langle r, v \rangle \in y\hspace{1pt}'\}$. It follows that for all $u \in x$, there is $v \in y$ such that $\exists r . \delta(r, u, v)$.
\end{proof}

\begin{prop}\label{Strong Delta_0 Coll proves Strong Sigma_1 Coll}
	$\mathrm{KP}^\mathcal{P} + \textnormal{Strong }\Delta_0^\mathcal{P} \textnormal{-Collection} \vdash  \textnormal{Strong }\Sigma_1^\mathcal{P} \textnormal{-Collection}$
\end{prop}
\begin{proof}
	This is similar as the previous proof. By Strong $\Delta_0$-Collection, for any $\delta(r, u, v) \in \Delta_0^\mathcal{P}[r, u, v]$:
	$$\forall x . \exists y\hspace{1pt}' . \forall u \in x . (\exists \langle r, v \rangle . \delta(r, u, v) \rightarrow \exists \langle r, v \rangle \in y\hspace{1pt}' . \delta(r, u, v)).$$
	Letting $y= \{v \mid \exists \langle r, v \rangle \in y\hspace{1pt}'\}$, it follows that
	$$\forall x . \exists y . \forall u \in x . (\exists v . \exists r . \delta(r, u, v) \rightarrow \exists v \in y . \exists r . \delta(r, u, v)),$$
	as desired.
\end{proof}

\begin{prop}
	$\mathrm{KP}^\mathcal{P} \vdash \Delta_1^\mathcal{P}\textnormal{-Separation}$
\end{prop}
\begin{proof}
Let $a$ be a set and let $\phi(x,y)$ and $\psi(x,y)$ be $\Delta_0^\mathcal{P}$-formulae such that $\forall x \in a . (\exists y . \phi(y) \leftrightarrow \neg \exists y . \psi(y))$. We need to show that $\exists b . \forall x . (x \in b \leftrightarrow (x \in a \wedge \exists y . \phi(x, y)))$. Note that $\exists y . \phi(x,y) \vee \exists y . \psi(x,y)$ is $\Sigma_1^\mathcal{P}$, equivalent to $\exists y . ( \phi(x,y) \vee \psi(x,y))$. Thus, $\forall x \in a . \exists y . ( \phi(x,y) \vee \psi(x,y))$, and by $\Sigma_1^\mathcal{P}$-Collection,  there is $c$ such that $\forall x \in a . \exists y \in c. ( \phi(x,y) \vee \psi(x,y))$. It follows that $\forall x \in a . (\exists y . \phi(x, y) \leftrightarrow \exists y \in c . \phi(x, y))$. But the right-hand side is $\Delta_0^\mathcal{P}$, so we obtain the desired $b$ by applying $\Delta_0^\mathcal{P}$-Separation to $\exists y \in c . \phi(x, y)$.
\end{proof}

\begin{prop}
	$\mathrm{KP}^\mathcal{P} + \Sigma_1^\mathcal{P} \textnormal{-Separation} \vdash  \textnormal{Strong }\Sigma_1^\mathcal{P} \textnormal{-Collection}$
\end{prop}
\begin{proof}
By Proposition \ref{Strong Delta_0 Coll proves Strong Sigma_1 Coll}, it suffices to prove Strong $\Delta_0$-Collection. Let $a$ be a set and let $\delta(x, y)$ be $\Delta_0^\mathcal{P}$. By $\Sigma_1$-Separation, there is $a\hspace{1pt}' \subseteq a$ such that $\forall x \in a . (x \in a\hspace{1pt}' \leftrightarrow \exists y . \delta(x, y))$. Hence, by $\Delta_0$-Collection, there is $b$ such that $\forall x \in a\hspace{1pt}' . \exists y \in b . \delta(x, y)$. By construction of $a\hspace{1pt}'$, we have $\forall x \in a . (\exists y . \delta(x, y) \rightarrow \exists y \in b . \delta(x, y))$, as desired.
\end{proof}

We shall now show that various operations are available in $\mathrm{KP}$ and $\mathrm{KP}^\mathcal{P}$. If $F(x, y) \in \mathcal{L}^0$ and some $\mathcal{L}^0$-theory $T$ proves that $\forall x . \forall y . \forall y\hspace{1pt}' . ((F(x, y) \wedge F(x, y\hspace{1pt}')) \rightarrow y = y\hspace{1pt}')$, then we say that $F$ is functional (over $T$), and we use functional notation, writing $F(x) = y$ for the formula $F(x, y)$, in the context of $T$. If, additionally, $T \vdash \forall x . \exists y . F(x) = y$, then we say that $F$ is total (over $T$).

We shall now present some results about introducing defined terms, functions and relations into $\mathrm{KP}$ and $\mathrm{KP}^\mathcal{P}$. A thorough examination is found e.g. in ch. 1 of \cite{Bar75}, working in $\mathrm{KP} + \textnormal{Foundation}$ (in our terminology), but it is easily seen that only $\mathrm{KP}$ is used. 

\begin{prop}\label{transitive closure}
	$\mathrm{KP} \vdash \textnormal{Transitive Containment}$. Moreover, there is a $\Sigma_1$-formula $\mathrm{TC}$, such that
	\begin{itemize}
		\item $\mathrm{KP} \vdash \forall x . \forall y . \forall y\hspace{1pt}' . ((\mathrm{TC}(x, y) \wedge \mathrm{TC}(x, y\hspace{1pt}')) \rightarrow y = y\hspace{1pt}'),$ 
		\item $\mathrm{KP} \vdash \forall x . \exists t . \mathrm{TC}(x) = t$,
		\item $\mathrm{KP} \vdash \forall x . (x \subseteq \mathrm{TC}(x) \wedge \text{``$\mathrm{TC}(x)$ is transitive''})$,
		\item $\mathrm{KP} \vdash \forall x . \forall t . ((x \subseteq t \wedge \text{``$t$ is transitive''}) \rightarrow \mathrm{TC}(x) \subseteq t)$.
	\end{itemize}
\end{prop}

{\em Remark.} Note that if $F$ is $\Sigma_1$ as well as functional and total over $\mathrm{KP}$, then $F$ is $\Delta_1$ over $\mathrm{KP}$: This is seen by considering the formula
\[
F'(x, y) \equiv \forall y\hspace{1pt}' . (F(x, y\hspace{1pt}') \rightarrow y = y\hspace{1pt}').
\]
$F'$ is clearly $\Pi_1$ over $\mathrm{KP}$. By functionality, $F(x, y) \Rightarrow F'(x, y)$, and by totality $F'(x, y) \Rightarrow F(x, y)$, so $F'$ is equivalent to $F$, showing that $F$ is $\Delta_1$ over $\mathrm{KP}$. Therefore, working in $\mathrm{KP}$, if $A$ is a set and $F$ is $\Sigma_1$ as well as functional and total, then by $\Delta_1$-Separation and $\Sigma_1$-Collection, $F\restriction_A =_\df \{ \langle x, y \rangle \mid F(x) = y \wedge x \in A \}$ exists as a set.

\begin{thm}[$\Sigma_1$-Recursion]
	Let $G(x, y)$ be a $\Sigma_1$-formula such that 
	\begin{itemize}
		\item $\mathrm{KP} \vdash \forall x . \forall y . \forall y\hspace{1pt}' . ((G(x, y) \wedge G(x, y\hspace{1pt}')) \rightarrow y = y\hspace{1pt}'),$
		\item $\mathrm{KP} \vdash \forall x . \exists y . G(x) = y.$
	\end{itemize}
	Then there is a $\Delta_1$-formula $F$, such that:
	\begin{itemize}
		\item $\mathrm{KP} \vdash \forall x . \forall y . \forall y\hspace{1pt}' . ((F(x, y) \wedge F(x, y\hspace{1pt}')) \rightarrow y = y\hspace{1pt}'),$
		\item $\mathrm{KP} \vdash \forall x . \exists y . F(x) = y,$
		\item $\mathrm{KP} \vdash \forall x . F(x) = G(F\restriction_x).$
	\end{itemize}
\end{thm}

Here is another important consequence of $\Sigma_1$-Recursion:

\begin{prop}\label{rank}
There is a $\Sigma_1$-formula $\rnk$, such that
\begin{itemize}
\item $\mathrm{KP} \vdash \forall x . \forall y . \forall y\hspace{1pt}' . ((\rnk(x, y) \wedge \rnk(x, y\hspace{1pt}')) \rightarrow y = y\hspace{1pt}'),$ 
\item $\mathrm{KP} \vdash \forall x . \exists \rho . \rnk(x) = \rho$,
\item $\mathrm{KP} \vdash \forall x . \mathrm{Ord}(\rnk(x))$,
\item $\mathrm{KP} \vdash \forall x . \rnk(x) = \sup\{\rnk(u)+1 \mid u \in x\}$.
\end{itemize}
\end{prop}

The following two results are proved in \cite{Mat01}, the latter being a direct consequence of the former. (The former is stated in the strong form, that $F$ is $\Delta_1^\mathcal{P}$, using the same trick as in the remark above the $\Sigma_1$-recursion theorem.)

\begin{thm}[$\Sigma_1^\mathcal{P}$-Recursion]
	Let $G(x, y)$ be a $\Sigma_1^\mathcal{P}$-formula such that 
	\begin{itemize}
		\item $\mathrm{KP}^\mathcal{P} \vdash \forall x . \forall y . \forall y\hspace{1pt}' . ((G(x, y) \wedge G(x, y\hspace{1pt}')) \rightarrow y = y\hspace{1pt}'),$
		\item $\mathrm{KP}^\mathcal{P} \vdash \forall x . \exists y . G(x) = y.$
	\end{itemize}
	Then there is a $\Delta_1^\mathcal{P}$-formula $F$, such that:
	\begin{itemize}
		\item $\mathrm{KP}^\mathcal{P} \vdash \forall x . \forall y . \forall y\hspace{1pt}' . ((F(x, y) \wedge F(x, y\hspace{1pt}')) \rightarrow y = y\hspace{1pt}'),$
		\item $\mathrm{KP}^\mathcal{P} \vdash \forall x . \exists y . F(x) = y,$
		\item $\mathrm{KP}^\mathcal{P} \vdash \forall x . F(x) = G(F\restriction_x).$
	\end{itemize}
\end{thm}

\begin{prop}
	There is a $\Sigma_1^\mathcal{P}$-formula $V(x, y)$, such that
	\begin{itemize}
		\item $\mathrm{KP}^\mathcal{P} \vdash \forall x . \forall y . \forall y\hspace{1pt}' . ((V(x, y) \wedge V(x, y\hspace{1pt}')) \rightarrow y = y\hspace{1pt}')$, we write $V_x = y$ for $V(x, y)$,
		\item $\mathrm{KP}^\mathcal{P} \vdash \forall \rho \in \mathrm{Ord} . \exists v . V_\rho = v.$
		\item $\mathrm{KP}^\mathcal{P} \vdash \forall \rho \in \mathrm{Ord} . \forall x . (x \in V_\rho \leftrightarrow \rnk(x) < \rho),$
		\item $\mathrm{KP}^\mathcal{P} \vdash V_0 = \varnothing \wedge \forall \rho \in \mathrm{Ord} . (\mathcal{P}(V_\rho) = V_{\rho + 1}) \wedge \forall \rho \in \mathrm{Ord} . V_\rho = \bigcup_{\xi < \rho} V_{\xi + 1}.$
	\end{itemize}
\end{prop}

The $V$-hierarchy given by the previous Proposition is very useful. For example, it enables the following result.

\begin{prop}
	For each $1 \leq k < \omega$, $\mathrm{KP}^\mathcal{P} + \Sigma_k^\mathcal{P} \textnormal{-Separation} \vdash \mathrm{B}(\Sigma_k^\mathcal{P}) \textnormal{-Foundation}$, where $\mathrm{B}(\Sigma_k^\mathcal{P})$ is the Boolean closure of $\Sigma_k^\mathcal{P}$.
\end{prop}
\begin{proof}
	Recall that $\Sigma_k^\mathcal{P} \textnormal{-Separation}$ implies $\mathrm{B}(\Sigma_k^\mathcal{P}) \textnormal{-Separation}$. 
	Let $\phi(x) \in \mathrm{B}(\Sigma_k^\mathcal{P}[x])$. Suppose there is $a$ such that $\phi(a)$. By $\mathrm{B}(\Sigma_k^\mathcal{P})$-Separation, let 
	$$A = \{ x \in V_{\rnk(a) + 1} \mid \phi(x) \},$$
	and note that $a \in A.$
	By $\Sigma_1$-Separation, let 
	\[R = \{ \xi < \rnk(a) + 1 \mid \exists x \in A . \rnk(x) = \xi \} . \]
	Since $R$ is a non-empty set of ordinals, it has a least element $\rho$. Let $a\hspace{1pt}' \in A$ such that $\rnk(a\hspace{1pt}') = \rho$. Then we have $\forall x \in a . \neg \phi(x)$, as desired.
\end{proof}

Many more facts about the Takahashi hierarchy in the context of $\mathrm{ZFC}$ are established in \cite{Tak72}. It appears like these results also hold in the context of $\mathrm{KP}^\mathcal{P}$ (apart from its Theorem 6, which might require $\mathrm{KP}^\mathcal{P} + \textnormal{Choice}$).

\section{First-order logic and partial satisfaction relations internal to $\mathrm{KP}^\mathcal{P}$}

By the $\Sigma_1$-Recursion Theorem above, it is straightforward to develop the machinery of first order logic within $\mathrm{KP}$. In the meta-theory, let $D$ be a recursive definition of a first order language $\mathcal{L}^*$. The recursive definition $D$ can be employed within $\mathrm{KP}$ to prove the existence (as a set) of the language defined by $D$, which we denote $\mathcal{L}$, in effect introducing a new constant symbol to the object language $\mathcal{L}^0$ of $\mathrm{KP}$. 

Now to clarify matters, let us distinguish between variables, terms, formulae, etc. of $\mathcal{L}^*$ in the meta-theory, and variables, terms, formulae, etc. of $\mathcal{L}$ in the object-theory. From the perspective of the meta-theory, $\mathcal{L}^*$ is a set equipped with appropriate structure that makes it an implementation of a first-order language, and its variables, terms, formulae, etc. are elements found in that structure. On the other hand, the meta-theory views $\mathcal{L}$ as a constant symbol (of the object language $\mathcal{L}^0$ of $\mathrm{KP}$) associated with a bunch of proofs in $\mathrm{KP}$ to the effect that $\mathcal{L}$ represents a first-order language in $\mathrm{KP}$. By a {\em standard} natural number is meant a natural number in the meta-theory. If $k$ is a standard natural number, then $\dot k$ denotes an introduced term for the implementation of that number as a set in the object-theory $\mathrm{KP}$. Similarly, by a {\em standard} variable, term, formula, etc. of $\mathcal{L}$, is meant a variable, term, formula, etc. of $\mathcal{L}^*$. From now on, we shall not mention $\mathcal{L}^*$. Instead, we talk about $\mathcal{L}$ and use the attribute {\em standard} when considering syntactical objects in $\mathcal{L}^*$.

Natural features of $\mathcal{L}$ can be implemented in $\mathrm{KP}$ as subsets of $V_\omega$. In particular, working in the object-theory $\mathrm{KP}$, let us highlight some important features:
\begin{enumerate}
	\item There is an infinite set of distinct variables $\mathrm{Var} = \{x_k \mid k \in \mathbb{N}\}$.
	\item If $x_k$ is a standard variable, then $\dot x_k$ is an introduced term for its representation $x_{\dot k}$.
	\item There are functions mapping function symbols and relation symbols to their respective arities in $\mathbb{N}$.
	\item To each function or relation symbol $S$ corresponds an introduced term $\dot S$.
	\item The set of terms, denoted $\mathrm{Term}_\mathcal{L}$, can be constructed recursively in such a manner that: 
	\begin{enumerate}
		\item $\mathrm{Var} \subseteq \mathrm{Term}$.
		\item For each $k$-ary function symbol $f$ (constants are considered to be $0$-ary functions), there is a function with domain $\mathrm{Term}^k$, sending tuples of terms to terms.
		\item There are functions by means of which terms can conversely be unpacked into immediate function symbol and immediate subterms.
		\item If $t = f\hspace{2pt}(\vec{x})$ is a standard term, then $\dot t$ is introduced to denote the term of the form $\dot f\hspace{2pt}(\dot\vec{x})$.
	\end{enumerate}
	\item For each $k$-ary relation symbol $R$, there is a function with domain $\mathrm{Term}^k_\mathcal{L}$, sending tuples of terms to atomic formulae (its range is denoted $\mathrm{Atom}_\mathcal{L}$).
	\item There are functions by means of which atomic formulae can conversely be unpacked into relation symbol and immediate subterms.
	\item The set of formulae, denoted $\mathrm{Form}_\mathcal{L}$ or simply $\mathcal{L}$, can be constructed recursively in such a manner that: 
	\begin{enumerate}
		\item Each standard formula $\phi$ of $\mathcal{L}$ has a representation as an introduced term $\dot\phi$.
		\item $\mathrm{Atom}_\mathcal{L} \subseteq \mathrm{Form}_\mathcal{L}$.
		\item For each $k$-ary propositional connective $\star$, there is a function $\dot\star : \mathrm{Form}^k_\mathcal{L} \rightarrow \mathrm{Form}_\mathcal{L}$, such that if $\phi$, $\psi$ and $\theta \cong \phi \wedge \psi$ are standard formulae, then $\dot\theta = \dot\phi \dot\wedge \dot\psi$, and similarly for the other connectives.
		\item For each quantifier $\boxminus$, there is a function $\dot\boxminus : \mathrm{Form}_\mathcal{L} \times \mathrm{Var} \rightarrow \mathrm{Form}_\mathcal{L}$, such that if $\phi$ and $\theta \cong \exists x_0 . \phi$ are standard formulae, then $\dot\theta = \dot\exists \dot x_0 \dot\phi$, and similarly for the other quantifier.
		\item There are functions by means of which formulae can be unpacked into immediate connective, or quantifier and bound variable, and immediate subformulae.
		\item For any formula, the occurrences of free and bound variables in it can be distinguished.
	\end{enumerate}
	\item There is a function of substitution from $\mathrm{Form}_\mathcal{L} \times \mathrm{Var} \times \mathrm{Term}_\mathcal{L}$ to $\mathrm{Form}_\mathcal{L}$, which substitutes a particular term for each free occurrence of a particular variable.

In the special case $\mathcal{L} = \mathcal{L}^0$, there are additional features worth highlighting:

 	\item The representations of equality and membership are denoted $\dot=$ and $\dot\in$, respectively.
	\item The respective sets of all $\bar\Sigma_k$-, $\bar\Pi_k$-, $\bar\Sigma_k^\mathcal{P}$- and $\bar\Pi_k^\mathcal{P}$-formulae exist, for all $k \in \mathbb{N}$, and they correspond in the natural way to their counterparts for standard $k$.
	\item For each $k \in \mathbb{N}$, there is a function $\dot\sim : \bar\Pi_k^\mathcal{P} \rightarrow \bar\Sigma_k^\mathcal{P}$, such that for any $\phi \in \bar\Pi_k^\mathcal{P}$, $\dot\sim\phi$ is the result of pushing the $\neg$-symbol in $\dot\neg\phi$ through all the unbounded quantifiers in the front of $\phi$, thus obtaining that $\dot\sim\phi$ is a $\bar \Sigma_k^\mathcal{P}$-formula equivalent to $\dot\neg\phi$ over $\mathrm{KP}$.
\end{enumerate}

Having examined the syntactical side of first-order logic internal to $\mathrm{KP}$, let us now look at the semantical side. 

In $\mathrm{KP}$, the satisfaction relation $\models$, between structures $\mathcal{M}$ and formulae of the language $\mathcal{L}$ of $\mathcal{M}$, can be defined in the usual way by $\Sigma_1$-recursion over the complexity of formulae. This is worked out in detail in Chapter III, Section 1 of \cite{Bar75}, even for the language $\mathcal{L}_{\omega_1, \omega}$, where countable disjunctions and conjunctions are allowed. In particular, $\mathrm{KP}$ proves that for any first-order language $\mathcal{L}$, for any $\phi \in \mathcal{L}$ and for any $\mathcal{L}$-structure $\mathcal{M}$, the compositional theory of satisfaction holds for $\mathcal{M} \models \phi$, that is to say: $\mathrm{KP}$ proves that if $\phi(\vec{x}), \psi(\vec{y}) \in \mathcal{L}$, $\star \in \{\vee, \wedge, \rightarrow \}$, $\boxminus \in \{\exists, \forall\}$, $\vec{m}, \vec{n}, m' \in \mathcal{M}$, and $f$ is a function symbol and $R$ is a relation symbol of $\mathcal{L}$, then
	\[
	\begin{array}{rcl}
	\mathcal{M} \models f\hspace{2pt}(\vec{m}) \dot = m' &\Leftrightarrow& f^\mathcal{M}(\vec{m}) = m' \\
	\mathcal{M} \models R(\vec{m}) &\Leftrightarrow& \vec{m} \in R^\mathcal{M} \\
	\mathcal{M} \models \dot\neg\phi(\vec{m}) &\Leftrightarrow& \neg \mathcal{M} \models \phi(\vec{m}) \\
	\mathcal{M} \models \phi(\vec{m}) \dot\star \psi(\vec{n}) &\Leftrightarrow& \big(\mathcal{M} \models \phi(\vec{m}) \big) \star \big(\mathcal{M} \models \psi(\vec{n}) \big) \\
	\mathcal{M} \models \dot\boxminus x . \phi(x, \vec{m}) &\Leftrightarrow& \boxminus m' \in \mathcal{M} . \mathcal{M} \models \phi(m', \vec{m}) . \\
	\end{array}
	\]

Note that in $\mathrm{KP}$, if $M$ is a set, then the structure $\mathcal{M} = (M, \in\restriction_M)$ can be constructed, where $\in_M = \{\langle x, y \rangle \in M^2 \mid x \in y\}$. Thus, for any standard $\mathcal{L}^0$-formula $\phi$ it makes sense to ask about the relationship between $\phi^M$ and $\mathcal{M} \models \dot\phi$:

\begin{lemma}\label{absoluteness with restricted domain}
	For any formula $\phi \in \mathcal{L}^0$, $\mathrm{KP} \vdash \forall M . \forall \vec{m} \in M . \big( \phi^M(\vec{m}) \leftrightarrow (M, \in\restriction_M) \models \dot\phi(\vec{m}) \big)$.
\end{lemma}
\begin{proof}
	This is proved by induction on the structure of $\phi$. In the atomic cases, $\phi^M$ equals $\phi$, and $(M, \in\restriction_M) \models \dot\phi$ is equivalent to $\phi$. In the inductive cases of the propositional connectives, the result follows by inspection from the compositionality of satisfaction explained above.
	
	For the existential quantifier case, suppose that $\phi(\vec{y})$ is $\exists x . \psi(x, \vec{y})$, and assume inductively that the result holds for $\psi(x, \vec{y})$. Note that $\phi^M$ is the formula $\exists x \in M . \psi^M(x, \vec{y})$, and $\mathrm{KP} \vdash \dot\phi(\vec{y}) = \dot\exists x . \dot\psi(x, \vec{y})$. On the other hand, by compositionality, $\mathrm{KP}$ proves
	\[
	\forall \vec{m} \in M . \big( (M, \in\restriction_M) \models \dot\exists x . \dot\psi(x, \vec{m}) \leftrightarrow \exists x \in M . (M, \in\restriction_M) \models \dot\psi(x, \vec{m}) \big).
	\]
	So by the induction hypothesis, $\mathrm{KP}$ proves
	\[
	\forall \vec{m} \in M . \big( (M, \in\restriction_M) \models \dot\exists x . \dot\psi(x, \vec{m}) \leftrightarrow \exists x \in M . \psi^M(x, \vec{m}) \big),
	\]
	as desired.
\end{proof}

We will now use the fact that the satisfaction relation $\mathcal{M} \models \phi$ is $\Sigma_1$ to show that appropriate partial satisfaction relations  are available for the Takahashi hierarchy in $\mathrm{KP}^\mathcal{P}$.

A set $a$ is {\em supertransitive} if it is transitive and $\forall x \in a . \forall y \subseteq x . y \in a$. Working in $\mathrm{KP}^\mathcal{P}$, note that supertransitivity is $\Delta_0^\mathcal
P$, and that the {\em supertransitive closure} of $a$, defined as 
\[
\mathrm{STC}(a) =_\df \bigcup \{ \mathcal{P}(x) \mid x \in \mathrm{TC}(a) \},
\] 
is the $\subseteq$-least supertransitive set such that $a \subseteq \mathrm{STC}(a)$. To see that $\mathrm{KP}^\mathcal{P}$ proves the existence of $\mathrm{STC}(a)$, recall that $\mathrm{KP}$ proves the existence of $\mathrm{TC}(a)$ and observe that the operation $x \mapsto \mathcal{P}(x)$ may be defined by a $\Delta_0^\mathcal{P}$-formula:
\[
\mathcal{P}(x) = y \Leftrightarrow (\forall z \subseteq x . z \in y) \wedge (\forall z \in y . z \subseteq x).
\]
So by $\Sigma_1^\mathcal{P}$-Replacement and the union axiom, $\mathrm{STC}(a)$ exists.

To make the definition of partial satisfaction relations more concise, we temporarily introduce the notation $\bar\Sigma_{k+1}^\mathcal{P} / \bar\Pi_{k}^\mathcal{P}$ for the set of pairs $\langle \phi, \psi \rangle$, such that $\phi \in \bar\Sigma_{k+1}^\mathcal{P}$ and $\psi \in \bar\Pi_{k}^\mathcal{P}$ and there is $p \in \mathbb{N}$ and a $p$-tuple $\vec{v} \in \mathrm{Var}$, such that $\phi = \dot\exists \pi_1(\vec{v}) . \dot\exists \pi_2(\vec{v}) . \dots . \dot\exists \pi_p(\vec{v}) . \psi$ (by $\Sigma_1$-recursion, $\mathrm{KP}$ proves the existence of this set). The formulae for partial satisfaction are defined as follows, by recursion over $k < \omega$: 
\begin{align*}
\mathrm{Sat}_{\Delta_0^\mathcal{P}}(\phi, \vec{m}) &\equiv_\df \phi \in \bar\Delta_0^\mathcal{P} \wedge \exists M . \big( (\text{``$M$ is supertransitive''} \wedge \vec{m} \in M) \wedge \\
&(M, \in\restriction_M) \models \phi(\vec{m}) \big) \\
\mathrm{Sat}'_{\Delta_0^\mathcal{P}}(\phi, \vec{m}) &\equiv_\df \phi \in \bar\Delta_0^\mathcal{P} \wedge \forall M . \big((\text{``$M$ is supertransitive''} \wedge \vec{m} \in M) \rightarrow \\
&(M, \in\restriction_M) \models \phi(\vec{m}) \big) \\
\mathrm{Sat}_{\Sigma_{k+1}^\mathcal{P}}(\phi, \vec{m}) &\equiv_\df \exists \psi \in \bar\Pi_{k}^\mathcal{P} . \exists \vec{n} . \big( \langle \phi, \psi \rangle \in \bar\Sigma_{k+1}^\mathcal{P} / \bar\Pi_{k}^\mathcal{P} \wedge \mathrm{Sat}_{\Pi_k^\mathcal{P}}(\psi, \vec{n}, \vec{m}) \big) \\
\mathrm{Sat}_{\Pi_{k+1}^\mathcal{P}}(\phi, {m}) &\equiv_\df \phi \in \bar\Pi_{k+1}^\mathcal{P} \wedge \neg \mathrm{Sat}_{\Sigma_{k+1}^\mathcal{P}}(\dot\sim\phi, {m}).
\end{align*}

\begin{prop}\label{complexity of satisfaction predicates}
	Let $k < \omega$. $\mathrm{Sat}_{\Delta_0^\mathcal{P}}$ is $\Delta_1$, $\mathrm{Sat}_{\Sigma_k^\mathcal{P}}$ is $\Sigma_k^\mathcal{P}$ and $\mathrm{Sat}_{\Pi_k^\mathcal{P}}$ is $\Pi_k^\mathcal{P}$ over $\mathrm{KP}^\mathcal{P}$. In particular, 
	\[
	\mathrm{KP}^\mathcal{P} \vdash \forall \phi . \forall {m} . (\mathrm{Sat}_{\Delta_0^\mathcal{P}}(\phi, {m}) \leftrightarrow \mathrm{Sat}'_{\Delta_0^\mathcal{P}}(\phi, {m})).
	\]
\end{prop}
\begin{proof}
	Since supertransitivity is $\Delta_0^\mathcal{P}$ and the satisfaction relation $\models$ is $\Delta_1$, we have that $\mathrm{Sat}_{\Delta_0^\mathcal{P}}$ is $\Sigma_1^\mathcal{P}$ and $\mathrm{Sat}'_{\Delta_0^\mathcal{P}}$ is $\Pi_1^\mathcal{P}$ over $\mathrm{KP}^\mathcal{P}$. Moreover, by definition of $\mathrm{Sat}_{\Sigma_k^\mathcal{P}}$ and $\mathrm{Sat}_{\Pi_k^\mathcal{P}}$, the result follows by induction on $k < \omega$ once we have established that
	\[
	\mathrm{KP}^\mathcal{P} \vdash \forall \phi . \forall {m} . (\mathrm{Sat}_{\Delta_0^\mathcal{P}}(\phi, {m}) \leftrightarrow \mathrm{Sat}'_{\Delta_0^\mathcal{P}}(\phi, {m})).
	\]
	
	We work in $\mathrm{KP}^\mathcal{P}$. Let $\phi \in \Delta_0^\mathcal{P}$ and let ${m}$ be a tuple of sets. As seen above, there is a $\subseteq$-least supertransitive set $\mathrm{STC}({m})$ containing ${m}$. Let $M$ be any supertransitive set containing ${m}$. It suffices to show that
	\[
	(M, \in\restriction_M)) \models \phi({m}) \Leftrightarrow (\mathrm{STC}({m}), \in\restriction_{\mathrm{STC}({m})})) \models \phi({m}),
	\]
	and we do so by induction on the complexity of $\phi$. For the atomic cases, this is immediate; and for the inductive cases of the propositional connectives, it follows from that these connectives commute with $\models$. Suppose that $\phi({x}) \equiv \exists x\hspace{1pt}' \in x . \psi(x\hspace{1pt}', {x})$. Then by induction hypothesis and transitivity,
	\begin{align*}
	\text{ }& (M, \in\restriction_M)) \models \phi({m}) \\
	\Leftrightarrow \text{ }& \exists m' \in m . (M, \in\restriction_M)) \models \psi(m', m) \\
	\Leftrightarrow \text{ }& \exists m' \in m . (\mathrm{STC}({m}), \in\restriction_{\mathrm{STC}({m})})) \models \psi(m', m) \\
	\Leftrightarrow \text{ }& (\mathrm{STC}({m}), \in\restriction_{\mathrm{STC}({m})})) \models \phi({m}).
	\end{align*}
	Suppose that $\phi({x}) \equiv \exists x\hspace{1pt}' \subseteq x . \psi(x\hspace{1pt}', {x})$.
	Then, similarly as above, we have by induction hypothesis and supertransitivity that
	\begin{align*}
	\text{ }& (M, \in\restriction_M)) \models \phi({m}) \\
	\Leftrightarrow \text{ }& \exists m' \subseteq m . (M, \in\restriction_M)) \models \psi(m', m) \\
	\Leftrightarrow \text{ }& \exists m' \subseteq m . (\mathrm{STC}({m}), \in\restriction_{\mathrm{STC}({m})})) \models \psi(m', m) \\
	\Leftrightarrow \text{ }& (\mathrm{STC}({m}), \in\restriction_{\mathrm{STC}({m})})) \models \phi({m}),
	\end{align*}
	as desired.
\end{proof}

\begin{thm}[Partial satisfaction relations]\label{partial satisfaction classes}
	For each $k < \omega$, each $\sigma \in \bar\Sigma_k^\mathcal{P}$ and each $\pi \in \bar\Pi_k^\mathcal{P}$,
	\begin{align*}
	\mathrm{KP}^\mathcal{P} &\vdash \sigma(\vec{x}) \leftrightarrow \mathrm{Sat}_{\Sigma_k^\mathcal{P}}(\dot\sigma, \vec{x}) \\
	\mathrm{KP}^\mathcal{P} &\vdash \pi(\vec{x}) \leftrightarrow \mathrm{Sat}_{\Pi_k^\mathcal{P}}(\dot\pi, \vec{x}).
	\end{align*}
\end{thm}
\begin{proof}[Proof-sketch]
	This theorem is essentially a consequence of the properties of the satisfaction relation $\models$ between structures and formulae. This is seen by combining Lemma \ref{absoluteness with restricted domain} with the definitions above of the formulae $\mathrm{Sat}_{\Delta_0^\mathcal{P}}$, $\mathrm{Sat}_{\Sigma_k^\mathcal{P}}$ and $\mathrm{Sat}_{\Pi_k^\mathcal{P}}$ in terms of the $\Delta_1$-relation $\models$.
	
	The proof is an induction on $k$. We start with the base case $k = 0$. We work in $\mathrm{KP}$: Let $\delta \in \bar\Delta_0^\mathcal{P}$ and let $\vec{m}$ be a tuple whose length matches the number of free variables of $\delta$. Let $M = \mathrm{STC}(\vec{m})$. It follows from supertransitivity that the range of any bounded quantifier in $\delta(\vec{m})$, as a set, is an element of $M$. Therefore, we have $\delta(\vec{m}) \leftrightarrow \delta^M(\vec{m})$. Now it follows from Lemma \ref{absoluteness with restricted domain} and the definition of $\mathrm{Sat}_{\Delta_0^\mathcal{P}}$ that $\delta(\vec{x}) \leftrightarrow \mathrm{Sat}_{\Delta_0^\mathcal{P}}(\dot\delta, \vec{x})$.
	
	For the inductive step, we concentrate on verifying the case of existential quantification. We work in $\mathrm{KP}$: Let $\sigma(\vec{y}) \in \bar\Sigma_{k+1}^\mathcal{P}[\vec{y}]$ be of the form $\exists \vec{x} . \pi(\vec{x}, \vec{y})$, where $\pi(\vec{x}, \vec{y}) \in \bar\Pi_k^\mathcal{P}[\vec{x}, \vec{y}]$. Let $\vec{m}$ be an arbitrary tuple of the same length as $\vec{y}$. First by definition of $\mathrm{Sat}_{\Sigma_{k+1}^\mathcal{P}}$, then by induction hypothesis, we have
	\[
	\begin{array}{rl}
	& \mathrm{Sat}_{\Sigma_{k+1}^\mathcal{P}}(\dot\exists \dot{\vec{x}} . \dot\pi, \vec{m}) \\ 
	\Leftrightarrow & \exists \vec{x} .  \mathrm{Sat}_{\Pi_k^\mathcal{P}}(\dot\pi, \vec{x}, \vec{m})  \\
	\Leftrightarrow & \exists \vec{x} . \pi(\vec{x}, \vec{m}),  \\
	\end{array}
	\]
	as desired.
\end{proof}

{\em Remark.} There is also a more general result to the effect that the partial satisfaction relations satisfy a compositional theory of satisfaction. 

The $\mathrm{Sat}$-relations have been defined so that they apply to formulae in the sets $\bar\Sigma_k^\mathcal{P}$ and $\bar\Pi_k^\mathcal{P}$, where all the unbounded quantifiers are in front. If we wish to apply them to an arbitrary formula $\phi$, we must first replace $\phi$ by an equivalent formula of such a form. But as this is a rather tedious step, we will usually consider that step to be done implicitly. We will only need to do so for $\Sigma_1^\mathcal{P}$-formulae. For these implicit steps, we rely on the following lemma, wherein $\Sigma^\mathcal{P}$ is defined as the least superset of $\Delta_0^\mathcal{P}$ closed under conjunction, disjunction, bounded quantifiers, $\mathcal{P}$-bounded quantifiers and existential quantification. 

\begin{lemma}
If $\phi$ is $\Sigma^\mathcal{P}$, then $\phi$ is $\Sigma_1^\mathcal{P}$, i.e. there is a $\bar\Sigma_1^\mathcal{P}$-formula $\phi\hspace{1pt}'$, such that $\mathrm{KP}^\mathcal{P} + \Sigma_1^\mathcal{P}\textnormal{-Separation} \vdash \phi \leftrightarrow\phi\hspace{1pt}'$.
\end{lemma}

The proof is omitted. It follows the corresponding proof for the L\'evy hierarchy given in Ch. 1 of \cite{Bar75}.

\section{Zermelo-Fraenkel set theory and G\"odel-Bernays class theory}

\begin{ax}[{\em Zermelo-Fraenkel set theory}, $\mathrm{ZF}$] 
	$\mathrm{ZF}$ is the $\mathcal{L}^0$-theory given by Extensionality, Pair, Union, Powerset, Infinity, Separation, Replacement, and Set Foundation.
\end{ax}
$\mathrm{ZFC}$ is $\mathrm{ZF} + \textnormal{Choice}$.

If $\mathcal{L}$ is an expansion of $\mathcal{L}^0$ with more symbols, then $\mathrm{ZF}(\mathcal{L})$ denotes the theory 
$$\mathrm{ZF} + \mathcal{L} \textnormal{-Separation} + \mathcal{L} \textnormal{-Replacement},$$ 
by which is meant that the schemata of Separation and Replacement are extended to all formulae in $\mathcal{L}$. $\mathrm{ZFC}(\mathcal{L})$ is defined analogously.

The following theorem schema of $\mathrm{ZF}$ will be useful for us.

\begin{thm}[Reflection]\label{reflection thm}
	For any formula $\phi(\vec{x}) \in \mathcal{L}^0$,
	$$\mathrm{ZF} \vdash \forall \alpha_0 \in \mathrm{Ord} . \exists \alpha \in \mathrm{Ord} . (\alpha_0 < \alpha \wedge \forall \vec{x} \in V_{\alpha} . (\phi(\vec{x}) \leftrightarrow \phi^{V_{\alpha}}(\vec{x}))).$$
\end{thm}

The following class theory is closely associated with $\mathrm{ZF}$.

\begin{ax}[{\em G\"odel-Bernays set theory}, $\mathrm{GB}$] $\mathrm{GB}$ is an $\mathcal{L}^1$-theory. Recall that in $\mathcal{L}^1$ we have a sort $\mathsf{Set}$ (over which lowercase variables range) and a sort $\mathsf{Class}$ (over which uppercase variables range); moreover $\mathsf{Set}$ is a subsort of $\mathsf{Class}$. The axioms presented in Definition \ref{axioms of set theory}, were all given with lowercase variables, so in the present context they are axioms on the sort $\mathsf{Set}$. $\mathrm{GB}$ may be given by these axioms and axiom schemata:
	\[
	\begin{array}{ll}
	\textnormal{Class Extensionality} & \forall X . \forall Y . \big((\forall u . ( u \in X \leftrightarrow u \in Y)) \rightarrow X = Y\big) \\
	\textnormal{Pair} & \\
	\textnormal{Union} & \\
	\textnormal{Powerset} & \\
	\textnormal{Infinity} & \\
%	\textnormal{Separation} & \\
%	\textnormal{Replacement} & \\
	\textnormal{Extended Separation} & \forall x . \exists y . \forall u . (u \in y \leftrightarrow (u \in x \wedge \phi(u))) \\
	\textnormal{Class Comprehension} & \exists Y . \forall u . (u \in Y \leftrightarrow \phi(u)) \\
	\textnormal{Class Replacement} & \forall F . \big( (\exists x . \exists Y . F : x \rightarrow Y) \rightarrow \\
	&\exists y . \forall v . (v \in y \leftrightarrow \exists u \in x . F(u) = v) \big) \\
	\textnormal{Class Foundation} & \forall X . (X \neq \varnothing \rightarrow \exists u \in X . u \cap X = \varnothing) \\	
	\end{array}
	\]
In the axiom schemata of $\textnormal{Extended Separation}$ and $\textnormal{Class Comprehension}$, $\phi$ ranges over $\mathcal{L}^1$-formulae in which all variables of sort $\mathsf{Class}$ are free, and in which the variables $y$ and $Y$ do not appear. $V$ denotes the class $\{x \mid \top \}$. $\mathrm{GBC}$ is $\mathrm{GB}$ plus this axiom:
\[
\begin{array}{ll}
	\textnormal{Global Choice} & \exists F . ( (F : V \setminus \{\varnothing\} \rightarrow V) \wedge \forall x \in V \setminus \{\varnothing\} . F(x) \in x ) \\
\end{array}
\]
It is well known that $\mathrm{GBC}$ is conservative over $\mathrm{ZFC}$.

We will also consider this axiom ``$\mathrm{Ord}$ is weakly compact'', in the context of $\mathrm{GBC}$. It is defined as ``Every binary tree of height $\mathrm{Ord}$ has a branch.'' The new notions used in the definiens are now to be defined. Let $\alpha$ be an ordinal. A {\em binary tree} is a (possibly class) structure $\mathcal{T}$ with a binary relation $<_\mathcal{T}$, such that:
\begin{enumerate}[(i)]
	\item Every element of $\mathcal{T}$ (called a {\em node}) is a function from an ordinal to $2$;
	\item For every $f \in \mathcal{T}$ and every ordinal $\xi < \dom(f\hspace{2pt})$, we have $f\restriction_\xi \in \mathcal{T}$;
	\item For all $f, g \in \mathcal{T}$, 
	\[
	f <_\mathcal{T} g \Leftrightarrow \dom(f\hspace{2pt}) < \dom(g) \wedge g\restriction_{\dom(f\hspace{2pt})} = f;
	\]
\end{enumerate}
Suppose that $\mathcal{T}$ is a binary tree. The {\em height} of $\mathcal{T}$, denoted $\mathrm{height}(\mathcal{T})$, is $\{ \xi \mid \exists f \in \mathcal{T} . \dom(f\hspace{2pt}) = \xi \}$ (which is either an ordinal or the class $\mathrm{Ord}$). A {\em branch} in $\mathcal{T}$ is a (possibly class) function $F : \mathrm{Ord} \rightarrow \mathcal{T}$, such that for all ordinals $\xi \in \mathrm{height}(\mathcal{T})$, $\dom(F(\alpha)) = \alpha$, and for all ordinals $\alpha < \beta \in \mathrm{height}(\mathcal{T})$, $F(\alpha) <_\mathcal{T} F(\beta)$. Moreover, for each ordinal $\alpha \in \mathrm{height}(\mathcal{T})$, we define 
\[
\mathcal{T}_\alpha =_\df \mathcal{T}\restriction_{\{ f \in \mathcal{T} \mid \dom(f\hspace{2pt}) < \alpha \}}.
\]
\end{ax}

In \cite{Ena04} it is shown that the $\mathcal{L}^0$-consequences of $\mathrm{GBC} + \textnormal{``$\mathrm{Ord}$ is weakly compact''}$ are the same as for $\mathrm{ZFC} + \Phi$, where 
$$\Phi = \{ \exists \kappa . \textnormal{``$\kappa$ is $n$-Mahlo and $V_\kappa \prec_{\Sigma_n} V$''} \mid n \in \mathbb{N}\}.$$ 
In particular, they are equiconsistent. Let us therefore define $n$-Mahlo and explain why $V_\kappa \prec_{\Sigma_n} V$ can be expressed as a sentence. 

$V_\kappa \prec_{\Sigma_n} V$ is expressed by a sentence saying that for all $\Sigma_n$-formulae $\phi(\vec{x})$ of set theory and for all $\vec{a} \in V_\kappa$ matching the length of $\vec{x}$, we have $((V_\kappa, \in\restriction_{V_\kappa}) \models \phi(\vec{a})) \leftrightarrow \mathrm{Sat}_{\Sigma_n}(\phi, \vec{a})$. Here we utilize the partial satisfaction relations $\mathrm{Sat}_{\Sigma_n}$, introduced to set theory in \cite{Lev65}.

Let $\kappa$ be a cardinal. $\kappa$ is {\em regular} if there is no unbounded function from a proper initial segment of $\kappa$ to $\kappa$. $\kappa$ is inaccessible if it is regular and for all cardinals $\lambda < \kappa$, we have $2^\lambda < \kappa$. If $\kappa$ is a regular cardinal, then we define that $C \subseteq \kappa$ is a {\em club} of $\kappa$ if it is unbounded in $\kappa$ and closed under suprema; and we define that $S \subseteq \kappa$ is {\em stationary} in $\kappa$ if it has non-empty intersection with every club of $\kappa$. $\kappa$ is {\em Mahlo} if it is inaccessible and the set $\{ \lambda < \kappa \mid \textnormal{``$\lambda$ is inaccessible''}\}$ is stationary in $\kappa$. $\kappa$ is $0${\em -Mahlo} if it is inaccessible. Recursively, for ordinals $\alpha > 0$, we define that $\kappa$ is $\alpha${\em -Mahlo} if for each $\beta < \alpha$ the set $\{ \lambda < \kappa \mid \textnormal{``$\lambda$ is $\beta$-Mahlo''}\}$ is stationary in $\kappa$. For example, $\kappa$ is Mahlo iff it is $1$-Mahlo.

In contrast to the result above, Enayat has communicated to the author that there are countable models of $\mathrm{ZFC} + \Phi$ that do not expand to models of $\mathrm{GBC} + \textnormal{``$\mathrm{Ord}$ is weakly compact''}$. In particular, such is the fate of {\em Paris models}, i.e. models of $\mathrm{ZFC}$ each of whose ordinals is definable in the model. An outline of a proof: Let $\mathcal{M}$ be a Paris model. There is no ordinal $\alpha$ in $\mathcal{M}$, such that $\mathcal{M}_\alpha \prec \mathcal{M}$, where $\mathcal{M}_\alpha =_\df ((V_\alpha^\mathcal{M})_\mathcal{M}, \in^\mathcal{M}\restriction_{(V_\alpha^\mathcal{M})_\mathcal{M}})$, because that would entail that a full satisfaction relation is definable in $\mathcal{M}$ contradicting Tarski\hspace{1pt}'s well-known theorem on the undefinability of truth. Suppose that $\mathcal{M}$ expands to a model $(\mathcal{M}, \mathcal{A}) \models \mathrm{GBC} + \textnormal{``$\mathrm{Ord}$ is weakly compact''}$. Then by the proof of Theorem 4.5(i) in \cite{Ena01}, $\mathcal{M}$ has a full satisfaction relation $\mathrm{Sat} \in \mathcal{A}$ (see the definition preceding Lemma \ref{rec sat reflection}) below. But then there is, by Lemma \ref{rec sat reflection}, unboundedly many ordinals $\alpha$ in $\mathcal{M}$ such that $\mathcal{M}_\alpha \prec \mathcal{M}$.

Moreover, it is shown in \cite{EH17} that for every model $\mathcal{M} \models \mathrm{ZFC}$ and the collection $\mathcal{A}$ of definable subsets of $\mathcal{M}$, we have $(\mathcal{M}, \mathcal{A}) \models \mathrm{GBC} + \neg \textnormal{``$\mathrm{Ord}$ is weakly compact''}$.

\section{Non-standard models of set theory}\label{Models of set theory}

If $\mathcal{M}$ is an $\mathcal{L}^0$-structure, and $a$ is a set or class in $\mathcal{M}$, then $a_\mathcal{M}$ denotes $\{x \in \mathcal{M} \mid x \in^\mathcal{M} a\}$. If $E = \in^\mathcal{M}$, then the notation $a_E$ is also used for $a_\mathcal{M}$. If $f \in \mathcal{M}$ codes a function internal to $\mathcal{M}$, then $f_\mathcal{M}$ also denotes the externalization of this function: $\forall x, y \in \mathcal{M} . (f_\mathcal{M}(x) = y \leftrightarrow \mathcal{M} \models f\hspace{2pt}(x) = y)$. Moreover, if $a \in \mathcal{M}$ codes a structure internal to $\mathcal{M}$, then $a_\mathcal{M}$ also denotes the externalization of this structure; in particular, if $R$ codes a relation in $\mathcal{M}$, then $\forall x, y \in \mathcal{M} . ( x R_\mathcal{M} y \leftrightarrow \mathcal{M} \models \langle x, y \rangle \in R)$. For example, recall that $\mathbb{N}^\mathcal{M}$ denotes the interpretation of $\mathbb{N}$ in $\mathcal{M}$ (assuming that $\mathcal{M}$ satisfies that the standard model of arithmetic exists); then $\mathbb{N}^\mathcal{M}_\mathcal{M}$ denotes the externalization of this model (which might be non-standard). 

Let $\mathcal{M} \models \mathrm{KP}$. Then by Proposition \ref{rank}, every element of $\mathcal{M}$ has a rank; so for any $\alpha \in \mathrm{Ord}^\mathcal{M}$, we can define 
\[
\mathcal{M}_\alpha =_\df \{m \in \mathcal{M} \mid \mathcal{M} \models \rnk(m) < \alpha\}.
\]
We say that an embedding $i : \mathcal{S} \rightarrow \mathcal{M}$ of an $\mathcal{L}^0$-structure $\mathcal{S}$ into $\mathcal{M}$ is {\em bounded} (by $\alpha \in \mathrm{Ord}^\mathcal{M}$) if $i(\mathcal{S}) \subseteq \mathcal{M}_\alpha$.

\begin{dfn}
Let $\mathcal{M} = (M, E)$ be a model in $\mathcal{L}^0$. It is {\em standard} if $E$ is well-founded. Assume that $\mathcal{M}$ is a model of $\mathrm{KP}$. Then the usual rank-function $\rnk : M \rightarrow \mathrm{Ord}^\mathcal{M}$ is definable in $\mathcal{M}$. Therefore $\mathcal{M}$ is non-standard iff $E \restriction_{\mathrm{Ord}^\mathcal{M}}$ is not well-founded. $m \in \mathcal{M}$ is {\em standard in} $\mathcal{M}$ if $\mathcal{M}\restriction_{m_E}$ is standard. 
\begin{itemize}
\item The {\em ordinal standard part} of $\mathcal{M}$, denoted $\mathrm{OSP}(\mathcal{M})$, is defined:
\[
\mathrm{OSP}(\mathcal{M}) =_\df \{\alpha \in \mathcal{M} \mid \text{``}\alpha \text{ is a standard ordinal of $\mathcal{M}$''}\}. 
\]
\item The {\em well-founded part} of $\mathcal{M}$, denoted $\mathrm{WFP}(\mathcal{M})$, is the substructure of $\mathcal{M}$ on the elements of standard rank:
\[
\mathrm{WFP}(\mathcal{M}) =_\df \mathcal{M}\restriction_{\{x \in \mathcal{M} \mid \text{``$x$ is standard in $\mathcal{M}$''}\}}.
\]
\item A set of the form $c_E \cap A$, where $c \in M$ and $A \subseteq M$, is said to be a subset of $A$ {\em coded in} $\mathcal{M}$. This notion is extended in the natural way to arbitrary injections into $M$. We define:
\[
\mathrm{Cod}_A(\mathcal{M}) =_\df \{c_E \cap A \mid c \in M\}.
\]
\item The {\em standard system of $\mathcal{M}$ over $A \subseteq \mathcal{M}$}, denoted $\mathrm{SSy}_A(\mathcal{M})$, is obtained by expanding $\mathcal{M}\restriction_A$ to an $\mathcal{L}^1$-structure, adding $\mathrm{Cod}_A(\mathcal{M})$ as classes:
\begin{align*}
\mathrm{SSy}_A(\mathcal{M}) &=_\df (\mathcal{M}\restriction_A, \mathrm{Cod}_A(\mathcal{M})), \\
x \in^{\mathrm{SSy}_A(\mathcal{M})} C &\Leftrightarrow_\df x \in^\mathcal{M} c, 
\end{align*}
for any $x \in A$, $c \in \mathcal{M}$ and $C \in \mathrm{Cod}_A(\mathcal{M})$, such that $C = c_E \cap A$.

Moreover, we define 
\[
\mathrm{SSy}(\mathcal{M}) =_\df \mathrm{SSy}_{\mathrm{WFP}(\mathcal{M})}(\mathcal{M}).
\]
\end{itemize}
\end{dfn}

Let $i : (\mathcal{M}, \mathcal{A}) \rightarrow (\mathcal{N, \mathcal{B}})$ be an embedding between $\mathcal{L}^1$-structures. Then $\forall x \in \mathcal{M} . \forall X \in \mathcal{A} . x \in^{(\mathcal{M}, \mathcal{A})} X \Leftrightarrow i(x) \in^{(\mathcal{N}, \mathcal{B})} i(X_\mathcal{M})$. Thus we can relate the two objects $i(X)$ and $i(X_\mathcal{M})$ as follows. Note that $i(X)$ is $i$ applied to the class $X$ as a member of $\mathcal{A}$, while $i(X_\mathcal{M})$ is the set $\{i(x) \mid x \in X_\mathcal{M}\}$ (i.e. the pointwise application of $i$ to $X_\mathcal{M}$ as a subset of $\mathcal{M}$). Indeed, we have $i(X_\mathcal{M}) = (i(X))_\mathcal{N} \cap i(\mathcal{M})$, for all $X \in \mathcal{A}$. 

Let $i : \mathcal{M} \rightarrow \mathcal{N}$ be an embedding that extends to an embedding of the $\mathcal{L}^1$-structures under consideration. Then $(\mathcal{M}, \mathcal{A}) \leq_i (\mathcal{N}, \mathcal{B})$. By the above, we have
\[
(\mathcal{M}, \mathcal{A}) \leq_i (\mathcal{N}, \mathcal{B}) \Leftrightarrow \forall X \in \mathcal{A} . \exists Y \in \mathcal{B} . i(X_\mathcal{M}) = Y_\mathcal{N} \cap i(\mathcal{M}).
\]
If $i : \mathcal{M} \rightarrow \mathcal{N}$ is an isomorphism that extends to an isomorphism of the $\mathcal{L}^1$-structures under consideration, then $(\mathcal{M}, \mathcal{A}) \cong_i (\mathcal{N}, \mathcal{B})$. By the above, we have 
\[
(\mathcal{M}, \mathcal{A}) \cong_i (\mathcal{N}, \mathcal{B}) \Leftrightarrow \forall X \in \mathcal{A} . i(X_\mathcal{M}) \in \mathcal{B} \wedge \forall Y \in \mathcal{B} . i^{-1}(Y_\mathcal{N}) \in \mathcal{A}.
\]
Note that if $(\mathcal{M}, \mathcal{A}) \leq_i (\mathcal{N}, \mathcal{B})$ and $(\mathcal{M}, \mathcal{A}) \geq_i (\mathcal{N}, \mathcal{B})$, then $(\mathcal{M}, \mathcal{A}) \cong_i (\mathcal{N}, \mathcal{B})$.

\begin{dfn}\label{special embeddings}
Let $i : \mathcal{M} \rightarrow \mathcal{N}$ be an embedding of $\mathcal{L}^0$-structures, where $\mathcal{N} \models \mathrm{KP}$.
\begin{itemize}
\item $i$ is {\em cofinal}, if 
\[
\forall  n \in \mathcal{N} . \exists m \in \mathcal{M} . n \in^\mathcal{N} i(m).
\]
We write $\mathcal{M} \leq^\mathrm{cf} \mathcal{N}$ if there is such an embedding. Moreover, $\llbracket \mathcal{M} \leq^\mathrm{cf} \mathcal{N} \rrbracket$ denotes the set of all such embeddings.
\item $i$ is {\em initial}, if
\[
\forall m \in \mathcal{M} . \forall n \in \mathcal{N} . (n \in^\mathcal{N} i(m) \rightarrow n \in i(\mathcal{M})).
\]
This is equivalent to:
\[
\forall m \in \mathcal{M} . i(m_\mathcal{M}) = i(m)_\mathcal{N}.
\]
We write $\mathcal{M} \leq^\mathrm{initial} \mathcal{N}$ if there is such an embedding. Moreover, $\llbracket \mathcal{M} \leq^\mathrm{initial} \mathcal{N} \rrbracket$ denotes the set of all such embeddings.
%\item $i$ is $H${\em -initial} (or {\em hereditarily inital}), if 
%\[
%\forall m \in \mathcal{M} . \forall n \in \mathcal{N} . (n \in^\mathcal{N} \mathrm
%{TC}^\mathcal{N}(m) \rightarrow n \in i(\mathcal{M})).
%\]
%We write $\mathcal{M} \leq^H \mathcal{N}$ if there is such an embedding. Moreover, $\llbracket \mathcal{M} \leq^H \mathcal{N} \rrbracket$ denotes the set of all such embeddings.
\item $i$ is {\em $\mathcal{P}$-initial} (or {\em power-initial}), if it is initial and {\em powerset preserving} in the sense:
\[
\forall m \in \mathcal{M} . \forall n \in \mathcal{N} . (n \subseteq^\mathcal{N} m \rightarrow n \in i(\mathcal{M})).
\]
We write $\mathcal{M} \leq^\mathcal{P} \mathcal{N}$ if there is such an embedding. Moreover, $\llbracket \mathcal{M} \leq^\mathcal{P} \mathcal{N} \rrbracket$ denotes the set of all such embeddings.
\item $i$ is {\em rank-initial}, if 
\[
\forall m \in \mathcal{M} . \forall n \in \mathcal{N} . (\rnk^\mathcal{N}(n) \leq^\mathcal{N} \rnk^\mathcal{N}(i(m)) \rightarrow n \in i(\mathcal{M})).
\]
Note that $\forall x \in \mathcal{M} . i(\rnk^\mathcal{M}(x)) = \rnk^\mathcal{N}(i(x))$. So the above is equivalent to 
\[
\forall \mu \in \mathrm{Ord}^\mathcal{M} . \forall n \in \mathcal{N} . (\rnk^\mathcal{N}(n) \leq^\mathcal{N} i(\mu) \rightarrow n \in i(\mathcal{M})).
\]
We write $\mathcal{M} \leq^\rnk \mathcal{N}$ if there is such an embedding. Moreover, $\llbracket \mathcal{M} \leq^\rnk \mathcal{N} \rrbracket$ denotes the set of all such embeddings.
\item $i$ is {\em topless}, if it is bounded and
\[
\forall \beta \in \mathrm{Ord}^\mathcal{N} \setminus i(\mathcal{M}) . \exists \beta\hspace{1pt}' \in \mathrm{Ord}^\mathcal{N} \setminus i(\mathcal{M}) . \beta\hspace{1pt}' <^\mathcal{N} \beta. 
\]
We write $\mathcal{M} \leq^\mathrm{topless} \mathcal{N}$ if there is such an embedding. Moreover, $\llbracket \mathcal{M} \leq^\mathrm{topless} \mathcal{N} \rrbracket$ denotes the set of all such embeddings.
\item $i$ is {\em strongly topless}, if it is bounded and for each $f \in \mathcal{N}$ with $\alpha, \beta \in \mathrm{Ord}^\mathcal{N}$ satisfying
\[
(\mathcal{N} \models f : \alpha \rightarrow \beta) \wedge \alpha_\mathcal{N} \supseteq i(\mathrm{Ord}^\mathcal{M}),
\]
there is $\nu \in \mathrm{Ord}^\mathcal{N} \setminus i(\mathcal{M})$ such that for all $\xi \in i(\mathrm{Ord}^\mathcal{M})$,
\[
f_\mathcal{N}(\xi) \not\in i(\mathcal{M}) \Leftrightarrow \mathcal{N} \models f\hspace{2pt}(\xi) > \nu. 
\]
We write $\mathcal{M} \leq^{\textnormal{s-topless}} \mathcal{N}$ if there is such an embedding. Moreover, $\llbracket \mathcal{M} \leq^{\textnormal{s-topless}} \mathcal{N} \rrbracket$ denotes the set of all such embeddings.
\item $i$ is $\omega${\em -topless}, if it is bounded and not $\omega${\em -coded from above}, meaning that for each $f \in \mathcal{N}$ with $\alpha, \beta \in \mathrm{Ord}^\mathcal{N}$ satisfying
\[
(\mathcal{N} \models f : \alpha \rightarrow \beta) \wedge \alpha_\mathcal{N} \supseteq \omega \wedge \forall k < \omega . f_\mathcal{N}(k) \in \mathrm{Ord}^\mathcal{N} \setminus i(\mathcal{M}),
\]
there is $\nu \in \mathrm{Ord}^\mathcal{N} \setminus i(\mathcal{M})$ such that $\nu <^\mathcal{N} f_\mathcal{N}(k)$, for all $k < \omega$.

We write $\mathcal{M} \leq^{\omega\textnormal{-topless}} \mathcal{N}$ if there is such an embedding. Moreover, $\llbracket \mathcal{M} \leq^{\omega\textnormal{-topless}} \mathcal{N} \rrbracket$ denotes the set of all such embeddings.
\end{itemize}

The notions of initiality are often combined with some notion of toplessness, yielding notions of {\em cut}. In particular, an embedding $i$ is a {\em rank-cut} if it is topless and rank-initial, and $i$ is a {\em strong rank-cut} if it is strongly topless and rank-initial.

For any $\mathrm{notion} \in \{\mathrm{initial}, H, \mathcal{P}, \rnk, \mathrm{topless}, \dots\}$, the symbols `$\leq^\mathrm{notion}$' and `$\llbracket \mathcal{M} \leq^\mathrm{notion} \mathcal{N} \rrbracket$' may be decorated as done in Section \ref{Basic logic and model theory}. For example, we will be concerned with the set $\llbracket \mathcal{M} \leq_\mathcal{S}^\rnk \mathcal{N} \rrbracket$ of rank-initial embeddings over $\mathcal{S}$ from $\mathcal{M}$ to $\mathcal{N}$, where $\mathcal{S}$ is some structure embeddable into both $\mathcal{M}$ and $\mathcal{N}$.

Note that the definitions of initiality and $\mathcal{P}$-initiality also make sense when $\mathcal{N}$ is a mere $\mathcal{L}^0$-structure, not necessarily satisfying $\mathrm{KP}$.
\end{dfn}

It is easily seen that if $i$ is rank-initial and proper, then it is bounded, so the first condition of toplessness is satisfied.

We immediately obtain the following implications:
\[
\text{$i$ is initial } \Leftarrow \text{$i$ is $\mathcal{P}$-initial $\Leftarrow$ $i$ is rank-initial},
\]
\[
\text{$i$ is topless} \Leftarrow \text{$i$ is $\omega$-topless} \Leftarrow \text{$i$ is strongly topless}.
\]

\begin{lemma}[$\Sigma_1^\mathcal{P}$-Overspill]
Suppose that $\mathcal{S}$ is a rank-initial topless substructure of $\mathcal{M} \models \mathrm{KP}^\mathcal{P}$, that $m \in \mathcal{M}$, and that $\sigma(x, y) \in \Sigma_1^\mathcal{P}[x, y]$. If for every $\xi \in \mathrm{Ord}(\mathcal{S})$, there is $\xi < \zeta \in \mathrm{Ord}(\mathcal{S})$ such that $\mathcal{M} \models \sigma(\zeta, m)$, then there is an ordinal $\mu \in \mathcal{M} \setminus \mathcal{S}$, such that $\mathcal{M} \models \sigma(\mu, m)$.
\end{lemma}
\begin{proof}
Let $\sigma\hspace{1pt}'(x, y)$ be the $\Sigma_1^\mathcal{P}[x, y]$-formula 
$$\mathrm{Ord}(x) \wedge \exists x\hspace{1pt}' . (\mathrm{Ord}(x\hspace{1pt}') \wedge x\hspace{1pt}' > x \wedge \sigma(x\hspace{1pt}', y)).$$ 
The antecedent of the claim implies that $\mathcal{M} \models \sigma\hspace{1pt}'(\xi, m)$, for all ordinals $\xi \in \mathcal{S}$. If $\mathcal{M} \models \forall \xi . (\mathrm{Ord}(\xi) \rightarrow \sigma\hspace{1pt}'(\xi, m))$, then we are done; so suppose not. By $\Sigma_1^\mathcal{P}$-Set induction, there is an ordinal $\alpha \in \mathcal{M} \setminus \mathcal{S}$, such that $\mathcal{M} \models \neg \sigma\hspace{1pt}'(\alpha, m)$, but $\mathcal{M} \models \sigma\hspace{1pt}'(\xi, m)$ for all ordinals $\xi < \alpha$ in $\mathcal{M}$. So since $\mathcal{S}$ is topless, there is a an ordinal $\alpha > \beta \in \mathcal{M} \setminus \mathcal{S}$ such that $\mathcal{M} \models \sigma\hspace{1pt}'(\beta, m)$, whence there is an ordinal $\mu > \beta$ in $\mathcal{M}$ such that $\mathcal{M} \models \sigma(\mu, m)$, as desired.
\end{proof}

\begin{prop}\label{elem initial is rank-initial}
	Let $\mathcal{M} \models \mathrm{KP}^\mathcal{P}$ and suppose that $i : \mathcal{M} \rightarrow \mathcal{N}$ is an elementary embedding. Then $i$ is initial if, and only if, it is rank-initial.
\end{prop}
\begin{proof}
	Suppose that $\beta \in \mathrm{Ord}^\mathcal{N}$ is in the image of $i$, so that $i(\alpha) = \beta$ for some $\beta \in \mathrm{Ord}^\mathcal{M}$. Let $n \in \mathcal{N}$, such that $\mathcal{N} \models \rnk(n) < \beta$. Since $i$ is elementary  $i(V_\alpha^\mathcal{M}) = V_{i(\alpha)}^\mathcal{N} = V_\beta^\mathcal{N}$. So by initiality, $n \in i(V_\alpha^\mathcal{M})$ as desired.
\end{proof}

\begin{prop}\label{WFP rank-initial topless}
Let $\mathcal{M} \models \mathrm{KP}$. An element $m_0 \in \mathcal{M}$ is standard iff $\rnk^\mathcal{M}(m_0)$ is standard. If $\mathcal{M}$ is non-standard, then $\mathrm{WFP}(\mathcal{M})$ is a topless rank-initial substructure of $\mathcal{M}$.
\end{prop}
\begin{proof}
$m_0$ is non-standard iff there is an infinite sequence 
$$m_0 \ni^\mathcal{M} m_1 \ni^\mathcal{M} \dots$$ 
of elements of $\mathcal{M}$. This holds iff 
$$\rnk^\mathcal{M}(m_0) >^\mathcal{M} \rnk^\mathcal{M}(m_1) >^\mathcal{M} \dots,$$ 
which in turn holds iff 
$$\rnk^\mathcal{M}(m_0), \rnk^\mathcal{M}(m_1), \dots$$ 
are all non-standard. This immediately yields the first claim and toplessness. It also yields rank-initiality: If $m \in \mathcal{M}$ and $\rnk^\mathcal{M}(m) < \rho$, for some standard $\rho$, then $\rnk^\mathcal{M}(m)$ is standard, whence $m$ is standard.
\end{proof}

\begin{prop}\label{comp emb}
	Let $i : \mathcal{M} \rightarrow \mathcal{N}$ and $j : \mathcal{N} \rightarrow \mathcal{O}$ be embeddings of models of $\mathrm{KP}$. 
\begin{enumerate}[{\normalfont (a)}]
	\item\label{comp emb topless} If $i$ is topless and $j$ is initial, then $j \circ i : \mathcal{M} \rightarrow \mathcal{O}$ is topless.
	\item\label{comp emb strongly topless} If $i$ is strongly topless and $j$ is rank-initial, then $j \circ i : \mathcal{M} \rightarrow \mathcal{O}$ is strongly topless.
\end{enumerate}		
\end{prop}
\begin{proof}
	(\ref{comp emb topless}) Let $\gamma \in \mathrm{Ord}(\mathcal{O}) \setminus j\circ i(\mathcal{M})$. By toplessness, there are $\beta\hspace{1pt}' <^\mathcal{N} \beta \in \mathrm{Ord}(\mathcal{N}) \setminus i(\mathcal{M})$. If $j(\beta\hspace{1pt}') <^\mathcal{O} \gamma$, then we are done. Otherwise, $\gamma <^\mathcal{O} j(\beta)$, so that by initiality $\gamma \in j(\mathcal{N})$. By toplessness, this yields $\beta\hspace{1pt}'' \in \mathrm{Ord}^\mathcal{N} \setminus i(\mathcal{M})$ such that $j(\beta\hspace{1pt}'') < \gamma$.
	
	(\ref{comp emb strongly topless}) Let $f \in \mathcal{O}$ with $\alpha, \beta \in \mathrm{Ord}^\mathcal{O}$ satisfying
\[
(\mathcal{O} \models f : \alpha \rightarrow \beta) \wedge \alpha_\mathcal{O} \supseteq (j \circ i)(\mathrm{Ord}^\mathcal{M}).
\]
By toplessness of $i$ and initiality of $j$, there is $\gamma \in \mathrm{Ord}^\mathcal{N}$, such that $\gamma_\mathcal{N} \supseteq i(\mathrm{Ord}^\mathcal{M})$ and $j(\gamma) \leq^\mathcal{O} \alpha$. In $\mathcal{O}$, let $f\hspace{2pt}' : j(\gamma) \rightarrow j(\gamma)$ be the ``truncation'' of $f$ defined by $f\hspace{2pt}'(\xi) = f\hspace{2pt}(\xi)$, if $f\hspace{2pt}(\xi) < j(\gamma)$, and by $f\hspace{2pt}'(\xi) = 0$, if $f\hspace{2pt}(\xi) \geq j(\gamma)$. Note that $\mathcal{O} \models \rnk(f\hspace{2pt}') \leq j(\gamma) + 2$. So by rank-initiality, there is $g\hspace{1pt}' \in \mathcal{N}$, such that $j(g\hspace{1pt}') = f\hspace{2pt}'$. Consequently, $\mathcal{N} \models g\hspace{1pt}' : \gamma \rightarrow \gamma$, and by initiality of $j$, $(j(g\hspace{1pt}'))_\mathcal{O} = f\hspace{2pt}'_\mathcal{O}$. By strong toplessness of $i$, there is $\nu \in \mathrm{Ord}^\mathcal{N} \setminus i(\mathcal{M})$ such that for all $\xi \in i(\mathrm{Ord}^\mathcal{M})$,
\[
g\hspace{1pt}'_\mathcal{N}(\xi) \not\in i(\mathcal{M}) \Leftrightarrow \mathcal{N} \models g\hspace{1pt}'(\xi) > \nu. 
\]
It follows that for all $\xi \in (j \circ i)(\mathrm{Ord}^\mathcal{M})$,
\[
f_\mathcal{O}(\xi) \not\in (j \circ i)(\mathcal{M}) \Leftrightarrow \mathcal{O} \models f\hspace{2pt}(\xi) > j(\nu). 
\]
So $j \circ i$ is strongly topless.
\end{proof} 

\begin{lemma}\label{omega-topless existence}
	Let $\mathcal{M} \models \mathrm{KP}^\mathcal{P}$ be $\omega$-non-standard and let $\alpha_0 \in \mathrm{Ord}^\mathcal{M}$. For each $k < \omega$, let $\alpha_k \in \mathrm{Ord}^\mathcal{M}$ such that $\mathcal{M} \models \alpha_k = \alpha_0 + k^\mathcal{M}$. Then $\bigcup_{k < \omega} \mathcal{M}_{\alpha_k}$ is an $\omega$-topless rank-initial substructure of $\mathcal{M}$. 
\end{lemma}
\begin{proof}
	$\bigcup_{k < \omega} \mathcal{M}_{\alpha_k}$ is obviously rank-initial in $\mathcal{M}$. Let $m \in^\mathcal{M} \omega^\mathcal{M}$ be non-standard. Since $\alpha_0 +^\mathcal{M} m \in \mathcal{M}$, we have that $\bigcup_{k < \omega} \mathcal{M}_{\alpha_k}$ is bounded in $\mathcal{M}$. Moreover, $\bigcup_{k < \omega} \mathcal{M}_{\alpha_k}$ is topless, because otherwise $\gamma =_\df \sup\{\alpha_k \mid k < \omega \}$ exists in $\mathcal{M}$ and $\mathcal{M} \models \gamma - \alpha_0 = \omega$, which contradicts that $\mathcal{M}$ is $\omega$-non-standard. 
	
	Let $f : \alpha \rightarrow \beta$ be a function in $\mathcal{M}$, where $\alpha, \beta \in \mathrm{Ord}^\mathcal{M}$ and 
	\[
	\alpha_\mathcal{M} \supseteq \omega \wedge \forall k < \omega . f_\mathcal{M}(k) \in \mathrm{Ord}^\mathcal{M} \setminus \bigcup_{k < \omega} \mathcal{M}_{\alpha_k}.
	\] 
	Note that for all $k < \omega$,
	\[
	\mathcal{M} \models \forall \xi \leq k . f\hspace{2pt}(\xi) > \alpha_0 + \xi.
	\]
	So by $\Delta_0^\mathcal{P}$-Overspill, there is a non-standard $\mathring{k} \in^\mathcal{M} \omega^\mathcal{M}$, such that 
	\[
	\mathcal{M} \models \forall \xi \leq \mathring{k} . f\hspace{2pt}(\xi) > \alpha_0 + \xi.
	\]
	Hence, $\alpha_0 +^\mathcal{M} \mathring{k}$ witnesses $\omega$-toplessness of $\bigcup_{k < \omega} \mathcal{M}_{\alpha_k}$.
\end{proof}

The following classic result is proved as Theorem 6.15 in \cite{Jeh02}:

\begin{thm}[Mostowski\hspace{1pt}'s Collapse]\label{Mostowski collapse}
If $\mathcal{M}$ is a well-founded model of $\textnormal{Extensionality}$, then there is a unique isomorphism $\mathrm{Mos} : \mathcal{M} \rightarrow (T, \in\restriction_T)$, such that $T$ is transitive. Moreover, 
$$\forall x \in \mathcal{M} . \mathrm{Mos}(x) = \{\mathrm{Mos}(u) \mid u \in^\mathcal{M} x\}.$$
\end{thm}

This theorem motivates the following simplifying assumption:

\begin{ass}
Every well-founded $\mathcal{L}^0$-model $\mathcal{M}$ of $\textnormal{Extensionality}$ is a transitive set, or more precisely, is of the form $(T, \in\restriction_T)$ where $T$ is transitive and unique. Every embedding between well-founded $\mathcal{L}^0$-models of $\textnormal{Extensionality}$ is an inclusion function. 
\end{ass}

In particular, for any model $\mathcal{M}$ of $\mathrm{KP}$, $\mathrm{WFP}(\mathcal{M})$ is a transitive set and $\mathrm{OSP}(\mathcal{M})$ is an ordinal.

\begin{prop}\label{unique embedding well-founded}
If $\mathcal{M}, \mathcal{N}$ are well-founded models of $\textnormal{Extensionality}$ and there is an initial embedding $i : \mathcal{M} \rightarrow \mathcal{N}$, then $\forall x \in \mathcal{M} . i(x) = x$.
\end{prop}
\begin{proof}
By Assumption 4.6.9, $\mathcal{M}$ and $\mathcal{N}$ are transitive models. Let $x \in \mathcal{M}$. By induction, we may assume that $\forall u \in x . i(u) = u$. So since $i$ is an embedding, $x \subseteq i(x)$. Conversely, let $v \in i(x)$. Since $i$ is an initial embedding, $x \ni i^{-1}(v) = v$. Thus $i(x) \subseteq x$. So $i(x) = x$.
\end{proof}

\begin{prop}\label{emb pres}
If $i : \mathcal{M} \rightarrow \mathcal{N}$ is an initial embedding between models of $\mathrm{KP}$, then:
\begin{enumerate}[{\normalfont (a)}]
\item\label{emb pres fixed} $i(x) = x$, for all $x \in \mathrm{WFP}(\mathcal{M})$.
\item\label{emb pres Delta_0} It is $\Delta_0$-elementary. In particular, for every $\sigma(\vec{x}) \in \Sigma_1[\vec{x}]$, and for every $\vec{a} \in \mathcal{M}$, 
\[
(\mathcal{M}, \vec{a}) \models \sigma(\vec{a}) \Rightarrow (\mathcal{N}, i(\vec{a})) \models \sigma(i(\vec{a})).
\]
\item\label{emb pres Delta_0^P} If $i$ is $\mathcal{P}$-initial, then it is $\Delta_0^\mathcal{P}$-elementary. In particular, for every $\sigma(\vec{x}) \in \Sigma_1^\mathcal{P}[\vec{x}]$, and for every $\vec{a} \in \mathcal{M}$, 
\[
(\mathcal{M}, \vec{a}) \models \sigma(\vec{a}) \Rightarrow (\mathcal{N}, i(\vec{a})) \models \sigma(i(\vec{a})).
\]
\item\label{emb pres SSy_S} If $\mathcal{S}$ is a substructure of $\mathcal{M}$, then $\mathrm{SSy}_\mathcal{S}(\mathcal{M}) \leq \mathrm{SSy}_{i(\mathcal{S})}(\mathcal{N})$.
\item\label{emb pres SSy} $\mathrm{SSy}(\mathcal{M}) \leq \mathrm{SSy}(\mathcal{N})$.
\item\label{emb pres SSy_S eq} If $i$ is rank-initial, $\mathcal{N} \models \mathrm{KP}^\mathcal{P}$ and $\mathcal{S}$ is a topless substructure of $\mathcal{M}$, then $\mathrm{SSy}_\mathcal{S}(\mathcal{M}) \cong \mathrm{SSy}_{i(\mathcal{S})}(\mathcal{N})$.
\item\label{emb pres SSy eq} If $i$ is rank-initial, $\mathcal{M}$ is non-standard and $\mathcal{N} \models \mathrm{KP}^\mathcal{P}$, then $\mathrm{SSy}(\mathcal{M}) \cong \mathrm{SSy}(\mathcal{N})$.
\end{enumerate}
\end{prop}
\begin{proof}
(\ref{emb pres fixed}) is clear from Proposition \ref{unique embedding well-founded}.

(\ref{emb pres Delta_0}) Suppose that $\sigma$ is of the form $\exists y . \varepsilon(y, \vec{x})$, where $\varepsilon \in \Delta_0[y, \vec{x}]$, and that $\mathcal{M} \models \exists y . \varepsilon(y, \vec{a})$. Let $b \in \mathcal{M}$ be a witness of that, so $\mathcal{M} \models \varepsilon(b, \vec{a})$. It suffices to show that $\mathcal{N} \models \varepsilon(i(b), i(\vec{a}))$. Hence, we may assume without loss of generality that $\sigma \in \Delta_0$. Moreover, by bundling existential quantifiers, we can even assume without loss of generality that $\sigma$ has no bounded existential quantifier in front. 

We proceed by induction on the complexity of $\sigma$: The atomic cases follow from that $i$ is an embedding. The cases of the propositional connectives follow from that these connectives ``commute with $\models$''. The only case remaining is bounded universal quantification. 

Suppose that $\sigma$ is $\forall y \in b . \varepsilon(y, \vec{a})$, for some $b \in \vec{a}$. Let $v \in^\mathcal{N} i(b)$. Since $i$ is an initial embedding, there is $u \in \mathcal{M}$, such that $i(u) = v$ and $u \in^\mathcal{M} b$. So $\mathcal{M} \models \varepsilon(u, \vec{a})$, whence by induction hypothesis, $\mathcal{N} \models \varepsilon(i(u), i(\vec{a}))$.

(\ref{emb pres Delta_0^P}) is proved like (\ref{emb pres Delta_0}). The only case remaining is the induction-step for $\mathcal{P}$-bounded universal quantification. For this we need the embedding to be $\mathcal{P}$-initial. Suppose that $\sigma$ is $\forall y \subseteq b . \varepsilon(y, \vec{a})$, for some $\varepsilon \in \Delta_0^\mathcal{P}[y, \vec{x}]$ and some $b \in \vec{a}$. Let $v \subseteq^\mathcal{N} i(b)$. Since $i$ is a $\mathcal{P}$-initial embedding, there is $u \in \mathcal{M}$, such that $i(u) = v$ and $u \subseteq^\mathcal{M} b$. So $\mathcal{M} \models \varepsilon(u, \vec{a})$, whence by induction hypothesis, $\mathcal{N} \models \varepsilon(i(u), i(\vec{a}))$.

For (\ref{emb pres SSy_S}), let $X$ be a class in $\mathrm{SSy}_\mathcal{S}(\mathcal{M})$ that is coded by $c \in \mathcal{M}$. Then we have, for all $x \in \mathcal{S}$, that $x \in c \leftrightarrow i(x) \in i(c)$, whence $i(c)$ codes $i(X_\mathcal{M}) \subseteq i(\mathcal{S})$ in $\mathcal{N}$ and $i(X_\mathcal{M})$ is a class in $\mathrm{SSy}_{i(\mathcal{S})}(\mathcal{N})$. So $i$ induces a witness of $\mathrm{SSy}_\mathcal{S}(\mathcal{M}) \leq \mathrm{SSy}_{i(\mathcal{S})}(\mathcal{N})$.

(\ref{emb pres SSy}) By Proposition \ref{unique embedding well-founded}, $\mathrm{WFP}(\mathcal{M})$ is an inital substructure of $\mathcal{M}$ and $\mathcal{N}$ pointwise fixed by $i$. Apply (\ref{emb pres SSy_S}).

(\ref{emb pres SSy_S eq}) By (\ref{emb pres SSy_S}), $\mathrm{SSy}_\mathcal{S}(\mathcal{M}) \leq \mathrm{SSy}_\mathcal{S}(\mathcal{N})$. Conversely, suppose that $Y \in \mathrm{SSy}_i(\mathcal{S})(\mathcal{N})$ is coded by $d$ in $\mathcal{N}$. By initiality of $i$ and toplessness of $\mathcal{S}$ in $\mathcal{M}$, $i(\mathcal{S})$ is topless in $\mathcal{N}$. So there are codes in $\mathcal{N}$ for $Y$ of arbitrarily small rank above $\mathrm{Ord}^{i(\mathcal{S})}$, found by intersecting $d$ with arbitrarily small $V_\alpha \supseteq \mathcal{S}$ internal to $\mathcal{N}$ ($\mathcal{N} \models \textnormal{Powerset}$, so we have the cumulative hierarchy). By toplessness of $\mathcal{S}$ in $\mathcal{M}$, some $\nu \in \mathrm{Ord}^\mathcal{N} \setminus i(\mathcal{S})$ is in $i(\mathcal{M})$. Thus, by rank-initiality of $i$, there is a code $d\hspace{1pt}' \in i(\mathcal{M})$ for $Y$ in $\mathcal{N}$. Since $i$ is an embedding that fixes $\mathcal{S}$ pointwise, $i^{-1}(d\hspace{1pt}')$ codes $i^{-1}(Y_\mathcal{N})$ in $\mathcal{M}$. We conclude that $\mathrm{SSy}_\mathcal{S}(\mathcal{M}) \cong \mathrm{SSy}_{i(\mathcal{S})}(\mathcal{N})$.

(\ref{emb pres SSy eq}) By Proposition \ref{unique embedding well-founded}, $\mathrm{WFP}(\mathcal{M}) \subseteq \mathrm{WFP}(\mathcal{N})$. So since $\mathcal{M}$ is non-standard, and $i$ is a rank-initial embedding, $\mathrm{WFP}(\mathcal{M})$ is a topless substructure of $\mathcal{M}$ and $\mathcal{N}$. By (a), $i$ is point-wise fixed on $\mathrm{WFP}(\mathcal{M})$. Apply (c).
\end{proof}

\begin{cor}\label{rank-init equivalences in KPP}
If $i: \mathcal{M} \rightarrow \mathcal{N}$ is an embedding between models of $\mathrm{KP}^\mathcal{P}$, then the following are equivalent:
\begin{enumerate}[{\normalfont (a)}]
\item $i$ is rank-initial.
\item $i$ is $\mathcal{P}$-initial.
\item $i$ is initial, and for each ordinal $\xi$ in $\mathcal{M}$, $i(V_\xi^\mathcal{M}) = V_{i(\xi)}^\mathcal{M}$.
\end{enumerate}
\end{cor}
\begin{proof}
(a) $\Rightarrow$ (b) is immediate from the definition.

(b) $\Rightarrow$ (c): The formula $x = V_\alpha$ is $\Sigma_1^\mathcal{P}$, so by Proposition \ref{emb pres} (\ref{emb pres Delta_0^P}), $i(V_\alpha^\mathcal{M}) = V_{i(\alpha)}^\mathcal{N}$, for all ordinals $\alpha$ in $\mathcal{M}$.

(c) $\Rightarrow$ (a): Let $a \in \mathcal{M}$ with $\mathcal{M}$-rank $\alpha$, let $b = i(a)$ with $\mathcal{N}$-rank $\beta$, and let $c \in \mathcal{N}$ be of $\mathcal{N}$-rank $\gamma \leq^\mathcal{N} \beta$. Since $i$ is an embedding preserving $(\xi \mapsto V_\xi)$, we have that $b \in^\mathcal{N} i(V_{\alpha + 1})$ and $i(V_{\alpha + 1}) = V_{\beta\hspace{1pt}'}$, for some $\beta\hspace{1pt}' >^\mathcal{N} \beta$. Therefore $c \in i(V_{\alpha+1})$, so since $i$ is initial, we get that $c \in i(\mathcal{M})$.
\end{proof}

\begin{lemma}\label{hierarchical type coded}
	Suppose that $\mathcal{M} \models \mathrm{KP}^\mathcal{P}$ is non-standard, and let $\mathcal{S}$ be a bounded substructure of $\mathcal{M}$. For each $\vec{a} \in \mathcal{M}$, we have that $\mathrm{tp}_{\Sigma_1^\mathcal{P}, \mathcal{S}}(\vec{a})$ and $\mathrm{tp}_{\Pi_1^\mathcal{P}, \mathcal{S}}(\vec{a})$ are coded in $\mathcal{M}$.
\end{lemma}
\begin{proof}
	Let $\alpha$ be a ordinal in $\mathcal{M}$, such that $\mathcal{S} \subseteq \mathcal{M}_\alpha$. Using $\Sigma_1^\mathcal{P}$-Separation and $\Pi_1^\mathcal{P}$-Separation in $\mathcal{M}$, let 
	\begin{align*}
	s &= \{ \phi(\vec{x}, \vec{v}) \mid \rnk(\vec{v}) < \alpha \wedge \mathrm{Sat}_{\Sigma_1^\mathcal{P}}( \phi(\vec{x}, \vec{y}), \vec{a}, \vec{v})\}^\mathcal{M}, \\
	p &= \{ \phi(\vec{x}, \vec{v}) \mid \rnk(\vec{v}) < \alpha \wedge \mathrm{Sat}_{\Pi_1^\mathcal{P}}( \phi(\vec{x}, \vec{y}), \vec{a}, \vec{v})\}^\mathcal{M}.
	\end{align*}
	By the properties of $\mathrm{Sat}_{\Sigma_1^\mathcal{P}}$, $s$ codes $\mathrm{tp}_{\Sigma_1^\mathcal{P}, \mathcal{S}}(\vec{a})$, and $p$ codes $\mathrm{tp}_{\Pi_1^\mathcal{P}, \mathcal{S}}(\vec{a})$.
\end{proof}

The following characterizations of recursively saturated models of $\mathrm{ZF}$ are sometimes useful. To state it we introduce this definition: Let $\mathcal{L}^0_\mathrm{Sat}$ be the language obtained by adding a new binary predicate $\mathrm{Sat}$ to $\mathcal{L}^0$. We say that $\mathcal{M}$ {\em admits a full satisfaction relation} if $\mathcal{M}$ expands to an $\mathcal{L}^0_\mathrm{Sat}$-structure $(\mathcal{M}, \mathrm{Sat}^\mathcal{M})$, such that
\begin{enumerate}[{\normalfont (i)}]
	\item $(\mathcal{M}, \mathrm{Sat}^\mathcal{M}) \models \mathrm{ZF}(\mathcal{L}^0_\mathrm{Sat})$,
	\item $(\mathcal{M}, \mathrm{Sat}^\mathcal{M}) \models \forall \sigma \in \bar\Sigma_n[x] . (\mathrm{Sat}(\sigma, x) \leftrightarrow \mathrm{Sat}_{\Sigma_n}(\sigma, x) \wedge \mathrm{Sat}(\neg\sigma, x) \leftrightarrow \mathrm{Sat}_{\Pi_n}(\neg\sigma, x))$, for each $n \in \mathbb{N}$.
\end{enumerate}
We say that $\mathrm{Sat}^\mathcal{M}$ is a {\em full satisfaction relation} on $\mathcal{M}$.

The following two results first appeared as Theorems 3.2 and 3.4 in \cite{Sch78}.

\begin{lemma}\label{rec sat reflection}
	Let $\mathcal{M}$ be a model of $\mathrm{ZF}$ that admits a full satisfaction relation. For each $\alpha_0 \in \mathrm{Ord}^\mathcal{M}$, there is $\alpha \in \mathrm{Ord}^\mathcal{M}$, such that $\alpha_0 <^\mathcal{M} \alpha$ and $\mathcal{M}_\alpha \preceq \mathcal{M}.$
\end{lemma}
\begin{proof}
	Let $\alpha_0 \in \mathrm{Ord}^\mathcal{M}$ be arbitrary. We work in $(\mathcal{M}, \mathrm{Sat}^\mathcal{M})$: By the Reflection Theorem, there is $\alpha > \sup\{\alpha_0, \omega\}$, such that for all $\vec{a} \in V_\alpha$ and for all $\delta(\vec{x}) \in \mathcal{L}^0$,
	$$\mathrm{Sat}(\delta, \vec{a}) \leftrightarrow
	\mathrm{Sat}^{V_\alpha}(\delta, \vec{a}).$$
	Now, by the properties of $\mathrm{Sat}$, we obtain
	$$\mathrm{Sat}(\delta, \vec{a}) \leftrightarrow \mathrm{Sat}(\delta^{V_\alpha}, \vec{a}).$$
	
	We now switch to working in the meta-theory: By correctness of $\mathrm{Sat}^\mathcal{M}$ for standard syntax, and by our work inside $\mathcal{M}$, for every standard $\delta$, and every $\vec{a} \in \mathcal{M}_\alpha$ of standard length:
	\[ \mathcal{M} \models \delta(\vec{a}) \leftrightarrow \mathcal{M}_\alpha \models \delta(\vec{a}). \]
	So $\mathcal{M}_\alpha \preceq \mathcal{M},$ as desired.
\end{proof}

\begin{thm}\label{rec sat char}
Let $\mathcal{M} \models \mathrm{ZF}$ be countable. The following conditions are equivalent:
\begin{enumerate}[{\normalfont (a)}]
\item\label{rec sat char rec} $\mathcal{M}$ is recursively saturated.
\item\label{rec sat char truth} $\mathcal{M}$ is $\omega$-non-standard and admits a full satisfaction relation.
\end{enumerate}
Moreover, even if $\mathcal{M}$ is not assumed countable, we have  (\ref{rec sat char truth}) $\Rightarrow$ (\ref{rec sat char rec}).
\end{thm}

While we are on the subject of recursively saturated models, it is worth giving the following theorem:

\begin{thm}
If $\mathcal{M} \models \mathrm{ZFC}$ is recursively saturated and $\mathcal{M} \preceq^\mathrm{cf} \mathcal{N}$, then $\mathcal{N}$ is recursively saturated.
\end{thm}
\begin{proof}
By the forward direction of Theorem \ref{rec sat char}, $\mathcal{M}$ expands to a structure $(\mathcal{M}, \mathrm{Sat}^\mathcal{M}, c^\mathcal{M})$ so that (\ref{rec sat char truth}) of that theorem holds for $\mathcal{M}$, and 
$$(\mathcal{M}, \mathrm{Sat}^\mathcal{M}, c^\mathcal{M}) \models k < c < \omega,$$ 
for all standard $k \in \mathbb{N}$. By Theorem \ref{cofinal expands} below, $\mathcal{N}$ also expands to $(\mathcal{N}, \mathrm{Sat}^\mathcal{N}, c^\mathcal{N})$ so that $(\mathcal{M}, \mathrm{Sat}^\mathcal{M}, c^\mathcal{M}) \preceq (\mathcal{N}, \mathrm{Sat}^\mathcal{N}, c^\mathcal{N})$ (to make that theorem applicable, the relation $\mathrm{Sat}$ and the constant $c$ can be formally merged into a single unary predicate that applies to ordered triples). This ensures that $\mathcal{N}$ is $\omega$-non-standard and expands as in the statement of Theorem \ref{rec sat char}(\ref{rec sat char truth}). So by the backward direction of Theorem \ref{rec sat char} applied to $(\mathcal{N}, \mathrm{Sat}^\mathcal{N})$, we have that $\mathcal{N}$ is recursively saturated.
\end{proof}

\begin{thm}\label{cofinal expands}
Let $\mathcal{L}^0_X$ be the language obtained by adding a new unary predicate $X$ to $\mathcal{L}^0$. If $(\mathcal{M}, X^\mathcal{M}) \models \mathrm{ZFC}(X)$ and $\mathcal{M} \preceq^\mathrm{cf} \mathcal{N}$, then there is $X^\mathcal{N} \subseteq \mathcal{N}$ such that $(\mathcal{M}, X^\mathcal{M}) \preceq (\mathcal{N}, X^\mathcal{N})$.
\end{thm}

In \cite{EKM18}, this is proved as Theorem 6.3.

\begin{dfn}
	Let $i : \mathcal{M} \rightarrow \mathcal{M}$ be a self-embedding of a model $\mathcal{M}$ of $\mathrm{KP}$.
	\begin{itemize}
		\item $x \in \mathcal{M}$ is a {\em fixed point} of $i$, if $i(x) = x$. The substructure of $\mathcal{M}$ of fixed points of $i$ is denoted $\mathrm{Fix}(i)$. 
		\item $X \subseteq \mathcal{M}$ is {\em pointwise fixed} by $i$, if every $x \in X$ is fixed by $i$. $x \in \mathcal{M}$ is {\em pointwise fixed} by $i$ (or an {\em initial fixed point} of $i$), if $x_\mathcal{M}$ is pointwise fixed by $i$. The substructure of $\mathcal{M}$ of elements pointwise fixed by $i$, is denoted $\mathrm{Fix}^\mathrm{initial}(i)$. 
		\item $x \in \mathcal{M}$ is an {\em $H$-initial  fixed point} of $\mathcal{M}$, if $\mathrm{TC}(x)_\mathcal{M}$ is pointwise fixed by $i$. The substructure of $\mathcal{M}$ of $H$-initial fixed points of $i$ is denoted $\mathrm{Fix}^H(i)$.
		\item $x \in \mathcal{M}$ is a {\em rank-initial  fixed point} of $\mathcal{M}$, if $\{y \in \mathcal{M} \mid \rnk(y) \leq \rnk(x)\}$ is pointwise fixed by $i$. The substructure of $\mathcal{M}$ of rank-initial fixed points of $i$ is denoted $\mathrm{Fix}^\rnk(i)$.
	\end{itemize}
	
	We say that $i$ is {\em contractive on $A \subseteq \mathcal{M}$} if for all $x \in A$, we have $\mathcal{M} \models \rnk(i(x)) < \rnk(x)$.
\end{dfn}

Assume that $\mathcal{M}$ is extensional and $i : \mathcal{M} \rightarrow \mathcal{M}$ is initial. Then $x \in \mathcal{M}$ is a fixed point of $i$ if it is pointwise fixed by $i$. It follows that 
$$\mathrm{Fix}(i) \supseteq \mathrm{Fix}^\mathrm{initial}(i) \supseteq \mathrm{Fix}^H(i) \supseteq \mathrm{Fix}^\rnk(i).$$

\begin{lemma}\label{rank-initial Fix is Sigma_1}
	Suppose that $\mathcal{M} \models \mathrm{KP}^\mathcal{P}$ and that $i$ is a rank-initial self-embedding of $\mathcal{M}$ such that $\mathcal{S} = \mathcal{M}\restriction_{\mathrm{Fix}(i)}$ is a rank-initial substructure of $\mathcal{M}$. Then $\mathcal{S} \preceq_{\Sigma_1^\mathcal{P}} \mathcal{M}$.
\end{lemma}
\begin{proof}
	We verify $\mathcal{S} \preceq_{\Sigma_1^\mathcal{P} } \mathcal{M}$ using The Tarski Test (it applies since $\Sigma_1^\mathcal{P}$ is closed under subformulae). Let $\delta(x, \vec{y}) \in \Delta_0^\mathcal{P}[x, \vec{y}]$, let $\vec{s} \in \mathcal{S}$, and assume that $\mathcal{M} \models \exists x . \delta(x, \vec{s})$. We shall now work in $\mathcal{M}$: Let $\xi$ be the least ordinal such that $\exists x \in V_{\xi + 1} . \delta(x, \vec{s})$. We shall show that $i(\xi) = \xi$. Suppose not, then either $i(\xi) < \xi$ or $i(\xi) > \xi$. If $i(\xi) < \xi$, then $\exists x \in V_{i(\xi) + 1} . \delta(x, \vec{s})$, contradicting that $\xi$ is the least ordinal with this property. If $i(\xi) > \xi$, then by rank-initiality there is an ordinal $\zeta < \xi$ such that $i(\zeta) = \xi$. But then $\exists x \in V_{\zeta + 1} . \delta(x, \vec{s})$, again contradicting that $\xi$ is the least ordinal with this property.
	
	By $\Delta_0^\mathcal{P}$-Separation in $\mathcal{M}$, let $D = \{ x \in V_{\xi + 1} \mid \delta(x, \vec{s}) \}^\mathcal{M}$. Since $i$ is $\Delta_0^\mathcal{P}$-elementary and $\xi, \vec{s} \in \mathrm{Fix}(i)$, we have 
	\[
	\mathcal{M} \models \big( (\rnk(x) = \xi \wedge \delta(x, \vec{s})) \leftrightarrow (\rnk(i(x)) = \xi \wedge \delta(i(x), \vec{s})) \big).
	\]
	It immediately follows that $i(D) \subseteq D$. But by rank-initiality, every $x$ of rank $\xi$ in $\mathcal{M}$ is a value of $i$, so we even get that $i(D) = D$. Let $d \in D$. By initiality and $D \in \mathrm{Fix}(i)$, we have $d \in \mathrm{Fix}(i)$; and by construction of $D$, $\mathcal{M} \models \delta(d, \vec{s})$, as desired.
\end{proof}

\begin{lemma}\label{rank-initial Fix is elementary}
	Suppose that $\mathcal{M} \models \mathrm{KP}^\mathcal{P}$ has definable Skolem functions and that $i$ is an automorphism of $\mathcal{M}$ such that $\mathcal{S} = \mathcal{M}\restriction_{\mathrm{Fix}(i)}$. Then $\mathcal{S} \preceq \mathcal{M}$.
\end{lemma}
\begin{proof}
Again, we apply The Tarski Test. Let $\phi(x, \vec{y}) \in \mathcal{L}^0$, let $\vec{s} \in \mathcal{S}$, and assume that $\mathcal{M} \models \exists x . \phi(x, \vec{s})$. Let $m \in \mathcal{M}$ be a witness of this fact. Let $f$ be a Skolem function for $\phi(x, \vec{y})$, defined in $\mathcal{M}$ by a formula $\psi(x, \vec{y})$. Then $\mathcal{M} \models \psi(m, \vec{s})$, and since $i$ is an automorphism fixing $\mathcal{S}$ pointwise, $\mathcal{M} \models \psi(i(m), \vec{s})$. But $\psi$ defines a function, so $\mathcal{M} \models m = i(m)$, whence $m \in \mathcal{S}$ as desired.
\end{proof}

\chapter{Embeddings between models of set theory}\label{ch emb set theory}

\section{Iterated ultrapowers with special self-embeddings}\label{Existence of models of set theory with automorphisms}

It is convenient to fix some objects which will be discussed throughout this section. Fix a countable model $(\mathcal{M}, \mathcal{A}) \models \mathrm{GBC}$. Fix $\mathbb{B}$ to be the boolean algebra $\{A \subseteq \mathrm{Ord}^\mathcal{M} \mid A \in \mathcal{A}\}$ induced by $\mathcal{A}$. Fix $\mathbb{P}$ to be the partial order of unbounded sets in $\mathbb{B}$ ordered under inclusion. Fix a filter $\mathcal{U}$ on $\mathbb{P}$. 

$\mathcal{U}$ is $\mathbb{P}$-{\em generic over} $(\mathcal{M}, \mathcal{A})$, or simply $(\mathcal{M}, \mathcal{A})$-{\em generic}, if it intersects every dense subset of $\mathbb{P}$ that is parametrically definable in $(\mathcal{M, A})$. A $\mathcal{U}$ is $(\mathcal{M, A})$-{\em complete} if for every $a \in \mathcal{M}$ and every $f : \mathrm{Ord}^\mathcal{M} \rightarrow a_\mathcal{M}$ that is coded in $\mathcal{A}$, there is $b \in a_\mathcal{M}$ such that $f^{-1}(b) \in \mathcal{U}$. Considering the characteristic functions of the classes in $\mathcal{A}$, we can easily see that if $\mathcal{U}$ is $(\mathcal{M, A})$-{\em complete}, then it is an {\em ultrafilter} on $\mathbb{B}$, i.e. for any $A \in \mathbb{B}$, we have $A \in \mathcal{U}$ or $\mathrm{Ord}^\mathcal{M} \setminus A \in \mathcal{U}$. 

Let $P : \mathrm{Ord}^\mathcal{M} \times \mathrm{Ord}^\mathcal{M} \rightarrow \mathrm{Ord}^\mathcal{M}$ be a bijection coded in $\mathcal{A}$. For each $g : \mathrm{Ord}^\mathcal{M} \rightarrow \{0,1\}^\mathcal{M}$ coded in $\mathcal{A}$, and each $\alpha \in \mathrm{Ord}^\mathcal{M}$, define $S^g_\alpha =_\mathrm{df} \{\xi \in \mathrm{Ord}^\mathcal{M} \mid g(P(\alpha, \xi)) = 1\}$. Thus, $g$ may be thought of as coding an $\mathrm{Ord}^\mathcal{M}$-sequence of sets; indeed $(\alpha \mapsto S^g_\alpha) : \mathrm{Ord}^\mathcal{M} \rightarrow \mathbb{B}$. $\mathcal{U}$ is $(\mathcal{M, A})$-{\em iterable} if for every $g : \mathrm{Ord}^\mathcal{M} \rightarrow \{0,1\}^\mathcal{M}$ coded in $\mathcal{A}$, we have $\{\alpha \mid S^g_\alpha \in \mathcal{U}\} \in \mathcal{A}$.

A filter $\mathcal{U}$ is $(\mathcal{M, A})$-{\em canonically Ramsey} if for every $n \in \mathbb{N}$ and $f : [\mathrm{Ord}^\mathcal{M}]^n \rightarrow \mathrm{Ord}^\mathcal{M}$ coded in $\mathcal{A}$, there is $H \in \mathcal{U}$ and $S \subseteq \{1, \dots, n\}$, such that for any $\alpha_1, \dots, \alpha_n$ and $\beta_1, \dots, \beta_n$ in $H$, 
$$f\hspace{2pt}(\alpha_1, \dots, \alpha_n) = f\hspace{2pt}(\beta_1, \dots, \beta_n) \leftrightarrow \forall m \in S . \alpha_m = \beta_m.$$
We say that $f$ is {\em canonical} on $H$.

The following theorem is proved in \cite[p. 48]{Ena04}. Combined with Lemma \ref{generic filter existence}, it establishes the existence of an ultrafilter on $\mathbb{B}$, which is $(\mathcal{M, A})$-complete, $(\mathcal{M, A})$-iterable and $(\mathcal{M, A})$-canonically Ramsey, under the assumption that $(\mathcal{M, A}) \models \text{``$\mathrm{Ord}$ is weakly compact''}$.

\begin{thm}\label{weakly compact gives nice ultrafilter}
Let $(\mathcal{M, A}) \models \text{``$\mathrm{Ord}$ is weakly compact''}$. If $\mathcal{U}$ is $(\mathcal{M, A})$-generic, then $\mathcal{U}$ is 
\begin{enumerate}[{\normalfont (a)}]
\item $(\mathcal{M, A})$-complete,
\item $(\mathcal{M, A})$-iterable, and
\item $(\mathcal{M, A})$-canonically Ramsey.
\end{enumerate}
\end{thm}

\cite{EKM17} gives a more detailed account of the following constructions.

\begin{constr}\label{constr iter ultra}
Suppose that $\mathcal{U}$ is a non-principle $(\mathcal{M}, \mathcal{A})$-iterable ultrafilter on $\mathbb{B}$. Then for any $n \in \mathbb{N}$, an ultrafilter $\mathcal{U}^n$ can be recursively constructed on $\mathbb{B}^n =_\mathrm{df} \{A \subseteq (\mathrm{Ord}^\mathcal{M})^n \mid A \in \mathcal{A}\}$ as follows:

First, we extend the definition of iterability. An ultrafilter $\mathcal{V}$ on $\mathbb{B}^n$ is $(\mathcal{M}, \mathcal{A})$-{\em iterable} if for any function $(\alpha \mapsto S_\alpha) : \mathrm{Ord}^\mathcal{M} \rightarrow \mathbb{B}^n$ coded in $\mathcal{A}$, we have $\{\alpha \mid S_\alpha \in \mathcal{V}\} \in \mathcal{A}$. 

$\mathcal{U}^0$ is the trivial (principle) ultrafilter $\{\{\langle\rangle\}\}$ on the boolean algebra $\{\langle\rangle, \{\langle\rangle\}\}$, where $\langle\rangle$ is the empty tuple.

For any $X \in \mathbb{B}^{n+1}$ and any $\alpha \in \mathrm{Ord}^\mathcal{M}$, define 
$$X_{\alpha} =_\mathrm{df} \{ \langle \alpha_2, \dots, \alpha_{n+1}\rangle \mid \langle \alpha, \alpha_2, \dots, \alpha_{n+1}\rangle \in X\}$$
$$X \in \mathcal{U}^{n+1} \Leftrightarrow_\mathrm{df} \big\{\alpha \mid X_{\alpha} \in \mathcal{U}^n \big\} \in \mathcal{U}.$$

Note that there are other equivalent definitions:
$$
\begin{array}{cl}
& X \in \mathcal{U}^{n} \\
\Leftrightarrow& \big\{\alpha_1 \mid \{ \langle \alpha_2, \dots, \alpha_{n}\rangle \mid \langle \alpha_1, \alpha_2, \dots, \alpha_{n}\rangle \in X\} \in \mathcal{U}^{n-1} \big\} \in \mathcal{U}  \\ 
\Leftrightarrow& \{ \alpha_1 \mid \{ \alpha_2 \mid \dots \{ \alpha_n \mid \langle \alpha_1, \dots, \alpha_n \rangle \in X\} \in \mathcal{U} \} \in \mathcal{U} \dots \} \in \mathcal{U}  \\
\Leftrightarrow& \{ \langle \alpha_1, \dots \alpha_{n-1} \rangle \mid \{ \alpha_n \mid \langle \alpha_1, \dots \alpha_{n} \rangle \in X \} \in \mathcal{U} \} \in \mathcal{U}^{n-1}
\end{array}
$$

By the setup, $\mathcal{U}^1 = \mathcal{U}$, which is an $(\mathcal{M}, \mathcal{A})$-iterable ultrafilter on $\mathbb{B}^1$ by assumption. Assuming that $\mathcal{U}^n$ is an $(\mathcal{M}, \mathcal{A})$-iterable ultrafilter on $\mathbb{B}^n$, we shall show that $\mathcal{U}^{n+1}$ is an $(\mathcal{M}, \mathcal{A})$-iterable ultrafilter on $\mathbb{B}^{n+1}$. Let $X \in \mathcal{U}^{n+1}$. 

If $X \subseteq Y \in \mathbb{B}^{n+1}$, then $X_\alpha \subseteq Y_\alpha$, for each $\alpha \in \mathrm{Ord}^\mathcal{M}$, and by iterability of $\mathcal{U}^{n}$, $\{\alpha \mid Y_\alpha \in \mathcal{U}\} \in \mathcal{A}$. So $Y \in \mathcal{U}^{n+1}$ by upwards closure of $\mathcal{U}^n$ and $\mathcal{U}$. Similarly, the iterability of $\mathcal{U}^n$ and the finite intersection and maximality properties of $\mathcal{U}^n$ and $\mathcal{U}$ imply that $\mathcal{U}^{n+1}$ has the finite intersection and maximality properties, respectively. To show iterability, suppose that the function $(\xi \mapsto S_\xi) : \mathrm{Ord}^\mathcal{M} \rightarrow \mathbb{B}^{n+1}$ is coded in $\mathcal{A}$. Then 
$$\{\xi \mid S_\xi \in \mathcal{U}^{n+1}\} = \big\{\xi \mid \{\alpha \mid (S_\xi)_\alpha \in \mathcal{U}^n \} \in \mathcal{U} \big\} \in \mathcal{A}$$ 
by iterability of  $\mathcal{U}$. So $\mathcal{U}^{n+1}$ is also $(\mathcal{M}, \mathcal{A})$-iterable. We have proved:

\begin{lemma}
If $\mathcal{U}$ is an $(\mathcal{M, A})$-iterable ultrafilter on $\mathbb{B}$, then $\mathcal{U}^n$ is an $(\mathcal{M, A})$-iterable ultrafilter on $\mathbb{B}^n$, for every $n \in \mathbb{N}$.
\end{lemma}

Since $\mathcal{U}$ is a non-principle ultrafilter, it contains all final segments of $\mathrm{Ord}^\mathcal{M}$. So by induction, we have 
$$\{ \langle \alpha_1, \dots, \alpha_n \rangle \mid \alpha_1 < \dots < \alpha_n \in \mathrm{Ord}^\mathcal{M}\} \in \mathcal{U}^n,$$ 
for every $n \in \mathbb{N}$.

Lastly, we shall extend the definition of completeness and show that each $\mathcal{U}^n$ has this property. An ultrafilter $\mathcal{V}$ on $\mathbb{B}^n$ is $(\mathcal{M, A})$-{\em complete} if for any $m < n$ and any functions $f : (\mathrm{Ord}^\mathcal{M})^m \rightarrow M$ and $g : (\mathrm{Ord}^\mathcal{M})^n \rightarrow M$ coded in $\mathcal{A}$, such that 
$$\{\langle \alpha_1, \dots, \alpha_n \rangle \mid g(\alpha_1, \dots, \alpha_n) \in f\hspace{2pt}(\alpha_1, \dots, \alpha_m)\} \in \mathcal{V}^n,$$ 
we have that $\mathcal{A}$ codes a function $f\hspace{2pt}' : (\mathrm{Ord}^\mathcal{M})^m \rightarrow M$, such that
$$\{\langle \alpha_1, \dots, \alpha_n \rangle \mid g(\alpha_1, \dots, \alpha_n) = f\hspace{2pt}'(\alpha_1, \dots, \alpha_m)\} \in \mathcal{V}^n.$$
\end{constr}

\begin{lemma}\label{iter complete}
If $\mathcal{U}$ is an $(\mathcal{M, A})$-complete and $(\mathcal{M, A})$-iterable ultrafilter, then $\mathcal{U}^n$ is $(\mathcal{M, A})$-complete, for every $n \in \mathbb{N}$.
\end{lemma}
\begin{proof}
Suppose that $m < n$ and that 
\begin{align*}
f &: (\mathrm{Ord}^\mathcal{M})^m \rightarrow M \\
g &: (\mathrm{Ord}^\mathcal{M})^n \rightarrow M
\end{align*}
satisfy
$$\{\langle \alpha_1, \dots, \alpha_n \rangle \mid g(\alpha_1, \dots, \alpha_n) \in f\hspace{2pt}(\alpha_1, \dots, \alpha_m)\} \in \mathcal{U}^n.$$
We may assume that $m + 1 = n$. The above is equivalent to
$$\{\langle \alpha_1, \dots, \alpha_m \rangle \mid \{\alpha_n \mid g(\alpha_1, \dots, \alpha_n) \in f\hspace{2pt}(\alpha_1, \dots, \alpha_m)\} \in \mathcal{U} \} \in \mathcal{U}^m.$$
Since there is a bijection (coded in $\mathcal{A}$) between $\mathrm{Ord}^\mathcal{M}$ and $M^n$, we have by iterability of $\mathcal{U}$ that
$$\big\{\langle \langle \alpha_1, \dots, \alpha_m \rangle, y \rangle \mid \{\alpha_n \mid g(\alpha_1, \dots, \alpha_n) = y\} \in \mathcal{U} \big\} \in \mathcal{A}.$$
Let $f\hspace{2pt}'$ be the function coded by this set. Since $\mathcal{U}$ is complete it follows that
$$\{\langle \alpha_1, \dots, \alpha_m \mid \{\alpha_n \mid g(\alpha_1, \dots, \alpha_n) = f\hspace{2pt}'(\alpha_1, \dots, \alpha_m)\} \in \mathcal{U}\} \in \mathcal{U}^m,$$
which is equivalent to
$$\{\langle \alpha_1, \dots, \alpha_n \rangle \mid g(\alpha_1, \dots, \alpha_n) = f\hspace{2pt}'(\alpha_1, \dots, \alpha_m)\} \in \mathcal{U}^n,$$
as desired.
\end{proof}

\begin{constr}\label{constr iter power}
$\mathcal{L}^0_\mathcal{A}$ be the language obtained from $\mathcal{L}^0$ by adding constant symbols for all elements of $\mathcal{M}$ and adding relation and function symbols for all relations and functions on $\mathcal{M}$ coded in $\mathcal{A}$. $(\mathcal{M}, A)_{A \in \mathcal{A}}$ denotes the canonical expansion of $\mathcal{M}$ to $\mathcal{L}^0_\mathcal{A}$ determined by $(\mathcal{M}, \mathcal{A})$. Assume that $\mathcal{U}$ is a non-principle $(\mathcal{M}, \mathcal{A})$-iterable ultrafilter on $\mathrm{Ord}^\mathcal{M}$ and let $\mathbb{L}$ be a linear order. We construct $\mathrm{Ult}_{\mathcal{U}, \mathbb{L}}(\mathcal{M}, \mathcal{A})$ as follows:

For each $n \in \mathbb{N}$, define 
\begin{align*}
\Gamma_n =_\mathrm{df} \big\{ & \phi(\xs) \in \mathcal{L}^0_\mathcal{A} \mid \\
&\{\langle \alpha_1, \dots, \alpha_n \rangle \mid (\mathcal{M}, A)_{A \in \mathcal{A}} \models \phi(\alpha_1, \dots, \alpha_n)\} \in \mathcal{U}^n\big\}.
\end{align*}

Since $\mathcal{U}^n$ is an ultrafilter on $(\mathrm{Ord}^\mathcal{M})^n$, each $\Gamma_n$ is a complete $n$-type over $\mathcal{M}$ in the language $\mathcal{L}^0_\mathcal{A}$. Moreover, each $\Gamma_n$ contains the elementary diagram of $(\mathcal{M}, A)_{A \in \mathcal{A}}$.

For each $l \in \mathbb{L}$, let $c_l$ be a new constant symbol, and let $\mathcal{L}^0_{\mathcal{A}, \mathbb{L}}$ be the language generated by $\mathcal{L}^0_\mathcal{A} \cup \{c_l \mid l \in \mathbb{L}\}$. Define 
\begin{align*}
T_{\mathcal{U}, \mathbb{L}} =_\mathrm{df} \{\phi(c_{l_1}, \dots, c_{l_n}) \in \mathcal{L}^0_{\mathcal{A}, \mathbb{L}} \mid & n \in \mathbb{N} \wedge (l_1 <_\mathbf{L} \dots <_\mathbf{L} l_n \in \mathbb{L}) \wedge \\
&\phi(\xs) \in \Gamma_n\}.
\end{align*}

$T_{\mathcal{U}, \mathbb{L}}$ is complete and contains the elementary diagram of $(\mathcal{M}, \mathcal{A})$, because the same holds for each $\Gamma_n$. By Construction \ref{constr iter ultra},
$$T_{\mathcal{U}, \mathbb{L}} \vdash c_{l_1} < c_{l_2} \in \mathrm{Ord} \text{, for any $l_1 <_\mathbb{L} l_2$.}$$ 
Moreover, $T_{\mathcal{U}, \mathbb{L}}$ has definable Skolem functions: For each $\mathcal{L}^0_{\mathcal{A}, \mathbb{L}}$-formula $\exists x . \phi(x)$, we can prove in $T_{\mathcal{U}, \mathbb{L}}$ that the set of witnesses of $\exists x . \phi(x, y)$ of least rank exists, and provided this set is non-empty an element is picked out by a global choice function coded in $\mathcal{A}$. Thus we can define the {\em iterated ultrapower of} $(\mathcal{M, A})$ {\em modulo} $\mathcal{U}$ {\em along} $\mathbb{L}$ as 
\[
\mathrm{Ult}_{\mathcal{U}, \mathbb{L}}(\mathcal{M}, \mathcal{A}) =_\mathrm{df} \text{``the prime model of $T_{\mathcal{U}, \mathbb{L}}$''}.
\]
In particular, every element of $\mathrm{Ult}_{\mathcal{U}, \mathbb{L}}(\mathcal{M}, \mathcal{A})$ is of the form $f\hspace{2pt}(c_{l_1}, \dots, c_{l_n})$, where $l_1 < \dots < l_n \in \mathbb{L}$ and $f \in \mathcal{A}$ (considered as a function symbol of $\mathcal{L}^0_{\mathcal{A}, \mathbb{L}}$). Note that for any $A \in \mathcal{A}$, any function $f$ coded in $\mathcal{A}$ and for any $l_1, \dots, l_n \in \mathbb{L}$, where $n \in \mathbb{N}$, we have
\[
\begin{array}{cl}
&\mathrm{Ult}_{\mathcal{U}, \mathbb{L}}(\mathcal{M}, \mathcal{A}) \models f\hspace{2pt}(c_{l_1}, \dots, c_{l_n}) \in A \\
\Leftrightarrow & \{ \xi \in \mathrm{Ord}^\mathcal{M} \mid (\mathcal{M}, A)_{A \in \mathcal{A}} \models f\hspace{2pt}(\xi) \in A \} \in \mathcal{U}.
\end{array}
\]
A different way of saying the same thing:
\begin{align*}
A^{\mathrm{Ult}_{\mathcal{U}, \mathbb{L}}(\mathcal{M}, \mathcal{A})} = \big\{ & (f\hspace{2pt}(c_{l_1}, \dots, c_{l_n}))^{\mathrm{Ult}_{\mathcal{U}, \mathbb{L}}(\mathcal{M}, \mathcal{A})} \mid \\
& \{ \xi \in \mathrm{Ord}^\mathcal{M} \mid (\mathcal{M}, A)_{A \in \mathcal{A}} \models f\hspace{2pt}(\xi) \in A \} \in \mathcal{U} \big\}.
\end{align*}

Since $T_{\mathcal{U}, \mathbb{L}}$ contains the elementary diagram of $(\mathcal{M}, \mathcal{A})$, the latter embeds elementarily in $\mathrm{Ult}_{\mathcal{U}, \mathbb{L}}(\mathcal{M}, \mathcal{A})$. For simplicity of presentation, we assume that this is an elementary extension. Note that if $\mathbb{L}$ is empty, then $\mathrm{Ult}_{\mathcal{U}, \mathbb{L}}(\mathcal{M}, \mathcal{A}) = (\mathcal{M, A})$. If $\mathcal{U}$ is non-principle, then it is easily seen from Construction \ref{constr iter ultra} that for any $l, l\hspace{1pt}' \in \mathbb{L}$ and any $\alpha \in \mathrm{Ord}^\mathcal{M}$,
$$l <_\mathbb{L} l\hspace{1pt}' \Leftrightarrow \alpha <_\mathbb{O} c_l <_\mathbb{O} c_{l\hspace{1pt}'},$$
where $\mathbb{O} = \mathrm{Ord}^{\mathrm{Ult}_{\mathcal{U}, \mathbb{L}}(\mathcal{M}, \mathcal{A})}$. So $\mathbb{L}$ embeds into the linear order of the ordinals in $\mathrm{Ult}_{\mathcal{U}, \mathbb{L}}(\mathcal{M}, \mathcal{A})$, above the ordinals of $\mathcal{M}$.

It will be helpful to think of the ultrapower as a function (actually functor) of $\mathbb{L}$ rather than as a function of $(\mathcal{M, A})$, so we introduce the alternative notation 
\[
\mathcal{G}_{\mathcal{U}, (\mathcal{M, A})}(\mathbb{L}) =_\mathrm{df} \mathrm{Ult}_{\mathcal{U}, \mathbb{L}}(\mathcal{M}, \mathcal{A}).
\]
Moreover, for each $A \in \mathcal{A}$, we define $A^{\mathcal{G}_{\mathcal{U}, (\mathcal{M, A})}(\mathbb{L})} =_\df A^\mathrm{Ult}_{\mathcal{U}, \mathbb{L}}(\mathcal{M}, \mathcal{A})$. 

Suppose that $(\mathcal{M, A}) \models \mathrm{GBC} + \text{``$\mathrm{Ord}$ is weakly compact''}$ and let $\mathcal{U}$ be an iterable non-principle ultrafilter on $\mathbb{B}$. Given an embedding $i : \mathbb{K} \rightarrow \mathbb{L}$, we construct an embedding 
$$\mathcal{G}_{\mathcal{U}, (\mathcal{M, A})}(i) : \mathcal{G}_{\mathcal{U}, (\mathcal{M, A})}(\mathbb{K}) \rightarrow \mathcal{G}_{\mathcal{U}, (\mathcal{M, A})}(\mathbb{L})$$ 
as follows: Note that any $a \in \mathcal{G}_{\mathcal{U}, (\mathcal{M, A})}(\mathbb{K})$ is of the form $f\hspace{2pt}(c_{k_1}, \dots, c_{k_n})$ for some $f \in \mathcal{A}$, $n \in \mathbb{N}$ and $k_1, \dots, k_n \in \mathbb{K}$. Define $\mathcal{G}_{\mathcal{U}, (\mathcal{M, A})}(i)(a) = f\hspace{2pt}(c_{i(k_1)}, \dots, c_{i(k_n)})$.

As shown in Theorem \ref{Gaifman thm}, $\mathcal{G}_{\mathcal{U}, (\mathcal{M, A})}(i)$ is an elementary embedding, and further more, $\mathcal{G}_{\mathcal{U}, (\mathcal{M, A})}$ is a functor from the category of linear orders, with embeddings as morphisms, to the category of models of the $\mathcal{L}^0_\mathcal{A}$-theory of $(\mathcal{M}, A)_{A \in \mathcal{A}}$, with elementary embeddings as morphisms. We call this the {\em Gaifman functor} of $\mathcal{U}, (\mathcal{M, A})$ and denote it by $\mathcal{G}_{\mathcal{U}, (\mathcal{M, A})}$, or just $\mathcal{G}$ for short.
\end{constr}

Gaifman \cite{Gai76} essentially proved the theorem below for models of arithmetic. A substantial chunk of its generalization to models of set theory was proved for specific needs in \cite{Ena04}.

\begin{thm}[Gaifman-style]\label{Gaifman thm}
Suppose that $(\mathcal{M}, \mathcal{A}) \models \mathrm{GBC} + $``$\mathrm{Ord}$ is weakly compact'' is countable and let $\mathcal{U}$ be an $(\mathcal{M}, \mathcal{A})$-generic ultrafilter. Write $\mathcal{G} = \mathcal{G}_{\mathcal{U}, (\mathcal{M, A})}$ for the corresponding Gaifman functor. Let $i : \mathbb{K} \rightarrow \mathbb{L}$ be an embedding of linear orders.
\begin{enumerate}[{\normalfont (a)}]
\item \label{satisfaction} For each $n \in \mathbb{N}$ and each $\phi(x_1, \dots, x_n) \in \mathcal{L}^0_{\mathcal{A}, \mathbb{L}}$:
\begin{align*}
&\mathcal{G}(\mathbb{L}) \models \phi(c_{l_1}, \dots, c_{l_n}) \Leftrightarrow \\
&\big\{ \langle \alpha_1, \dots, \alpha_n \rangle \in (\mathrm{Ord}^\mathcal{M})^n \mid (\mathcal{M}, A)_{A \in \mathcal{A}} \models \phi(\alpha_1, \dots, \alpha_n) \big\} \in \mathcal{U}^n.
\end{align*}
\item \label{elem} $\mathcal{G}(i) : \mathcal{G}(\mathbb{K}) \rightarrow \mathcal{G}(\mathbb{L})$ is an elementary embedding.
\item \label{func} $\mathcal{G}$ is a functor.
\item \label{cons} If $\mathbb{L} \neq \varnothing$, then $\mathrm{SSy}_\mathcal{M}(\mathcal{G}(\mathbb{L})) \cong (\mathcal{M}, \mathcal{A})$.
\item \label{card} If $|\mathbb{L}| \geq \aleph_0$, then $|\mathcal{G}(\mathbb{L})| = |\mathbb{L}|$.
\item \label{init} $i$ is initial iff $\mathcal{G}(i)$ is rank-initial. 
\item \label{iso} $i$ is an isomorphism iff $\mathcal{G}(i)$ is an isomorphism.
\item \label{bnd} Let $l_0 \in \mathbb{L}$. $i$ is strictly bounded above by $l_0$ iff $\mathcal{G}(i)\restriction_{\mathrm{Ord}^{\mathcal{G}(\mathbb{K})}}$ is strictly bounded above by $c_{l_0}$.
\item \label{downcof} If $\mathbb{L} \setminus i(\mathbb{K})$ has no least element, then $\{ c_l \mid l \in \mathbb{L} \setminus i(\mathbb{K})\}$ is downwards cofinal in $\mathrm{Ord}^{\mathcal{G}(\mathbb{L})} \setminus \mathrm{Ord}^{\mathcal{G}(i(\mathbb{K}))}$.
\item \label{equal} Let $\mathbb{L}'$ be a linear order and let $j, j\hspace{1pt}' : \mathbb{L} \rightarrow \mathbb{L}'$ be embeddings. $i$ is an equalizer of $j, j\hspace{1pt}' : \mathbb{L} \rightarrow \mathbb{L}'$ iff $\mathcal{G}(i)$ is an equalizer of $\mathcal{G}(j), \mathcal{G}(j\hspace{1pt}') : \mathcal{G}(\mathbb{L}) \rightarrow \mathcal{G}(\mathbb{L}')$.
\item \label{contr} Let $i\hspace{1pt}' : \mathbb{K} \rightarrow \mathbb{L}$ be an embedding. We have $\forall k \in \mathbb{K} . i(k) < i\hspace{1pt}'(k)$ iff $\forall \xi \in \mathrm{Ord}^{\mathcal{G}(\mathbb{K})} \setminus \mathrm{Ord}^\mathcal{M} . \mathcal{G}(i)(\xi) < \mathcal{G}(i\hspace{1pt}')(\xi)$.
\end{enumerate}
\end{thm}

{\em Remark. } (\ref{elem}) and (\ref{init}) imply that $(\mathcal{M, A})$ is a rank-initial elementary substructure of $\mathcal{G}(\mathbb{L})$. It follows from (\ref{equal}) that if $j : \mathbb{L} \rightarrow \mathbb{L}$ is a self-embedding with no fixed point, then the fixed point set of $\mathcal{G}(j)$ is $\mathcal{M}$ (consider the equalizer of $j$ and $\id_\mathbb{L}$).

\begin{proof}
(\ref{satisfaction}) This is immediate from Construction \ref{constr iter power}.

(\ref{elem}) We may assume that $\mathbb{K} \subseteq \mathbb{L}$ and that $i$ is the corresponding inclusion function. This has the convenient consequence that $\mathcal{L}^0_{\mathcal{A}, \mathbb{K}} \subseteq \mathcal{L}^0_{\mathcal{A}, \mathbb{L}}$. Let $\phi(\vec{c}) \in \mathcal{L}^0_{\mathcal{A}, \mathbb{K}}$ be a sentence, where $\vec{c}$ is a tuple of constants. By (\ref{satisfaction}), 
\[
\mathcal{G}(\mathbb{K}) \models \phi(\vec{c}) \Leftrightarrow \mathcal{G}(\mathbb{L}) \models \phi(\vec{c}).
\]
Since every element of $\mathcal{G}(\mathbb{K})$ interprets a term, this equivalence establishes $\mathcal{G}(\mathbb{K}) \preceq \mathcal{G}(\mathbb{L})$, as $\mathcal{L}^0_{\mathcal{A}, \mathbb{K}}$-structures.

(\ref{func}) It is clear that $\mathcal{G}(\id_\mathbb{L}) = \id_{\mathcal{G}(\mathbb{L})}$. It only remains to verify that composition is preserved. Let $j : \mathbb{L} \rightarrow \mathbb{L}'$ and $j\hspace{1pt}' : \mathbb{L}' \rightarrow \mathbb{L}''$ be embeddings of linear orders. Let $a$ be an arbitrary element of $\mathcal{G}(\mathbb{L})$. Then $a = f\hspace{2pt}(c_{l_1}, \dots, c_{l_n})$, for some $f \in \mathcal{A}$, $n \in \mathbb{N}$ and $l_1, \dots, l_n \in \mathbb{L}$. $\mathcal{G}(j\hspace{1pt}'\circ j)(a) = f\hspace{2pt}(c_{j\hspace{1pt}'\circ j(l_1)}, \dots, c_{j\hspace{1pt}'\circ j(l_n)}) = f\hspace{2pt}(c_{j\hspace{1pt}'(j(l_1))}, \dots, c_{j\hspace{1pt}'(j(l_n))}) = (\mathcal{G}(j\hspace{1pt}') \circ \mathcal{G}(j))(a)$, as desired.

(\ref{cons}) We start with $\mathcal{A} \subseteq \mathrm{Cod}_\mathcal{M}(\mathcal{G}(\mathbb{L}))$: Let $A \in \mathcal{A}$. Since $(\mathcal{M}, \mathcal{A}) \models \mathrm{GBC}$, the function $f_A : \mathrm{Ord}^\mathcal{M} \rightarrow \mathcal{M}$, defined by $f_A(\xi) = V_\xi \cap A$ for all $\xi \in \mathrm{Ord}^\mathcal{M}$, is coded in $\mathcal{A}$. Since $\mathbb{L} \neq \varnothing$, let $l \in \mathbb{L}$. Now by (\ref{satisfaction}), for each $a \in \mathcal{M}$,
\[
\begin{array}{cl}
&\mathcal{G}(\mathbb{L}) \models a \in f_A(c_l) \\
\Leftrightarrow & \{ \alpha \in \mathrm{Ord^\mathcal{M}} \mid (\mathcal{M}, A)_{A \in \mathcal{A}} \models a \in f_A(\alpha) \} \in \mathcal{U} \\
\Leftrightarrow & a \in A,
\end{array}
\]
so $f_A(l)$ codes $A$.

We proceed with $\mathrm{Cod}_\mathcal{M}(\mathcal{G}(\mathbb{L})) \subseteq \mathcal{A}$: Let $b \in \mathcal{G}(\mathbb{L})$. Then $\mathcal{G}(\mathbb{L}) \models b = f\hspace{2pt}(c_{l_1}, \dots, c_{l_n})$, for some $n \in \mathbb{N}$ and $l_1, \dots, l_n \in \mathbb{L}$. We need to show that 
\[
\{ x \in \mathcal{M} \mid \mathcal{G}(\mathbb{L}) \models x \in f\hspace{2pt}(c_{l_1}, \dots, c_{l_n}) \} \in \mathcal{A}.
\]
By (\ref{satisfaction}), this amounts to showing that
\begin{align*}
\big\{ & x \in \mathcal{M} \mid \\
&\{ \langle \alpha_1, \dots, \alpha_n \rangle \in (\mathrm{Ord}^\mathcal{M})^n \mid (\mathcal{M}, A)_{A \in \mathcal{A}} \models x \in f\hspace{2pt}(\alpha_1, \dots, \alpha_n) \} \in \mathcal{U}^n \\
\big\} & \in \mathcal{A}.
\end{align*}
Letting $w : \mathrm{Ord}^\mathcal{M} \rightarrow \mathcal{M}$ be a well-ordering of $\mathcal{M}$ coded in $\mathcal{A}$, the above is equivalent to
\begin{align*}
\big\{ & \xi \in \mathrm{Ord}^\mathcal{M} \mid \\
&\{ \langle \alpha_1, \dots, \alpha_n \rangle \in (\mathrm{Ord}^\mathcal{M})^n \mid \\
&\phantom{\{}(\mathcal{M}, A)_{A \in \mathcal{A}} \models w(\xi) \in f\hspace{2pt}(\alpha_1, \dots, \alpha_n) \} \in \mathcal{U}^n \\
\big\} & \in \mathcal{A}.
\end{align*}
This last statement holds since $\mathcal{U}^n$ is $(\mathcal{M}, \mathcal{A})$-iterable on $\mathbb{B}^n$.

(\ref{card}) Suppose that $|\mathbb{L}| \geq \aleph_0$. Since $(\mathcal{M}, \mathcal{A})$ is countable, $|\mathbb{L}| = |\mathcal{L}^0_{\mathcal{A}, \mathbb{L}}|$. So since $\mathcal{G}(\mathbb{L})$ is a prime model in that language, we have $|\mathcal{G}(\mathbb{L})| = |\mathbb{L}|$.

(\ref{init}) We may assume that $\mathbb{L}$ extends $\mathbb{K}$ and that $\mathcal{G}(\mathbb{L})$ extends $\mathcal{G}(\mathbb{K})$. By Proposition \ref{elem initial is rank-initial}, it suffices to show that $i$ is initial. Let $a \in \mathcal{G}(\mathbb{K})$ and $b \in \mathcal{G}(\mathbb{L})$, such that $\mathcal{G}(\mathbb{L}) \models b \in a$. We need to show that $b \in \mathcal{G}(\mathbb{K})$. Note that $\mathcal{G}(\mathbb{K}) \models a = f\hspace{2pt}(c_{l_1}, \dots, c_{l_m})$ and $\mathcal{G}(\mathbb{L}) \models b = g(c_{l_1}, \dots, c_{l_n})$, for some $m \leq n \in \mathbb{N}$, $l_1, \dots, l_m \in \mathbb{K}$, $ l_{m+1}, \dots, l_n \in \mathbb{L}$ and $f, g \in \mathcal{A}$. By (\ref{satisfaction}), we have that 
$$\big\{\alpha_1, \dots \alpha_n \mid (\mathcal{M, A}) \models g(\alpha_1, \dots, \alpha_n) \in f\hspace{2pt}(\alpha_1, \dots, \alpha_m)\big\} \in \mathcal{U}^n.$$
So by $(\mathcal{M, A})$-completeness of $\mathcal{U}$ and Lemma \ref{iter complete},
\begin{align*}
& \text{there is $f\hspace{2pt}' : (\mathrm{Ord}^\mathcal{M})^m \rightarrow M$ in $\mathcal{A}$, such that} \\
& \big\{\alpha_1, \dots \alpha_n \mid (\mathcal{M, A}) \models g(\alpha_1, \dots, \alpha_n) = f\hspace{2pt}'(\alpha_1, \dots, \alpha_m)\big\} \in \mathcal{U}^n,
\end{align*}
whence $\mathcal{G}(\mathbb{L}) \models b = f\hspace{2pt}'(c_{l_1}, \dots, c_{l_m})$. But $f\hspace{2pt}'(c_{l_1}, \dots, c_{l_m}) \in \mathcal{G}(\mathbb{K})$. So $i$ is initial.

(\ref{iso}) ($\Leftarrow$) follows from that the orderings embed into the respective sets of ordinals of the models, and that any isomorphism of the models preserves the order of their ordinals. ($\Rightarrow$) follows from that functors preserve isomorphisms. 

(\ref{bnd}) ($\Leftarrow$) is obvious. For ($\Rightarrow$), we may assume that $\mathbb{K}$ is a linear suborder of $\mathbb{L}$ that is strictly bounded above by $l_0 \in \mathbb{L}$. Note that $\mathcal{G}(\mathbb{K}) \prec \mathcal{G}(\mathbb{L}_{< l_0}) \prec \mathcal{G}(\mathbb{L})$. So every ordinal of $\mathcal{G}(\mathbb{K})$ is an ordinal of $\mathcal{G}(\mathbb{L}_{< l_0})$, and by (\ref{init}), every ordinal of $\mathcal{G}(\mathbb{L}_{< l_0})$ is an ordinal of $\mathcal{G}(\mathbb{L})$ below $c_{l_0}$.

(\ref{downcof}) We may assume that $\mathbb{K} \subseteq \mathbb{L}$. Suppose that $\mathbb{L} \setminus \mathbb{K}$ has no least element. Let $\alpha \in \mathrm{Ord}^{\mathcal{G}(\mathbb{L})} \setminus \mathrm{Ord}^{\mathcal{G}(\mathbb{K})}$. Then $\mathcal{G}(\mathbb{L}) \models \alpha = f\hspace{2pt}(c_{l_1}, \dots, c_{l_n})$, for some $n \in \mathbb{N}$ and $l_1 < \dots < l_n \in \mathbb{L}$. Let $1 \leq n^\circ \leq n$ be the least natural number such that there is $l \in \mathbb{L}\setminus \mathbb{K}$ with $l < l_{n^\circ}$. Let $l^* \in \mathbb{L}\setminus \mathbb{K}$ witness this for $n^\circ$. To show that $\mathcal{G}(\mathbb{L}) \models l^* < \alpha$, it suffices to show that
$$\{\langle \xi_1, \dots, \xi_{n^\circ - 1}, \xi^*, \xi_{n^\circ}, \dots, \xi_n \rangle \mid \xi^* < f\hspace{2pt}(\xi_1, \dots, \xi_n) \} \in \mathcal{U}^{n+1}.$$
Suppose not. Then
$$\{\langle \xi_1, \dots, \xi_{n^\circ - 1}, \xi^*, \xi_{n^\circ}, \dots, \xi_n \rangle \mid \xi^* \geq f\hspace{2pt}(\xi_1, \dots, \xi_n) \} \in \mathcal{U}^{n+1 },$$
so by completeness
\begin{align*}
\big\{ & \langle \xi_1, \dots, \xi_{n^\circ - 1} \rangle \mid \\
& \exists \xi . \{ \langle \xi_{n^\circ}, \dots, \xi_n \rangle \mid \xi = f\hspace{2pt}(\xi_1, \dots, \xi_n)\} \in \mathcal{U}^{n - n^\circ + 1} \\
\big\} & \in \mathcal{U}^{n^\circ - 1}.
\end{align*}
Hence, by iterability, we can code a function $f\hspace{2pt}' : (\mathrm{Ord}^\mathcal{M})^{n^\circ -1} \rightarrow \mathrm{Ord}^\mathcal{M}$ in $\mathcal{A}$ by
$$\{\langle \langle \xi_1, \dots, \xi_{n^\circ - 1} \rangle, \xi \rangle \mid \{ \langle \xi_{n^\circ}, \dots, \xi_n \rangle \mid \xi = f\hspace{2pt}(\xi_1, \dots, \xi_n)\} \in \mathcal{U}^{n - n^\circ + 1} \},$$
and $\mathcal{G}(\mathbb{L}) \models f\hspace{2pt}'(l_{c_1}, \dots, l_{c_{n^\circ - 1}}) = \alpha$. But this means that $\alpha \in \mathcal{G}(\mathbb{K})$, contradicting assumption. 

(\ref{equal}) ($\Leftarrow$) is obvious. For ($\Rightarrow$), assume that $i : \mathbb{K} \rightarrow \mathbb{L}$ is an equalizer of $j, j\hspace{1pt}' : \mathbb{L} \rightarrow \mathbb{L}'$, i.e. we may assume that $\mathbb{K}$ is the linear suborder of $\mathbb{L}$ on $\{l \in \mathbb{L} \mid j(l) = j\hspace{1pt}'(l)\}$. It suffices to show that for all elements $x$ of $\mathcal{G}(\mathbb{L})$, we have $\mathcal{G}(j)(x) = \mathcal{G}(j\hspace{1pt}')(x) \leftrightarrow x \in \mathcal{G}(\mathbb{K})$. ($\leftarrow$) is obvious. For ($\rightarrow$), suppose that $a \in \mathcal{G}(\mathbb{L}) \setminus \mathcal{G}(\mathbb{K})$. Let $n^*$ be the least natural number such that $\mathcal{G}(\mathbb{L}) \models a = f\hspace{2pt}(c_{l_1}, \dots, c_{l_{n^*}})$, for some $f \in \mathcal{A}$ and $l_1 < \dots < l_{n^*} \in \mathbb{L}$. Suppose that $\mathcal{G}(\mathbb{L}) \models f\hspace{2pt}(c_{j(l_1)}, \dots, c_{j(l_{n^*})}) = f\hspace{2pt}(c_{j\hspace{1pt}'(l_1)}, \dots, c_{j\hspace{1pt}'(l_{n^*})})$. Since $\mathcal{U}$ is $(\mathcal{M, A})$-canonically Ramsey, since $f$ is coded in $\mathcal{A}$ and since there is a bijection between the universe and the ordinals coded in $\mathcal{A}$, there is $H \in \mathcal{U}$ and $S \subseteq \{1, \dots, n^*\}$, such that for any $\alpha_1 < \dots < \alpha_n$ and $\beta_1 < \dots < \beta_n$ in $H$, 
$$f\hspace{2pt}(\alpha_1, \dots, \alpha_{n^*}) = f\hspace{2pt}(\beta_1, \dots, \beta_{n^*}) \leftrightarrow \forall m \in S . \alpha_m = \beta_m.$$
It follows from $f\hspace{2pt}(c_{j(l_1)}, \dots, c_{j(l_{n^*})}) = f\hspace{2pt}(c_{j\hspace{1pt}'(l_1)}, \dots, c_{j\hspace{1pt}'(l_{n^*})})$ (by a routine argument based on the constructions of this section) that 
$$S = \{ m \mid 1 \leq m \leq n^* \wedge j(l_m) = j\hspace{1pt}'(l_m) \}.$$ 
Since $a \not\in \mathcal{G}(\mathbb{K})$, there is $1 \leq n^\circ \leq n^*$, such that $j(l_{n^\circ}) \neq j\hspace{1pt}'(l_{n^\circ})$. Note that $f\hspace{2pt}' : (\mathrm{Ord}^\mathcal{M})^{n^*-1} \rightarrow \mathcal{M} $, defined by 
\begin{align*}
& f\hspace{2pt}'(\alpha_1, \dots, \alpha_{n^\circ - 1}, \alpha_{n^\circ +1}, \dots, \alpha_{n^*}) = \\
& f\hspace{2pt}(\alpha_1, \dots, \alpha_{n^\circ-1}, \alpha_{n^\circ-1}+1, \alpha_{n^\circ + 1}, \dots, \alpha_{n^*}),
\end{align*}
is coded in $\mathcal{A}$; and note that 
$$f\hspace{2pt}(\alpha_1, \dots, \alpha_{n^\circ-1}, \alpha_{n^\circ-1}+1, \alpha_{n^\circ + 1}, \dots, \alpha_{n^*}) = f\hspace{2pt}(\alpha_1, \dots, \alpha_{n^*})$$ 
for all $\alpha_1< \dots< \alpha_{n^*} \in H$. Since $H \in \mathcal{U}$, it follows that 
$$\mathcal{G}(\mathbb{L}) \models a = f\hspace{2pt}'(c_{l_1}, \dots, c_{l_{n^\circ -1}}, c_{l_{n^\circ +1}}, \dots, c_{l_n}),$$ 
contradicting minimality of $n^*$. 

(\ref{contr}) ($\Leftarrow$) is obvious. For ($\Rightarrow$), let $\alpha$ be an arbitrary internal ordinal of $\mathcal{G}(\mathbb{K})$ not in $\mathcal{M}$. Let $k^*$ be the least element of $\mathbb{K}$, such that $\mathcal{G}(\mathbb{K}) \models \alpha =  f\hspace{2pt}(c_{k_1}, \dots, c_{k_n})$, where $f \in \mathcal{A}$, $n \in \mathbb{N}$ and $k_1 < \dots < k_n = k^* \in \mathbb{K}$. Note that $i\hspace{1pt}'(\alpha)$ is in $\mathcal{G}(\mathbb{L}_{\leq i\hspace{1pt}'(k^*)})$, that $i(\alpha)$ is in $\mathcal{G}(\mathbb{L}_{\leq i(k^*)})$, but that $i\hspace{1pt}'(\alpha)$ is not in $\mathcal{G}(\mathbb{L}_{\leq i(k^*)})$, because $i(k^*) < i\hspace{1pt}'(k^*)$. So by (\ref{elem}) and (\ref{init}), $\mathcal{G}(\mathbb{L}) \models i(\alpha) < i\hspace{1pt}'(\alpha)$.
\end{proof}

This theorem is quite powerful when applied to the set of rational numbers $\mathbb{Q}$, with the usual ordering $<_\mathbb{Q}$. For any structure $\mathcal{K}$, and $S \subseteq \mathcal{K}$, we define $\mathrm{End}_S(\mathcal{K})$ as the monoid of endomorphisms of $\mathcal{K}$ that fix $S$ pointwise, and we define $\mathrm{Aut}_S(\mathcal{K})$ as the group of automorphisms of $\mathcal{K}$ that fix $S$ pointwise.

\begin{cor}\label{end extend model of weakly compact to model with autos}
	If $\mathcal{M} \models \mathrm{ZFC}$ expands to a countable model $(\mathcal{M}, \mathcal{A})$ of $\mathrm{GBC} + $ ``$\mathrm{Ord}$ is weakly compact'', then there is $\mathcal{M} \prec^\textnormal{rank-cut} \mathcal{N}$, such that $\mathrm{SSy}_\mathcal{M}(\mathcal{N}) = (\mathcal{M}, \mathcal{A})$, and such that for any countable linear order $\mathbb{L}$, there is an embedding of $\mathrm{End}(\mathbb{L})$ into $\mathrm{End}_\mathcal{M}(\mathcal{N})$. Moreover, this embedding sends every automorphism of $\mathbb{L}$ to an automorphism of $\mathcal{N}$, and sends every contractive self-embedding of $\mathbb{L}$ to a self-embedding of $\mathcal{N}$ that is contractive on $\mathrm{Ord}^\mathcal{N} \setminus \mathcal{M}$ and whose fixed-point set is $\mathcal{M}$.
\end{cor}
\begin{proof}
	Since $(\mathcal{M}, \mathcal{A})$ is countable, Lemma \ref{generic filter existence} and Theorem \ref{weakly compact gives nice ultrafilter} tell us that there is an $(\mathcal{M}, \mathcal{A})$-complete ultrafilter $\mathcal{U}$. Let $\mathcal{N} = \mathcal{G}_{\mathcal{U}, (\mathcal{M}, \mathcal{A})}(\mathbb{Q})$. By Theorem \ref{Gaifman thm} (\ref{elem}), (\ref{cons}), (\ref{init}) and (\ref{downcof}), $\mathcal{M} \prec^\textnormal{rank-cut} \mathcal{N}$ and $\mathrm{SSy}_\mathcal{M}(\mathcal{N}) = (\mathcal{M}, \mathcal{A})$. By Theorem \ref{Gaifman thm} (\ref{func}) and (\ref{equal}), there is an embedding of $\mathrm{End}(\mathbb{Q})$ into $\mathrm{End}_\mathcal{M}(\mathcal{N})$. Moreover, it is well-known that for any countable linear order $\mathbb{L}$, there is an embedding of $\mathrm{End}(\mathbb{L})$ into $\mathrm{End}(\mathbb{Q})$. Composing these two embeddings gives the result. The last sentence in the statement follows from Theorem \ref{Gaifman thm} (\ref{iso}), (\ref{equal}) and (\ref{contr}).
\end{proof}

\begin{lemma}\label{contractive proper selfembedding of Q}
For any $q \in \mathbb{Q}$, there is an initial topless contractive self-embedding of the usual linear order on  $\mathbb{Q}$ that is strictly bounded by $q$.
\end{lemma}
\begin{proof}
It suffices to show that there is an initial topless contractive self-embedding of $\mathbb{Q}$, because by toplessness that would be bounded by some $q' \in \mathbb{Q}$ and we can compose it with the self-embedding $(x \mapsto x - |q' - q|)$ to obtain an initial topless contractive self-embedding bounded by $q$. Thus, we proceed to show that the usual linear order on $\mathbb{Q}$ can be expanded to a model of the following theory $T$, in the language of a binary relation $<$ and a unary function $f$:
\begin{align*}
&\text{``$<$ is a dense linear order without endpoints''}; \\
&\forall x, y . ( x = y \leftrightarrow f\hspace{2pt}(x) = f\hspace{2pt}(y)); \\
&\forall x, y . ( x < y \leftrightarrow f\hspace{2pt}(x) < f\hspace{2pt}(y)); \\
&\forall x, y . (x < f\hspace{2pt}(y) \rightarrow \exists z . f\hspace{2pt}(z) = x); \\
&\exists y . \forall x . f\hspace{2pt}(x) < y; \\
&\forall y . \big((\forall x . f\hspace{2pt}(x) < y) \rightarrow \exists y\hspace{1pt}' . (y\hspace{1pt}' < y \wedge \forall x . f\hspace{2pt}(x) < y\hspace{1pt}') \big); \\
&\forall x . f\hspace{2pt}(x) < x.
\end{align*}
Let $\mathcal{R}$ be the expansion of the order of the punctured reals $\mathbb{R} \setminus \{0\}$ inherited from the usual order of $\mathbb{R}$, interpreting $f$ by the function $f^\mathcal{R} : \mathbb{R} \setminus \{0\} \rightarrow \mathbb{R} \setminus \{0\}$, defined by $f^\mathcal{R}(x) = -2^{-x}$, for all $x \in \mathbb{R} \setminus \{0\}$. Note that $\mathcal{R} \models T$. Now by the Downward L\"owenheim-Skolem Theorem, there is a countable model $\mathcal{Q}$ of $T$. Since every countable dense linear order without endpoints is isomorphic to $(\mathbb{Q}, <_\mathbb{Q})$, it follows that $f^\mathcal{Q}$ induces an initial topless contractive self-embedding of $\mathbb{Q}$.
\end{proof}

\begin{cor}\label{end extend model of weakly compact to model with endos}
	Suppose that $\mathcal{M} \models \mathrm{ZFC}$ expands to a countable model $(\mathcal{M}, \mathcal{A})$ of $\mathrm{GBC} + $``$\mathrm{Ord}$ is weakly compact''. Then there is a model $\mathcal{M} \prec^\textnormal{rank-cut} \mathcal{N}$, with $\mathrm{SSy}_\mathcal{M}(\mathcal{N}) = (\mathcal{M}, \mathcal{A})$, such that for any $\nu \in \mathrm{Ord}^\mathcal{N} \setminus \mathcal{M}$, there is a rank-initial topless elementary self-embedding $j$ of $\mathcal{N}$, which is contractive on $\mathrm{Ord}^\mathcal{N} \setminus \mathcal{M}$, bounded by $\nu$, and satisfies $\mathcal{M} = \mathrm{Fix}(j)$.
\end{cor}
\begin{proof}
	Let $\mathcal{U}$ be an $(\mathcal{M}, \mathcal{U})$-generic ultrafilter, and let $\mathcal{N}$ be the model $\mathcal{G}_{\mathcal{U}, (\mathcal{M}, \mathcal{A})}(\mathbb{Q})$. As in Corollary \ref{end extend model of weakly compact to model with autos}, $\mathcal{M} \prec^\textnormal{rank-cut} \mathcal{N}$ and $\mathrm{SSy}_\mathcal{M}(\mathcal{N}) = (\mathcal{M}, \mathcal{A})$. Let $\nu \in \mathrm{Ord}^\mathcal{N} \setminus \mathcal{M}$. By Theorem \ref{Gaifman thm} (\ref{downcof}), there is $q \in \mathbb{Q}$, such that $\mathcal{N} \models c_q < \nu$. Using Lemma \ref{contractive proper selfembedding of Q}, let $\hat j$ be an initial topless contractive self-embedding of $\mathbb{Q}$ that is strictly bounded by $q$. Let $j = \mathcal{G}(\hat j)$. The result now follows from Theorem \ref{Gaifman thm}: By (\ref{elem}), $j$ is an elementary embedding; by (\ref{init}), $j$ is rank-initial; by (\ref{bnd}) and  (\ref{downcof}), $j$ is bounded by $\nu$ and topless; by (\ref{contr}), $j$ is contractive on $\mathrm{Ord}^\mathcal{N} \setminus \mathcal{M}$; and by (\ref{equal}) $\mathcal{M} = \mathrm{Fix}(j)$.
\end{proof}

We will also have use of a slight generalization of the Gaifman construction described above. We consider a set-up where $\mathcal{S} <^\textnormal{rank-cut} \mathcal{M} \models \mathrm{KP} + \textnormal{Choice}$ and $(\mathcal{S}, \mathcal{A}) =_\df \mathrm{SSy}_\mathcal{S}(\mathcal{M})$ is a model of $\mathrm{GBC} + $ ``$\mathrm{Ord}$ is weakly compact''. The partial order $\mathbb{P}$ and the boolean algebra $\mathbb{B}$ are now constructed as above, based on $(\mathcal{S}, \mathcal{A})$. By Lemma \ref{generic filter existence} and Theorem \ref{weakly compact gives nice ultrafilter}, there is an $(\mathcal{S}, \mathcal{A})$-generic ultrafilter $\mathcal{U}$. By Construction \ref{constr iter ultra}, this ultrafilter can be iterated. 

Now the goal is essentially to construct, given any linear order $\mathbb{L}$, an elementary extension $\mathcal{N}$ of $\mathcal{M}$, such that $\mathcal{S}$ is also rank-initial in $\mathcal{N}$ and such that $\mathbb{L}$ order-embeds ``nicely into the set of ordinals of $\mathcal{N}$ above $\mathcal{S}$ and below $\mathcal{M} \setminus \mathcal{S}$''. To this end, we proceed with a modification of Construction \ref{constr iter power}.

We say that a function $f : \mathcal{S}^n  \rightarrow \mathcal{M}$, for some standard $n \in \mathbb{N}$, is {\em coded in} $\mathcal{M}$ if there is a function $g$ in $\mathcal{M}$ with $\dom(g_\mathcal{M}) \supseteq \mathcal{S}^n$ and $f = g_\mathcal{M} \cap (\mathcal{S}^n \times \mathcal{M})$. Let $\mathcal{F}$ be the set of all functions from $\mathcal{S}$ to $\mathcal{M}$ coded in $\mathcal{M}$. Let $\mathcal{L}^0_\mathcal{F}$ be the language obtained from $\mathcal{L}^0$ by adding new constant symbols for the elements of $\mathcal{M}$ and new function symbols for the elements of $\mathcal{F}$. Then we may canonically expand $\mathcal{M}$ to an $\mathcal{L}^0_\mathcal{F}$-structure $(\mathcal{M}, f\hspace{2pt})_{f \in \mathcal{F}}$. Now, just as before, for each $n \in \mathbb{N}$ we define
\begin{align*}
\Gamma_n =_\df \big\{ & \phi(\xs) \in \mathcal{L}^0_\mathcal{F} \mid \\
& \{\langle \alpha_1, \dots, \alpha_n \rangle \mid (\mathcal{M}, f\hspace{2pt})_{f \in \mathcal{F}} \models \phi(\alpha_1, \dots, \alpha_n)\} \in \mathcal{U}^n\big\}.
\end{align*}
For each $l \in \mathbb{L}$, let $c_l$ be a new constant symbol, and let $\mathcal{L}^0_{\mathcal{F}, \mathbb{L}}$ be the language generated by $\mathcal{L}^0_\mathcal{F} \cup \{c_l \mid l \in \mathbb{L}\}$. Define 
\begin{align*}
T_{\mathcal{U}, \mathbb{L}} =_\mathrm{df} \{\phi(c_{l_1}, \dots, c_{l_n}) \in \mathcal{L}^0_{\mathcal{F}, \mathbb{L}} \mid & n \in \mathbb{N} \wedge (l_1 <_\mathbf{L} \dots <_\mathbf{L} l_n \in \mathbb{L}) \wedge \\
& \phi(\xs) \in \Gamma_n\}.
\end{align*}
Note that this theory contains the elementary diagram of $\mathcal{M}$ and it has definable Skolem functions (a global choice function on $\mathcal{S}$ is found in $\mathcal{F}$). So we may define the {\em iterated ultrapower of} $(\mathcal{M, F})$ {\em modulo} $\mathcal{U}$ {\em along} $\mathbb{L}$ as 
\[
\mathcal{G}_{\mathcal{U}, (\mathcal{M, F})}(\mathbb{L}) =_\mathrm{df} \mathrm{Ult}_{\mathcal{U}, \mathbb{L}}(\mathcal{M}, \mathcal{F}) =_\mathrm{df} \text{``the prime model of $T_{\mathcal{U}, \mathbb{L}}$''}.
\]
In particular, every element of $\mathrm{Ult}_{\mathcal{U}, \mathbb{L}}(\mathcal{M}, \mathcal{F})$ is of the form $f\hspace{2pt}(c_{l_1}, \dots, c_{l_n})$, where $l_1 < \dots < l_n \in \mathbb{L}$ and $f \in \mathcal{F}$. 

If $\mathcal{U}$ is non-principle, then by definition of $T_{\mathcal{U}, \mathbb{L}}$, we have: For any $l, l\hspace{1pt}' \in \mathbb{L}$, any $\alpha \in \mathrm{Ord}^\mathcal{S}$ and any $\mu \in \mathrm{Ord}^\mathcal{M} \setminus \mathcal{S}$,
\[
l <_\mathbb{L} l\hspace{1pt}' \Leftrightarrow \alpha <_\mathbb{O} c_l <_\mathbb{O} c_{l\hspace{1pt}'} <_\mathbb{O} \mu,
\]
where $\mathbb{O} = \mathrm{Ord}^{\mathrm{Ult}_{\mathcal{U}, \mathbb{L}}(\mathcal{M}, \mathcal{F})}$. So $\mathbb{L}$ embeds into the linear order of ordinals in $\mathrm{Ult}_{\mathcal{U}, \mathbb{L}}(\mathcal{M}, \mathcal{F})$ that are above the ordinals of $\mathcal{S}$ and below the other ordinals of $\mathcal{M}$.

The generalization of Theorem \ref{Gaifman thm} may now be stated like this:

\begin{thm}\label{Gaifman thm general}
Suppose that $\mathcal{S} <^\textnormal{rank-cut} \mathcal{M} \models \mathrm{KP} + \textnormal{Choice}$, where $\mathcal{M}$ is countable, and that $(\mathcal{S}, \mathcal{A}) =_\df \mathrm{SSy}_\mathcal{S}(\mathcal{M})$ is a model of $\mathrm{GBC} + $``$\mathrm{Ord}$ is weakly compact''. Let $\mathcal{U}$ be an $(\mathcal{S}, \mathcal{A})$-generic ultrafilter and let $\mathcal{F}$ be the set of functions from $\mathcal{S}$ to $\mathcal{M}$ coded in $\mathcal{M}$. Let $i : \mathbb{K} \rightarrow \mathbb{L}$ be an embedding of linear orders. Write $\mathcal{G} = \mathcal{G}_{\mathcal{U}, (\mathcal{M, F})}$ for the corresponding Gaifman functor, and write $\mathcal{S}(\mathbb{K})$ for the set of $x \in \mathcal{G}(\mathbb{K})$ of the form $x = f\hspace{2pt}(c_{l_1}, \dots, c_{l_n})$, where $\mathrm{image}(f\hspace{2pt}) \subseteq \mathcal{S}$ and $n \in \mathbb{N}$.
\begin{enumerate}[{\normalfont (a)}]
\item \label{satisfaction gen} For each $n \in \mathbb{N}$ and each $\phi(x_1, \dots, x_n) \in \mathcal{L}^0_{\mathcal{F}, \mathbb{L}}$:
\begin{align*}
& \mathcal{G}(\mathbb{L}) \models \phi(c_{l_1}, \dots, c_{l_n}) \Leftrightarrow \\
& \big\{ \langle \alpha_1, \dots, \alpha_n \rangle \in (\mathrm{Ord}^\mathcal{M})^n \mid (\mathcal{M}, f\hspace{2pt})_{f \in \mathcal{F}} \models \phi(\alpha_1, \dots, \alpha_n) \big\} \in \mathcal{U}^n.
\end{align*}
\item \label{elem gen} $\mathcal{G}(i) : \mathcal{G}(\mathbb{K}) \rightarrow \mathcal{G}(\mathbb{L})$ is an elementary embedding.
\item \label{func gen} $\mathcal{G}$ is a functor.
\item \label{cons gen} If $\mathbb{L} \neq \varnothing$, then $\mathrm{SSy}_\mathcal{S}(\mathcal{G}(\mathbb{L})) \cong (\mathcal{S}, \mathcal{A})$.
\item \label{card gen} If $|\mathbb{L}| \geq \aleph_0$, then $|\mathcal{G}(\mathbb{L})| = |\mathbb{L}|$.
\item \label{init gen} $i$ is initial iff $\mathcal{G}(i)\restriction_{\mathcal{S}(\mathbb{K})}$ is rank-initial. Moreover, $\mathcal{S}(\mathbb{K}) \subseteq \mathcal{G}(\mathbb{K}) \setminus (\mathcal{M} \setminus \mathcal{S})$ and $\mathcal{S}(\mathbb{K}) <^\textnormal{rank} \mathcal{G}(\mathbb{K})$.
\item \label{iso gen} $i$ is an isomorphism iff $\mathcal{G}(i)$ is an isomorphism.
\item \label{bnd gen} Let $l_0 \in \mathbb{L}$. $i$ is strictly bounded above by $l_0$ iff $\mathcal{G}(i)\restriction_{\mathrm{Ord}^{\mathcal{G}(\mathbb{K})} \cap \mathcal{S}(\mathbb{K})}$ is strictly bounded above by $c_{l_0}$ in $\mathcal{G}(\mathbb{L})$.
\item \label{downcof gen} If $\mathbb{L} \setminus i(\mathbb{K})$ has no least element, then $\{ c_l \mid l \in \mathbb{L} \setminus i(\mathbb{K})\}$ is downwards cofinal in $\mathrm{Ord}^{\mathcal{G}(\mathbb{L})} \setminus \mathrm{Ord}^{\mathcal{G}(i(\mathbb{K}))}$.
\item \label{equal gen} Let $\mathbb{L}'$ be a linear order and let $j, j\hspace{1pt}' : \mathbb{L} \rightarrow \mathbb{L}'$ be embeddings. $i$ is an equalizer of $j, j\hspace{1pt}' : \mathbb{L} \rightarrow \mathbb{L}'$ iff $\mathcal{G}(i)$ is an equalizer of $\mathcal{G}(j), \mathcal{G}(j\hspace{1pt}') : \mathcal{G}(\mathbb{L}) \rightarrow \mathcal{G}(\mathbb{L}')$.
\item \label{contr gen} Let $i\hspace{1pt}' : \mathbb{K} \rightarrow \mathbb{L}$ be an embedding. We have $\forall k \in \mathbb{K} . i(k) < i\hspace{1pt}'(k)$ iff $\forall \xi \in \mathrm{Ord}^{\mathcal{G}(\mathbb{K})} \setminus \mathrm{Ord}^\mathcal{M} . \mathcal{G}(i)(\xi) < \mathcal{G}(i\hspace{1pt}')(\xi)$.
\end{enumerate}
\end{thm}
\begin{proof}[Proof-modification]
Essentially, only (\ref{init gen}) and (\ref{bnd gen}) are stated differently. The proofs of the others go through verbatim after replacing certain instances of `$\mathcal{A}$' by `$\mathcal{F}$', where appropriate. We proceed with the proofs of the new versions of (\ref{init gen}) and (\ref{bnd gen}): 

(\ref{init gen}) Let us start with the second claim. Let $a \in \mathcal{S}(\mathbb{K}).$ Note that $\mathcal{G}(\mathbb{K}) \models a = f\hspace{2pt}(c_{l_1}, \dots, c_{l_m})$, for some $m \in \mathbb{N}$, $l_1, \dots, l_m \in \mathbb{K}$ and $f \in \mathcal{F}$, such that $\mathrm{image}(f\hspace{2pt}) \subseteq \mathcal{S}$. Let $\mu \in \mathrm{Ord}^\mathcal{M} \setminus \mathcal{S}$. Since $\mathrm{image}(f\hspace{2pt}) \subseteq \mathcal{S}$, we have by (\ref{satisfaction gen}) that $\mathcal{G}(\mathbb{K}) \models \mathrm{rank}(a) = \mathrm{rank}(f\hspace{2pt}(c_{l_1}, \dots, c_{l_m})) < \mu$. So $a \in \mathcal{S}(\mathbb{K}) \subseteq \mathcal{G}(\mathbb{K}) \setminus (\mathcal{M} \setminus \mathcal{S})$ as desired. Note that the third claim follows from the first claim.

We proceed with the first claim. We may assume that $\mathbb{L}$ extends $\mathbb{K}$ and that $\mathcal{G}(\mathbb{L})$ extends $\mathcal{G}(\mathbb{K})$. By Proposition \ref{elem initial is rank-initial}, it suffices to show that $i\restriction_{\mathcal{S}(\mathbb{K})}$ is initial. Let $a \in \mathcal{S}(\mathbb{K})$ and $b \in \mathcal{G}(\mathbb{L})$, such that $\mathcal{G}(\mathbb{L}) \models b \in a$. We need to show that $b \in \mathcal{G}(\mathbb{K})$. Note that $\mathcal{G}(\mathbb{K}) \models a = f\hspace{2pt}(c_{l_1}, \dots, c_{l_m})$ and $\mathcal{G}(\mathbb{L}) \models b = g(c_{l_1}, \dots, c_{l_n})$, for some $m \leq n \in \mathbb{N}$, $l_1, \dots, l_m \in \mathbb{K}$, $ l_{m+1}, \dots, l_n \in \mathbb{L}$ and $f, g \in \mathcal{F}$, such that $\mathrm{image}(f\hspace{2pt}) \subseteq \mathcal{S}$. By (\ref{satisfaction}), we have that 
$$\big\{\alpha_1, \dots \alpha_n \mid ((\mathcal{M}, f\hspace{2pt})_{f \in \mathcal{F}}) \models g(\alpha_1, \dots, \alpha_n) \in f\hspace{2pt}(\alpha_1, \dots, \alpha_m)\big\} \in \mathcal{U}^n.$$
It now follows from initiality of $\mathcal{S}$ in $\mathcal{M}$, that there is $g\hspace{1pt}' \in \mathcal{F}$, such that $\mathrm{image}(g\hspace{1pt}') \subseteq \mathcal{S}$ and
$$\big\{\alpha_1, \dots \alpha_n \mid ((\mathcal{M}, f\hspace{2pt})_{f \in \mathcal{F}}) \models g(\alpha_1, \dots, \alpha_n) = g\hspace{1pt}'(\alpha_1, \dots, \alpha_n)\big\} \in \mathcal{U}^n.$$
(This last step is the crucial new ingredient of the modified proof.)
Combining the two last statements with the $(\mathcal{M, A})$-completeness of $\mathcal{U}$ and Lemma \ref{iter complete}, we conclude that there is $f\hspace{2pt}' : \mathcal{S}^m \rightarrow S$ in $\mathcal{F}$ such that 
$$\big\{\alpha_1, \dots \alpha_n \mid ((\mathcal{M}, f\hspace{2pt})_{f \in \mathcal{F}}) \models g(\alpha_1, \dots, \alpha_n) = f\hspace{2pt}'(\alpha_1, \dots, \alpha_m)\big\} \in \mathcal{U}^n,$$ 
whence $\mathcal{G}(\mathbb{L}) \models b = f\hspace{2pt}'(c_{l_1}, \dots, c_{l_m})$. But $f\hspace{2pt}'(c_{l_1}, \dots, c_{l_m}) \in \mathcal{S}(\mathbb{K})$. So $i_{\mathcal{S}(\mathbb{K})}$ is initial.

(\ref{bnd gen}) ($\Leftarrow$) follows from that $(l \mapsto c_l)$ is an embedding of $\mathbb{L}$ into $\mathrm{Ord}^{\mathcal{G}(\mathbb{K})} \cap \mathcal{S}(\mathbb{K})$. For ($\Rightarrow$), we may assume that $\mathbb{K}$ is a linear suborder of $\mathbb{L}$ that is strictly bounded above by $l_0 \in \mathbb{L}$. By (\ref{init gen}), $\mathcal{S}(\mathbb{K})$ is rank-initial in $\mathcal{G}(\mathbb{L}_{< l_0})$. So since $\mathcal{G}(\mathbb{K}) \prec \mathcal{G}(\mathbb{L}_{< l_0}) \prec \mathcal{G}(\mathbb{L})$, we have that $\mathrm{Ord}^{\mathcal{G}(\mathbb{K})} \subseteq \mathrm{Ord}^{\mathcal{G}(\mathbb{L}_{< l_0})} \subseteq \mathrm{Ord}^{\mathcal{G}(\mathbb{L})}$ and that for each $\alpha \in \mathrm{Ord}^{\mathcal{G}(\mathbb{K})} \cap \mathcal{S}(\mathbb{K})$, $\mathcal{G}(\mathbb{L}) \models \alpha < c_{l_0}$.
\end{proof}

\begin{cor}\label{end extend model of weakly compact to model with endos gen}
	Suppose that $\mathcal{M}$ is a countable model of $\mathrm{KP} + \textnormal{Choice}$ and $\mathcal{S} <^\textnormal{rank-cut} \mathcal{M}$, such that $(\mathcal{S}, \mathcal{A}) =_\df \mathrm{SSy}_\mathcal{S}(\mathcal{M})$ is a model of $\mathrm{GBC} + \mathrm{Ord} \text{ is weakly compact}$.  There is $\mathcal{M} \prec \mathcal{N}$, such that $\mathcal{S} <^\textnormal{rank-cut} \mathcal{N}$ and $\mathrm{SSy}_\mathcal{S}(\mathcal{N}) = (\mathcal{S}, \mathcal{A})$, with a rank-initial self-embedding $j : \mathcal{N} \rightarrow \mathcal{N}$, such that for some $\mathcal{S} \subsetneq \mathcal{S}' <^\rnk \mathcal{N}$, we have that $j$ is contractive on $\mathcal{S}' \setminus \mathcal{S}$ and that $\mathrm{Fix}(j) \cap \mathcal{S'} = \mathcal{S}$. 
\end{cor}
\begin{proof}
	Since $\mathcal{M}$ is countable, there is an $(\mathcal{S}, \mathcal{A})$-generic ultrafilter $\mathcal{U}$. Let $\mathcal{N} = \mathcal{G}_{\mathcal{U}, (\mathcal{M}, \mathcal{F})}(\mathbb{Q})$. We apply Theorem \ref{Gaifman thm general}: Let $\mathcal{S}' = \mathcal{S}(\mathbb{Q})$. Let $\hat j$ be a contractive self-embedding of $\mathbb{Q}$. Now the result follows from (\ref{elem gen}), (\ref{cons gen}), (\ref{init gen}), (\ref{equal gen}) and (\ref{contr gen}).
\end{proof}

\section{Embeddings between models of set theory}\label{Existence of embeddings between models of set theory}

In \S 4 of \cite{Fri73}, a back-and-forth technique was pioneered that utilizes partial satisfaction relations and the ability of non-standard models to code types over themselves (as indicated in Lemma \ref{hierarchical type coded}). Here we will prove refinements of set theoretic results in \S 4 of \cite{Fri73}, as well as generalizations of arithmetic results in \cite{BE18} and \cite{Res87b} to set theory. We will do so by casting the results in the conceptual framework of forcing. We do so because: 
\begin{itemize}
	\item The conceptual framework of forcing allows a {\em modular} design of the proofs, clarifying which assumptions are needed for what, and whereby new pieces can be added to a proof without having to re-write the other parts. So it serves as an efficient bookkeeping device.
	\item It enables us to look at these results from a different angle, and potentially apply theory that has been developed for usage in forcing.
\end{itemize}

\begin{lemma}\label{Friedman lemma}
	Let $\mathcal{M} \models \mathrm{KP}^\mathcal{P} + \Sigma_1^\mathcal{P}\textnormal{-Separation}$ and $\mathcal{N} \models \mathrm{KP}^\mathcal{P}$ be countable and non-standard, and let $\mathcal{S}$ be such that $\mathcal{S} \leq^{\rnk, \mathrm{topless}} \mathcal{M}$ and $\mathcal{S} \leq^{\rnk, \mathrm{topless}} \mathcal{N}$. Moreover, let $\mathbb{P} = \llbracket \mathcal{M} \preceq_{\Sigma_1^\mathcal{P}, \mathcal{S}} \mathcal{N}_\beta \rrbracket^{<\omega}$ and let $\beta \in \mathrm{Ord}^\mathcal{N} \setminus \mathcal{S}$.
	\begin{enumerate}[{\normalfont (a)}]
		\item\label{Friedman lemma forwards} If $\mathrm{SSy}_\mathcal{S}(\mathcal{M}) \leq \mathrm{SSy}_{\mathcal{S}}(\mathcal{N})$, then
		\[
		\mathcal{C}_{m} =_\df \big\{ f \in \mathbb{P} \mid m \in \dom(f\hspace{2pt}) \big\}
		\]
		is dense in $\mathbb{P}$, for each $m \in \mathcal{M}$.
		\item\label{Friedman lemma backwards} If $\mathrm{SSy}_\mathcal{S}(\mathcal{M}) \cong \mathrm{SSy}_{\mathcal{S}}(\mathcal{N})$, then 
		\begin{align*}
		\mathcal{D}_{m, n} =_\df \big\{ f \in \mathbb{P} \mid & m \in \dom(f\hspace{2pt}) \wedge \\
		& ((\mathcal{N} \models \rnk(n) \leq \rnk(m)) \rightarrow n \in \mathrm{image}(f\hspace{2pt})) \big\}
		\end{align*}		
		is dense in $\mathbb{P}$, for each $m \in \mathcal{M}$ and $n \in \mathcal{N}$.
		\item\label{Friedman lemma downwards} If $\mathcal{N} = \mathcal{M}$, then 
		\[
		\mathcal{E}_{\alpha} =_\df \big\{ f \in \mathbb{P} \mid \exists m \in \dom(f\hspace{2pt}) . (f\hspace{2pt}(m) \neq m \wedge \mathcal{M} \models \rnk(m) = \alpha) \big\}
		\]
		is dense in $\mathbb{P}$, for each $\alpha \in \mathrm{Ord}^\mathcal{M} \setminus \mathcal{S}$.
 	\end{enumerate}	
\end{lemma}
Note that $\mathcal{N}_\beta$ is rank-initial in $\mathcal{N}$, so by absoluteness of $\Delta_0^\mathcal{P}$-formulas over rank-initial substructures, we have for any $n \in \mathcal{N}_\beta$, for any $s \in \mathcal{S}$ and for any $\delta(x, y, z) \in \Delta_0^\mathcal{P}[x, y, z]$, that
\[
\mathcal{N} \models \exists x \in V_\beta . \delta(x, n, s) \Leftrightarrow \mathcal{N}_\beta \models \exists x . \delta(x, n, s).
\]
\begin{proof}
	We may assume that $\mathcal{S} \subseteq \mathcal{M}$ and $\mathcal{S} \subseteq \mathcal{N}$, rank-initially and toplessly.
	
	(\ref{Friedman lemma forwards}) Let $g \in \mathbb{P}$. Unravel $g$ as a $\gamma$-sequence of ordered pairs $\langle m_\xi, n_\xi\rangle_{\xi < \gamma}$, where $\gamma < \omega$. Let $m_\gamma \in \mathcal{M}$ be arbitrary. We need to find $f$ in $\mathbb{P}$ extending $g$, such that $m_\gamma \in \dom(f\hspace{2pt})$.
	
	Using $\mathrm{Sat}_{\Sigma_1^\mathcal{P}}$, we have by Lemma \ref{hierarchical type coded} and $\Sigma_1^\mathcal{P}$-Separation that there is a code $c$ in $\mathcal{M}$ for 
	\begin{align*}
	\{\delta(x, \langle y_\xi \rangle_{\xi < \gamma}, s) \mid & \delta \in \Delta_0^\mathcal{P}[x, \langle y_\xi \rangle_{\xi < \gamma}, z] \cap \mathcal{S} \wedge s \in \mathcal{S} \wedge \\
	& \mathcal{M} \models \exists x . \delta(x, \langle m_\xi \rangle_{\xi < \gamma}, s)\}.
	\end{align*}
	By $\mathrm{SSy}(\mathcal{M}) \leq \mathrm{SSy}(\mathcal{N})$, this set has a code $d$ in $\mathcal{N}$ as well.
	We define the formulae 
	\begin{align*}
	\phi(\zeta) &\equiv \exists y . \forall  \delta  \in c \cap V_\zeta . \exists x . \mathrm{Sat}_{\Delta_0^\mathcal{P}}( \delta, x, \langle m_\xi \rangle_{\xi < \gamma}, y), \\
	\phi_{<\beta}(\zeta) &\equiv \exists y \in V_\beta . \forall  \delta  \in d \cap V_\zeta . \exists x \in V_\beta . \mathrm{Sat}_{\Delta_0^\mathcal{P}}( \delta, x, \langle n_\xi \rangle_{\xi < \gamma}, y).
	\end{align*}
	Since $\mathrm{Sat}_{\Delta_0^\mathcal{P}} \in \Delta_1^\mathcal{P}$, we have $\phi \in \Sigma_1^\mathcal{P}$ and $\phi_{< \beta} \in \Delta_1^\mathcal{P}$. For every ordinal $\zeta \in \mathcal{S}$, we have $c \cap V_\zeta = d \cap V_\zeta \in \mathcal{S}$, and as witnessed by $m_{\gamma}$, $\mathcal{M} \models \phi(\zeta)$. So by $\Sigma_1^\mathcal{P}$-elementarity of $g$, $\mathcal{N} \models \phi_{<\beta}(\zeta)$ for every ordinal $\zeta \in \mathcal{S}$. Since $\mathcal{S}$ is topless in $\mathcal{N}$, there is by $\Delta_1^\mathcal{P}$-Overspill a non-standard ordinal $\nu$ in $\mathcal{N}$, such that $\mathcal{N} \models \phi_{<\beta}(\nu)$. Set $n_{\gamma}$ to some witness of this fact and note that $n_\gamma \in^\mathcal{N} V_\beta^\mathcal{N}$. Put $f = g \cup \{\langle m_\gamma, n_\gamma \rangle \}$. We proceed to verify that $f$ is $\Sigma_1^\mathcal{P}$-elementary. Let $s \in \mathcal{S}$ and let $\delta(x, \langle y_\xi \rangle_{\xi < \gamma+1}, s) \in \Delta_0^\mathcal{P}[x, \langle y_\xi \rangle_{\xi < \gamma+1}, z]$. Now, as desired,
	\begin{align*}
	\mathcal{M} \models \exists x . \delta(x, \langle m_\xi \rangle_{\xi < \gamma + 1}, s) & \Rightarrow \ulcorner \delta(x, \langle y_\xi \rangle_{\xi < \gamma + 1}, s) \urcorner \in d \\
	& \Rightarrow \mathcal{N} \models \exists x \in V_\beta . \delta(x, \langle n_\xi \rangle_{\xi < \gamma + 1}, s).
	\end{align*}
	The second implication follows from the properties of $\mathrm{Sat}$.
	
	(\ref{Friedman lemma backwards}) Let $g \in \mathbb{P}$. Unravel $g$ as a $\gamma$-sequence of ordered pairs $\langle m_\xi, n_\xi\rangle_{\xi < \gamma}$, where $\gamma < \omega$. Let $n_\gamma \in \mathcal{N}$, such that there is $\xi < \gamma$ for which $\mathcal{N} \models \rnk(n_\gamma) \leq \rnk(m_\xi)$. We need to find $f$ in $\mathbb{P}$ extending $g$, such that $n_\gamma \in \mathrm{image}(f\hspace{2pt})$.
	
	Let $d\hspace{1pt}'$ be a code in $\mathcal{N}$ for 
	\begin{align*}
	\{\delta(x, \langle y_\xi \rangle_{\xi < \gamma}, s) \mid & \delta \in \Delta_0^\mathcal{P}[x, \langle y_\xi \rangle_{\xi < \gamma}, z] \cap \mathcal{S} \wedge s \in \mathcal{S} \wedge \\
	& \mathcal{N} \models \forall x \in V_\beta . \delta(x, \langle m_\xi \rangle_{\xi < \gamma}, s)\}.
	\end{align*}
	and let $c'$ be its code in $\mathcal{M}$. We define the formulae
	\begin{align*}
	\psi_{< \beta}(\zeta) \equiv \phantom{.} & \exists y \subseteq V_{\sup(\rnk(\langle n_\xi \rangle_{\xi < \gamma}))}. \forall  \delta  \in d\hspace{1pt}' \cap V_\zeta . \forall x \in V_\beta . \\
	&\mathrm{Sat}_{\Delta_0^\mathcal{P}}( \delta , x, \langle n_\xi \rangle_{\xi < \gamma}, y), \\
	\psi(\zeta) \equiv \phantom{.} & \exists y \subseteq V_{\sup(\rnk(\langle m_\xi \rangle_{\xi < \gamma}))}. \forall  \delta  \in c' \cap V_\zeta . \forall x .\\
	& \mathrm{Sat}_{\Delta_0^\mathcal{P}}( \delta , x, \langle m_\xi \rangle_{\xi < \gamma}, y).
	\end{align*}
	Since $\mathrm{Sat}_{\Delta_0^\mathcal{P}}$ is $\Delta_1^\mathcal{P}$, $\psi_{< \beta}$ is $\Delta_1^\mathcal{P}$ and $\psi$ is $\Pi_1^\mathcal{P}$. Moreover, $d\hspace{1pt}' \cap V_\zeta = c' \cap V_\zeta$, and $\psi_{< \beta}$ is witnessed by $n_{\gamma}$, for every ordinal $\zeta \in \mathcal{S}$. So it follows from the (dual of the) $\Sigma_1^\mathcal{P}$-elementarity of $g$ that $\psi$ is satisfied in $\mathcal{M}$ for every ordinal $\zeta \in \mathcal{S}$, whence by $\Pi_1^\mathcal{P}$-Overspill we have $\mathcal{M} \models \psi(\mu)$ for some non-standard ordinal $\mu \in \mathcal{M}$. Let $m_\gamma$ be some witness of this fact, and put $f = g \cup \{\langle m_\gamma, n_\gamma \rangle \}$. We proceed to verify that $f$ is $\Sigma_1^\mathcal{P}$-elementary. Let $s \in \mathcal{S}$ and let $\delta(x, \langle y_\xi \rangle_{\xi < \gamma}, s) \in \Delta_0^\mathcal{P}[x, \langle y_\xi \rangle_{\xi < \gamma}, z]$. Now, as desired,
	\begin{align*}
	\mathcal{N} \models \forall x \in V_\beta . \delta(x, \langle n_\xi \rangle_{\xi < \gamma + 1}, s) & \Rightarrow \ulcorner \delta(x, \langle y_\xi \rangle_{\xi < \gamma + 1}, s) \urcorner \in d\hspace{1pt}' \\
	& \Rightarrow \mathcal{M} \models \forall x . \delta(x, \langle m_\xi \rangle_{\xi < \gamma + 1}, s).
	\end{align*}
	The second implication follows from the properties of $\mathrm{Sat}$.
	
	(\ref{Friedman lemma downwards}) Let $\alpha \in \mathrm{Ord}^\mathcal{M} \setminus \mathcal{S}$, and let $g \in \mathbb{P}$. Unravel $g$ as a $\gamma$-sequence of ordered pairs $\langle m_\xi, m'_\xi\rangle_{\xi < \gamma}$, where $\gamma < \omega$. We need to find $m_\gamma \neq n_\gamma$, such that $\mathcal{M} \models \rnk(m_\gamma) = \alpha$ and $g \cup \{\langle m_\gamma, n_\gamma \rangle \} \in \mathbb{P} = \llbracket \mathcal{M} \preceq_{\Sigma_1^\mathcal{P}, \mathcal{S}} \mathcal{N}_\beta \rrbracket^{< \omega}$. Note that by rank-initiality and toplessness, there is $\alpha\hspace{1pt}' \in \mathrm{Ord}^\mathcal{M} \setminus \mathcal{S}$, such that $(\alpha\hspace{1pt}' + 3 \leq \alpha)^\mathcal{M}$ and $(V_{\alpha\hspace{1pt}'})^\mathcal{M}_\mathcal{M} \supseteq \mathcal{S}$. 
	
	We proceed to work in $\mathcal{M}$: The set $V_{\alpha+1} \setminus V_\alpha$ of sets of rank $\alpha$ has cardinality $\beth_{\alpha+1}^\mathcal{M}$, while the set $\mathcal{P}(V_{\alpha\hspace{1pt}'} \times V_{\alpha\hspace{1pt}'}) \subseteq V_{\alpha\hspace{1pt}' + 3}$ has the strictly smaller cardinality $\beth_{\alpha\hspace{1pt}' + 3}$. (Here we used $\mathcal{M} \models \textnormal{Powerset}$, and the recursive definition $\beth_0 = 0$, $\beth_{\xi + 1} = 2^{\beth_\xi}$, $\beth_\xi = \sup\{\beth_\zeta \mid \zeta < \xi\}$ for limits $\xi$.) 
	We define a function $t : V_{\alpha+1} \setminus V_\alpha \rightarrow V_{\alpha\hspace{1pt}'+3}$ by
	\begin{align*}
	t(v) = \{ & \langle \delta, s \rangle \in (\Delta_0^\mathcal{P}[x, \langle y_\xi \rangle_{\xi < \gamma}, y_\gamma, z] \cap V_{\alpha\hspace{1pt}'}) \times V_{\alpha\hspace{1pt}'} \mid \\
	& \exists x . \mathrm{Sat}_{\Delta_0^\mathcal{P}}(\delta, x, \langle m_\xi \rangle_{\xi < \gamma}, v, s) \},
	\end{align*}
	for each $v \in V_{\alpha+1} \setminus V_\alpha$. $t$ exists by $\Sigma_1^\mathcal{P}$-Separation. Since $t$ has a domain of strictly larger cardinality than its co-domain, there are $m, m'$ of rank $\alpha$, such that $m \neq m'$ and $t(m) = t(m')$.
	
	We return to working in the meta-theory: $m$ and $m'$ have the same $\Sigma_1^\mathcal{P}$-type with parameters in $\mathcal{S} \cup \langle m_\xi \rangle_{\xi < \gamma}$. In other words, for every $s \in \mathcal{S}$ and every $\delta(x, \langle y_\xi \rangle_{\xi < \gamma}, y_\gamma, z) \in  \Delta_0^\mathcal{P}[x, \langle y_\xi \rangle_{\xi < \gamma}, y_\gamma, z]$, we have 
	\[
	\mathcal{M} \models \exists x . \delta(x, \langle m_\xi \rangle_{\xi < \gamma}, m, s) \leftrightarrow \exists x . \delta(x, \langle m_\xi \rangle_{\xi < \gamma}, m', s). \tag{$\dagger$}
	\]
	On the other hand, by (\ref{Friedman lemma forwards}) and by $g \in \mathbb{P}$, there are $n$, $n'$, such that for every $s \in \mathcal{S}$ and every $\delta(x, \langle y_\xi \rangle_{\xi < \gamma}, y_\gamma, y_{\gamma + 1}, z) \in  \Delta_0^\mathcal{P}[x, \langle y_\xi \rangle_{\xi < \gamma}, y_\gamma, y_{\gamma + 1}, z]$, we have
	\[
	\mathcal{M} \models \exists x . \delta(x, \langle m_\xi \rangle_{\xi < \gamma}, m, m', s) \rightarrow \exists x \in V_\beta . \delta(x, \langle n_\xi \rangle_{\xi < \gamma}, n, n', s). \tag{$\ddagger$}
	\]
	By ($\ddagger$), $n \neq n'$, whence $m \neq n$ or $m \neq n'$. If $m \neq n$, then $g \cup \{ \langle m, n \rangle \} \in \mathcal{E}_\alpha$ by ($\ddagger$). If $m \neq n'$, then by ($\dagger$) and ($\ddagger$),
	\begin{align*}
	\mathcal{M} \models \exists x . \delta(x, \langle m_\xi \rangle_{\xi < \gamma}, m, s) &\Leftrightarrow \mathcal{M} \models \exists x . \delta(x, \langle m_\xi \rangle_{\xi < \gamma}, m', s) \\
	&\Rightarrow \mathcal{M} \models \exists x \in V_\beta . \delta(x, \langle n_\xi \rangle_{\xi < \gamma}, n', s),
	\end{align*}
	so $g \cup \{ \langle m, n' \rangle \} \in \mathcal{E}_\alpha$. In either case we are done.
\end{proof}

Based on this Lemma, we can prove a theorem that refines results in \S 4 of \cite{Fri73}.

\begin{thm}[Friedman-style]\label{Friedman thm}
Let $\mathcal{M} \models \mathrm{KP}^\mathcal{P} + \Sigma_1^\mathcal{P}\textnormal{-Separation}$ and $\mathcal{N} \models \mathrm{KP}^\mathcal{P}$ be countable and non-standard, and let $\mathcal{S}$ be a shared rank-initial topless substructure of $\mathcal{M}$ and $\mathcal{N}$. Moreover, let $m_0 \in \mathcal{M}$, let $n_0 \in \mathcal{N}$, and let $\beta \in \mathrm{Ord}^\mathcal{N}$. Then the following are equivalent:
\begin{enumerate}[{\normalfont (a)}]
\item\label{Friedman thm emb} There is $i \in \llbracket \mathcal{M} \leq^\rnk \mathcal{N} \rrbracket$, such that $i(m_0) = n_0$ and $i(\mathcal{M}) \subseteq \mathcal{N}_\beta$.
\item\label{Friedman thm pres} $\mathrm{SSy}_\mathcal{S}(\mathcal{M}) \cong \mathrm{SSy}_{\mathcal{S}}(\mathcal{N})$, and for all $s \in \mathcal{S}$ and $\delta(x, y, z) \in \Delta_0^\mathcal{P}[x, y, z]$:
$$\mathcal{M} \models \exists x . \delta(x, m_0, s) \Rightarrow \mathcal{N} \models \exists x \in V_\beta . \delta(x, n_0, s).$$
\item\label{Friedman thm extra continuum} There is a map $g \mapsto i_g$, from sequences $g : \omega \rightarrow 2$, to embeddings $i_g : \mathcal{M} \rightarrow \mathcal{N}$ satisfying \textnormal{(\ref{Friedman thm emb})}, such that for any $g <^\mathrm{lex} g\hspace{1pt}' : \omega \rightarrow 2$, we have $i_g <^\rnk i_{g\hspace{1pt}'}$.
\item\label{Friedman thm extra topless} There is a topless embedding $i : \mathcal{M} \rightarrow \mathcal{N}$ satisfying \textnormal{(\ref{Friedman thm emb})}.
\end{enumerate}
\end{thm}

\begin{proof}
Most of the work has already been done for (\ref{Friedman thm emb}) $\Leftrightarrow$ (\ref{Friedman thm pres}). The other equivalences are proved as Lemma \ref{Friedman thm extra} below.

(\ref{Friedman thm emb}) $\Rightarrow$ (\ref{Friedman thm pres}): The first conjunct follows from Proposition \ref{emb pres}. The second conjunct follows from Proposition \ref{emb pres} and $i(\mathcal{M}) \subseteq \mathcal{N}_\beta$.

(\ref{Friedman thm pres}) $\Rightarrow$ (\ref{Friedman thm emb}): Let $\mathbb{P} = \llbracket \mathcal{M} \preceq_{\Sigma_1^\mathcal{P}, \mathcal{S}} \mathcal{N}_\beta \rrbracket^{<\omega}$. By the second conjunct of (\ref{Friedman thm pres}), the function $f_0$ defined by $(m_0 \mapsto n_0)$, with domain $\{m_0\}$, is in $\mathbb{P}$. Using Lemma \ref{generic filter existence} and Lemma \ref{Friedman lemma} (\ref{Friedman lemma forwards}, \ref{Friedman lemma backwards}), we obtain a filter $\mathcal{I}$ on $\mathbb{P}$ which contains $f_0$ and is $\{\mathcal{C}_m \mid m \in \mathcal{M}\} \cup \{\mathcal{D}_{m, n} \mid m \in \mathcal{M} \wedge n \in \mathcal{N}\}$-generic. Let $i = \bigcup \mathcal{I}$. Since $\mathcal{I}$ is downwards directed, $i$ is a function. Clearly $\mathrm{image}(i) \subseteq \mathcal{N}_\beta$. Since $\mathcal{I}$ is $\{\mathcal{C}_m \mid m \in \mathcal{M}\}$-generic, $i$ has domain $\mathcal{M}$; and since $f_0 \in \mathcal{I}$, $i(m_0) = n_0$. To see that $i$ is rank-initial, let $m \in \mathcal{M}$, and let $n \in \mathcal{N}$ such that $\mathcal{N} \models \rnk(n) \leq \rnk(i(m))$. Since $\mathcal{I} \cap \mathcal{D}_{m, n} \neq \varnothing$, we have that $n$ is in the image of $i$.
\end{proof}

Friedman's theorem is especially powerful in conjunction with the following lemma.

\begin{lemma}\label{Bound for sigma_1 lemma}
	Let $\mathcal{N} \models \mathrm{KP}^\mathcal{P} + \Sigma_1^\mathcal{P}\textnormal{-Separation}$, let $n \in \mathcal{N}$ and let $\mathcal{S}$ be a bounded substructure of $\mathcal{N}$. Then there is an ordinal $\beta \in \mathcal{N}$, such that for each $s \in \mathcal{S}$, and for each $\delta(x, y, z) \in \Delta_0^\mathcal{P}[x, y, z]$:
	$$(\mathcal{N}, n, s) \models (\exists x . \delta(x, n, s)) \leftrightarrow (\exists x \in V_\beta . \delta(x, n, s)).$$
\end{lemma}
\begin{proof}
	Let $\nu$ be an infinite ordinal in $\mathcal{N}$ such that $\mathcal{S} \subseteq \mathcal{N}_\nu$. We work in $\mathcal{N}$: Let $A = \Delta_0^\mathcal{P}[x, y, z] \times V_\nu$. By Strong $\Sigma_1^\mathcal{P}$-Collection there is a set $B$, such that for all $\langle  \delta , t\rangle \in A$, if $\exists x . \mathrm{Sat}_{\Delta_0^\mathcal{P}}( \delta , x, n, t)$, then there is $b \in B$ such that $\mathrm{Sat}_{\Delta_0^\mathcal{P}}(\delta, b, n, t)$. Setting $\beta = \rnk(B)$, the claim of the lemma follows from the properties of $\mathrm{Sat}_{\Delta_0^\mathcal{P}}$.
\end{proof}

\begin{lemma}\label{Bound for pi_2 lemma}
Let $\mathcal{N} \models \mathrm{KP}^\mathcal{P} + \Sigma_2^\mathcal{P}\textnormal{-Separation}$, let $n \in \mathcal{N}$ and let $\mathcal{S}$ be a bounded substructure of $\mathcal{N}$. Then there is an ordinal $\beta \in \mathcal{N}$, such that for each $s \in \mathcal{S}$, and for each $\delta(x, x\hspace{1pt}', y, z) \in \Delta_0^\mathcal{P}[x, x\hspace{1pt}', y, z]$:
	$$(\mathcal{N}, n, s) \models (\forall x . \exists x\hspace{1pt}' . \delta(x, x\hspace{1pt}', n, s)) \rightarrow (\forall x \in V_\beta . \exists x\hspace{1pt}' \in V_\beta . \delta(x, x\hspace{1pt}', n, s)).$$
\end{lemma}
\begin{proof}
Let $\beta_0$ be an infinite ordinal in $\mathcal{N}$ such that $\mathcal{S} \subseteq \mathcal{N}_{\beta_0}$. We work in $\mathcal{N}$: By $\Sigma_2^\mathcal{P}$-Separation (which is equivalent to $\Pi_2^\mathcal{P}$-Separation), let 
$$D = \{ \langle \delta, t \rangle \in \Delta_0^\mathcal{P}[x, x\hspace{1pt}', y, z] \times V_{\beta_0} \mid \forall x . \exists x\hspace{1pt}' . \mathrm{Sat}_{\Delta_0}^\mathcal{P}(\delta, x, x\hspace{1pt}', n, t) \}.$$ 
Recursively, for each $k < \omega$, let $\beta_{k+1}$ be the least ordinal such that 
$$\forall \langle \delta, t \rangle \in D . \big( \forall x \in V_{\beta_k} . \exists x\hspace{1pt}' \in V_{\beta_{k+1}} . \mathrm{Sat}_{\Delta_0^\mathcal{P}}(\delta, x, x\hspace{1pt}', n, s) \big).$$ 
The existence of the set $\{ \beta_k \mid k < \omega \}$ follows from $\Sigma_1^\mathcal{P}$-Recursion, because the functional formula defining the recursive step is $\Sigma_1^\mathcal{P}$, as seen when written out as $\phi(\beta_k, \beta_{k+1}) \wedge \forall \gamma < \beta_{k+1} . \neg \phi(\beta_k, \gamma),$ where $\phi(\beta_k, \beta_{k+1})$ is the formula $\forall \langle \delta, t \rangle \in D . \big( \forall x \in V_{\beta_k} . \exists x\hspace{1pt}' \in V_{\beta_{k+1}} . \mathrm{Sat}_{\Delta_0^\mathcal{P}}(\delta, x, x\hspace{1pt}', n, s) \big)$. Put $\beta = \mathrm{sup}\{ \beta_k \mid k < \omega \}$.

Let $s \in \mathcal{S}$ and let $\delta(x, x\hspace{1pt}', y, z) \in \Delta_0^\mathcal{P}[x, x\hspace{1pt}', y, z]$. To verify that 
$$(\mathcal{N}, n, s) \models (\forall x . \exists x\hspace{1pt}' . \delta(x, x\hspace{1pt}', n, s)) \rightarrow (\forall x \in V_\beta . \exists x\hspace{1pt}' \in V_\beta . \delta(x, x\hspace{1pt}', n, s)),$$ 
we work in $(\mathcal{N}, n, s)$: Suppose that $\forall x . \exists x\hspace{1pt}' . \delta(x, x\hspace{1pt}', n, s)$, and let $x \in V_\beta$. Then $x \in V_{\beta_k}$ for some $k < \omega$. By construction, there is $x\hspace{1pt}' \in V_{\beta_{k+1}}$, such that $\mathrm{Sat}_{\Delta_0^\mathcal{P}}(\delta, x, x\hspace{1pt}', n, s)$. So by the properties of $\mathrm{Sat}_{\Delta_0^\mathcal{P}}$, we have $\delta(x, x\hspace{1pt}', n, s)$, as desired.
\end{proof}

\begin{lemma}\label{Friedman thm extra 2}
	Under the assumptions of Theorem \ref{Friedman thm}, for each embedding $i_1 : \mathcal{M} \rightarrow \mathcal{N}$ satisfying \textnormal{(\ref{Friedman thm emb})} of Theorem \ref{Friedman thm}, there is an embedding $i_0 <^\rnk i_1$ satisfying \textnormal{(\ref{Friedman thm emb})}.
\end{lemma}
\begin{proof}
	By Lemma \ref{Bound for sigma_1 lemma}, there is $\alpha \in \mathrm{Ord}^\mathcal{M}$, such that 
	for all $s \in \mathcal{S}$ and all $\delta(x, y, z) \in \Delta_0^\mathcal{P}[x, y, z]$:
	\[
	\mathcal{M} \models \exists x . \delta(x, m, s) \Leftrightarrow \mathcal{M} \models \exists x \in V_\alpha . \delta(x, m, s).
	\]
	By Proposition \ref{emb pres}(\ref{emb pres Delta_0^P}) applied to $i_1$, we have for all $s \in \mathcal{S}$ and all $\delta(x, y, z) \in \Delta_0^\mathcal{P}[x, y, z]$ that
	\[
	\mathcal{M} \models \exists x \in V_\alpha . \delta(x, m, s) \Rightarrow \mathcal{N} \models \exists x \in V_{i_1(\alpha)} . \delta(x, n, s),
	\]
	and consequently that 
	\[
	\mathcal{M} \models \exists x . \delta(x, m, s) \Rightarrow \mathcal{N} \models \exists x \in V_{i_1(\alpha)} . \delta(x, n, s).
	\]
	So by (\ref{Friedman thm pres}) $\Rightarrow$ (\ref{Friedman thm emb}), there is a rank-initial embedding $i_0 : \mathcal{M} \rightarrow \mathcal{N}$, such that $i_0(m) = n$ and $i_0(\mathcal{M}) \subseteq V_{i_1(\alpha)}^\mathcal{N}$. Since $i_1(\alpha) \in i_1(\mathcal{M}) \setminus i_0(\mathcal{M})$, we are done.
\end{proof}

\begin{lemma}\label{Friedman thm extra}
	These statements are equivalent to (\ref{Friedman thm emb}) in Theorem \ref{Friedman thm}:
	\begin{enumerate}[{\normalfont (a)}]
		\setcounter{enumi}{2}
		\item There is a map $g \mapsto i_g$, from sequences $g : \omega \rightarrow 2$, to embeddings $i_g : \mathcal{M} \rightarrow \mathcal{N}$ satisfying \textnormal{(\ref{Friedman thm emb})}, such that for any $g <^\mathrm{lex} g\hspace{1pt}' : \omega \rightarrow 2$, we have $i_g <^\rnk i_{g\hspace{1pt}'}$.
		\item There is a topless embedding $i : \mathcal{M} \rightarrow \mathcal{N}$ satisfying \textnormal{(\ref{Friedman thm emb})}.
	\end{enumerate}
\end{lemma}
\begin{proof}
	It suffices to show that (\ref{Friedman thm emb}) $\Rightarrow$ (\ref{Friedman thm extra continuum}) $\Rightarrow$ (\ref{Friedman thm extra topless}).
	
	(\ref{Friedman thm emb}) $\Rightarrow$ (\ref{Friedman thm extra continuum}): 
	 Let $(a_\xi)_{\xi < \omega}$ and $(b_\xi)_{\xi < \omega}$ be enumerations of $\mathcal{M}$ and $\mathcal{N}$, respectively, with infinitely many repetitions of each element. For each $g : \omega \rightarrow 2$, we shall construct a distinct $i_g : \mathcal{M} \rightarrow \mathcal{N}$. To do so, we first construct approximations of the $i_g$.
	 
	 For any $\gamma < \omega$, we allow ourselves to denote any function $f : \gamma \rightarrow 2$ as an explicit sequence of values $f\hspace{2pt}(0),f\hspace{2pt}(1), \dots, f\hspace{2pt}(\gamma - 1)$. For each $\gamma < \omega$, we shall construct a finite subdomain $D_\gamma \subseteq \mathcal{M}$, and for each $f : \gamma \rightarrow 2$, we shall construct an embedding $i_f$. We do so by this recursive construction on $\gamma < \omega$:
	 \begin{enumerate}
	 	\item For $i_\varnothing : \mathcal{M} \rightarrow \mathcal{N}$, choose any embedding satisfying (\ref{Friedman thm emb}).
	 	\item Put $D_\varnothing = \{m\}$.
	 	\item Suppose that $i_f$ has been constructed for some $f : \gamma \rightarrow 2$, where $\gamma < \omega$. Put $i_{f, 1} = i_f$. Applying Lemma \ref{Friedman thm extra 2} to $i_{f,1}$, with $D_\gamma^\mathcal{M}$ in place of $m$ and with $i_f\hspace{2pt}(D_\gamma^\mathcal{M})$ in place of $n$, we choose an embedding $i_{f, 0} : \mathcal{M} \rightarrow \mathcal{N}$ such that 
	 	\begin{enumerate}[(i)]
	 		\item $i_{f, 0}, i_{f, 1}$ are rank-initial, with all values of rank below $\beta$ in $\mathcal{N}$.
	 		\item $i_{f, 0}\restriction_{D_\gamma} = i_{f, 1}\restriction_{D_\gamma} = i_f\restriction_{D_\gamma}$,
	 		\item $i_{f,0} <^\rnk i_{f,1}$.
	 	\end{enumerate}
	 	\item Put $D_{\gamma + 1}$ to be a finite subdomain of $\mathcal{M}$, such that
	 	\begin{enumerate}[(i)]
	 		\item $D_{\gamma} \subseteq D_{\gamma + 1}$, 
	 		\item $a_\gamma \in D_{\gamma + 1}$,
	 		\item If $\mathcal{N} \models \rnk(b_\gamma) \leq \sup\{\rnk(a) \mid a \in D_\gamma\}$, then we have that $i_{f, 0}^{-1}(b_\gamma), i_{f, 1}^{-1}(b_\gamma) \in D_{\gamma + 1}$,
	 		\item For each $f : \gamma \rightarrow 2$, we have that $i_{f, 1}^{-1}(\nu) \in D_{\gamma + 1}$, for some $\nu \in i_{f, 1}(\mathrm{Ord}^\mathcal{M}) \setminus i_{f, 0}(\mathrm{Ord}^\mathcal{M})$.
	 	\end{enumerate}
	 \end{enumerate}
	 Note that every $a \in \mathcal{M}$ is in $D_\gamma$ for some $\gamma < \omega$. Moreover, for every $\gamma < \omega$, if $f <^\mathrm{lex} f\hspace{2pt}' : \gamma \rightarrow 2$, then $i_f <^\rnk i_{f\hspace{2pt}'}$.
	 
	 Now, for each $g : \omega \rightarrow 2$, define $i_g : \mathcal{M} \rightarrow \mathcal{N}$ by
	 \[
	 	i_g(a) = i_{g\restriction_{\gamma}}(a), 
	 \]
	 for each $a \in \mathcal{M}$, where $\gamma < \omega$ is the least such that $a \in D_\gamma$.
	 Note that for each $\gamma < \omega$, $i_g\restriction_{D_\gamma} = i_{g\restriction_\gamma}$, so  if $a \in D_\gamma$, then $i_g(a) = i_{g\restriction_\gamma}(a)$. We now verify that these $i_g$ have the desired properties. Let $g : \omega \rightarrow 2$.
	 \begin{enumerate}
	 	\item $i_g$ is an embedding: Let $\phi(x)$ be a quantifier free formula and let $a \in \mathcal{M}$. Then $a \in D_\gamma$ for some $\gamma < \omega$, so since $i_{g\restriction_{\gamma}}$ is an embedding, $\mathcal{M} \models \phi(a) \Rightarrow \mathcal{N} \models \phi(i_g(a)).$
	 	\item $i_g(m) = n$: $m \in D_\varnothing$ and $i_g\restriction_{D_\varnothing} = i_\varnothing\restriction_{D_\varnothing}$.
	 	\item $i_g(\mathcal{M}) \subseteq \mathcal{N}_\beta$: Let $a \in \mathcal{M}$ and pick $\gamma < \omega$ such that $a \in D_\gamma$. Then $i_g(a) = i_{g\restriction_\gamma}(a) \in \mathcal{N}_\beta$.
	 	\item $i_g$ is rank-initial: Let $a \in \mathcal{M}$ and $b \in \mathcal{N}$, such that $\mathcal{N} \models \rnk(b) \leq \rnk(i_g(a))$. Pick $\gamma < \omega$ such that $a \in D_\gamma$. Then $i_{g\restriction_{\gamma + 1}}^{-1}(b) \in D_{\gamma + 1}$. So $b \in i_g(\mathcal{M})$.
	 	\item If $g <^\mathrm{lex} g\hspace{1pt}' : \omega \rightarrow 2$, then $i_g <^\rnk i_{g\hspace{1pt}'}$: Let $\gamma < \omega$ be the least such that $g(\gamma) < g\hspace{1pt}'(\gamma)$. Then by construction of the approximations, $i_{g\hspace{1pt}'} >^\rnk i_{g\restriction_{\gamma + 1}} \geq^\rnk i_g$.
	 \end{enumerate}
	 
	 (\ref{Friedman thm extra continuum}) $\Rightarrow$ (\ref{Friedman thm extra topless}): Since $\mathcal{N}$ is countable, there are only $\aleph_0$ many ordinals in $\mathcal{N}$ which top a substructure, so by (\ref{Friedman thm extra continuum}) we are done.
\end{proof}

The following two corollaries are sharpen the celebrated results in \S 4 of \cite{Fri73}.

\begin{cor}\label{Friedman cor}
Let $\mathcal{M} \models \mathrm{KP}^\mathcal{P}  + \Sigma_1^\mathcal{P}\textnormal{-Separation}$ and $\mathcal{N} \models \mathrm{KP}^\mathcal{P} + \Sigma_1^\mathcal{P}\textnormal{-Separation}$ be countable and non-standard. Let $\mathcal{S}$ be a common rank-initial topless substructure of $\mathcal{M}$ and $\mathcal{N}$. Then the following are equivalent:
\begin{enumerate}[{\normalfont (a)}]
\item There is $i \in \llbracket \mathcal{M} \leq^{\rnk}_\mathcal{S} \mathcal{N} \rrbracket$.
\item[{\normalfont (a')}] There is $i \in \llbracket \mathcal{M} \leq^{\rnk, \mathrm{topless}}_\mathcal{S} \mathcal{N} \rrbracket$.
\item $\mathrm{SSy}(\mathcal{M}) = \mathrm{SSy}(\mathcal{N})$, and $\mathrm{Th}_{\Sigma_1^\mathcal{P}, S}(\mathcal{M}) \subseteq \mathrm{Th}_{\Sigma_1^\mathcal{P}, S}(\mathcal{N})$.
\end{enumerate}
\end{cor}
\begin{proof}
(\ref{Friedman thm emb}) $\Rightarrow$ (\ref{Friedman thm pres}) is proved just as for Theorem \ref{Friedman thm}.

(\ref{Friedman thm pres}) $\Rightarrow$ (\ref{Friedman thm emb}') follows from Theorem \ref{Friedman thm} by letting $\beta \in \mathcal{N}$ be as obtained from Lemma \ref{Bound for sigma_1 lemma}, and setting $m_0 = \varnothing^\mathcal{M}$ and $n_0 = \varnothing^\mathcal{N}$.
\end{proof}

\begin{cor}\label{Friedman selfembedding}
Let $\mathcal{M} \models \mathrm{KP}^\mathcal{P} + \Sigma_1^\mathcal{P}\textnormal{-Separation}$ be a countable non-standard. Let $\mathcal{S}$ be a rank-initial topless substructure of $\mathcal{M}$. Then there is a proper $i \in \llbracket \mathcal{M} \leq^{\rnk, \mathrm{topless}}_\mathcal{S} \mathcal{M} \rrbracket$, such that
\[
\forall \alpha \in \mathrm{Ord}^\mathcal{M} \setminus \mathcal{S} . \exists m \in \mathcal{M} . (\rnk^\mathcal{M}(m) = \alpha \wedge i(m) \neq m).
\]
\end{cor}
\begin{proof}
Let $\beta \in \mathcal{M}$ be be the ordinal bound obtained from Lemma \ref{Bound for sigma_1 lemma}, and let $\mathbb{P} = \llbracket \mathcal{M} \preceq_{\Sigma_1^\mathcal{P}, \mathcal{S}} \mathcal{M}_\beta \rrbracket^{<\omega}$. Put $m_0 = n_0 = \varnothing^\mathcal{M}$. We adjust the proof of (\ref{Friedman thm pres}) $\Rightarrow$ (\ref{Friedman thm emb}) in Theorem \ref{Friedman thm}, by setting $\mathcal{I}$ to be a $\{\mathcal{C}_m \mid m \in \mathcal{M}\} \cup \{\mathcal{D}_{m, n} \mid m \in \mathcal{M} \wedge n \in \mathcal{N}\} \cup \{\mathcal{E}_{\alpha} \mid \alpha \in \mathrm{Ord}^\mathcal{M} \setminus \mathcal{S}\}$-generic filter on $\mathbb{P}$ (utilizing Lemma \ref{Friedman lemma} (\ref{Friedman lemma downwards})). Put $i = \bigcup \mathcal{I}$. It only remains to verify that
\[
\forall \alpha \in \mathrm{Ord}^\mathcal{M} \setminus \mathcal{S} . \exists m \in \mathcal{M} . (\rnk^\mathcal{M}(m) = \alpha \wedge i(m) \neq m).
\]
But this follows from that $\mathcal{I}$ intersects $\mathcal{E}_\alpha$, for each $\alpha \in \mathrm{Ord}^\mathcal{M} \setminus \mathcal{S}$.
\end{proof}

The result above says in particular that every countable non-standard model of $\mathrm{KP}^\mathcal{P} + \Sigma_1^\mathcal{P}\textnormal{-Separation}$ has a proper rank-initial self-embedding. As a remark, there is a related theorem by Hamkins, where no initiality is required from the embedding, established in \cite{Ham13}. Citing from this article's abstract: ``every countable model of set theory $\langle M, \in^M \rangle$, including every well-founded model, is isomorphic to a submodel of its own constructible universe $\langle L^M, \in^M \rangle$ by means of an embedding $j : M \rightarrow L^M$''.

\begin{thm}[Wilkie-style]\label{Wilkie theorem}
	Suppose that $\mathcal{M} \models \mathrm{KP}^\mathcal{P}  + \Sigma_1^\mathcal{P}\textnormal{-Separation}$ $+ \Pi_2^\mathcal{P}\textnormal{-Foundation}$ and $\mathcal{N} \models \mathrm{KP}^\mathcal{P}$ are countable and non-standard. Let $\mathcal{S}$ be a common rank-initial topless substructure of $\mathcal{M}$ and $\mathcal{N}$. Let $\beta \in \mathrm{Ord}^{\mathcal{N}}$. Then the following are equivalent:
	\begin{enumerate}[{\normalfont (a)}]
		\item\label{Wilkie emb} For any ordinal $\alpha <^\mathcal{N} \beta$, there is $i \in \llbracket \mathcal{M} \leq^{\rnk}_\mathcal{S} \mathcal{N} \rrbracket$, such that  $\mathcal{N}_\alpha \subseteq i(\mathcal{M}) \subseteq \mathcal{N}_\beta$.
		\item[{\normalfont (a')}] For any ordinal $\alpha <^\mathcal{N} \beta$, there is $i \in \llbracket \mathcal{M} \leq^{\rnk, \mathrm{topless}}_\mathcal{S} \mathcal{N} \rrbracket$, such that  $\mathcal{N}_\alpha \subseteq i(\mathcal{M}) \subseteq \mathcal{N}_\beta$.
		\item\label{Wilkie pres} $\mathrm{SSy}_\mathcal{S}(\mathcal{M}) = \mathrm{SSy}_\mathcal{S}(\mathcal{N})$, and for all $s \in \mathcal{S}$ and $\delta(x, y, z) \in \Delta_0^\mathcal{P}[x, y, z]$:
$$(\mathcal{M}, s) \models \forall x . \exists y . \delta(x, y, s) \Rightarrow (\mathcal{N}, s) \models \forall x \in V_\beta . \exists y \in V_\beta . \delta(x, y, s).$$
	\end{enumerate}
\end{thm}
\begin{proof}
(\ref{Wilkie emb}) $\Rightarrow$ (\ref{Wilkie pres}): It is easy to see that (\ref{Wilkie emb}) $\Rightarrow$ $\beta$ is a limit ordinal in $\mathcal{N}$. Let $s \in \mathcal{S}$. Let $\delta(x, y, z)$ be $\Delta_0^\mathcal{P}[x, y, z]$ and assume that $\mathcal{M} \models \forall x . \exists y . \delta(x, y, s)$. Given Theorem \ref{Friedman thm}, it only remains to show that $\mathcal{N} \models \forall x \in V_\beta . \exists y \in V_\beta . \delta(x, y, s)$. Let $a \in \mathcal{N}_\beta$ be arbitrary and set $\alpha = (\rnk(a)+1)^\mathcal{N}$. Since $\beta$ is a limit, $\alpha < \beta$. By (\ref{Wilkie emb}), there is a rank-initial embedding $i : \mathcal{M} \rightarrow \mathcal{N}$, such that $a \in i(\mathcal{M}) \subseteq \mathcal{N}_\beta$. Pick $m \in \mathcal{M}$ such that $\mathcal{M} \models \delta(i^{-1}(a), m, s)$. Then $i(m) \in \mathcal{N}_\beta$, and by Proposition \ref{emb pres}, $\mathcal{N} \models \delta(a, i(m), s)$, as desired.

(\ref{Wilkie pres}) $\Rightarrow$ (\ref{Wilkie emb}'): Let $n = V_\alpha^\mathcal{N}$. Using Lemma \ref{hierarchical type coded} and $\Delta_1^\mathcal{P}$-Collection, let $d$ be a code in $\mathcal{N}$ for the following set:
\[
\{\delta(x, y, s) \mid \delta \in \Delta_0^\mathcal{P}, s \in \mathcal{S}, \mathcal{N} \models \forall x \in V_\beta . \delta(x, n, s)\}.
\]

Using $\mathrm{SSy}(\mathcal{M}) = \mathrm{SSy}(\mathcal{N})$ let $c$ be a code for this set in $\mathcal{M}$. Define the formulae
\begin{align*}
\phi_{< \beta}(\zeta) &\equiv \exists y \in V_\beta . \forall \delta \in d \cap V_\zeta . \forall x \in V_\beta . \mathrm{Sat}_{\Delta_0^\mathcal{P}}(\delta, x, n, s) \\
\phi(\zeta) &\equiv \exists y . \forall \delta \in c \cap V_\zeta . \forall x . \mathrm{Sat}_{\Delta_0^\mathcal{P}}(\delta, x, n, s)
\end{align*}
Note that $\phi$ is $\Sigma_2^\mathcal{P}$. Moreover, $c \cap V_\zeta = d \cap V_\zeta$ and $n$ witnesses $\phi_{<\beta}(\zeta)$, for all ordinals $\zeta \in \mathcal{S}$. So by the second conjunct of (\ref{Wilkie pres}), $\mathcal{M} \models \phi(\zeta)$ for all ordinals $\zeta \in \mathcal{S}$, whence by $\Sigma_2^\mathcal{P}$-Overspill, $\mathcal{M} \models \phi(\mu)$ for some non-standard ordinal $\mu \in \mathcal{M}$. Letting $m \in \mathcal{M}$ be a witness of this fact, we have that $m$ realizes $\mathrm{tp}_{\Pi_1^\mathcal{P}, \mathcal{S}}(n)$. Now (\ref{Wilkie emb}) is obtained by plugging $m$ and $n$ into Theorem $\ref{Friedman thm}$.
\end{proof}

\begin{cor}
Let $\mathcal{M} \models \mathrm{KP}^\mathcal{P}  + \Sigma_1^\mathcal{P}\textnormal{-Separation} + \Pi_2^\mathcal{P}\textnormal{-Foundation}$ and $\mathcal{N} \models \mathrm{KP}^\mathcal{P} + \Sigma_2^\mathcal{P}\textnormal{-Separation}$ be countable and non-standard. Let $\mathcal{S}$ be a common rank-initial topless substructure of $\mathcal{M}$ and $\mathcal{N}$. Then the following are equivalent:
\begin{enumerate}[{\normalfont (a)}]
\item\label{Wilkie emb 2} For any ordinal $\alpha \in \mathcal{N}$, there is $i \in \llbracket \mathcal{M} \leq^{\rnk}_\mathcal{S} \mathcal{N} \rrbracket$, such that $\mathcal{N}_\alpha \subseteq i(\mathcal{M})$.
\item[{\normalfont (a')}] For any ordinal $\alpha \in \mathcal{N}$, there is a proper $i \in \llbracket \mathcal{M} \leq^{\rnk, \mathrm{topless}}_\mathcal{S} \mathcal{N} \rrbracket$, such that $\mathcal{N}_\alpha \subseteq i(\mathcal{M})$.
\item\label{Wilkie pres 2} $\mathrm{SSy}(\mathcal{M}) = \mathrm{SSy}(\mathcal{N})$, and $\mathrm{Th}_{\Pi_2^\mathcal{P}, S}(\mathcal{M}) \subseteq \mathrm{Th}_{\Pi_2^\mathcal{P}, S}(\mathcal{N})$.
\end{enumerate}
\end{cor}
\begin{proof}
(\ref{Wilkie emb 2}) $\Rightarrow$ (\ref{Wilkie pres 2}) is proved just as for Theorem \ref{Wilkie theorem}.

(\ref{Wilkie pres 2}) $\Rightarrow$ (\ref{Wilkie emb 2}') follows from Theorem \ref{Wilkie theorem} by letting $\beta \in \mathrm{Ord}^\mathcal{N}$ be as obtained from Lemma \ref{Bound for pi_2 lemma}.
\end{proof}

\begin{cor}\label{Wilkie selfembedding}
Let $\mathcal{M} \models \mathrm{KP}^\mathcal{P} + \Sigma_2^\mathcal{P}\textnormal{-Separation} + \Pi_2^\mathcal{P}\textnormal{-Foundation}$ be countable and non-standard. Let $\mathcal{S}$ be a rank-initial topless substructure of $\mathcal{M}$. For any $\alpha \in \mathrm{Ord}^\mathcal{M}$ there is $\beta \in \mathrm{Ord}^\mathcal{M}$ and $i \in \llbracket \mathcal{M} \leq^{\rnk, \mathrm{topless}}_\mathcal{S} \mathcal{M} \rrbracket$, such that $\mathcal{M}_\alpha \subseteq i(\mathcal{M}) \subseteq \mathcal{M}_\beta$ and
\[
\forall \alpha \in \mathrm{Ord}^\mathcal{M} \setminus \mathcal{S} . \exists m \in \mathcal{M} . (\rnk^\mathcal{M}(m) = \alpha \wedge i(m) \neq m).
\]
\end{cor}
\begin{proof}
Let $\mathcal{N} = \mathcal{M}$ and let $\beta \in \mathrm{Ord}^\mathcal{M}$ be as obtained from Lemma \ref{Bound for pi_2 lemma}. Then condition (\ref{Wilkie pres}) of Theorem \ref{Wilkie theorem} is satisfied. Repeat the proof of Theorem \ref{Wilkie theorem} (\ref{Wilkie pres}) $\Rightarrow$ (\ref{Wilkie emb}') with $\mathcal{N} = \mathcal{M}$, except that at the last step: apply Corollary \ref{Friedman selfembedding} instead of Theorem \ref{Friedman thm}.
\end{proof}

Now that we have explored necessary and sufficient conditions for constructing embeddings between models, we turn to the question of constructing isomorphisms between models. For this purpose we shall restrict ourselves to recursively saturated models of $\mathrm{ZF}$.

\begin{lemma}\label{rec sat iso lemma}
Let $\mathcal{M}$ and $\mathcal{N}$ be countable recursively saturated models of $\mathrm{ZF}$, and let $\mathcal{S}$ be a common rank-initial $\omega$-topless substructure of $\mathcal{M}$ and $\mathcal{N}$. Moreover, let $\mathbb{P} = \llbracket \mathcal{M} \preceq_\mathcal{S} \mathcal{N} \rrbracket^{< \omega}$.

If $\mathrm{SSy}_\mathcal{S}(\mathcal{M}) \leq \mathrm{SSy}_\mathcal{S}(\mathcal{N})$, then 
\[
\mathcal{C}'_m =_\df \{ f \in \mathbb{P} \mid m \in \dom(f\hspace{2pt}) \}
\]
is dense in $\mathbb{P}$, for each $m \in \mathcal{M}$.
\end{lemma}
\begin{proof}
By Theorem \ref{rec sat char}, $\mathcal{M}$ and $\mathcal{N}$ are $\omega$-non-standard and there are expansions $(\mathcal{M}, \mathrm{Sat}^\mathcal{M})$ and $(\mathcal{N}, \mathrm{Sat}^\mathcal{N})$ satisfying condition (\ref{rec sat char truth}) of that theorem. Recall that (informally) this condition says that these are satisfaction classes that are correct for all formulae in $\mathcal{L}^0$ of standard complexity, and that the expanded structures satisfy Separation and Replacement for all formulae in the expanded language $\mathcal{L}^0_\mathrm{Sat}$.

Let $g \in \mathbb{P}$. We unravel it as $g = \{\langle m_\xi, n_\xi \rangle \mid \xi < \gamma \}$, for some $\gamma < \omega$. Let $m_\gamma \in \mathcal{M}$ be arbitrary. By $\mathcal{L}^0_\mathrm{Sat}$-Separation, there is a code $c$ in $\mathcal{M}$ for the set
\begin{align*}
\{ & \langle \delta, s \rangle \in (\mathcal{L}^0[\langle x_\xi \rangle_{\xi < \gamma}, x_\gamma, z] \cap \mathcal{S}) \times \mathcal{S} \mid \\
& (\mathcal{M}, \mathrm{Sat}^\mathcal{M}) \models \mathrm{Sat}(\delta, \langle m_\xi \rangle_{\xi < \gamma}, m_\gamma, s) \}.
\end{align*}
Since $\mathrm{SSy}_\mathcal{S}(\mathcal{M}) \cong \mathrm{SSy}_\mathcal{S}(\mathcal{N})$, this set is also coded by some $d$ in $\mathcal{N}$.

We define a formula:
\begin{align*}
\phi(\zeta, k, \langle x_\xi \rangle_{\xi < \gamma}, q) \equiv & \mathrm{Ord}(\zeta) \wedge k < \omega \wedge \\
& \exists x_\gamma . \forall \delta \in \Sigma_{k} . \forall t \in V_\zeta . \\
& (\mathrm{Sat}(\delta, \langle x_\xi \rangle_{\xi < \gamma}, x_\gamma, t) \leftrightarrow \langle \delta, t \rangle \in q).
\end{align*}
By construction of $c$ and correctness of $\mathrm{Sat}^\mathcal{M}$, we have that 
$$\mathcal{M} \models \phi(\zeta, k, \langle m_\xi \rangle_{\xi < \gamma}, c),$$ 
for each $\zeta \in \mathrm{Ord}^\mathcal{M} \cap \mathcal{S}$ and each $k < \omega = \mathrm{OSP}(\mathcal{M})$. So since $g$ is elementary, and since $\mathcal{M}_\zeta = \mathcal{N}_\zeta$ and $c_\mathcal{M} \cap \mathcal{M}_\zeta = d_\mathcal{N} \cap \mathcal{N}_\zeta$ for each $\zeta \in \mathrm{Ord}^\mathcal{N} \cap \mathcal{S}$, we also have that $\mathcal{N} \models \phi(\zeta, k, \langle n_\xi \rangle_{\xi < \gamma}, d)$ for each $\zeta \in \mathrm{Ord}^\mathcal{N} \cap \mathcal{S}$ and each $k < \omega$. Pick some $\nu \in \mathrm{Ord}^\mathcal{N} \setminus \mathcal{S}$. Now by Overspill on $\mathcal{S}$, for each $k < \omega$ there is $\nu\hspace{1pt}'_k \in \mathrm{Ord}^\mathcal{N} \setminus \mathcal{S}$ such that $(\mathcal{N}, \mathrm{Sat}^\mathcal{N}) \models \nu\hspace{1pt}'_k < \nu \wedge \phi(\nu\hspace{1pt}'_k, k, \langle n_\xi \rangle_{\xi < \gamma}, d)$.

Pick some non-standard $o <^\mathcal{N} \omega^\mathcal{N}$. Working in $\mathcal{N}$, we construct a partial function $(k \mapsto \nu_k) : o \rightarrow \nu + 1$, such that for each $k < o$,
\[
\nu_k = \sup \{ \zeta \mid \zeta < \nu \wedge \phi(\zeta, k, \langle n_\xi \rangle_{\xi < \gamma}, d) \}.
\]
We return to reasoning in the meta-theory. By the Overspill-argument above, this function is total on $\omega$, and $\nu_k \not\in \mathcal{S}$ for each $k < \omega$. Moreover, by logic, $\nu_k \geq \nu_l$ for all $k \leq l < \omega$. So by $\omega$-toplessness, there is $\nu_\infty \in \mathrm{Ord}^\mathcal{N} \setminus \mathcal{S}$, such that for each $k < \omega$, $\nu_\infty <^\mathcal{N} \nu_k$. So for each $k < \omega$, we have $(\mathcal{N}, \mathrm{Sat}^\mathcal{N}) \models \phi(\nu_\infty, k, \langle n_\xi \rangle_{\xi < \gamma}, d)$, whence by Overspill on $\mathrm{WFP}(\mathcal{N})$, there is a non-standard $k_\infty \in^\mathcal{N} \omega^\mathcal{N}$ such that $(\mathcal{N}, \mathrm{Sat}^\mathcal{N}) \models \phi(\nu_\infty, k_\infty, \langle n_\xi \rangle_{\xi < \gamma}, d)$. Let $n_\gamma \in \mathcal{N}$ be a witness of this fact. Note that for all $s \in \mathcal{S}$ and for all $\delta \in \mathcal{L}^0[\langle x_\xi \rangle_{\xi < \gamma}, x_\gamma, z]$,
\[
\mathcal{N} \models \delta(\langle n_\xi \rangle_{\xi < \gamma}, n_\gamma, s) \Leftrightarrow \mathcal{N} \models \langle \delta, s \rangle \in d.
\]

Let $f = g \cup \{\langle m_\gamma, n_\gamma \rangle \}$. We need to show that $f \in \mathcal{C}'_{m_\gamma}$; it only remains to verify that $f$ is elementary. Now observe that for any $s \in \mathcal{S}$, and any formula $\delta(\langle x_\xi \rangle_{\xi < \gamma + 1}, z) \in \mathcal{L}^0[\langle x_\xi \rangle_{\xi < \gamma + 1}, z]$,
\begin{align*}
\mathcal{M} \models \delta(\langle m_\xi \rangle_{\xi < \gamma + 1}, s) &\Leftrightarrow \mathcal{M} \models \langle \delta, s \rangle \in c \\
&\Leftrightarrow \mathcal{N} \models \langle \delta, s \rangle \in d \\
&\Leftrightarrow \mathcal{N} \models \delta(\langle n_\xi \rangle_{\xi < \gamma + 1}, s).
\end{align*}
Therefore, $f \in \mathcal{C}'_{m_\gamma}$ as desired.
\end{proof}

\begin{thm}\label{rec sat iso thm}
Let $\mathcal{M}$ and $\mathcal{N}$ be countable recursively saturated models of $\mathrm{ZF}$, and let $\mathcal{S}$ be a common rank-initial $\omega$-topless substructure of $\mathcal{M}$ and $\mathcal{N}$. Let $m_0 \in \mathcal{M}$ and let $n_0 \in \mathcal{N}$. The following are equivalent:

\begin{enumerate}[{\normalfont (a)}]
	\item There is  $i \in \llbracket \mathcal{M} \cong_\mathcal{S} \mathcal{N} \rrbracket$ such that $i(m_0) = n_0$.
	\item $\mathrm{SSy}_\mathcal{S}(\mathcal{M}) \cong \mathrm{SSy}_\mathcal{S}(\mathcal{N})$ and $\mathrm{Th}_\mathcal{S}((\mathcal{M}, m_0)) = \mathrm{Th}_\mathcal{S}((\mathcal{N}, n_0))$.
\end{enumerate}
\end{thm}
\begin{proof}
The forward direction is clear since $i$ is an isomorphism.

Let $\mathbb{P} = \llbracket \mathcal{M} \preceq_\mathcal{S} \mathcal{N} \rrbracket^{< \omega}$. Since $\mathrm{Th}_\mathcal{S}((\mathcal{M}, m_0)) = \mathrm{Th}_\mathcal{S}((\mathcal{N}, n_0))$, the function $(m_0 \mapsto n_0)$ is in $\mathbb{P}$. For each $m \in \mathcal{M}$ and each $n \in  \mathcal{N}$, let
\begin{align*}
\mathcal{C}'_m &=_\df \{ f \in \mathbb{P} \mid m \in \dom(f\hspace{2pt}) \}, \\
\mathcal{D}'_n &=_\df \{ f \in \mathbb{P} \mid n \in \mathrm{image}(f\hspace{2pt}) \}. \\
\end{align*}
By Lemma \ref{rec sat iso lemma}, $\mathcal{C}'_m$ and $\mathcal{D}'_n$ are dense in $\mathbb{P}$ for all $m \in \mathcal{M}$ and all $n \in \mathcal{N}$. By Lemma \ref{generic filter existence}, there is a $\mathcal{C}'_m \cup \mathcal{D}'_n$-generic filter $\mathcal{I}$ on $\mathbb{P}$ containing $(m_0 \mapsto n_0)$. Let $i = \bigcup \mathcal{I}$. 

By the genericity, $\dom(i) = \mathcal{M}$ and $\mathrm{image}(i) = \mathcal{N}$. Moreover, by the filter properties, for any $\vec{m} \in \mathcal{M}$, some finite extension $f \in \mathbb{P}$ of $i\restriction_{\vec{m}}$ is in $\mathcal{I}$. So by elementarity of $f$ and arbitrariness of $\vec{m}$, we have that $i$ is an isomorphism.
\end{proof}

The following Theorem is an improvement of Lemma \ref{rec sat reflection}.

\begin{thm}[Ressayre]\label{Ressayre thm}
	Let $\mathcal{M}$ be a countable recursively saturated model of $\mathrm{ZF}$. For any $\alpha_0 \in \mathrm{Ord}(\mathcal{M})$ there is $\alpha \in \mathrm{Ord}^\mathcal{M}$, such that $\alpha_0 <^\mathcal{M} \alpha$ and for all $S \in \mathcal{M}_\alpha$ we have $\mathcal{M} \cong_{S_\mathcal{M}} \mathcal{M}_\alpha \preceq \mathcal{M}$.
\end{thm}
\begin{proof}
	Let $\alpha >^\mathcal{M} \omega^\mathcal{M}$ be as obtained from Lemma \ref{rec sat reflection}. Thus $\mathcal{M}_\alpha \preceq \mathcal{M}$. Let $S \in \mathcal{M}_\alpha$ be arbitrary. For each $m \in^\mathcal{M} \omega^\mathcal{M}$, let $\sigma_k = \rnk(S) + m$ as evaluated in $\mathcal{M}$. Since $\mathcal{M}_\alpha \preceq \mathcal{M}$, we have $\sigma_m < \alpha$ for each $m \in^\mathcal{M} \omega^\mathcal{M}$.
	
	Since $\mathcal{M}$ is recursively saturated, it is $\omega$-non-standard. Let $\mathcal{M}_{S, \omega} = \bigcup_{k < \omega} \mathcal{M}_{\sigma_k}$ (note that we take this union only over standard $k$). By Lemma \ref{omega-topless existence}, $\mathcal{M}_{S, \omega}$ is a common rank-initial $\omega$-topless substructure of $\mathcal{M}$ and $\mathcal{M}_\alpha$; and obviously $S_\mathcal{M} \subseteq \mathcal{M}_{S, \omega}$. 
	
	By rank-initiality, $\mathrm{SSy}_{\mathcal{M}_{S, \omega}}(\mathcal{M_\alpha}) = \mathrm{SSy}_{\mathcal{M}_{S, \omega}}(\mathcal{M})$, and by $\mathcal{M}_\alpha \preceq \mathcal{M}$, we have $\mathrm{Th}_{\mathcal{M}_{S, \omega}}(\mathcal{M_\alpha}) = \mathrm{Th}_{\mathcal{M}_{S, \omega}}(\mathcal{M})$. So it follows from Theorem \ref{rec sat iso thm} that $\mathcal{M} \cong_{S_\mathcal{M}} \mathcal{M}_\alpha$.
\end{proof}

\section{Characterizations}\label{Characterizations}

\begin{thm}[Ressayre-style]\label{Ressayre characterization Sigma_1-Separation}
	Let $\mathcal{M} \models \mathrm{KP}^\mathcal{P}$ be countable and non-standard. The following are equivalent:
	\begin{enumerate}[{\normalfont (a)}]
		\item\label{Ressayre characterization Sigma_1-Separation Separation} $\mathcal{M} \models \Sigma_1^\mathcal{P} \textnormal{-Separation}.$
		\item\label{Ressayre characterization Sigma_1-Separation Self-embedding} For every $\alpha \in \mathrm{Ord}^\mathcal{M}$, there is a rank-initial self-embedding of $\mathcal{M}$ which fixes $\mathcal{M}_\alpha$ pointwise.
	\end{enumerate}
\end{thm}
\begin{proof}
	(\ref{Ressayre characterization Sigma_1-Separation Separation}) $\Rightarrow$ (\ref{Ressayre characterization Sigma_1-Separation Self-embedding}): Let $\mathcal{N} = \bigcup_{\xi \in \mathrm{OSP}(\mathcal{M})} \mathcal{M}_{\alpha + \xi}.$ Note that $\mathcal{N}$ is a rank-cut of $\mathcal{M}$. So by Corollary \ref{Friedman selfembedding}, we are done.
	
	(\ref{Ressayre characterization Sigma_1-Separation Self-embedding}) $\Rightarrow$ (\ref{Ressayre characterization Sigma_1-Separation Separation}): Let $\phi(x) \in \Sigma_1^\mathcal{P}[x]$ and let $a \in \mathcal{M}$. Let $i$ be a rank-initial self-embedding of $\mathcal{M}$ which fixes $\mathcal{M}_{\rnk(a) + 1}$ pointwise and which satisfies $i(\mathcal{M}) \subseteq \mathcal{M}_\mu$, for some $\mu \in \mathrm{Ord}^\mathcal{M}$. Let us write $\phi(x)$ as $\exists y . \delta(y, x)$, where $\delta \in \Delta_0^\mathcal{P}$. Since $i$ fixes $a_\mathcal{M}$ pointwise, we have
	\[
	\mathcal{M} \models \forall x \in a . ( \exists y . \delta(y, x) \leftrightarrow \exists y \in V_\mu . \delta(y, x)).
	\]
	So $\{ x \in a \mid \phi(x) \} = \{ x \in a \mid \exists y \in V_\mu . \delta(y, x)) \}$, which exists by $\Delta_0^\mathcal{P}$-Separation.
\end{proof}

\begin{thm}[Bahrami-Enayat-style]\label{characterize topless substructure}
	Let $\mathcal{M} \models \mathrm{KP}^\mathcal{P}  + \Sigma_1^\mathcal{P}\textnormal{-Separation}$ be countable and non-standard, and let $\mathcal{S}$ be a topless substructure of $\mathcal{M}$. The following are equivalent:
	\begin{enumerate}[{\normalfont (a)}]
		\item\label{characterize topless substructure emb} $\mathcal{S} = \mathrm{Fix}^\rnk(i)$, for some rank-initial self-embedding $i : \mathcal{M} \rightarrow \mathcal{M}$.
		\item[{\normalfont (\ref{characterize topless substructure emb}')}] $\mathcal{S} = \mathrm{Fix}^\rnk(i)$, for some proper topless rank-initial self-embedding $i : \mathcal{M} \rightarrow \mathcal{M}$.
		\item\label{characterize topless substructure substr} $\mathcal{S}$ is rank-initial in $\mathcal{M}$.
	\end{enumerate}
\end{thm}
\begin{proof}
	(\ref{characterize topless substructure emb}) $\Rightarrow$ (\ref{characterize topless substructure substr}) is immediate from the definition of $\mathrm{Fix}^\rnk$.
	
	(\ref{characterize topless substructure substr}) $\Rightarrow$ (\ref{characterize topless substructure emb}') follows from Corollary \ref{Friedman selfembedding}.
\end{proof}

\begin{thm}[Kirby-Paris-style]\label{Kirby Paris thm}
	Let $\mathcal{M} \models \mathrm{KP}^\mathcal{P} + \textnormal{Choice}$ be countable and let $\mathcal{S} \leq^\textnormal{rank-cut} \mathcal{M}$. The following are equivalent:
	\begin{enumerate}[{\normalfont (a)}]
		\item $\mathcal{S}$ is a strong rank-cut in $\mathcal{M}$ and $\omega^\mathcal{M} \in \mathcal{S}$.
		\item $\mathrm{SSy}_\mathcal{S}(\mathcal{M}) \models \mathrm{GBC} + \text{``$\mathrm{Ord}$ is weakly compact''}$.
	\end{enumerate}
\end{thm}
\begin{proof}
	The two directions are proved as Lemmata \ref{Kirby Paris forward} and \ref{Kirby Paris backward} below.
\end{proof}

\begin{lemma}\label{Kirby Paris lemma}
	Let $\mathcal{M} \models \mathrm{KP}^\mathcal{P} + \textnormal{Choice}$, let $\mathcal{S}$ be a strongly topless rank-initial substructure of $\mathcal{M}$ and let us write $\mathrm{SSy}_\mathcal{S}(\mathcal{M})$ as $(\mathcal{S}, \mathcal{A})$. For any $\phi(\vec{x}, \vec{Y}) \in \mathcal{L}^1$ and for any $\vec{A} \in \mathcal{A}$, there are $\theta_\phi(\vec{x}, \vec{y}) \in \Delta_0^\mathcal{P} \subseteq \mathcal{L}^0$ and $\vec{a} \in \mathcal{M}$, such that for all $\vec{s} \in \mathcal{S}$,
	\[
	(\mathcal{S}, \mathcal{A}) \models \phi(\vec{s}, \vec{A}) \Leftrightarrow \mathcal{M} \models \theta_\phi(\vec{s}, \vec{a}).
	\]
\end{lemma}
\begin{proof}
	We construct $\theta_\phi$ recursively on the structure of $\phi$. Let $\vec{A} \in \mathcal{A}$ be arbitrary and let $\vec{a}$ be a tuple of codes in $\mathcal{M}$ for $\vec{A}$. In the base cases, given a coordinate $A$ in $\vec{A}$ and its code $a$ in $\vec{a}$, put:
	\begin{align*}
	\theta_{x=y} &\equiv x = y, \\
	\theta_{x \in y} &\equiv x \in y, \\
	\theta_{x \in A} &\equiv x \in a. 
	\end{align*}
	It is clear that the result holds in the first two cases, and also in the third case since $a$ codes $A$. 
	
	Assume inductively that the result holds for $\phi(\vec{x}, \vec{Y}), \psi(\vec{x}, \vec{Y}) \in \mathcal{L}^1$ and $\theta_\phi(\vec{x}, \vec{y}), \theta_\psi(\vec{x}, \vec{y}) \in \Delta_0^\mathcal{P}$, and put:
	\begin{align*}
	\theta_{\neg \phi} &\equiv \neg \theta_\phi, \\
	\theta_{\phi \vee \psi} &\equiv \theta_\phi \vee \theta_\psi, \\
	\theta_{\exists x . \phi} &\equiv \exists x \in b . \phi, 
	\end{align*}
	where $b \in \mathcal{M}$ is next to be constructed. But before doing so, note that the result holds for the first two cases, simply because the connectives commute with $\models$.
	
	Let us write $\vec{x}$ as $x, \vec{x}'$, and let $k = \mathrm{arity}(\vec{x}')$. Since $\mathcal{S}$ is bounded in $\mathcal{M} \models \textnormal{Infinity}$, there is a limit $\mu \in \mathrm{Ord}^\mathcal{M} \setminus \mathcal{S}$. Therefore, by letting $d = V_\mu^\mathcal{M}$, we obtain $d \in \mathcal{M}$, $\mathcal{S} \subseteq d_\mathcal{M}$, and $\mathcal{M} \models d^k \subseteq d$. Working in $\mathcal{M}$, by Choice and $\Delta_0^\mathcal{P}$-Separation, there is a function $f : d \rightarrow d$, such that for all $\vec{t} \in d^k$,
	\[
	\exists x \in d . \theta_\phi(x, \vec{t}, \vec{a}) \Leftrightarrow f\hspace{2pt}(\vec{t}) \in \{ u \in d \mid \theta_\phi(u, \vec{t}, \vec{a}) \} \Leftrightarrow \theta_\phi(f\hspace{2pt}(\vec{t}), \vec{t}, \vec{a}). \tag{$\dagger$}
	\]
	By strong toplessness and rank-initiality, there is $\beta \in \mathrm{Ord}^\mathcal{M} \setminus \mathcal{S}$ such that for all $s \in \mathcal{S}$,
	\[
	f\hspace{2pt}(s) \in \mathcal{S} \Leftrightarrow \rnk(f\hspace{2pt}(s)) < \beta.  \tag{$\ddagger$}
	\]
	Put $b = V_\beta^\mathcal{M}$. Note that by toplessness, $\mathcal{S}$ is closed under ordered pair. Putting ($\dagger$) and ($\ddagger$) together, we have that, for all $\vec{s} \in \mathcal{S}$,
	\[
	\mathcal{M} \models \exists x \in b . \theta_\phi(x, \vec{s}, \vec{a}) \Leftrightarrow \exists s_0 \in \mathcal{S} . \mathcal{M} \models \theta_\phi(s_0, \vec{s}, \vec{a}),
	\]
	and by induction hypothesis,
	\begin{align*}
	\exists s_0 \in \mathcal{S} . \mathcal{M} \models \theta_\phi(s_0, \vec{s}, \vec{a}) & \Leftrightarrow \exists s_0 \in \mathcal{S} . \mathcal{S} \models \phi(s_0, \vec{s}, \vec{a}) \\
	& \Leftrightarrow  \mathcal{S} \models \exists x . \phi(x, \vec{s}, \vec{a}).
	\end{align*}
	Putting these equivalences together yields the desired result for $\exists x . \phi$. The parameters appearing in $\theta_{\exists x . \phi}$ are $b$ and $\vec{a}$.
\end{proof}

\begin{lemma}\label{Kirby Paris forward}
	Let $\mathcal{M} \models \mathrm{KP}^\mathcal{P} + \textnormal{Choice}$ and let $\mathcal{S} \leq \mathcal{M}$. If $\mathcal{S}$ is strongly topless and rank-initial in $\mathcal{M}$ and $\omega^\mathcal{M} \in \mathcal{S}$, then $\mathrm{SSy}_\mathcal{S}(\mathcal{M}) \models \mathrm{GBC} + \text{``$\mathrm{Ord}$ is weakly compact''}$.
\end{lemma}
\begin{proof}
	Let us write $\mathrm{SSy}_\mathcal{S}(\mathcal{M})$ as $(\mathcal{S}, \mathcal{A})$. By toplessness and rank-initiality, there is $d \in \mathcal{M}$ such that $d_\mathcal{M} \supset \mathcal{S}$. 
	
	$(\mathcal{S}, \mathcal{A}) \models \textnormal{Class Extensionality, Pair, Union, Powerset, Infinity}$ are inherited from $\mathcal{M}$, because $\omega^\mathcal{M} \in \mathcal{S}$ and for any $\alpha \in \mathrm{Ord}^\mathcal{S}$, $\alpha +^\mathcal{M} 2 \in \mathcal{S}$ by toplessness, and $\mathcal{M}_{\alpha +^\mathcal{M} 2} \subseteq \mathcal{S}$ by rank-initiality.
	
	$(\mathcal{S}, \mathcal{A}) \models \textnormal{Class Foundation}$: Let $A \in \mathcal{A}$ such that $(\mathcal{S}, \mathcal{A}) \models A \neq \varnothing$. Let $a$ be a code for $A$ in $\mathcal{M}$. By $\Pi_1^\mathcal{P}$-Foundation, there is an $\in^\mathcal{M}$-minimal element $m \in^\mathcal{M} a$. Since $A$ is non-empty, we have by rank-initiality that $m \in \mathcal{S}$. If there were $s \in \mathcal{S}$ such that $(\mathcal{S}, \mathcal{A}) \models s \in m \cap A$, then we would have $\mathcal{M} \models s \in m \cap a$, contradicting $\in^\mathcal{M}$-minimality of $m$. Hence, $m$ is an $\in$-minimal element of $A$ in $(\mathcal{S}, \mathcal{A})$.
		
	$(\mathcal{S}, \mathcal{A}) \models \textnormal{Global Choice}$: By Choice in $\mathcal{M}$, there is a choice function $f$ on $(d \setminus \{\varnothing\})^\mathcal{M}$. Note that $f$ codes a global choice function $F \in \mathcal{A}$ on $(V \setminus \{\varnothing\})^{(\mathcal{M}, \mathcal{A})} \in \mathcal{A}$.
	
	$(\mathcal{S}, \mathcal{A}) \models \textnormal{Class Comprehension}$: Let $\phi(x, \vec{Y}) \in \mathcal{L}^1$, in which all variables of sort $\mathsf{Class}$ are free, and let $\vec{A} \in \mathcal{A}$. By Lemma \ref{Kirby Paris lemma}, there are $\theta_\phi \in \Delta_0^\mathcal{P}$ and $\vec{a} \in \mathcal{M}$, such that for all $s \in \mathcal{S}$,
	\[
	(\mathcal{S}, \mathcal{A}) \models \phi(s, \vec{A}) \Leftrightarrow \mathcal{M} \models \theta_\phi(s, \vec{a}).
	\]
	Working in $\mathcal{M}$, let $c = \{ t \in d \mid \theta_\phi(t, \vec{a}) \}$. Let $C \in \mathcal{A}$ be the class coded by $c$. It follows that 
	\[
	\mathcal{S} \models \forall x . (x \in c \leftrightarrow \phi(x, \vec{A})).
	\]
	
	$(\mathcal{S}, \mathcal{A}) \models \textnormal{Extended Separation}$: Simply observe that if $s \in \mathcal{S}$, $\vec{A} \in \mathcal{S}$ and $\phi(x, \vec{Y}) \in \mathcal{L}^1$, in which all variables of sort $\mathsf{Class}$ are free, then by Class Comprehension, the class $\{ x \mid \phi(x, \vec{A}) \}^{(\mathcal{S}, \mathcal{A})}$ exists in $\mathcal{A}$ and is coded in $\mathcal{M}$ by $c$, say. So by rank-initiality, $c \cap s \in \mathcal{S}$, and $c \cap s$ clearly witnesses the considered instance of $\textnormal{Extended Separation}$.
	
	$(\mathcal{S}, \mathcal{A}) \models \textnormal{Class Replacement}$: Let $F \in \mathcal{A}$ be a class function such that $\dom(F) \in \mathcal{S}$, and let $f$ be a code in $\mathcal{M}$ for $F$. In $\mathcal{M}$, using $\dom(F)$ and $f$ as parameters, we can construct a function $f\hspace{2pt}'$ such that
	\begin{align*}
	\dom(f\hspace{2pt}'_\mathcal{M}) = \phantom{.}& d \supseteq \mathcal{S}, \\
	\mathcal{M} \models \phantom{.}& \forall x \in \dom(F) . \big( (x \in \dom(F) \rightarrow f\hspace{2pt}'(x) = f\hspace{2pt}(x)) \wedge \\
	& (x \not\in \dom(F) \rightarrow f\hspace{2pt}'(x) = 0).
	\end{align*}
	Note that $f\hspace{2pt}'_\mathcal{M}\restriction_{\mathcal{S}} \subseteq \mathcal{S}$. Suppose that 
	$$(\mathcal{S}, \mathcal{A}) \models \forall \xi \in \mathrm{Ord} . \exists x \in \dom(F) . \rnk(F(x)) > \xi.$$ 
	Then we have for all $\xi \in \mathrm{Ord}^\mathcal{S}$ that 
	$$\mathcal{M} \models \exists x \in \dom(F) . \rnk(d) > \rnk(f\hspace{2pt}'(x)) > \xi.$$ 
	So by Overspill, there is $\mu \in \mathrm{Ord}^\mathcal{M} \setminus \mathcal{S}$ such that 
	$$\mathcal{M} \models \exists x \in \dom(F) . \rnk(d) > \rnk(f\hspace{2pt}'(x)) > \mu.$$
	But this contradicts that $f\hspace{2pt}'_\mathcal{M}\restriction_{\mathcal{S}} \subseteq \mathcal{S}$. Therefore, 
	$$(\mathcal{S}, \mathcal{A}) \models \exists \xi \in \mathrm{Ord} . \forall x \in \dom(F) . \rnk(F(x)) < \xi.$$
	Now it follows by $\textnormal{Extended Separation}$ that $\mathrm{image}^{(\mathcal{S}, \mathcal{A})}(F) \in \mathcal{S}$.
	
	$(\mathcal{S}, \mathcal{A}) \models$ ``$\mathrm{Ord}$ is weakly compact'': Let $\mathcal{T}$ be a binary tree of height $\mathrm{Ord}$ in $\mathcal{S}$, coded in $\mathcal{M}$ by $\tau \in \mathcal{M} \setminus \mathcal{S}$. Note that for all $\zeta \in \mathrm{Ord}^\mathcal{S}$, $\tau_\zeta$ is a binary tree of height $\zeta$. So by $\Delta_0$-Overspill, there is $\mu \in \mathrm{Ord}^\mathcal{M} \setminus \mathcal{S}$, such that $\tau_\mu$ is a binary tree of height $\mu$. Let $f \in^\mathcal{M} \tau_\mu$ such that $\dom^\mathcal{M}(f\hspace{2pt}) \in \mathrm{Ord}^\mathcal{M} \setminus \mathcal{S}$. Let $F \in \mathcal{A}$ be the class coded by $f$. Since $\mathcal{S}$ is rank-initial in $\mathcal{M}$, we have for each $\zeta \in \mathrm{Ord}^\mathcal{S}$, that $F_{(\mathcal{S}, \mathcal{A})}(\zeta) = f_\mathcal{M}\restriction_\zeta$. It follows that $F$ is a branch in $\mathcal{T}$.  
\end{proof}

\begin{lemma}\label{embed in Gaifman model}
	Let $\mathcal{M} \models \mathrm{KP}^\mathcal{P}  + \Sigma_1^\mathcal{P}\textnormal{-Separation}$ be countable and non-standard, and let $\mathcal{S}$ be a  topless rank-initial $\Sigma_1^\mathcal{P}$-elementary substructure of $\mathcal{M}$, such that 
	\[
	\mathrm{SSy}_\mathcal{S}(\mathcal{M}) \models \mathrm{GBC} + \textnormal{``$\mathrm{Ord}$ is weakly compact''}.
	\]
	Then there is a structure $\mathcal{N} \succ \mathcal{S}$, and a rank-initial topless self-embedding $i : \mathcal{N} \rightarrow \mathcal{N}$, such that $\mathrm{Fix}(i) = \mathcal{S}$, $i$ is contractive on $\mathcal{N} \setminus \mathcal{S}$, and
	\[
	\mathcal{S} <^{\textnormal{rank-cut}} i(\mathcal{M}) <^{\textnormal{rank-cut}} i(\mathcal{N})  <^{\textnormal{rank-cut}} \mathcal{M} <^{\textnormal{rank-cut}} \mathcal{N}.
	\] 
\end{lemma}
\begin{proof}
	By the assumption $\mathcal{S} \preceq_{\Sigma_1^\mathcal{P}} \mathcal{M}$, $\mathrm{Th}_{\Sigma_1^\mathcal{P}, \mathcal{S}}(\mathcal{S}) = \mathrm{Th}_{\Sigma_1^\mathcal{P}, \mathcal{S}}(\mathcal{M})$. Let us write $\mathrm{SSy}_\mathcal{S}(\mathcal{M})$ as $(\mathcal{S}, \mathcal{A})$. Since
	\[
	(\mathcal{S}, \mathcal{A}) \models \mathrm{GBC} + \textnormal{``$\mathrm{Ord}$ is weakly compact''},
	\]
	we can apply Corollary \ref{end extend model of weakly compact to model with endos}
	 to obtain a model $\mathcal{N}$, such that $\mathcal{S} \preceq^{\rnk} \mathcal{N}$ and for each $\nu \in \mathrm{Ord}^\mathcal{N} \setminus \mathcal{S}$, there is a proper rank-initial self-embedding $i_\nu$ of $\mathcal{N}$, which is contractive on $\mathcal{N} \setminus \mathcal{S}$ and which satisfies $\mathrm{image}(i_\nu) \subseteq \mathcal{N}_\nu$ and $\mathrm{Fix}(i_\nu) = \mathcal{S}$.
	
	Since $\mathrm{Th}_{\Sigma_1^\mathcal{P}, \mathcal{S}}(\mathcal{S}) = \mathrm{Th}_{\Sigma_1^\mathcal{P}, \mathcal{S}}(\mathcal{M})$ and $\mathcal{S} \preceq \mathcal{N}$, we have $\mathrm{Th}_{\Sigma_1^\mathcal{P}, \mathcal{S}}(\mathcal{M}) = \mathrm{Th}_{\Sigma_1^\mathcal{P}, \mathcal{S}}(\mathcal{N})$. So by Corollary \ref{Friedman cor}, there is a proper topless rank-initial embedding $j : \mathcal{M} \rightarrow \mathcal{N}$ which fixes $\mathcal{S}$ pointwise. Identify $\mathcal{M}$, pointwise, with the image of this embedding, and pick $\mu \in \mathrm{Ord}^\mathcal{M} \setminus \mathcal{S}$. Let $i = i_\mu$. Since $\mathcal{M}$ is topless in $\mathcal{N}$, we have by Proposition \ref{comp emb} that $i(\mathcal{M})$ is topless in $\mathcal{M}$. Now note that	
	\[
	\mathcal{S} <^{\textnormal{rank-cut}} i(\mathcal{M}) <^{\textnormal{rank-cut}} i(\mathcal{N})  <^{\textnormal{rank-cut}} \mathcal{M} <^{\textnormal{rank-cut}} \mathcal{N},
	\] 
	as desired.
\end{proof}

\begin{thm}[Bahrami-Enayat-style]\label{characterize strongly topless substructure}
	Let $\mathcal{M} \models \mathrm{KP}^\mathcal{P}  + \Sigma_1^\mathcal{P}\textnormal{-Separation}$ be countable and non-standard, and let $\mathcal{S}$ be a proper rank-initial substructure of $\mathcal{M}$. The following are equivalent:
	\begin{enumerate}[{\normalfont (a)}]
		\item\label{characterize strongly topless substructure emb} $\mathcal{S} = \mathrm{Fix}(i) \cap \mathcal{S}'$, for some $\mathcal{S} \subsetneq \mathcal{S}' <^{\rnk} \mathcal{M}$ and some self-embedding $i : \mathcal{M} \rightarrow \mathcal{M}$.
		\item[{\normalfont (\ref{characterize topless substructure emb}')}] $\mathcal{S} = \mathrm{Fix}(i)$, for some self-embedding $i : \mathcal{M} \rightarrow \mathcal{M}$. 
		\item[{\normalfont (\ref{characterize topless substructure emb}'')}] $\mathcal{S} = \mathrm{Fix}(i)$, for some topless rank-initial self-embedding $i : \mathcal{M} \rightarrow \mathcal{M}$, which is contractive on $\mathcal{M} \setminus \mathcal{S}$.
		\item\label{characterize strongly topless substructure substr} $\mathcal{S}$ is a strongly topless $\Sigma_1^\mathcal{P}$-elementary substructure of $\mathcal{M}$.
	\end{enumerate}
\end{thm}
\begin{proof}
	(\ref{characterize strongly topless substructure emb}) $\Rightarrow$ (\ref{characterize strongly topless substructure substr}): We start by observing that $\mathcal{S}$ is topless: It is assumed to be a proper substructure. If there were a least $\lambda \in \mathrm{Ord}^\mathcal{M} \setminus \mathcal{S}$, then by initiality of $i$ and $\mathcal{S} \subseteq \mathrm{Fix}(i)$, we would have $i(\lambda) = \lambda$, contradicting $\mathcal{S} \supseteq \mathrm{Fix}(i) \cap \mathcal{S}'$ and $\lambda  \in \mathcal{S}'$. 
	
	Let $\alpha, \beta \in \mathrm{Ord}^\mathcal{M}$, with $\mathrm{Ord}^\mathcal{M} \cap \mathcal{S} \subseteq \alpha_\mathcal{M}$, and let $f \in \mathcal{M}$ code a function from $\alpha$ to  $\beta$ in $\mathcal{M}$. Note that $\mathrm{Ord}^\mathcal{M} \cap \mathcal{S} \subseteq i(\alpha)_\mathcal{M}$ and that $i(f\hspace{2pt})$ codes a function from $i(\alpha)$ to $i(\beta)$ in $\mathcal{M}$. Let $\gamma \in (\mathrm{Ord}^\mathcal{M} \cap \mathcal{S}') \setminus \mathcal{S}$. By $\mathcal{S} = \mathrm{Fix}(i) \cap \mathcal{S}'$, for all $\zeta \in \mathrm{Ord}^\mathcal{M} \cap \mathcal{S}$ we have
	\[
	f\hspace{2pt}(\zeta) \not\in \mathcal{S} \Leftrightarrow f\hspace{2pt}(\zeta) \neq i(f\hspace{2pt})(\zeta) \vee f\hspace{2pt}(\zeta) \geq \gamma. \tag{$\dagger$}
	\]
	We define a formula, with $f$ and $i(f\hspace{2pt})$ as parameters:
	\[
	\phi(\xi) \equiv \mathrm{Ord}(\xi) \wedge \forall \zeta < \xi . \big( f\hspace{2pt}(\zeta) \neq i(f\hspace{2pt})(\zeta) \rightarrow f\hspace{2pt}(\zeta) > \xi \big)
	\]
	Note that $\phi$ is $\Pi_1$ and that $\mathcal{M} \models \phi(\zeta)$, for all $\zeta \in \mathrm{Ord}^\mathcal{M} \cap \mathcal{S}$. So by $\mathcal{M} \models \Pi_1^\mathcal{P} \text{-Overspill}$ and toplessness of $\mathcal{S}$, there is $\mu \in \mathrm{Ord}^\mathcal{M} \setminus \mathcal{S}$ such that $\mathcal{M} \models \phi(\mu)$. Combining $\mathcal{M} \models \phi(\mu)$ with ($\dagger$), we have for all $\zeta \in \mathrm{Ord}^\mathcal{M} \cap \mathcal{S}$ that
	\[
	f\hspace{2pt}(\zeta) \not\in \mathcal{S} \Rightarrow f\hspace{2pt}(\zeta) > \mu.
	\]
	On the other hand, by $\mu \in \mathrm{Ord}^\mathcal{M} \setminus \mathcal{S}$, the converse is obvious. Hence, $\mathcal{S}$ is strongly topless.
	
	Finally, it follows from Lemma \ref{rank-initial Fix is Sigma_1} that $\mathcal{S} \preceq_{\Sigma_1^\mathcal{P}} \mathcal{M}$.
	
	(\ref{characterize strongly topless substructure substr}) $\Rightarrow$ (\ref{characterize strongly topless substructure emb}''): By Lemma \ref{Kirby Paris forward}, we can apply Lemma \ref{embed in Gaifman model}. The restriction $i\restriction_\mathcal{M}$ of $i : \mathcal{N} \rightarrow \mathcal{N}$ (from Lemma \ref{embed in Gaifman model}) to $\mathcal{M}$, is a topless rank-initial self-embedding of $\mathcal{M}$ with fixed-point set $\mathcal{S}$, which is contractive on $\mathcal{M} \setminus \mathcal{S}$. 
\end{proof}

\begin{lemma}\label{Kirby Paris backward}
	Let $\mathcal{M} \models \mathrm{KP}^\mathcal{P} + \textnormal{Choice}$ be countable and let $\mathcal{S}$ be a rank-cut of $\mathcal{M}$. If $\mathrm{SSy}_\mathcal{S}(\mathcal{M}) \models \mathrm{GBC} + \text{``$\mathrm{Ord}$ is weakly compact''}$, then $\mathcal{S}$ is a strong rank-cut of $\mathcal{M}$.
\end{lemma}
\begin{proof}
	By Corollary \ref{end extend model of weakly compact to model with endos gen}, there is $\mathcal{N} \succ \mathcal{M}$ such that $\mathcal{S}$ is a rank-cut of $\mathcal{N}$ and there is a self-embedding $i : \mathcal{N} \rightarrow \mathcal{N}$ with $\mathrm{Fix}(i) \cap \mathcal{S'} = \mathcal{S}$, for some $\mathcal{S} \subsetneq \mathcal{S}' <^\rnk \mathcal{N}$. So by Theorem \ref{characterize strongly topless substructure}, $\mathcal{S}$ is a strong rank-cut of $\mathcal{N}$.
	
	Let $f : \alpha \rightarrow \beta$ be a function in $\mathcal{M}$, where $\alpha, \beta \in \mathrm{Ord}^\mathcal{M}$ and $\alpha_\mathcal{M} \supseteq \mathrm{Ord}^\mathcal{M} \cap \mathcal{S}$. $f$ may also be considered as a function in $\mathcal{N}$, so since $\mathcal{N}$ elementarily extends $\mathcal{M}$, there is by Theorem \ref{Friedman thm} a rank-initial embedding $j : \mathcal{M} \rightarrow \mathcal{N}$ which fixes $f$ and fixes $\mathcal{S}$ pointwise. Let $\mathcal{M}'$ be the isomorphic copy of $\mathcal{M}$ given by the image of this embedding.
	
	Since $\mathcal{S}$ is a strong rank-cut in $\mathcal{N}$, there is $\nu \in \mathrm{Ord}^\mathcal{N} \setminus \mathcal{S}$ such that for all $\zeta \in \mathrm{Ord}^\mathcal{N} \setminus \mathcal{S}$,
	\[
	f\hspace{2pt}(\zeta) \not\in \mathcal{S} \Leftrightarrow f\hspace{2pt}(\zeta) > \nu.
	\]
	But by rank-initiality of $\mathcal{M}'$ in $\mathcal{N}$, we have that $\nu \in \mathrm{Ord}^\mathcal{M'}$. Now since $j$ is an embedding fixing $f$ and fixing $\mathcal{S}$ pointwise, we have for all $\xi \in \mathrm{Ord}^\mathcal{M} \setminus \mathcal{S}$,
	\[
	f\hspace{2pt}(\xi) \not\in \mathcal{S} \Leftrightarrow f\hspace{2pt}(\xi) > j^{-1}(\nu).
	\]
	So $\mathcal{S}$ is strongly topless in $\mathcal{M}$.
\end{proof}

\begin{thm}\label{strongly topless self-embedding iff GBC weakly compact}
	Suppose that $\mathcal{M} \models \mathrm{KP}^\mathcal{P} + \Sigma_1^\mathcal{P}\textnormal{-Separation} + \textnormal{Choice}$ is countable and non-standard. The following are equivalent:
	\begin{enumerate}[{\normalfont (a)}]
		\item\label{strongly topless self-embedding iff GBC weakly compact strongly topless} There is a strongly topless rank-initial self-embedding $i$ of $\mathcal{M}$.
		\item\label{strongly topless self-embedding iff GBC weakly compact ZFC} $\mathcal{M}$ expands to a model $(\mathcal{M}, \mathcal{A})$ of $\mathrm{GBC} + \text{``$\mathrm{Ord}$ is weakly compact''}$.
	\end{enumerate}
\end{thm}
\begin{proof}
	(\ref{strongly topless self-embedding iff GBC weakly compact strongly topless}) $\Rightarrow$ (\ref{strongly topless self-embedding iff GBC weakly compact ZFC}): If $i$ is strongly topless and rank-initial, then by the Lemma \ref{Kirby Paris forward}, we have $\mathrm{SSy}_{i(\mathcal{M})}(\mathcal{M}) \models \mathrm{GBC} + \textnormal{``$\mathrm{Ord}$ is weakly compact''}$. So since $\mathcal{M} \cong i(\mathcal{M})$, we have that $\mathcal{M}$ expands to a model of $\mathrm{GBC} + $``$\mathrm{Ord}$ is weakly compact''.

	(\ref{strongly topless self-embedding iff GBC weakly compact ZFC}) $\Rightarrow$ (\ref{strongly topless self-embedding iff GBC weakly compact strongly topless}): Expand $\mathcal{M}$ to a countable model $(\mathcal{M}, \mathcal{A})$ of $\mathrm{GBC} + $ ``$\mathrm{Ord}$ is weakly compact''. Let $\mathcal{N} \succ^\textnormal{rank-cut} \mathcal{M}$ be a model obtained from Theorem \ref{Gaifman thm} by putting $\mathbb{L}$ to be a countable linear order without a least element, e.g. $\mathbb{Q}$. By Theorem \ref{characterize strongly topless substructure}, $\mathcal{M}$ is strongly topless in $\mathcal{N}$. Note that $\mathrm{Th}(\mathcal{N}) = \mathrm{Th}(\mathcal{M})$ and $\mathrm{SSy}(\mathcal{N}) = \mathrm{SSy}(\mathcal{M})$. So by Theorem \ref{Friedman thm}, there is a rank-initial embedding $i : \mathcal{N} \rightarrow \mathcal{M}$. By Proposition \ref{comp emb}, it now follows that $i(\mathcal{M})$ is strongly topless in $\mathcal{M}$. 
\end{proof}

\begin{lemma}\label{rec sat elementary strongly topless}
	Let $\mathcal{M}$ be a countable recursively saturated model of $\mathrm{ZFC}$. If $\mathcal{S}$ is a strongly topless rank-initial elementary substructure of $M$, then $S \cong \mathcal{M}$, and a full satisfaction relation on $\mathcal{S}$ is coded in $\mathrm{SSy}_\mathcal{S}(\mathcal{M})$.
\end{lemma}
\begin{proof}
We start by showing that a full satisfaction relation on $\mathcal{S}$ is coded in $\mathrm{SSy}_\mathcal{S}(\mathcal{M})$. By the forward direction of Theorem \ref{rec sat char}, $\mathcal{M}$ is $\omega$-non-standard and admits a full satisfaction relation $\mathrm{Sat}^\mathcal{M}$. Put $\mathrm{Sat}^\mathcal{S} = \mathrm{Sat}^\mathcal{M} \cap \mathcal{S}$. Note that $\mathrm{Sat}^\mathcal{S}$ is coded in $\mathcal{M}$, so the relation $\mathrm{Sat}^\mathcal{S}$ is coded as a class in $\mathrm{SSy}_\mathcal{S}(\mathcal{M})$. Since $\mathcal{S} \prec \mathcal{M}$, we have $\omega^\mathcal{S} = \omega^\mathcal{M}$. So since $\mathcal{S}$ is rank-initial in $\mathcal{M}$ and $\mathcal{M}$ is $\omega$-non-standard, $\mathcal{S}$ is $\omega$-non-standard. 

Since $\mathcal{S}$ is a strongly topless rank-initial elementary substructure of $M$, we have by Lemma \ref{Kirby Paris forward} that $\mathrm{SSy}_\mathcal{S}(\mathcal{M}) \models \mathrm{GBC}$. Therefore we have $(\mathcal{S}, \mathrm{Sat}^\mathcal{S}) \models \mathrm{ZF}(\mathcal{L}^0_\mathrm{Sat})$. To establishes that $\mathrm{Sat}^\mathcal{S}$ is a full satisfaction relation on $\mathcal{S}$, it remains only to check that $(\mathcal{S}, \mathrm{Sat}^\mathcal{S}) \models \forall \sigma \in \bar\Sigma_n[x] . (\mathrm{Sat}(\sigma, x) \leftrightarrow \mathrm{Sat}_{\Sigma_n}(\sigma, x) \wedge \mathrm{Sat}(\neg\sigma, x) \leftrightarrow \mathrm{Sat}_{\Pi_n}(\neg\sigma, x))$, for each standard $n \in \mathbb{N}$. But this follows from that $\mathrm{Sat}_{\Sigma_n}^\mathcal{S} = \mathrm{Sat}_{\Sigma_n}^\mathcal{M} \cap \mathcal{S}^2$ and $\mathrm{Sat}_{\Pi_n}^\mathcal{S} = \mathrm{Sat}_{\Pi_n}^\mathcal{M} \cap \mathcal{S}^2$, for each standard $n \in \mathbb{N}$, which in turn follows from that $\mathcal{S} \prec \mathcal{M}$. 

By the backward direction of Theorem \ref{rec sat char}, it now follows that $\mathcal{S}$ is recursively saturated. Since recursively saturated models are $\omega$-non-standard, we have by Lemma \ref{omega-topless existence} that $\mathrm{WFP}(\mathcal{M})$ is $\omega$-topless in $\mathcal{M}$ and in $\mathcal{S}$. So by Theorem \ref{rec sat iso thm}, $\mathcal{S} \cong_{\mathrm{WFP}(\mathcal{M})} \mathcal{M}$.
\end{proof}

\begin{thm}[Kaye-Kossak-Kotlarski-style]\label{characterize strongly topless substructure of rec sat}
	Let $\mathcal{M} \models \mathrm{ZFC} + V = \mathrm{HOD}$ be countable and recursively saturated, and let $\mathcal{S}$ be a proper rank-initial substructure of $\mathcal{M}$. The following are equivalent:
	\begin{enumerate}[{\normalfont (a)}]
		\item\label{characterize strongly topless substructure of rec sat emb} $\mathcal{S} = \mathrm{Fix}(i)$, for some automorphism $i : \mathcal{M} \rightarrow \mathcal{M}$. 
		\item\label{characterize strongly topless substructure of rec sat substr} $\mathcal{S}$ is a strongly topless elementary substructure of $\mathcal{M}$.
		\item[{\normalfont (\ref{characterize strongly topless substructure of rec sat substr}')}] $\mathcal{S}$ is a strongly topless elementary substructure of $\mathcal{M}$ isomorphic to $\mathcal{M}$.
	\end{enumerate}
\end{thm}
\begin{proof}
(\ref{characterize strongly topless substructure of rec sat emb}) $\Rightarrow$ (\ref{characterize strongly topless substructure of rec sat substr}'): Since $\mathcal{M} \models V = \mathrm{HOD}$, it has definable Skolem functions, whence Lemma \ref{rank-initial Fix is elementary} may be applied to the effect that $\mathcal{S} \prec \mathcal{M}$. Strong toplessness of $\mathcal{S}$ follows from the forward direction of Theorem \ref{characterize strongly topless substructure}. By Lemma \ref{rec sat elementary strongly topless}, we now have that $\mathcal{S} \cong \mathcal{M}$.

(\ref{characterize strongly topless substructure of rec sat substr}) $\Rightarrow$ (\ref{characterize strongly topless substructure of rec sat emb}): Let $(\mathcal{S}, \mathcal{A}) = \mathrm{SSy}_\mathcal{S}(\mathcal{M})$. Since $\mathcal{S} \prec \mathcal{M}$, we have $\omega^\mathcal{M} \in \mathcal{S}$. Now, by Lemma \ref{Kirby Paris forward}, $(\mathcal{S}, \mathcal{A}) \models \mathrm{GBC} + \text{``$\mathrm{Ord}$ is weakly compact''}$. Thus, we may apply Theorem \ref{Gaifman thm} (say with $\mathbb{L} = \mathbb{Q}$) to obtain a countable model $\mathcal{S} \prec \mathcal{N}$ with an automorphism $j : \mathcal{N} \rightarrow \mathcal{N}$ such that $\mathrm{Fix}(j) = \mathcal{S}$. By Lemma \ref{Kirby Paris backward}, $\mathcal{S}$ is strongly topless in $\mathcal{N}$.

Moreover, we have by Lemma \ref{rec sat elementary strongly topless} that $\mathcal{S}$ is recursively saturated with a full satisfaction relation $\mathrm{Sat}^\mathcal{S}$ coded in $\mathcal{A}$. By part (\ref{elem}) of Theorem \ref{Gaifman thm}, $\mathrm{Sat}^\mathcal{S}$ corresponds to a full satisfaction class $\mathrm{Sat}^\mathcal{N}$ on $\mathcal{N}$. So by Theorem \ref{rec sat char}, $\mathcal{N}$ is recursively saturated. Since $\mathcal{S}$ is strongly topless in both $\mathcal{M}$ and $\mathcal{N}$, it now follows from Theorem \ref{rec sat iso thm} that there is an isomorphism $k \in \llbracket \mathcal{M} \cong_\mathcal{S} \mathcal{N} \rrbracket$. The desired automorphism of $\mathcal{M}$ is now obtained as $i = k^{-1} \circ j \circ k$.
\end{proof}

%-----------------------------------------------

\chapter{Stratified set theory and categorical semantics}\label{ch prel NF cat}

\section{Stratified set theory and class theory}\label{Stratified set theory and class theory}

Let $\mathcal{L_\mathsf{Set}} = \{\in, S, \langle -, - \rangle\}$ be the language of set theory augmented with a unary predicate symbol $S$ of ``sethood\hspace{1pt}'' and a binary function symbol $\langle - , - \rangle$ of ``ordered pair''. We introduce notation for the ``set-many quantifier'': 
$$\setmany z . \phi \text{ abbreviates } \exists x . \big( S(x) \wedge \forall z . (z \in x \leftrightarrow \phi(z)) \big),$$
where $x$ is chosen fresh, i.e. not free in $\phi$.

\begin{dfn}\label{DefStrat}
Let $\phi$ be an $\mathcal{L_\mathsf{Set}}$-formula. $\phi$ is {\em stratified} if there is a function $s : \mathrm{term}(\phi) \rightarrow \mathbb{N}$, where $\mathrm{term}(\phi)$ is the set of terms occurring in $\phi$, such that for any $u, v, w \in \mathrm{term(\phi)}$ and any atomic subformula $\theta$ of $\phi$,
\begin{enumerate}[(i)]
\item if $u \equiv \langle v, w \rangle$, then $s(u) = s(v) = s(w)$,
\item if $\theta \equiv (u = v)$, then $s(u) = s(v)$,
\item if $\theta \equiv (u \in v)$, then $s(u) + 1 = s(v)$,
\end{enumerate}
where $\equiv$ denotes literal equality (of terms or formulae). Such an $s$ is called a {\em stratification} of $\phi$. $s(u)$ is called the {\em type} of $u$. Clearly, if $\phi$ is stratified, then there is a {\em minimal} stratification in the sense that $s(v) = 0$ for some variable $v$ occurring in $\phi$. Also note that the formula $\langle v, w \rangle = \{\{v\}, \{v, w\}\}$, stipulating that the ordered pair is the Kuratowski ordered pair, is not stratified. Therefore, it is condition (i), read as ``type-level ordered pair'', that gives power to axiom P below. 
\end{dfn}

\begin{ntn}
In the axiomatizations below, $\mathrm{NFU}_\mathsf{Set}$ is the theory thus axiomatized in classical logic, while $\mathrm{INFU}_\mathsf{Set}$ is the theory thus axiomatized in intuitionistic logic. For brevity we simply write $\mathrm{(I)NFU}_\mathsf{Set}$, and similarly for $\mathrm{(I)NF}_\mathsf{Set}$, to talk about the intuitionistic and classical theories in parallel. More generally, any statement that $\mathrm{(I)XX(U)}_\mathrm{K}$ relates to $\mathrm{(I)YY(U)}_\mathrm{L}$ in some way, means that each of the four theories $\mathrm{IXXU}_\mathrm{K}$, $\mathrm{XXU}_\mathrm{K}$, $\mathrm{IXX}_\mathrm{K}$, $\mathrm{XX}_\mathrm{K}$ relates in that way to $\mathrm{IXXU}_\mathrm{L}$, $\mathrm{XXU}_\mathrm{L}$, $\mathrm{IXX}_\mathrm{L}$, $\mathrm{XX}_\mathrm{L}$, respectively. Since we will be proving equiconsistency results between theories in different languages, the language is emphasized as a subscript to the name of the theory. This is why we write $\mathrm{(I)NF(U)}_\mathsf{Set}$ for the set theoretic theory $\mathrm{(I)NF(U)}$. 
\end{ntn}

\begin{ax}[$\mathrm{(I)NFU}_\mathsf{Set}$] \label{SCAx}
\[
\begin{array}{rl}
\mathrm{Ext}_S & (S(x) \wedge S(y) \wedge \forall z . z \in x \leftrightarrow z \in y) \rightarrow x = y \\
\mathrm{SC}_S & \text{For all stratified $\phi$: } \setmany z . \phi(z) \\
\mathrm{P} & \langle x, y \rangle = \langle x\hspace{1pt}', y\hspace{1pt}' \rangle \rightarrow ( x = x\hspace{1pt}' \wedge y = y\hspace{1pt}' ) \\
\textnormal{{ Sethood}} & z \in x \rightarrow S(x) \\
\end{array}
\]
\end{ax}
$\mathrm{Ext}_S$ stands for Extensionality (for Sets), $\mathrm{SC}_S$ stands for Stratified Comprehension (yielding Sets), and $\mathrm{P}$ stands for Ordered Pair. In order to keep the treatment uniform, we axiomatize $\mathrm{(I)NF}_\mathsf{Set}$ as $\mathrm{(I)NFU}_\mathsf{Set}$ + $\forall x . S(x)$. Obviously, $\mathrm{(I)NF}_\mathsf{Set}$ can be axiomatized in the language without the predicate $S$, simply as $\mathrm{Ext} + \mathrm{SC} + \mathrm{P}$ (where $\mathrm{Ext}$ and $\mathrm{SC}$ are like $\mathrm{Ext}_S$ and $\mathrm{SC}_S$, respectively, but without the $S$-conjuncts). Less obviously, $\mathrm{NF}$ proves the negation of Choice \cite{Spe53}, which entails the axiom of Infinity, which in turn enables implementation of type-level ordered pairs. So $\mathrm{NF}$ can be axiomatized as $\mathrm{Ext} + \mathrm{SC}$ in the plain language $\{\in\}$ of set theory.

Note that $\mathrm{SC}_S$ implies the existence of a universal set, denoted $V$. In the context of the sethood predicate, it is natural to restrict the definition of subset to sets. So define
$$x \subseteq y \Leftrightarrow_{\mathrm{df}} S(x) \wedge S(y) \wedge \forall z . (z \in x \rightarrow z \in y).$$
The power set, $\pow y$, of $y$ is defined as $\{z \mid z \subseteq y\}$, and exists by $\mathrm{SC}_S$. Therefore, only sets are elements of power sets. An important special case of this is that $S(x) \leftrightarrow x \in \pow V$. So the axiom $\forall x . S(x)$, yielding $\mathrm{(I)NF}$, may alternatively be written $V = \pow V$. In the meta-theory $\subseteq$ and $\pow$ are defined in the standard way. When proving the existence of functions (coded as sets of ordered pairs) in $\mathrm{(I)NF(U)}$, the type-level requirement of ordered pairs means that the defining formula (in addition to being stratified) needs to have the argument- and value-variable at the same type.

$\mathrm{(I)ML(U)}_\mathsf{Class}$ is the impredicative theory of classes corresponding to $\mathrm{(I)NF(U)}_\mathsf{Set}$. $\mathrm{ML}$ was introduced by Quine in his book \cite{Qui40}. Apparently $\mathrm{ML}$ stands for ``Mathematical Logic'' (the title of that book). There is both a predicative and an impredicative version of $\mathrm{ML}$, and both are equiconsistent with $\mathrm{NF}_\mathsf{Set}$, as proved in \cite{Wan50}. One obtains a model of $\mathrm{ML}$ simply by taking the power set of a model of $\mathrm{NF}$, along with a natural interpretation that suggests itself, so the proof requires enough strength in the meta-theory to handle sets of the size of the continuum. (The equiconsistency between predicative $\mathrm{ML}$ and $\mathrm{NF}$ can be proved in a weaker meta-theory that is only strong enough to handle countable sets.) Without difficulty, the proof extends to equiconsistency between each of the theories $\mathrm{(I)ML(U)}_\mathsf{Class}$ and $\mathrm{(I)NF(U)}_\mathsf{Set}$, respectively. For the purpose of completeness, a proof of $\mathrm{Con}(\mathrm{(I)NF(U)}_\mathsf{Set}) \Rightarrow  \mathrm{Con}(\mathrm{(I)ML(U)}_\mathsf{Class})$ is provided below.

The theory $\mathrm{(I)ML(U)}_\mathsf{Cat}$, which the author introduces in this research as an algebraic set theory of $\mathrm{(I)NF(U)}$, probably corresponds better to predicative $\mathrm{(I)ML(U)}$. The difficult and interesting direction of the proof of equiconsistency between $\mathrm{(I)ML(U)}_\mathsf{Cat}$ and $\mathrm{(I)NF(U)}_\mathsf{Set}$ is the interpretation of $\mathrm{(I)NF(U)}_\mathsf{Set}$ in $\mathrm{(I)ML(U)}_\mathsf{Cat}$.

We axiomatize $\mathrm{(I)ML(U)}_\mathsf{Class}$ in a one-sorted language $\mathcal{L}_\mathsf{Class}$ that augments $\mathcal{L}_\mathsf{Set}$ with a unary predicate $C$ and a unary predicate $\mathrm{Setom}$. We read $C(x)$ as ``$x$ is a {\em class}'' and read $S(x)$ as ``$x$ is a {\em set}''.  Moreover, ``Setom'' is a portmanteau for ``sets and atoms''. $\mathrm{Setom}(x) \wedge \neg S(x)$ is read as ``$x$ is an {\em atom}''. We treat the pairing function as a partial function; formally we take it to be a ternary relation symbol, but we write it in functional notation. For convenience, we introduce the abbreviations $\exists \vec{x} \in \mathrm{Setom} . \phi$ and $\forall \vec{x} \in \mathrm{Setom} . \phi$ for $\exists \vec{x} . ((\mathrm{Setom}(x_1) \wedge \dots \wedge \mathrm{Setom}(x_n)) \wedge \phi)$ and $\forall \vec{x} . ((\mathrm{Setom}(x_1) \wedge \dots \wedge \mathrm{Setom}(x_n)) \rightarrow \phi)$, respectively, where $\vec{x} = (x_1, \dots, x_n)$ for some $n \in \mathbb{N}$. We say that such quantifiers are {\em bounded} to $\mathrm{Setom}$.

\begin{ax}[$\mathrm{(I)MLU}_\mathsf{Class}$] \label{AxIMLU}
\[
\begin{array}{rl}
\textnormal{C-hood} & z \in x \rightarrow C(x) \\
\textnormal{Sm-hood} & z \in x \rightarrow \mathrm{Setom}(z) \\
\mathrm{Ext}_C & (C(x) \wedge C(y) \wedge \forall z . (z \in x \leftrightarrow z \in y)) \rightarrow x = y \\
\mathrm{CC}_C & \text{For all $\phi$: }  \exists x . \big( C(x) \wedge \forall z \in \mathrm{Setom} . (z \in x \leftrightarrow \phi(z))\big) \\
\mathrm{SC}_S & \text{For all stratified $\phi$ with only $z, \vec{y}$ free: } \\
& \forall \vec{y} \in \mathrm{Setom} . \exists x \in \mathrm{Setom} . \forall z \in \mathrm{Setom} . (z \in x \leftrightarrow \phi(z, \vec{y})) \\
\textnormal{P} & \forall x, y, x\hspace{1pt}', y\hspace{1pt}' \in \mathrm{Setom} . \\ 
& ( \langle x, y \rangle = \langle x\hspace{1pt}', y\hspace{1pt}' \rangle \leftrightarrow ( x = x\hspace{1pt}' \wedge y = y\hspace{1pt}' ) ) \\
S = \mathrm{Sm} \cap C & S(x) \leftrightarrow (\mathrm{Setom}(x) \wedge C(x))
\end{array}
\]
\end{ax}
C-hood stands for Classhood, S-hood stands for Setomhood, $\mathrm{Ext}_C$ stands for Extensionality (for classes), $\mathrm{CC}_C$ stands for Class Comprehension (yielding classes), and $S = \mathrm{Sm} \cap C$ stands for Set equals Setom Class. In $\mathrm{CC}_C$ and $\mathrm{SC}_S$, we assume that $x$ is fresh, i.e. not free in $\phi$. We obtain $\mathrm{(I)ML}_\mathsf{Class}$ by adding the axiom $\forall x \in \mathrm{Setom} . S(x)$. Predicative $\mathrm{(I)ML(U)}_\mathsf{Class}$ is obtained by requiring in $\mathrm{CC}_C$ that all quantifiers in $\phi$ are bounded to $\mathrm{Setom}$.

The leftwards arrow has been added to the Ordered Pair axiom, because the partial function of ordered pair is formally treated as a ternary relation symbol. One might find it natural to add the axiom $\neg C(x) \rightarrow \mathrm{Setom}(x)$, but since we will not need it, the author prefers to keep the axiomatization more general and less complicated.

The extension of $\mathrm{Setom}$ may be thought of as the collection of sets and atoms, but although $\forall x \in\mathrm{Setom} . (S(x) \vee \neg S(x))$ follows from the law of excluded middle in $\mathrm{MLU}_\mathsf{Class}$, this proof does not go through intuitionistically; the author does not expect it to be provable in $\mathrm{IMLU}_\mathsf{Class}$. Note that it follows from the axioms that Sethood (restricted to Setoms) holds, i.e. that $\forall x \in \mathrm{Setom} . ( z \in x \rightarrow S(x) )$. 

The predicate $S$ is clearly redundant in the sense that it is definable, but it is convenient to have it in the language. This more detailed presentation is chosen because it makes it easy to see that $\mathrm{(I)ML(U)}_\mathsf{Class}$ interprets $\mathrm{(I)NF(U)}_\mathsf{Set}$: For any axiom of $\mathrm{(I)NF(U)}_\mathsf{Set}$, simply interpret it as the formula obtained by replacing each subformula of the form $\boxminus x . \phi$ by $\boxminus x \in \mathrm{Setom} . \phi$, for each $\boxminus \in \{\exists, \forall\}$. One may also obtain a model of $\mathrm{(I)NF(U)}_\mathsf{Set}$ from a model of $\mathrm{(I)ML(U)}_\mathsf{Class}$, by restricting its domain to the extension of $\mathrm{Setom}$ and then taking the reduct to $\mathcal{L}_\mathsf{Set}$. 

We now proceed towards showing that the consistency of $\mathrm{(I)NF(U)}_\mathsf{Set}$ implies the consistency of $\mathrm{(I)ML(U)}_\mathsf{Class}$. The idea of the proof is straightforward: we start with a model of $\mathrm{(I)NF(U)}_\mathsf{Set}$ and add all the possible subsets of this structure as new elements to model the classes, with the obvious extension of the $\in$-relation. However, the proof involves some detail of presentation, especially if we do it directly for intuitionistic Kripke models. So here we start off with the classical case, showing how to construct a model of $\mathrm{ML(U)}_\mathsf{Class}$ from a model of $\mathrm{NF(U)}_\mathsf{Set}$. After the categorical semantics has been introduced, we will be able to perform the same proof in the categorical semantics of any topos (Theorem \ref{ToposModelOfML}). The proof below is therefore redundant, but it may help the reader unfamiliar with categorical semantics to compare the two.

\begin{prop}\label{ConSetConClass prop}
If there is a model of $\mathrm{NF(U)}_\mathsf{Set}$, then there is a model of $\mathrm{ML(U)}_\mathsf{Class}$.
\end{prop}
\begin{proof}
We concentrate on the case $\mathrm{Con(NFU)}_\mathsf{Set} \Rightarrow \mathrm{Con(MLU)}_\mathsf{Class}$. Afterwards it will be easy to see the modifications required for the other case. We take care to do this proof in intuitionistic logic, as it will be relevant later on.

Let $\mathcal{N} = ( N, S^\mathcal{N}, \in^\mathcal{N}, P^\mathcal{N} )$ be a model of $\mathrm{NFU}$. Define a model 
$$\mathcal{M} = ( M , C^\mathcal{N} , \mathrm{Setom}^\mathcal{M} , S^\mathcal{N} , \in^\mathcal{M} , P^\mathcal{N}   )$$ 
as follows. Since $\mathcal{N} \models \mathrm{Ext}_S$, it is straightforward to construct a set $M$ with an injection $p : \mathcal{P}(N) \rightarrow M$ and an injection $t : N \rightarrow M$, such that 
$$\forall x \in N . \forall y \in \mathcal{P}(N) . \big( t(x) = p(y) \leftrightarrow (x \in S^\mathcal{N} \wedge y = \{u \in N \mid u \in^\mathcal{N} x \} ) \big).$$ 
Take $M$ as the domain of $\mathcal{M}$.
$$
\begin{array}{rcl}
C^\mathcal{M}  & =_\mathrm{df} & \{p(y) \mid y \in \mathcal{P}(N ) \} \\
\mathrm{Setom}^\mathcal{M}  & =_\mathrm{df} & \{t (x) \mid x \in N \}  \\
S^\mathcal{M}  & =_\mathrm{df} & \{t (x) \mid x \in S^\mathcal{N} \} \\
u \in^\mathcal{M} v & \Leftrightarrow_\mathrm{df} & \exists x \in N . \exists y \in \mathcal{P}(N) . (u = t(x) \wedge v = p(y) \wedge x \in y) \\ 
P^\mathcal{M}  & =_\mathrm{df} & \{\langle t (x), t (y), t (z) \rangle \mid P^\mathcal{N} (x, y) = z\} \\
\end{array}
$$

We now proceed to verify that $\mathcal{M} \models \mathrm{MLU}_\mathsf{Class}$.

Classhood follows from the construction of $C^\mathcal{M}$ and $\in^\mathcal{M} $.

Setomhood follows from the construction of $\mathrm{Setom}^\mathcal{M} $ and $\in^\mathcal{M} $.

Note that $t $ witnesses 
$$\langle N , \in^\mathcal{N} , S^\mathcal{N} , P^\mathcal{N}  \rangle \cong \langle \mathrm{Setom}^\mathcal{M} , \in^\mathcal{M}  \restriction_{\mathrm{Setom}^\mathcal{M} }, S^\mathcal{M} , P^\mathcal{M}  \rangle.$$ 
For by construction, it is easily seen that it is a bijection and that the isomorphism conditions for $S$ and $P$ are satisfied. Moreover, for any $x, x\hspace{1pt}' \in N $, we have 
\[
\begin{array}{rl}
& t (x) \in^\mathcal{M}  t (x\hspace{1pt}') \\ 
\Leftrightarrow & \exists y \in \mathcal{P}(N) . (t(x\hspace{1pt}') = p(y) \wedge x \in y) \\
\Leftrightarrow & \exists y \in \mathcal{P}(N) . ( x\hspace{1pt}' \in S^\mathcal{N} \wedge y = \{u \in N \mid u \in^\mathcal{N} x\hspace{1pt}' \} \wedge x \in y) \\
\Leftrightarrow & x \in^\mathcal{N}  x\hspace{1pt}'.
\end{array}
\]

Since the axioms $\mathrm{SC}_S$ and Ordered Pair in effect have all quantifiers restricted to the extension of $\mathrm{Setom}$, and $\mathcal{N}$ satisfies these axioms, the isomorphism $t$ yields that $\mathcal{M}$ satisfies these axioms as well.

$\mathrm{Ext}_C$ follows from that $t$ is injective and $\forall x \in N . \forall y \in \mathcal{P}(N) . (t(x) \in^\mathcal{M} p(y) \leftrightarrow x \in y)$.

Set equals Setom Class follows from that $v \in C^\mathcal{M} \cap \mathrm{Setom}^\mathcal{M} \Leftrightarrow \exists x \in N . \exists y \in \mathcal{P}(N) . (t(x) = p(y) = v) \Leftrightarrow \exists x \in N . \exists y \in \mathcal{P}(N) . (t(x) = p(y) = v \wedge x \in S^\mathcal{N}) \Leftrightarrow v \in S^\mathcal{M}$.

It only remains to verify that $\mathcal M$ satisfies $\mathrm{CC}_C$. Let $\phi(z)$ be an $\mathcal{L}_\mathsf{Class}$-formula. Let 
$$A = \{u \in N  \mid \mathcal{M}  \models \phi(t (u))\},$$
and note that $\mathcal{M} \models C(p(A))$.

The following implications complete the proof.
$$
\begin{array}{rl}
& A = \{x \in N  \mid \mathcal{M}  \models \phi(t (x))\} \\
\Rightarrow & \forall x \in N  . \big( x \in A \leftrightarrow \mathcal{M}  \models \phi(t (x)) \big) \\
\Rightarrow & \forall u \in \mathrm{Setom}^\mathcal{M}  . \big( u \in^\mathcal{M} p( A) \leftrightarrow \mathcal{M}  \models \phi(u) \big) \\
\Rightarrow & \forall u \in \mathrm{Setom}^\mathcal{M}  . \big( \mathcal{M}  \models (u \in p(A) \leftrightarrow \phi(u)) \big) \\
\Rightarrow & \mathcal{M}  \models \exists X . \big( C(X) \wedge \forall u \in \mathrm{Setom}. (u \in X \leftrightarrow (\phi(u))) \big) \\
\end{array}
$$

To verify the case $\mathrm{Con}(\mathrm{NF}) \Rightarrow \mathrm{Con}(\mathrm{ML})$, note that if $S^\mathcal{N} = N$, then $\mathrm{Setom}^\mathcal{M} = S^\mathcal{M}$.
\end{proof}

For the predicative version of $\mathrm{ML(U)}$, it suffices to consider the set of definable subsets of a model of $\mathrm{NF(U)}$. Thus, a slightly modified version of the above proof can be carried out for the predicative case in an appropriate set theory of countable sets.

\section{Categorical semantics} \label{CatAxioms}

Categories may be viewed as structures in the basic language of category theory. Traditionally, a theory in the first order language of category theory (or an expansion of that language) is formulated as a definition of a class of models. Such definitions, that can be turned into first order axiomatizations, are called elementary. The definitions of classes of categories made in this section are all easily seen to be elementary. 

Now follows a presentation of the categorical semantics of first order logic in Heyting (intuitionistic logic) and Boolean (classical logic) categories. A full account can be found e.g. in \cite[pp. 807-859]{Joh02}. 

It is assumed that the reader is familiar with basic category theoretic notions: Most importantly, the notions of {\em diagram}, {\em cone}, {\em limit} and their duals (in particular, the special cases of {\em terminal object}, {\em initial object}, {\em product} and {\em pullback}), as well as the notions of {\em functor}, {\em natural transformation} and {\em adjoint} functors.

Since the definition of Heyting categories below uses the notion of adjoint functors between partial orders, let us explicitly define this particular case of adjoint functors: Let $\mathbb{A}$ and $\mathbb{B}$ be partial orders with orderings $\leq_\mathbb{A}$ and $\leq_\mathbb{B}$, respectively. They may be considered as categories with the elements of the partial order as objects, and with a single morphism $x \rightarrow y$ if $x \leq y$, and no morphism from $x$ to $y$ otherwise, for all elements $x, y$ in the partial order. The composition of morphisms is the only one possible. Note that a functor from $\mathbb{A}$ to $\mathbb{B}$, as categories, is essentially the same as an order-preserving function from $\mathbb{A}$ to $\mathbb{B}$, as partial orders. Let $\mathbf{F} : \mathbb{A} \leftarrow \mathbb{B}$ and $\mathbf{G} : \mathbb{A} \rightarrow \mathbb{B}$ be functors. $\mathbf{F}$ is {\em left adjoint} to $\mathbf{G}$, and equivalently $\mathbf{G}$ is {\em right adjoint} to $\mathbf{F}$, written $\mathbf{F} \dashv \mathbf{G}$, if for all objects $X$ in $\mathbb{A}$ and all objects $Y$ in $\mathbb{B}$,
\[
\mathbf{F}Y \leq_\mathbb{A} X \Leftrightarrow Y \leq_\mathbb{B} \mathbf{G}X.
\]

A morphism $f$ is a {\em cover} if whenever $f = m \circ g$ for a mono $m$, then $m$ is an isomorphism. A morphism $f$ has an {\em image} if it factors as $f = m \circ e$, where $m$ is a mono with the universal property that if $f = m' \circ e'$ is some factorization with $m'$ mono, then there is a unique $k$ such that $m = m' \circ k$.

\begin{dfn} \label{HeytingBooleanCategory}
A category is a {\em Heyting category} if it satisfies the following axioms (HC).
\begin{enumerate}
\item[(F1)] It has finite limits.
\item[(F2)] It has images.
\item[(F3)] The pullback of any cover is a cover.
\item[(F4)] Each $\mathrm{Sub}_X$ is a sup-semilattice.
\item[(F5)] For each $f : X \rightarrow Y$, the {\em inverse image functor} $f\hspace{2pt}^* : \mathrm{Sub}_Y \rightarrow \mathrm{Sub}_X$ (defined below) preserves finite suprema and has left and right adjoints: $\exists_f \dashv f\hspace{2pt}^* \dashv \forall_f $.
\end{enumerate}

We call this theory HC. $\mathrm{Sub}_X$ and $f\hspace{2pt}^*$ are explained below. One can prove from these axioms that that each $\mathrm{Sub}_X$ is a Heyting algebra. A {\em Boolean category} is a Heyting category such that each $\mathrm{Sub}_X$ is a Boolean algebra. We call that theory BC.

A {\em Heyting} ({\em Boolean}) {\em functor}, is a functor between Heyting (Boolean) categories that preserves the structure above. $\mathbf{C}$ is a {\em Heyting} ({\em Boolean}) {\em subcategory} of $\mathbf{D}$ if it is a subcategory and the inclusion functor is Heyting (Boolean).
\end{dfn}

Let $\mathbf{C}$ be any Heyting category. It has a terminal object $\mathbf{1}$ and an initial object $\mathbf{0}$, as well as a product $X_1 \times \dots \times X_n$, for any $n \in \mathbb{N}$ (in the case $n = 0$, $X_1 \times \dots \times X_n$ is defined as the the terminal object $\mathbf{1}$). Given an $n \in \mathbb{N}$ and a product $P$ of $n$ objects, the $i$-th projection morphism, for $i = 1, \dots , n$, is denoted $\pi_P^i$ (the subscript $P$ will sometimes be dropped when it is clear from the context). If $f_i : Y \rightarrow X_i$ are morphisms in $\mathbb{C}$, for each $i \in \{1, \dots, n\}$ with $n \in \mathbb{N}$, then $\langle f_1, \dots, f_n\rangle : Y \rightarrow X_1 \times \dots \times X_n$ denotes the unique morphism such that $\pi^i \circ \langle f_1, \dots, f_n\rangle = f_i$, for each $i \in \{1, \dots, n\}$. An important instance of this is that $\mathbf{C}$ has a diagonal mono $\Delta_X : X \rightarrowtail X \times X$, for each $X$, defined by $\Delta_X = \langle \id_X, \id_X \rangle$. If $g_i : Y_i \rightarrow X_i$ are morphisms in $\mathbf{C}$, for each $i \in \{1, \dots, n\}$ with $n \in \mathbb{N}$, then $g_1 \times \dots \times g_n : Y_1 \times \dots \times Y_n \rightarrow X_1 \times \dots \times X_n$ denotes the morphism $\langle g_1 \circ \pi^1, \dots, g_n \circ \pi^n \rangle$.

A {\em subobject} of an object $X$ is an isomorphism class of monos $m : Y \rightarrowtail X$ in the slice category $\mathbf{C}/X$. (Two monos $m : Y \rightarrowtail X$ and $m' : Y\hspace{1pt}' \rightarrowtail X$ are isomorphic in $\mathbf{C}/X$ iff there is an isomorphism $f : Y \rightarrow Y\hspace{1pt}'$ in $\mathbf{C}$, such that $m = m' \circ f$.) It is often convenient to denote such a subobject by $Y$, although it is an abuse of notation; in fact we shall do so immediately. The subobjects of $X$ are endowed with a partial order: If $m : Y \rightarrowtail X$ and $m' : Y\hspace{1pt}' \rightarrowtail X$ represent two subobjects $Y$ and $Y\hspace{1pt}'$ of $X$, then we write $Y \leq_X Y\hspace{1pt}'$ if there is a mono from $m$ to $m'$ in $\mathbf{C}/X$ (i.e. if there is a mono $f : Y \rightarrow Y\hspace{1pt}'$ in $\mathbf{C}$, such that $m = m' \circ f$). 

The axioms (F1)--(F5) ensure that for any object $X$, the partial order of subobjects of $X$, denoted $\mathrm{Sub}(X)$, with its ordering denoted $\leq_X$ and its equality relation denoted $\cong_X$ (or just $=$ when the context is clear), is a Heyting algebra, with constants $\bot_X$, $\top_X$ and operations $\wedge_X$, $\vee_X$, $\rightarrow_X$ (we often suppress the subscript when it is clear from the context). Given a morphism $f : X \rightarrow Y$ in $\mathbf{C}$, the functor $f\hspace{2pt}^* : \mathrm{Sub}(Y) \rightarrow \mathrm{Sub}(X)$ is defined by sending any subobject of $Y$, represented by $m_B : B \rightarrowtail Y$, say, to the subobject of $X$ represented by the pullback of $m_B$ along $f$. Given a subobject $A$ of $Y$, represented by a mono $m_A$ with co-domain $Y$, we may write $A^* : \mathrm{Sub}(Y) \rightarrow \mathrm{Sub}(X)$ as an alternative notation for the functor $m_A^*$.

A {\em structure} (or {\em model}) $\mathcal M$, in the categorical semantics of $\mathbf{C}$, in a sorted signature $\mathcal{S}$, is an assignment of sorts, relation symbols and function symbols of $\mathcal{S}$ to objects, subobjects and morphisms of $\mathbf{C}$, respectively, as now to be explained. 

Sorts: Any sort in $\mathcal{S}$ is assigned to an object of $\mathbf{C}$.

Relation symbols: Any relation symbol $R$ on a sort $S_1 \times \dotsc \times S_n$, where $n \in \mathbb{N}$, is assigned to a subobject $R^\mathcal{M} \leq S^\mathcal{M}_1 \times \dotsc \times S^\mathcal{M}_n$. In particular, the equality symbol $=_S$ on the sort $S \times S$ is assigned to the subobject of $S^\mathcal{M} \times S^\mathcal{M}$ determined by $\Delta_{S^\mathcal{M}} : S^\mathcal{M} \rightarrowtail S^\mathcal{M} \times S^\mathcal{M}$. In the case $n = 0$, $S^\mathcal{M}_1 \times \dotsc \times S^\mathcal{M}_n$ is the terminal object $\mathbf{1}$. Thus, we can handle $0$-ary relation symbols. By the above, such a symbol is assigned to a subobject of $\mathbf{1}$. For example, the unique morphism $\mathbf{1} \rightarrow \mathbf{1}$ and the unique morphism $\mathbf{0} \rightarrow \mathbf{1}$ represent subobjects of $\mathbf{1}$. In the semantics explained below, the former corresponds to truth and the latter corresponds to falsity.

Function symbols: Any function symbol $f : S_1 \times \dotsc \times S_n \rightarrow T$, where $n \in \mathbb{N}$, is assigned to a morphism $f^\mathcal{M} : S^\mathcal{M}_1 \times \dotsc \times S^\mathcal{M}_n \rightarrow T^\mathcal{M}$. Note that in the case $n = 0$, $f$ is assigned to a morphism $1 \rightarrow T$. In this case, we say that $f$ is a {\em constant symbol}.

Let $m, n \in \mathbb{N}$ and let $k \in \{1, \dots, n\}$. The $\mathcal M$-{\em interpretation} $\llbracket \vec{x} : S_1 \times \dotsc \times S_n \mid t \rrbracket^\mathcal{M}$ (which may be abbreviated $\llbracket \vec{x} \mid t \rrbracket$ when the structure and the sorts of the variables are clear) of a term $t$ of sort $T$ in context $\vec x$ of sort $S_1 \times \dotsc \times S_n$ is a morphism $S^\mathcal{M}_1 \times \dotsc \times S^\mathcal{M}_n \rightarrow T^\mathcal{M}$ defined recursively:
\[
\begin{array}{rcl}
\llbracket \vec{x} \mid x_k \rrbracket &=_\mathrm{df}& \pi^k : S^\mathcal{M}_1 \times \dotsc \times S^\mathcal{M}_n \rightarrow S^\mathcal{M}_k \\
\llbracket \vec{x} \mid f\hspace{2pt}(t_1, \dotsc, t_m) \rrbracket &=_\mathrm{df}&
\mathcal M f \circ \langle \llbracket \vec x \mid t_1 \rrbracket, \dotsc, \llbracket \vec{x} \mid t_m \rrbracket \rangle : \\
&& S^\mathcal{M}_1 \times \dotsc \times S^\mathcal{M}_n \rightarrow W^\mathcal{M}, 
\end{array}
\]
where $t_1, \dotsc, t_m$ are terms of sorts $T_1, \dotsc, T_m$, respectively, and $f$ is a function symbol of sort $T_1 \times \dotsc \times T_m \rightarrow W$. 

The $\mathcal M$-{\em interpretation} $\llbracket \vec{x} : S_1 \times \dotsc \times S_n \mid \phi \rrbracket^\mathcal{M}$ (which may be abbreviated $\llbracket \vec{x} \mid \phi \rrbracket$ when the structure and the sorts of the variables are clear) of a formula $\phi$ in context $\vec x$ of sort $S_1 \times \dotsc \times S_n$ is defined recursively:
\[
\begin{array}{rcl}
\llbracket \vec{x} \mid \bot \rrbracket &=_\mathrm{df}& [\bot \rightarrowtail S^\mathcal{M}_1 \times \dotsc \times S^\mathcal{M}_n] \\
\llbracket \vec{x} \mid R(\vec{t}) \rrbracket &=_\mathrm{df}& \llbracket \vec{x} \mid \vec{t} \rrbracket^* (R^\mathcal{M}), \\
&& \text{where $R$ is a relation symbol and $\vec{t}$ are terms.} \\
\llbracket \vec{x} \mid \chi \odot \psi \rrbracket &=_\mathrm{df}& \llbracket \vec{x} \mid \chi \rrbracket \odot \llbracket \vec{x} \mid \psi \rrbracket \text{, where $\odot \in \{\wedge, \vee, \rightarrow\}$.} \\
\llbracket \vec{x}\setminus \{x_k\} \mid \boxminus x_k . \psi \rrbracket &=_\mathrm{df}& \boxminus_{\langle \pi^1, \dotsc, \pi^{k-1}, \pi^{k+1}, \dotsc, \pi^n \rangle} (\llbracket \vec{x} \mid \psi \rrbracket), \\
&& \text{where $\boxminus \in \{\forall, \exists\}$.}
\end{array}
\]
Recall that $\llbracket \vec{x} \mid \vec{t} \rrbracket^*(R^\mathcal{M})$ is obtained by taking the pullback of a representative of $R^\mathcal{M}$ along $\llbracket \vec{x} \mid \vec{t} \rrbracket$. The denotation of $\boxminus_{\langle \pi^1, \dotsc, \pi^{k-1}, \pi^{k+1}, \dotsc, \pi^n \rangle}$ is given in axiom (F5) of Heyting categories above. 

We say that $\phi(\vec x)$ is {\em valid} in $\mathcal{M}$, and write $\mathcal{M} \models \phi$, whenever $\llbracket \vec{x} \mid \phi(\vec{x}) \rrbracket$ equals the maximal subobject $S^\mathcal{M}_1 \times \dotsc \times S^\mathcal{M}_n$ of $\mathrm{Sub}(S^\mathcal{M}_1 \times \dotsc \times S^\mathcal{M}_n)$. In particular, if $\phi$ is a sentence, then $\mathcal{M} \models \phi$ iff $\llbracket \cdot : \cdot \mid \phi \rrbracket = \mathbf{1}$, where the notation ``$\cdot : \cdot$'' stands for the empty sequence of variables in the $0$-ary context. It is of course more convenient to write $\llbracket \cdot : \cdot \mid \phi \rrbracket$ simply as $\llbracket \phi \rrbracket$.

When working with this semantics it is sometimes convenient to use the following well-known rules:
\[
\begin{array}{rcl}
\llbracket \vec{x} \mid \chi \rrbracket \wedge \llbracket \vec{x} \mid \psi \rrbracket &=& \llbracket \vec{x} \mid \chi \rrbracket^*(\llbracket \vec{x} \mid \psi \rrbracket) \\
\llbracket \vec{x} \mid \chi \rightarrow \psi \rrbracket &\Leftrightarrow& \llbracket \vec{x} \mid \chi \rrbracket \leq \llbracket \vec{x} \mid \psi \rrbracket \\
\llbracket \forall x_1 \dots \forall x_n . \psi \rrbracket = \mathbf{1}  &\Leftrightarrow& \llbracket \vec{x} \mid \psi \rrbracket = S^\mathcal{M}_1 \times \dots \times S^\mathcal{M}_1n,
\end{array}
\]
In the last equivalence, it is assumed that $x_1, \dots, x_n$ are the only free variables of $\phi$.
 
When an interpretation $\mathcal{M}$ of $\mathcal{S}$ in a Heyting category $\mathbf{C}$ is given, we will often simply write ``$\mathbf{C} \models \phi$''. Sometimes it is convenient to extend $\mathcal{S}$ with some objects, morphisms and subobjects of $\mathbf{C}$ as new sorts, function symbols and relation symbols, respectively.

\begin{dfn} \label{NaturalSignature}
Let $\mathbf{C}$ be a Heyting category and let $\mathbf{D}$ be a subcategory of $\mathbf{C}$ with finite products. We define the {\em $\mathbf{D}$-signature with respect to $\mathbf{C}$}, denoted $\mathcal{S}^\mathbf{C}_\mathbf{D}$, as the following signature.
\begin{itemize}
\item Sorts: For each object $A$ of $\mathbf{D}$, $A$ is a sort in $\mathcal{S}^\mathbf{C}_\mathbf{D}$.
\item Function symbols: For each morphism $f : A \rightarrow B$ of $\mathbf{D}$, $f : A \rightarrow B$ is a function symbol in $\mathcal{S}^\mathbf{C}_\mathbf{D}$ from the sort $A$ to the sort $B$.
\item Relation symbols: For each $n \in \mathbb{N}$, and for each morphism $m : A \rightarrow B_1 \times \dots \times B_n$ of $\mathbf{D}$, such that $m$ is monic in $\mathbf{C}$, $m$ is an $n$-ary relation symbol in $\mathcal{S}^\mathbf{C}_\mathbf{D}$ on the sort $B_1 \times \dots \times B_n$. (Note that in the case $n = 0$, $B_1 \times \dots \times B_n$ is the terminal object $\mathbf{1}$ of $\mathbf{D}$ and $m$ is a $0$-ary relation symbol.)
\end{itemize}
Given $\mathcal{S}^\mathbf{C}_\mathbf{D}$, the {\em natural $\mathcal{S}^\mathbf{C}_\mathbf{D}$-structure} is defined by assigning each sort $A$ to the object $A$, assigning each function symbol $f$ to the morphism $f$, and assigning each relation symbol $m$ on the sort $B_1 \times \dots \times B_n$ to the subobject of $B_1 \times \dots \times B_n$ in $\mathbf{C}$ determined by $m$. Let $\phi$ be an $\mathcal{S}^\mathbf{C}_\mathbf{D}$-formula. We write $\mathbf{C} \models \phi$ for the statement that $\phi$ is satisfied in the natural $\mathcal{S}^\mathbf{C}_\mathbf{D}$-structure. If no signature has been specified, then $\mathbf{C} \models \phi$ means that $\phi$ is satisfied in the natural $\mathcal{S}^\mathbf{C}_\mathbf{C}$-structure (and it is assumed that $\phi$ is an $\mathcal{S}^\mathbf{C}_\mathbf{C}$-formula).
\end{dfn}

The importance of Heyting categories lies in this well-known result:

\begin{thm}[Completeness for categorical semantics]\label{Completeness}
Intuitionistic and classical first order logic are sound and complete for the categorical semantics of Heyting and Boolean categories, respectively.
\end{thm}

As a first application of the categorical semantics, we shall generalize Proposition \ref{ConSetConClass prop} to the intuitionistic case. This can be done efficiently through the machinery of topos theory.

\begin{dfn}\label{power object}
A {\em topos} is a category with finite limits and power objects. A {\em power object} of an object $A$, is an object $\mathbf{P}A$ along with a mono $m : \hspace{2pt} \in_A \hspace{2pt} \rightarrowtail A \times \mathbf{P}A$ such that for any mono $r : R \rightarrowtail A \times B$, there is a unique morphism $\chi : B \rightarrow \mathbf{P}A$ making this a pullback square:
\[
\begin{tikzcd}[ampersand replacement=\&, column sep=small]
R \ar[rr] \ar[d, rightarrowtail, "r"'] \&\& \in_A \ar[d, rightarrowtail, "m"] \\
A \times B \ar[rr, "{\id \times \chi}"] \&\& A \times \mathbf{P}A 
\end{tikzcd} 
\]
\end{dfn}

The expression ``morphism $\chi : B \rightarrow \mathbf{P}A$ making this a pullback square'' with a pullback-diagram drawn underneath (as above), will be used several times in this text. More formally, it is taken as an abbreviation of ``morphism $\chi : B \rightarrow \mathbf{P}A$ such that $r$ is a pullback of $m$ along $\id \times \chi$'' (where $m$ and $r$ depend as above on the pullback-diagram drawn underneath).

A {\em small} category is a category that can be implemented as a set (i.e. it does not require a proper class). If $\mathbf{C}$ is a small category, then the category $\mathbf{Set}^\mathbf{C}$, of functors from $\mathbf{C}$ to the usual category of sets, with natural transformations as morphisms, is called the {\em category of presheaves of} $\mathbf{C}^\mathrm{op}$.

Here we collect some well-known facts about topoi, found in introductory textbooks, that are needed for the proof of Theorem \ref{ToposModelOfML}.

\begin{prop}\label{topos prop}
Let $\mathbf{C}$ be a small category. Let $\mathbf{Set}$ be the usual category of sets. Let $\mathbf{E}$ be a topos and let $Z$ be an object in $\mathbf{E}$. Let $\mathbf{P}Z$ along with $p_Z : \in_Z \rightarrowtail Z \times \mathbf{P}Z$ be a power object of $Z$ in $\mathbf{E}$.
\begin{enumerate}[{\normalfont (a)}]
\item $\mathbf{Set}^\mathbf{C}$ is a topos.
\item $\mathbf{E}$ is a Heyting category.
\item $\mathbf{E} \models \forall x, y : \mathbf{P}Z . \big( (\forall z : Z . (z \in x \leftrightarrow z \in y)) \rightarrow x = y \big)$
\item For each $\mathcal{S}^\mathbf{E}_\mathbf{E}$-formula $\phi(z, y)$, $\mathbf{E} \models \forall y : Y . \exists x : \mathbf{P}Z . \forall z : Z . (z \in x \leftrightarrow \phi(z, y))$.
\item The pushout of any mono in $\mathbf{E}$ is a mono.
\item The pushout of any mono in $\mathbf{E}$ also forms a pullback diagram.
\end{enumerate}
\end{prop}

An {\em intuitionistic Kripke structure} in a first-order language $\mathcal{L}$ on a partial order $\mathbb{P}$, is an $\mathcal{L}$-structure in the categorical semantics of $\mathbf{Set}^\mathbb{P}$. It is well-known and easily verified that this definition is equivalent to the traditional definition, as given e.g. in \cite{Mos15}.

\begin{thm}\label{ToposModelOfML}
Let $\mathbf{E}$ be a topos. In the categorical semantics of $\mathbf{E}$: If there is a model of $\mathrm{(I)NF(U)}_\mathsf{Set}$, then there is a model of $\mathrm{(I)ML(U)}_\mathsf{Class}$.
\end{thm}
\begin{proof}
This result follows immediately from the proof of Proposition \ref{ConSetConClass prop}, because that proof can literally be carried out in the internal language of any topos. (It is well-known that one can safely reason from the axioms of a weak intuitionistic set theory in this internal language.) However, for the reader's convenience we shall also give the proof in its interpreted form, in the language of category theory.

The intuitionistic and classical cases correspond to the cases that $\mathbf{E}$ is Heyting and Boolean, respectively. The symbol $\in$ is used for the element-relations associated with power objects in $\mathbf{E}$, and use the symbol $\epsin$ for the element-relation symbol in $\mathcal{L}_\mathsf{Set}$ and $\mathcal{L}_\mathsf{Class}$. The object interpreting the domain of $\mathcal{N}$ is denoted $N$. This means that we have a mono $n_S : S^\mathcal{N} \rightarrowtail N$ interpreting the sethood predicate $S$, a morphism $n_P : N \times N \rightarrow N$ interpreting the pairing function $\langle -, - \rangle$, and a mono $n_{\epsin} : \epsin^\mathcal{N} \rightarrow N \times N$ interpreting the element-relation $\epsin$. 

Sethood: $\mathcal{N} \models z \epsin x \rightarrow S(x)$ means that 
$$\llbracket x, y : N \mid x \epsin y \rrbracket \leq \llbracket x, y : N \mid S(y) \rrbracket = N \times S^\mathcal{N},$$ 
so there is a mono $n'_{\epsin} : \epsin^\mathcal{N} \rightarrowtail N \times S^\mathcal{N}$, such that $(\id_N \times n_S) \circ n'_{\epsin} = n_{\epsin}$.

Ordered Pair: $\mathcal{N} \models \langle x, y \rangle = \langle x\hspace{1pt}', y\hspace{1pt}' \rangle \rightarrow ( x = x\hspace{1pt}' \wedge y = y\hspace{1pt}' )$ means that $n_P : N \times N \rightarrowtail N$ is monic.

$\mathrm{Ext}_S$: $\mathcal{N} \models (S(x) \wedge S(y) \wedge \forall z . (z \epsin x \leftrightarrow z \epsin y)) \rightarrow x = y$ implies that for any pullback-square of the form below, $\chi$ is the unique morphism making this a pullback-square:
\[
\begin{tikzcd}[ampersand replacement=\&, column sep=small]
R \ar[rr] \ar[d, rightarrowtail, "r"'] \&\& {\epsin^\mathcal{N}} \ar[d, rightarrowtail, "n'_{\epsin}"] \\
N \times B \ar[rr, "{\id \times \chi}"] \&\& N \times S^\mathcal{N}
\end{tikzcd} 
\tag{A}
\]
To see this, we shall work with the natural $\mathcal{S}^{\mathbf{E}}_{\mathbf{E}}$-structure, which expands $\mathcal{N}$. Let $b, b\hspace{1pt}' : B \rightarrow S^\mathcal{N}$, such that $r$ is a pullback of $n'_{\epsin}$, both along $\id \times b$ and along $\id \times b\hspace{1pt}'$. By the categorical semantics, $r$ then represents both $\llbracket z : N, v : B \mid z \epsin b(v) \rrbracket$ and $\llbracket z : N, v : B \mid z \epsin b\hspace{1pt}'(v) \rrbracket$. So 
$$\mathbf{E} \models \forall v : B . \forall z : N . (z \epsin b(v) \leftrightarrow z \epsin b\hspace{1pt}'(v)),$$ whence by $\mathbf{E} \models \mathrm{Ext}_S$, we have $\mathbf{E} \models \forall v : B . b(v) = b\hspace{1pt}'(v)$. It follows that $b = b\hspace{1pt}'$.

$\mathrm{SC}_S$: For all stratified $\phi$, $\mathcal{N} \models \setmany z . \phi(z, y)$. Although this remark is not needed for the proof, it may help to clarify: $\mathcal{N} \models \mathrm{SC}_S$ implies that for any stratified $\mathcal{L}_\mathsf{Set}$-formula $\phi(z, y)$, there is a morphism $\chi : N \rightarrow S^\mathcal{N}$ making this a pullback-square:
\[
\begin{tikzcd}[ampersand replacement=\&, column sep=small]
{\llbracket z, y : N \mid \phi(z, y) \rrbracket} \ar[rr] \ar[d, rightarrowtail, "r"'] \&\& {\epsin^\mathcal{N}} \ar[d, rightarrowtail, "n'_{\epsin}"] \\
N \times N \ar[rr, "{\id \times \chi}"] \&\& N \times S^\mathcal{N}
\end{tikzcd}
\]

By the pullback-property of the power object $\mathbf{P}N$ of $N$, there is a unique $\chi_S$ making this a pullback-square:
\[
\begin{tikzcd}[ampersand replacement=\&, column sep=small]
{\epsin^\mathcal{N}} \ar[rr] \ar[d, rightarrowtail, "n'_{\epsin}"'] \&\& {\in_N} \ar[d, rightarrowtail, "n_{\in_N}"] \\
N \times S^\mathcal{N} \ar[rr, "{\id \times \chi_S}"] \&\& N \times \mathbf{P}N
\end{tikzcd} 
\tag{B}
\]
By combining (A) with (B), we find that $\chi_S$ is monic: Let $b, b\hspace{1pt}' : B \rightarrow S^\mathcal{N}$, such that $\chi_S \circ b = \chi_S \circ b\hspace{1pt}'$. Let $r : R \rightarrowtail N \times B$ and $r' : R' \rightarrowtail N \times B$ be the pullbacks of $n'_{\epsin}$ along $\id_N \times b$ and $\id_N \times b\hspace{1pt}'$, respectively. Consider these pullbacks as instances of (A) above. ``Gluing\hspace{1pt}'' each of these pullback-diagrams with (B) along the common morphism $n'_{\epsin}$, yields two new pullback-diagrams with the bottom morphisms $\mathrm{id}_N \times (\chi_S \circ b)$ and $\mathrm{id}_N \times (\chi_S \circ b\hspace{1pt}')$, respectively. (It is a basic and well-known property of pullbacks that such a ``gluing\hspace{1pt}'' of two pullback yields another pullback.) We know that these bottom morphisms are equal. Thus, by uniqueness of pullbacks up to isomorphism, we may assume that $r = r'$ and $R = R'$. Now it follows from the uniqueness of $\chi$ in (A) that $b = b\hspace{1pt}'$.

We proceed to construct an $\mathcal{L}_\mathsf{Class}$-structure $\mathcal{M}$ in $\mathbf{E}$, such that $\mathcal{M} \models \mathrm{(I)MLU}$. The domain $M$ of the structure is constructed as this pushout:
\[
\begin{tikzcd}[ampersand replacement=\&, column sep=small]
S^\mathcal{N} \ar[rr, rightarrowtail, "\chi_S"] \ar[d, rightarrowtail, "n_S"'] \&\& {\mathbf{P}N} \ar[d, rightarrowtail, "m_C"] \\
N \ar[rr, rightarrowtail, "m_{\mathrm{Setom}}"] \&\& M
\end{tikzcd} 
\tag{C}
\]
By Proposition \ref{topos prop}, $m_{\mathrm{Setom}}$ and $m_C$ are monic.

We interpret the predicate $\mathrm{Setom}$ by the mono $m_{\mathrm{Setom}} : N \rightarrowtail M$ and the classhood predicate $C$ by the mono $m_C : \mathbf{P}N \rightarrowtail M$. Naturally, we interpret $S$ by the mono $m_S =_\mathrm{df} n_\mathrm{Setom} \circ n_S : S^\mathcal{N} \rightarrowtail M$ and the partial ordered pair function by the mono $m_P =_\mathrm{df} m_{\mathrm{Setom}} \circ n_P : N \times N \rightarrow M$. (Formally, ordered pair is treated as a ternary relation symbol in $\mathcal{L}_\mathsf{Class}$, which is interpreted by $\llbracket x, y, z : M \mid \exists x\hspace{1pt}', y\hspace{1pt}', z\hspace{1pt}' : N . (m_{\mathrm{Setom}}(x\hspace{1pt}') = x \wedge m_{\mathrm{Setom}}(y\hspace{1pt}') = y \wedge m_{\mathrm{Setom}}(z\hspace{1pt}') = z \wedge m_P(x\hspace{1pt}', y\hspace{1pt}') = z\hspace{1pt}' \rrbracket$.)

We interpret the element-relation $\epsin$ by the mono
\begin{align*}
m_{\epsin} &=_\mathrm{df} (m_{\mathrm{Setom}} \times \id_M) \circ (\id_N \times m_C) \circ n_{\in_N} \\
&= (m_{\mathrm{Setom}} \times m_C) \circ n_{\in_N},
\end{align*} 
from  $\in_N$ to $M \times M$. Moreover, let $m'_{\epsin} = (\id_N \times m_C) \circ n_{\in_N}$, so that $m_{\epsin} = (m_{\mathrm{Setom}} \times \id_M) \circ m'_{\epsin}$. Now consider this diagram, obtained by gluing (B) on top of the diagram ``$N \times$ (C)'':
\[
\begin{tikzcd}[ampersand replacement=\&, column sep=small]
{\epsin^\mathcal{N}} \ar[rr] \ar[d, rightarrowtail, "n'_{\epsin}"'] \&\& {\in_N} \ar[d, rightarrowtail, "n_{\in_N}"] \\
N \times S^\mathcal{N} \ar[rr, rightarrowtail, "\id \times \chi_S"] \ar[d, rightarrowtail, "\id \times n_S"'] \&\& {N \times \mathbf{P}N} \ar[d, rightarrowtail, "\id \times m_C"] \\
N \times N \ar[rr, rightarrowtail, "\id \times m_{\mathrm{Setom}}"] \&\& N \times M
\end{tikzcd} 
\tag{D}
\]

The lower square is a pushout because (C) is, so by Proposition \ref{topos prop} it is a pullback. Since (B) is also a pullback, we have by a basic well-known result that (D) is also a pullback. It follows that $\llbracket z, x : N \mid z \epsin x \rrbracket ^\mathcal{N} \cong_{N \times N} \llbracket z, x : N \mid m_{\mathrm{Setom}}(z) \epsin m_{\mathrm{Setom}}(x) \rrbracket ^\mathcal{M}$. In other words, the interpretations of $\epsin$ in $\mathcal{N}$ and $\mathcal{M}$ agree on $N$ as a subobject of $M$ represented by $m_\mathrm{Setom}$. We can now easily verify the axioms of $\mathrm{(I)MLU}_\mathsf{Class}$.

Classhood: $\mathcal{M} \models z \epsin x \rightarrow C(x)$ follows from that $m_{\epsin} = (m_{\mathrm{Setom}} \times \id_M) \circ (\id_N \times m_C) \circ n_{\in_N}$.

Setomhood: $\mathcal{M} \models z \epsin x \rightarrow \mathrm{Setom}(x)$ also follows from that $m_{\epsin} = (m_{\mathrm{Setom}} \times \id_M) \circ (\id_N \times m_C) \circ n_{\in_N}$.

$\mathrm{Ext}_C$: $\mathcal{M} \models (C(x) \wedge C(y) \wedge \forall z . (z \epsin x \leftrightarrow z \epsin y)) \rightarrow x = y$ follows from that $\mathbf{E} \models \forall x, y : \mathbf{P}N . \big( (\forall z : N . (z \in x \leftrightarrow z \in y)) \rightarrow x = y \big)$ (see Proposition \ref{topos prop}), that $m_\mathrm{Setom}$ is monic, and that $m_{\epsin} = (m_{\mathrm{Setom}} \times \id_M) \circ (\id_N \times m_C) \circ n_{\in_N}$.

$\mathrm{CC}_C:$ For all $\mathcal{L}_\mathsf{Class}$-formulae $\phi$, $\mathcal{M} \models \exists x . \big( C(x) \wedge \forall z \in \mathrm{Setom} . (z \epsin x \leftrightarrow \phi(z))\big)$, follows from that $\mathbf{E} \models \exists x : \mathbf{P}N . \forall z : N . (z \in x \leftrightarrow \phi(z))$ and that $m_{\epsin} = (m_{\mathrm{Setom}} \times \id_M) \circ (\id_N \times m_C) \circ n_{\in_N}$.

$\mathrm{SC}_S:$ For all stratified $\mathcal{L}_\mathsf{Class}$-formulae $\phi$ with only $z, \vec{y}$ free, $\mathcal{M} \models \forall \vec{y} \in \mathrm{Setom} . \setmany z . \phi(z, \vec{y})$, follows from that $\mathcal{N} \models \mathrm{SC}_S$, and that 
$$\llbracket x, y : N \mid x \epsin y \rrbracket ^\mathcal{N} \cong_{N \times N} \llbracket x, y : N \mid m_{\mathrm{Setom}}(x) \epsin m_{\mathrm{Setom}}(y) \rrbracket ^\mathcal{M}.$$

Ordered Pair: $\mathcal{M} \models \forall x, y, x\hspace{1pt}', y\hspace{1pt}' \in \mathrm{Setom} . ( \langle x, y \rangle = \langle x\hspace{1pt}', y\hspace{1pt}' \rangle \leftrightarrow ( x = x\hspace{1pt}' \wedge y = y\hspace{1pt}' ) )$, follows from that $\mathcal{N} \models \textnormal{Ordered Pair}$, that $m_\mathrm{Setom}$ is monic, and that $m_P = m_\mathrm{Setom} \circ n_P$.

Set equals Setom Class: $\mathcal{M} \models S(x) \leftrightarrow (\mathrm{Setom}(x) \wedge C(x))$, follows from that $m_S = m_\mathrm{Setom} \circ n_S$ and that $n_S$ is a pullback of $m_C$ along $m_\mathrm{Setom}$, as seen in diagram (C).

This concludes the verification of $\mathcal{M} \models \mathrm{(I)MLU}_\mathsf{Class}$. For the case without atoms, note that if $ \mathcal{N} \models \forall x . S(x)$, then $n_S$ is an iso, so since $m_S = m_\mathrm{Setom} \circ n_S$, we have $\mathcal{M} \models \forall x . (S(x) \leftrightarrow \mathrm{Setom}(x))$.
\end{proof}

\begin{cor}\label{ConSetConClass}
$\mathrm{(I)NF(U)}_\mathsf{Set}$ is equiconsistent to $\mathrm{(I)ML(U)}_\mathsf{Class}$.
\end{cor}
\begin{proof}
The $\Leftarrow$ direction was established directly after Axioms \ref{AxIMLU}. For the $\Rightarrow$ direction: By the completeness theorem for intuitionistic predicate logic and Kripke models, there is a Kripke model of $\mathrm{(I)NF(U)}_\mathsf{Set}$, i.e. there is a partial order $\mathbb{P}$ and an $\mathcal{L}_\mathsf{Set}$-structure $\mathcal{N}$ in $\mathbf{Set}^\mathbb{P}$, such that $\mathcal{N} \models \mathrm{(I)NF(U)}_\mathsf{Set}$. By Proposition \ref{topos prop}, $\mathbf{Set}^\mathbb{P}$ is a topos, so it follows from Theorem \ref{ToposModelOfML} that there is a Kripke model of $\mathrm{(I)ML(U)}_\mathsf{Class}$. The classical cases are obtained by setting $\mathbb{P}$ to a singleton.
\end{proof}

\begin{rem}\label{remML}
An equiconsistency statement is trivial unless the consistency strength of the theories considered is at least that of the meta-theory. It is folklore that the consistency strength of $\mathrm{NFU}_\mathsf{Set}$ is at least that of a weak set theory called Mac Lane set theory (by \cite{Jen69} it is at most that), and that the category of presheaves is a topos with Mac Lane set theory as meta-theory, so for the classical case the statement is non-trivial. Moreover, if one unpacks the above equiconsistency proof, one finds that the full Powerset axiom is not needed. It suffices that powersets of countable sets exist, to construct the needed Kripke structure. The strengths of $\mathrm{INF(U)}_\mathsf{Set}$ have not been studied much, so the non-triviality of the above equiconsistency statement needs to be taken as conditional in that case. Regardless of these matters, the proof of Theorem \ref{ToposModelOfML} is constructive and yields information on the close relationship between $\mathrm{(I)NF(U)_\mathsf{Set}}$ and $\mathrm{(I)ML(U)}_\mathsf{Class}$ (also in the intuitionistic case).
\end{rem}

\chapter{Stratified algebraic set theory}\label{ch equicon set cat NF}

\section{Stratified categories of classes}\label{Formulation of ML_CAT}

We now proceed to introduce a new categorical theory, intended to characterize the categorical content of predicative $\mathrm{(I)ML(U)}_\mathrm{Class}$. For comparison, let us first recall the definition of topos.

We need a relativized notion of power object, for the axiomatization to be presented below:

\begin{dfn}
Let $\mathbf{C}$ be a category, and let $\mathbf{D}$ be a subcategory of $\mathbf{C}$. A {\em power object in} $\mathbf{C}$ {\em with respect to} $\mathbf{D}$, of an object $A$ in $\mathbf{D}$, is defined as in Definition \ref{power object}, except that $r$ is assumed to be in $\mathbf{D}$ and $m, \chi$ are required to be in $\mathbf{D}$. More precisely, it is an object $\mathbf{P}A$ along with a morphism $m : \hspace{2pt} \in \hspace{2pt} \rightarrowtail A \times \mathbf{P}A$ in $\mathbf{D}$ which is monic in $\mathbf{C}$, such that for any $r : R \rightarrowtail A \times B$ in $\mathbf{D}$ which is monic in $\mathbf{C}$, there is a morphism $\chi : B \rightarrow \mathbf{P}A$ in $\mathbf{D}$, which is the unique morphism in $\mathbf{C}$ making this a pullback square in $\mathbf{C}$:
\[
\begin{tikzcd}[ampersand replacement=\&, column sep=small]
R \ar[rr] \ar[d, rightarrowtail, "r"'] \&\& \in \ar[d, rightarrowtail, "m"] \\
A \times B \ar[rr, "{\id \times \chi}"] \&\& A \times \mathbf{P}A 
\end{tikzcd} 
\]
\end{dfn}

We need a couple of more definitions: A functor $\mathbf{F} : \mathbf{C} \rightarrow \mathbf{D}$ is {\em conservative} if for any morphism $f$ in $\mathbf{C}$, if $\mathbf{F}(f\hspace{2pt})$ is an isomorphism then $f$ is an isomorphism. A subcategory is {\em conservative} if its inclusion functor is conservative. A {\em universal object} in a category $\mathbf{C}$ is an object $X$, such that for every object $Y$ there is a mono $f : Y \rightarrowtail X$. The theory $\mathrm{IMLU}_\mathsf{Cat}$ is axiomatized as follows. 

\begin{dfn}[$\mathrm{IMLU}_\mathsf{Cat}$] \label{IMLU_CAT}
A {\em stratified category of classes} (or an $\mathrm{IMLU}$-category) is a pair of Heyting categories $( \mathbf{M}, \mathbf{N})$, such that
\begin{itemize}
\item $\mathbf{N}$ is a conservative Heyting subcategory of $\mathbf{M}$, 
\item there is an object $U$ in $\mathbf{N}$ which is universal in $\mathbf{N}$,
\item there is an endofunctor $\mathbf{T}$ on $\mathbf{M}$, restricting to an endofunctor of $\mathbf{N}$ (also denoted $\mathbf{T}$), along with a natural isomorphism $\iota : \id_\mathbf{M} \xrightarrow{\sim} \mathbf{T}$ on $\mathbf{M}$, 
\item there is an endofunctor $\mathbf{P}$ on $\mathbf{N}$, such that for each object $A$ in $\mathbf{N}$, $\mathbf{T}A$ has a power object $\mathbf{P}A$, $m_{\subseteq^{\mathbf{T}}_A} : \hspace{1pt} \subseteq^{\mathbf{T}}_A \hspace{1pt} \rightarrowtail \mathbf{T} A \times \mathbf{P} A$ in $\mathbf{M}$ with respect to $\mathbf{N}$; spelling this out:
	\begin{itemize}
	\item $m_{\subseteq^{\mathbf{T}}_A}$ is a morphism in $\mathbf{N}$ which is monic in $\mathbf{M}$, such that
	\item for any $r : R \rightarrowtail \mathbf{T} A \times B$ in $\mathbf{N}$ which is monic in $\mathbf{M}$, there is $\chi : B \rightarrow \mathbf{P} A$ in $\mathbf{N}$, which is the unique morphism in $\mathbf{M}$ making this a pullback square in $\mathbf{M}$:
	\end{itemize}
\begin{equation} \label{PT}
\begin{tikzcd}[ampersand replacement=\&, column sep=small]
R \ar[rr] \ar[d, rightarrowtail, "r"'] \&\& \subseteq^{\mathbf{T}}_A \ar[d, rightarrowtail, "{m_{\subseteq^{\mathbf{T}}_A}}"] \\
\mathbf{T} A \times B \ar[rr, "{\id \times \chi}"] \&\& \mathbf{T} A \times \mathbf{P} A 
\end{tikzcd} 
\tag{PT}
\end{equation}
\item  there is a natural isomorphism $\mu : \mathbf{P} \circ \mathbf{T} \xrightarrow{\sim} \mathbf{T} \circ \mathbf{P}$ on $\mathbf{N}$.
\end{itemize}

If ``Heyting\hspace{1pt}'' is replaced with ``Boolean'' throughout the definition, then we obtain the theory $\mathrm{MLU}_\mathsf{Cat}$. If $U \cong \mathbf{P}U$ is added to $\mathrm{(I)MLU}_\mathsf{Cat}$, then we obtain the theories $\mathrm{(I)ML}_\mathsf{Cat}$, respectively.
\end{dfn}

In order to carry over some intuitions from a stratified set theory such as $\mathrm{NFU}$, $\mathbf{T}A$ may be thought of as $\{\{x\} \mid x \in A\}$ and $\mathbf{P}A$ may be thought of as $\{ X \mid X \subseteq A\}$. Now $\subseteq^{\mathbf{T}}_A$ corresponds to the subset relation on $\mathbf{T}A \times \mathbf{P}A$. Note that on this picture, $\subseteq^{\mathbf{T}}$ is very similar to the $\in$-relation. Thus (\ref{PT}) is intended to be the appropriate variant for stratified set theory of the power object axiom of topos theory. These intuitions are made precise in the proof of Theorem \ref{ConClassConCat}, where we interpret $\mathrm{(I)ML(U)}_\mathsf{Cat}$ in $\mathrm{(I)ML(U)}_\mathsf{Class}$.

It is easily seen that this axiomatization is elementary, i.e. it corresponds to a theory in a first order language $\mathcal{L}_\mathsf{Cat}$. Its precise specification involves quite some detail. Suffice to say that the language of category theory is augmented with relation symbols $\mathbf{M}_\mathrm{Ob}$, $\mathbf{M}_\mathrm{Mor}$, $\mathbf{N}_\mathrm{Ob}$, and $\mathbf{N}_\mathrm{Mor}$; a constant symbol $U$; and function symbols $\mathbf{T}_\mathrm{Ob}$, $\mathbf{T}_\mathrm{Mor}$, $\iota$, $\mu$, $\mathbf{P}_\mathrm{Ob}$ and $\mathbf{P}_\mathrm{Mor}$ (using the same names for the symbols and their interpretations, and where the subscripts $\mathrm{Ob}$ and $\mathrm{Mor}$ indicate the component of the functor acting on objects and morphisms, respectively).

Note that the definition can easily be generalized, so that we merely require that $\mathbf{N}$ is a Heyting category that is mapped into $\mathbf{M}$ by a faithful conservative Heyting functor $\mathbf{F} : \mathbf{N} \rightarrow \mathbf{M}$. This would not hinder any of the results below. We choose the more specific definition in terms of a subcategory because it simplifies the statements of the results.

We shall now collect a few useful properties of $\mathrm{(I)ML(U)}$-categories. First a definition: A functor $\mathbf{F} : \mathbf{B} \rightarrow \mathbf{C}$ {\em reflects finite limits} if for any finite diagram $\mathbf{D} : \mathbf{I} \rightarrow \mathbf{B}$ and for any cone $\Lambda$ of $\mathbf{D}$ in $\mathbf{B}$, if $\mathbf{F}\Lambda$ is a limit in $\mathbf{C}$ of $\mathbf{F} \circ \mathbf{D} : \mathbf{I} \rightarrow \mathbf{C}$, then $\Lambda$ is a limit of $\mathbf{D} : \mathbf{I} \rightarrow \mathbf{B}$ in $\mathbf{B}$.

\begin{prop} \label{strenghten def}
Let $( \mathbf{M}, \mathbf{N})$ along with $U$, $\mathbf{T}$, $\iota$, $\mathbf{P}$ and $\mu$ be an $\mathrm{IMLU}$-category.
\begin{enumerate}[{\normalfont (a)}]
\item \label{strengthen mono} For any morphism $f : A \rightarrow B$ in $\mathbf{N}$, $f$ is monic in $\mathbf{N}$ iff $f$ is monic in $\mathbf{M}$.
\item \label{strengthen reflect limits} The inclusion functor of $\mathbf{N}$ (as a subcategory) into $\mathbf{M}$ reflects finite limits.
\item \label{strengthen power} $\mathbf{P}A$ along with $m_{\subseteq^{\mathbf{T}}_A}$, as in {\em (\ref{PT})} above, is a power object of $\mathbf{T}A$ in $\mathbf{N}$, for any $A$ in $\mathbf{N}$.
\item \label{strengthen T heyting} $\mathbf{T} : \mathbf{M} \rightarrow \mathbf{M}$ is a Heyting endofunctor. If $( \mathbf{M}, \mathbf{N})$ is an $\mathrm{MLU}$-category, then $\mathbf{T} : \mathbf{M} \rightarrow \mathbf{M}$ is a Boolean endofunctor.
\item \label{strengthen T limits} $\mathbf{T} : \mathbf{N} \rightarrow \mathbf{N}$ preserves finite limits.
\end{enumerate}
\end{prop}
\begin{proof}
(\ref{strengthen mono}) ($\Leftarrow$) follows immediately from that $\mathbf{N}$ is a subcategory of $\mathbf{M}$. ($\Rightarrow$) follows from that $\mathbf{N}$ is a Heyting subcategory of $\mathbf{M}$, and that Heyting functors preserve pullbacks, because in general, a morphism $m : A \rightarrow B$ is monic iff $\id_A : A \rightarrow A$ and $\id_A : A \rightarrow A$ form a pullback of $m$ and $m$ (as is well known and easy to check).

(\ref{strengthen reflect limits}) Let $L$, along with some morphisms in $\mathbf{N}$, be a cone in $\mathbf{N}$ of a finite diagram $\mathbf{D} : \mathbf{I} \rightarrow \mathbf{N}$, such that this cone is a limit of $\mathbf{D} : \mathbf{I} \rightarrow \mathbf{N}$ in $\mathbf{M}$. Let $K$ be a limit in $\mathbf{N}$ of $\mathbf{D} : \mathbf{I} \rightarrow \mathbf{N}$. Since $\mathbf{N}$ is a Heyting subcategory of $\mathbf{M}$, $K$ is also such a limit in $\mathbf{M}$. Let $f : L \rightarrow K$ be the universal morphism in $\mathbf{N}$ obtained from the limit property of $K$ in $\mathbf{N}$. By the limit properties of $K$ and $L$ in $\mathbf{M}$, $f$ is an isomorphism in $\mathbf{M}$. Since $\mathbf{N}$ is a conservative subcategory of $\mathbf{M}$, $f$ is also an isomorphism in $\mathbf{N}$, whence $L$ is a limit in $\mathbf{N}$ of $\mathbf{D} : \mathbf{I} \rightarrow \mathbf{N}$, as desired.

(\ref{strengthen power}) By (\ref{strengthen mono}), any morphism in $\mathbf{N}$ that is monic in $\mathbf{M}$ is also monic in $\mathbf{N}$. Let $A$ be an object of $\mathbf{N}$. By (\ref{strengthen reflect limits}), (\ref{PT}) is a pullback square in $\mathbf{N}$. Suppose that $\chi\hspace{1pt}'$ in $\mathbf{N}$ makes (\ref{PT}) a pullback in $\mathbf{N}$ (in place of $\chi$). Since $\mathbf{N}$ is a Heyting subcategory of $\mathbf{M}$, $\chi\hspace{1pt}'$ also makes (\ref{PT}) a pullback square in $\mathbf{M}$. So by the uniqueness property in $\mathbf{M}$, $\chi\hspace{1pt}' = \chi$.

(\ref{strengthen T heyting}) Since $\mathbf{T} : \mathbf{M} \rightarrow \mathbf{M}$ is naturally isomorphic to the identity functor, which is trivially a Heyting (Boolean) functor, $\mathbf{T}$ is also a Heyting (Boolean) endofunctor of $\mathbf{M}$.

(\ref{strengthen T limits}) Let $L$ be a limit in $\mathbf{N}$ of a finite diagram $\mathbf{D} : \mathbf{I} \rightarrow \mathbf{N}$. By (\ref{strengthen T heyting}), $\mathbf{T} : \mathbf{M} \rightarrow \mathbf{M}$ preserves limits, so $\mathbf{T}L$ is a limit in $\mathbf{M}$ of $\mathbf{T} \circ \mathbf{D} : \mathbf{I} \rightarrow \mathbf{N}$. By (\ref{strengthen reflect limits}), $\mathbf{T}L$ is also a limit in $\mathbf{N}$ of $\mathbf{T} \circ \mathbf{D} : \mathbf{I} \rightarrow \mathbf{N}$.
\end{proof}

We now proceed to show $\mathrm{Con}(\mathrm{(I)NF(U)}_\mathsf{Class}) \Rightarrow \mathrm{Con}(\mathrm{(I)ML(U)}_\mathsf{Cat})$. This is the easy and perhaps less interesting part of the equiconsistency proof, but it has the beneficial spin-off of showing how the axioms of $\mathrm{(I)ML(U)}_\mathsf{Cat}$ correspond to set theoretic intuitions. Given Corollary \ref{ConSetConClass}, it suffices to find an interpretation of $\mathrm{(I)ML(U)}_\mathsf{Cat}$ in $\mathrm{(I)ML(U)}_\mathsf{Class}$, as is done in the proof below. This proof actually shows that $\mathrm{(I)ML(U)}_\mathsf{Cat}$ can be interpreted in {\em predicative} $\mathrm{(I)ML(U)}_\mathsf{Class}$; the formulae used in the class-abstracts of the proof only need quantifiers bounded to the extension of $\mathrm{Setom}$.

\begin{thm}\label{ConClassConCat}
$\mathrm{(I)ML(U)}_\mathsf{Cat}$ is interpretable in $\mathrm{(I)ML(U)}_\mathsf{Class}$.
\end{thm}
\begin{proof}
We go through the case of $\mathrm{IMLU}$ in detail, and then explain the modifications required for the other cases. Throughout the interpretation, we work in $\mathrm{IMLU}_\mathsf{Class}$, introducing class and set abstracts $\{x \mid \phi(x, p)\}$, whose existence are justified by the axioms $\mathrm{CC}_C$ and $\mathrm{SC}_S$, respectively. Such class and set abstracts satisfy $\forall x . ( x \in \{x\hspace{1pt}' \mid \phi(x\hspace{1pt}', p)\} \leftrightarrow (\phi(x, p) \wedge x  \in \mathrm{Setom}) )$. Whenever $\phi(x, p)$ is stratified and we have $\mathrm{Setom}(p)$, then the corresponding set abstract exists (and is also a class). Because of the stratification constraint on ordered pairs, when showing that a function $(x \mapsto y)$ defined by $\phi(x, y, p)$ is coded as a set, we have to verify that $\phi(x, y, p)$ can be stratified with the same type assigned to $x$ and $y$. There are no constraints on $\phi$, for a class abstract to exist. Throughout the proof, these $\phi$ are written out explicitly, but for the most part the stratification verifications are simple and left to the reader. 

The interpretation proceeds as follows.
	\begin{enumerate}
	\item Interpret $\mathbf{M}_\mathrm{Ob}(x)$ as $C(x)$, i.e. ``$x$ is a class''.
	\item Interpret $\mathbf{M}_\mathrm{Mor}(m)$ as ``$m$ is a disjoint union of three classes $A$, $B$ and $f$, such that $f$ is a set of pairs coding a function with domain $A$ and co-domain $B$''.\footnote{A disjoint union of three classes may be implemented as a class using the formula $\langle i, x \rangle \in m \leftrightarrow \big( (i = 1 \wedge x \in A) \vee (i = 2 \wedge x \in B) \vee (i = 3 \wedge x \in f\hspace{2pt})\big)$. In order to be able to interpret the domain and co-domain function symbols, we need to include information about the domain class and co-domain class in the interpretation of the morphisms. Otherwise, the same functional class will often interpret many morphisms with different co-domains.} For convenience, we extend the functional notation to $m$ in this setting, i.e. $m(x) =_\mathrm{df} f\hspace{2pt}(x)$, for all $x \in A$, and we also say that $m$ codes this function/morphism from $A$ to $B$.
	\item Interpret the remaining symbols of the language of category theory in the obvious way. Most importantly, composition of morphisms is interpreted by composition of functions. The resulting interpretations of the axioms of category theory are now easily verified for $\mathbf{M}$.
	\item Interpret $\mathbf{N}_\mathrm{Ob}(x)$ as ``$x$ is a set''; and interpret $\mathbf{N}_\mathrm{Mor}(m)$ as ``[insert the interpretation of $\mathbf{M}_\mathrm{Mor}(m)$] and $m$ is a set''. The axioms of category theory are now easily verified for $\mathbf{N}$.
	\item We need to show that the interpretation of the axioms of Heyting categories hold for $\mathbf{M}$ and $\mathbf{N}$. It is well-known that these axioms hold for the categories of classes and sets in conventional class and set theory, see for example \cite{Gol06}. Here we use the same class and set constructions, we just need to check that the axioms $\mathrm{CC}_C$ and $\mathrm{SC}_S$ of $\mathrm{IMLU}_\mathsf{Cat}$ are strong enough to yield the needed sets. $\mathrm{Ext}_C$ ensures the uniqueness conditions in the axioms. 
	
	Existence conditions are supported by class/set abstracts $\{x \mid \phi(x)\}$, where the formula $\phi$ is stratified. We write out each such $\phi$ explicitly and let the reader do the simple verification that $\phi$ is stratified. Thus, in the case of $\mathbf{N}$ we can rely on $\mathrm{SC}_S$, and in the case of $\mathbf{M}$ we can rely on CC$_C$. The only difference is that in the latter case the formula $\phi$ in the class abstract may have parameters which are proper classes. So we can do the verifications for $\mathbf{M}$ and $\mathbf{N}$ simultaneously. 
	
	Let $m : A \rightarrow B$ and $n : C \rightarrow B$ be morphisms in $\mathbf{M}$ or $\mathbf{N}$. Note that for $\mathbf{M}$ and $\mathbf{N}$, subobjects are represented by subclasses and subsets, respectively. Moreover, in both $\mathbf{M}$ and $\mathbf{N}$, any morphism is monic iff injective, and is a cover iff surjective. 
		\begin{enumerate}
		\item[(F1)] Finite limits: It is well-known that the existence of all finite limits follows from the existence of a terminal object and the existence of all pullbacks. $\{\varnothing\}$ is a terminal object. $D =_\mathrm{df} \{ \langle x, z  \rangle \in A \times C \mid m(x) = n(z)\}$, along with the restricted projection morphisms $\pi^1\restriction_D : D \rightarrow A$ and $\pi^2\restriction_D : D \rightarrow C$, is a pullback of the morphisms $m : A \rightarrow B$ and $n : C \rightarrow B$.
		\item[(F2)] Images: The class or set $\{m(x) \mid x \in A\} \subseteq B$, along with its inclusion function into $B$ is the image of $m : A \rightarrow B$.
		\item[(F3)] The pullback of any cover is a cover: Consider the pullback of $m$ and $n$ considered above, and suppose that $m$ is surjective. Then, for any $c \in C$, there is $a \in A$ such that $m(a) = n(c)$, whence $\langle a, c \rangle \in D$. So the projection $D \rightarrow C$ is surjective, as required.
		\item[(F4)] Each $\mathrm{Sub}_X$ is a sup-semilattice under $\subseteq$: Since subobjects are represented by subclasses/subsets, each $\mathrm{Sub}_X$ is the partial order of subclasses/subsets of $X$. Binary union, given by the set abstract $\{z \mid z \in A \vee z \in B\}$, yields the binary suprema required for $\mathrm{Sub}_X$ to be a sup-semilattice. (Note that $\mathrm{Sub}_X$ does not need to be implemented as a set or a class.)
		\item[(F5)] For each morphism $f : X \rightarrow Y$, the functor $f\hspace{2pt}^* : \mathrm{Sub}_Y \rightarrow \mathrm{Sub}_X$ preserves finite suprema and has left and right adjoints, $\exists_f \dashv f\hspace{2pt}^* \dashv \forall_f\hspace{2pt} $: 
		
		$f\hspace{2pt}^*$ is the inverse image functor, mapping any subset $Y\hspace{1pt}' \subseteq Y$ to $\{x \in X \mid f\hspace{2pt}(x) \in Y\hspace{1pt}'\} \subseteq Y$, which clearly preserves finite suprema (unions). 
		
		$\exists_f$ is the image functor, which maps any subset $X\hspace{1pt}' \subseteq X$ to $\{f\hspace{2pt}(x) \mid x \in X\hspace{1pt}'\} \subseteq Y $. 
		
		$\forall_f$ is the functor mapping any subset $X\hspace{1pt}' \subseteq X$ to the set abstract $\{y \in Y \mid \forall x \in X .(f\hspace{2pt}(x) = y \rightarrow x \in X\hspace{1pt}')\} \subseteq Y$. 
		
		Let $X\hspace{1pt}' \subseteq X$ and $Y\hspace{1pt}' \subseteq Y$. It is easily seen that $\exists_f\hspace{2pt}(X\hspace{1pt}') \subseteq Y\hspace{1pt}' \iff X\hspace{1pt}' \subseteq f\hspace{2pt}^*(Y\hspace{1pt}')$, i.e. $\exists_f \dashv f\hspace{2pt}^*$. It is also easily seen that $f\hspace{2pt}^*(Y\hspace{1pt}') \subseteq X\hspace{1pt}' \iff Y\hspace{1pt}' \subseteq \forall_f\hspace{2pt}(X\hspace{1pt}')$, i.e. $f\hspace{2pt}^* \dashv \forall_f $.
		\end{enumerate}
		
	\item In the verification of the HC axioms above, when the objects and morphisms are in $\mathbf{N}$, the same sets are constructed regardless if the HC axioms are verified for $\mathbf{M}$ or $\mathbf{N}$. It follows that $\mathbf{N}$ is a Heyting subcategory of $\mathbf{M}$.
	\item In both $\mathbf{M}$ and $\mathbf{N}$, a morphism is an isomorphism iff it is bijective. Hence, $\mathbf{N}$ is a conservative subcategory of $\mathbf{M}$.
	\item Interpret $U$ as $V$, the set $\{x \mid x = x\}$, which is a superset of every set, and hence a universal object in $\mathbf{N}$.
	\item For any object $x$ and morphism $m : A \rightarrow B$ of $\mathbf{M}$, interpret $\mathbf{T}_\mathrm{Ob}(x)$ as $\{\{u\} \mid u \in x\}$; and interpret $\mathbf{T}_\mathrm{Mor}(m)$ as ``the class coding the morphism $(\{x\} \mapsto \{m(x)\}) : \mathbf{T}A \rightarrow \mathbf{T}B$''. Since these formulae stratified, $\mathbf{T}$ restricts appropriately to $\mathbf{N}$. It is easily verified that the interpreted axioms of a functor hold.
	\item For each object $x$ in $\mathbf{M}$, interpret $\iota_x$ as the code of the morphism $(z \mapsto \{z\}) : x \rightarrow \mathbf{T}(x)$, which is a class. Since the inverse of $\iota_x$ is similarly interpretable, we obtain that the interpretation of $\iota_x$ is an isomorphism in the category theoretic sense. That $\iota$ is a natural isomorphism on $\mathbf{M}$ is clear from its definition and the definition of $\mathbf{T}$. (A word of caution: $\iota_x$ is not generally a set even if $x$ is, in fact $\iota_V$ is a proper class.)
	\item For each object $x$ in $\mathbf{N}$, interpret $\mathbf{P}_\mathrm{Ob}(x)$ as $\pow x$. For each morphism $m : A \rightarrow B$ in $\mathbf{N}$, interpret $\mathbf{P}_\mathrm{Mor}(m)$ as ``the set coding the morphism $(x \mapsto \{m(z) \mid z \in x\}) : \pow A \rightarrow \pow B$''. It is easily seen that this makes $\mathbf{P}$ an endofunctor on $\mathbf{N}$.
	\item Let $x$ be an object in $\mathbf{N}$. Note that $\mathbf{P}\mathbf{T}x = \pow \{\{z\} \mid z \in x\}$. Interpret $\mu_x : \mathbf{P}\mathbf{T}x \rightarrow \mathbf{T}\mathbf{P}x$ by the set coding the morphism $(u \mapsto \{\cup u\}) : \pow \{\{z\} \mid z \in y\} \rightarrow \{\{v\} \mid v \in \pow y\}$. Union and singleton are defined by stratified formulae. Because the union operation lowers type by one and the singleton operation raises type by one, argument and value are type-level in the formula defining $\mu_x$, so $\mu_x$ is coded by a set and is therefore a morphism in $\mathbf{N}$. It is easily seen from the constructions of $\mathbf{T}$, $\mathbf{P}$ and $\mu$, that $\mu$ is a natural isomorphism.
	\item Define $x \subseteq^{\mathbf{T}} y$ set theoretically by $\exists u . (x = \{u\} \wedge u \in y)$. For each object $A$ of $\mathbf{N}$, interpret $\subseteq^{\mathbf{T}}_A \hookrightarrow \mathbf{T} A \times \mathbf{P} A$ as the set coding the inclusion function of $\{\langle x, y \rangle \in \mathbf{T}A \times \mathbf{P}A \mid x \subseteq^{\mathbf{T}} y\} \subseteq \mathbf{T}A \times \mathbf{P}A$.
	\item We proceed to verify that $\mathbf{T}$, $\mathbf{P}$ and $\subseteq^\mathbf{T}$ satisfy the property (\ref{PT}). Suppose that $r : R \rightarrowtail\mathbf{T} A \times B$ in $\mathbf N$ is monic in $\mathbf{M}$. In both $\mathbf{N}$ and $\mathbf{M}$, a morphism is monic iff it is injective, so $r$ is monic in $\mathbf{N}$. Let $\chi : B \rightarrow \pow A$ code the function $(y \mapsto \big\{u \mid \exists c \in R . r(c) = \langle \{u\}, y \rangle \big\})$. Since this is a stratified definition, where argument and value have equal type, $\chi$ is a morphism in $\mathbf{N}$. The proof that $\chi$ is the unique morphism making (\ref{PT}) a pullback in $\mathbf{M}$ is just like the standard proof in conventional set theory; it proceeds as follows. We may assume that $R \subseteq A \times B$ and $r$ is the inclusion function. Then $\chi$ is $(y \mapsto \{u \mid \{u\} R y \})$. For the top arrow in (\ref{PT}) we choose $(\id \times \chi)\restriction_R$. Since $\forall \langle \{u\}, y\rangle \in R . \{u\} \subseteq^{\mathbf{T}} \chi(y)$, (\ref{PT}) commutes. 
	
	For the universal pullback property: Suppose that $\langle f, g\rangle : Q \rightarrow \mathbf{T}A \times B$ and $\langle d, e\rangle : Q \rightarrow \subseteq^{\mathbf{T}}_A$ are morphisms in $\mathbf{N}$ making the diagram commute in $\mathbf{M}$. Let $q \in Q$ be arbitrary. Then $f\hspace{2pt}(q) = d(q)$, $\chi(g(q)) = e(q)$ and $d(q) \subseteq^\mathbf{T} e(q)$, so $f\hspace{2pt}(q) \subseteq^{\mathbf{T}} \chi(g(q))$, whence by definition of $\chi$ we have $f\hspace{2pt}(q) R g(q)$. Thus, $(q \mapsto \langle f\hspace{2pt}(q), g(q)\rangle)$ defines the unique morphism from $Q$ to $R$ in $\mathbf{M}$, witnessing the universal pullback property. Since its definition is stratified, it is also a morphism in $\mathbf{N}$.
	
	It remains to show that if $\chi\hspace{1pt}'$ is a morphism in $\mathbf{N}$ that (in place of $\chi$) makes (\ref{PT}) a pullback in $\mathbf{M}$, then $\chi\hspace{1pt}' = \chi$. Let $\chi\hspace{1pt}'$ be such a morphism, and let $u \in A$ and $y \in B$. Since $\{u\} R y \Leftrightarrow \{u\} \subseteq^{\mathbf{T}} \chi(y)$, it suffices to show that $\{u\} R y \Leftrightarrow \{u\} \subseteq^{\mathbf{T}} \chi\hspace{1pt}'(y)$. By commutativity $\{u\} R y \Rightarrow \{u\} \subseteq^{\mathbf{T}} \chi\hspace{1pt}'(y)$. Conversely, applying the universal pullback property to the inclusion function $\{\langle \{u\}, y\rangle\} \hookrightarrow \mathbf{T}A \times B$, we find that $\{u\} \subseteq^{\mathbf{T}} \chi\hspace{1pt}'(y) \Rightarrow \{u\} R y$.
	\end{enumerate}
This completes the interpretation of $\mathrm{IMLU}_\mathsf{Cat}$ in $\mathrm{IMLU}_\mathsf{Class}$. For $\mathrm{MLU}$, simply observe that $\mathrm{MLU}_\mathsf{Class} \vdash \forall X . \forall X\hspace{1pt}' \subseteq X . X\hspace{1pt}' \cup (X - X\hspace{1pt}') = X$, so each $\mathrm{Sub}_X$ is Boolean. For $\mathrm{(I)ML}$,  the fact that $V = \pow V$ ensures that the interpretation of $U \cong \mathbf{P}U$ holds.
\end{proof}

\section{Interpretation of the $\mathsf{Set}$-theories in the $\mathsf{Cat}$-theories} \label{interpret set in cat}

For the rest of the paper, fix an $\mathrm{IMLU}$-category $( \mathbf{M}, \mathbf{N})$ -- along with $U$ in $\mathbf{N}$, $\mathbf{T}: \mathbf{M} \rightarrow \mathbf{M}$ (restricting to an endofunctor of $\mathbf{N}$), $\iota : \mathbf{id} \xrightarrow{\sim} \mathbf{T}$ on $\mathbf{M}$, $\mathbf{P}: \mathbf{N} \rightarrow \mathbf{N}$, $\mu : \mathbf{P}\circ\mathbf{T} \xrightarrow{\sim} \mathbf{T}\circ\mathbf{P}$, and $\subseteq^{\mathbf{T}}_X \rightarrowtail \mathbf{T} X \times \mathbf{P}X$ (for each object $X$ in $\mathbf{N}$) -- all satisfying the conditions in Definition \ref{IMLU_CAT}. Moreover, fix an object $\mathbf{1}$ which is terminal in both $\mathbf{M}$ and $\mathbf{N}$ and fix a product functor $\times$ on $\mathbf{M}$ which restricts to a product functor on $\mathbf{N}$. This can be done since $\mathbf{N}$ is a Heyting subcategory of $\mathbf{M}$. Given an $n \in \mathbb{N}$ and a product $P$ of $n$ objects, the $i$-th projection morphism, for $i = 1, \dots , n$, is denoted $\pi_P^i$.

In this section, we shall establish that 
$$\mathrm{Con}(\mathrm{(I)ML(U)}_\mathsf{Cat}) \Rightarrow \mathrm{Con}(\mathrm{(I)NF(U)}_\mathsf{Set}).$$ 
We do so by proving that the axioms of $\mathrm{(I)NF(U)}_\mathsf{Set}$ can be interpreted in the internal language of $\mathrm{(I)ML(U)}_\mathsf{Cat}$. In particular, we construct a structure in the categorical semantics of $\mathbf{M}$ which satisfies the axioms of $\mathrm{(I)NF(U)}_\mathsf{Set}$. The variation between the intuitionistic and the classical case is handled by Theorem \ref{Completeness}, so we will concentrate on proving 
$$\mathrm{Con}(\mathrm{IMLU}_\mathsf{Cat}) \Rightarrow \mathrm{Con}(\mathrm{INFU}_\mathsf{Set}),$$ 
and $\mathrm{Con}(\mathrm{MLU}_\mathsf{Cat}) \Rightarrow \mathrm{Con}(\mathrm{NFU}_\mathsf{Set})$ is thereby obtained as well, simply by assuming that $( \mathbf{M}, \mathbf{N})$ is an $\mathrm{MLU}$-category. By Lemma \ref{SforNF} below this also establishes $\mathrm{Con}(\mathrm{(I)ML}_\mathsf{Cat}) \Rightarrow \mathrm{Con}(\mathrm{(I)NF}_\mathsf{Set})$.

\begin{constr} \label{InConstruction}
For each object $A$ of $\mathbf{N}$ let $\in_A$, along with $m_{\in_A}$, be this pullback in $\mathbf{M}$:
$$
\begin{tikzcd}[ampersand replacement=\&, column sep=small]
{\in_A} \ar[rr, "{\sim}"] \ar[d, rightarrowtail, "{m_{\in_A}}"] \&\& {\subseteq^{\mathbf{T}}_A} \ar[d, rightarrowtail] \\
{A \times \mathbf{P} A} \ar[rr, "{\sim}", "{\iota \times \id}"'] \&\& {\mathbf{T} A \times \mathbf{P} A}
\end{tikzcd}
$$
\end{constr}

In order to avoid confusing the $\in_A$ defined above with the membership symbol of $\mathcal{L}_\mathsf{Set}$, the latter is replaced by the symbol $\epsin$.

\begin{constr} \label{StructureConstruction}
This $\mathcal{L}_\mathsf{Set}$-structure, in the categorical semantics of $\mathbf{M}$, is denoted $\mathcal{U}$:
	\begin{enumerate}
	\item The single sort of $\mathcal{L}_\mathsf{Set}$ is assigned to the universal object $U$ of $\mathbf{N}$.
	\item Fix a mono $m_S : \mathbf{P}U \rightarrowtail U$ in $\mathbf{N}$. The sethood predicate symbol $S$ is identified with the predicate symbol $m_S$ in $\mathcal{S}^\mathbf{M}_\mathbf{N}$, and is assigned to the subobject of $U$ determined by $m_S$.
	\item Fix the mono $m_{\epsin} =_{\mathrm{df}} (\id_U \times m_S) \circ m_{\in_U} : \in_U \rightarrowtail U \times \mathbf{P} U \rightarrowtail U \times U.$ The membership symbol $\epsin$ is identified with the symbol $m_{\epsin}$ in $\mathcal{S}^\mathbf{M}_\mathbf{M}$, and is assigned to the subobject of $U \times U$ determined by $m_{\epsin}$.
	\item Fix a mono $m_P : U \times U \rightarrowtail U$ in $\mathbf{N}$. The function symbol $\langle -, -\rangle$ is identified with the symbol $m_P$ in $\mathcal{S}^\mathbf{M}_\mathbf{N}$ and is assigned to the subobject of $U$ determined by $m_P$.
	\end{enumerate}

By the identifications of symbols, the signature of $\mathcal{L}_\mathsf{Set}$ is a subsignature of $\mathcal{S}^\mathbf{M}_\mathbf{M}$.
\end{constr}

We will usually omit subscripts such as in $\in_A$ and $\subseteq^{\mathbf{T}}_A$, as they tend to be obvious. Similarly, sort declarations are sometimes omitted when considering formulae of the internal language. Note that the symbol $\epsin$ in $\mathcal{L}_\mathsf{Set}$ is interpreted by the subobject of $U \times U$ determined by $m_{\epsin}$, not by the subobject of $U \times \mathbf{P}U$ determined by $m_{\in_U}$. In the categorical setting it tends to be more natural to have a membership relation of sort $A \times \mathbf{P}A$ for each object $A$, while in the set-theoretical setting it tends to be more natural to have just one sort, say Universe, and just one membership relation of sort Universe $\times$ Universe.

To prove $\mathrm{Con}(\mathrm{(I)ML(U)}_\mathsf{Cat}) \Rightarrow \mathrm{Con}(\mathrm{(I)NF(U)}_\mathsf{Set})$, we need to establish that $\mathcal{U}$ satisfies Axioms \ref{SCAx}. $\mathcal{U} \models \phi$ is the statement that the $\mathcal{L}_\mathsf{Set}$-structure $\mathcal{U}$ satisfies $\phi \in \mathcal{L}_\mathsf{Set}$, in the categorical semantics of $\mathbf{M}$. For the major axioms, Extensionality and Stratified Comprehension, we will first prove the more general (and more naturally categorical) statements in terms of the $\in_A$, and second obtain the required statements about $\epsin$ as corollaries. The general results will be stated in the form $\mathbf{M} \models \phi$, where $\phi$ is a formula in the language of $\mathcal{S}^M_M$ or some subsignature of it. In particular, the subsignature $\mathcal{S}^\mathbf{M}_\mathbf{N}$ is of interest. Since $\mathbf{N}$ is a Heyting subcategory of $\mathbf{M}$, if $\phi(\vec{x})$ is an $\mathcal{S}^\mathbf{M}_\mathbf{N}$-formula (with $\vec{x} : X^n$, for some $X$ in $\mathbf{N}$ and $n \in \mathbb{N}$), then $\llbracket \vec{x} : X^n \mid \phi(\vec{x}) \rrbracket$ is assigned to the same subobject of $X^n$ by the natural $\mathcal{S}^\mathbf{M}_\mathbf{M}$-structure as by the natural $\mathcal{S}^\mathbf{M}_\mathbf{N}$-structure. Therefore, we do not need to specify which of these structures is used when referring to a subobject by such an expression.

The following proposition is the expression of Construction \ref{InConstruction} in the categorical semantics.

\begin{prop}\label{membership as iota subset}
Let $X$ be an object of $\mathbf{N}$. 
$$\mathbf{M} \models \forall x : X . \forall y : \mathbf{P} X . ( x \in y \leftrightarrow \iota x \subseteq^{\mathbf{T}} y).$$
\end{prop}
\begin{proof}
$\llbracket x, y \mid \iota x \subseteq^{\mathbf{T}} y \rrbracket = (\iota \times \id)^*\llbracket u, y \mid u \subseteq^{\mathbf{T}} y \rrbracket = \llbracket x, y \mid x \in y \rrbracket$. 
\end{proof}

Let us start the proof of $\mathcal{U} \models \mathrm{INFU}_\mathsf{Set}$ with the easy axioms of Sethood and Ordered Pair.

\begin{prop}[Sethood] \label{Sethood}
$\mathcal{U} \models \forall z . \forall x . (z \epsin x \rightarrow S(x))$
\end{prop}
\begin{proof}
By construction of $m_{\epsin}$, $\llbracket z, x : U \mid z \epsin x \rrbracket \leq_{U \times U} U \times \mathbf{P}U$, and by construction of  $m_S$, $U \times \mathbf{P}U \cong_{U \times U} \llbracket z, x : U \mid S(x)\rrbracket$, so 
$$\llbracket z, x : U \mid z \epsin x \rrbracket \leq_{U \times U} \llbracket z, x : U \mid S(x)\rrbracket,$$ 
as desired.
\end{proof}

\begin{prop}[Ordered Pair] \label{OrderedPair}
$\mathcal{U} \models \forall x, x\hspace{1pt}', y, y\hspace{1pt}' . \big( \langle x, y \rangle = \langle x\hspace{1pt}', y\hspace{1pt}' \rangle \rightarrow ( x = x\hspace{1pt}' \wedge y = y\hspace{1pt}' ) \big)$
\end{prop}
\begin{proof}
Let $\langle a, a\hspace{1pt}', b, b\hspace{1pt}' \rangle$ be a mono with co-domain $U^4$, representing 
\[
\llbracket x, x\hspace{1pt}', y, y\hspace{1pt}' : U \mid \forall x, x\hspace{1pt}', y, y\hspace{1pt}' . \big( \langle x, y \rangle = \langle x\hspace{1pt}', y\hspace{1pt}' \rangle \rightarrow ( x = x\hspace{1pt}' \wedge y = y\hspace{1pt}' ) \big) \rrbracket .
\]
We need to derive $\langle a, b \rangle = \langle a\hspace{1pt}', b\hspace{1pt}' \rangle$ from the assumption $m_P \circ \langle a, b \rangle = m_P \circ \langle a\hspace{1pt}', b\hspace{1pt}' \rangle$. But this follows immediately from that $m_P$ is monic.
\end{proof}

The following Lemma yields $\mathrm{Con}(\mathrm{(I)ML}_\mathsf{Cat}) \Rightarrow \mathrm{Con}(\mathrm{(I)NF}_\mathsf{Set})$ for free, if we successfully prove that $\mathcal{U} \models \mathrm{(I)NFU_\mathsf{Set}}$.

\begin{lemma}[$\mathrm{(I)NF}$ for free]\label{SforNF}
If $U \cong \mathbf{P}U$, then we can choose $m_S : \mathbf{P}U \rightarrowtail U$ (i.e. the interpretation of the predicate symbol $S$) to be an isomorphism. If so, then $\mathcal{U} \models \forall x . S(x)$.
\end{lemma}
\begin{proof}
Since $m_S$ is an isomorphism, $m_s : \mathbf{P}U \rightarrowtail U$ and $\id : U \rightarrowtail U$ represent the same subobject of $U$.
\end{proof}

Note that we do not need $U = \mathbf{P}U$ for this result; $U \cong \mathbf{P}U$ suffices. This means that our results will actually give us that $\mathrm{(I)NF_\mathsf{Set}}$ is equiconsistent with $\mathrm{(I)NFU_\mathsf{Set}} + \big( |V| = |\mathcal P(V)| \big)$, with essentially no extra work. See Corollary \ref{ConFewAtoms} below. This result has been proved previously in \cite{Cra00} using the conventional set-theoretical semantics. In the present categorical setting, this result is transparently immediate.

\begin{prop}[Extensionality]\label{Extensionality}
Let $Z$ be an object of $\mathbf{N}$. 
$$\mathbf{M} \models \forall x : \mathbf{P} Z . \forall y : \mathbf{P} Z . \big[ \big( \forall z : Z . (z \in x \leftrightarrow z \in y) \big) \rightarrow x = y \big].$$
\end{prop}
\begin{proof}
We use the fact that $\mathbf{N}$ is a Heyting subcategory of $\mathbf{M}$. By Proposition \ref{membership as iota subset}, it suffices to establish that in $\mathbf{N}$:
\[ \llbracket x : \mathbf{P} Z, y : \mathbf{P} Z \mid \forall z : \mathbf{T}Z . (z \subseteq^{\mathbf{T}} x \leftrightarrow z \subseteq^{\mathbf{T}} y) \rrbracket \leq \llbracket x, y \mid x = y \rrbracket .\]
Let $\langle a, b \rangle : E \rightarrowtail \mathbf{P} Z \times \mathbf{P} Z$ represent $\llbracket x, y \mid \forall z . (z \subseteq^{\mathbf{T}} x \leftrightarrow z \subseteq^{\mathbf{T}} y) \rrbracket$. We need to show that $a = b$.

Consider $\llbracket w, u, v \mid w \subseteq^{\mathbf{T}} u \rrbracket$ and $\llbracket w, u, v \mid w \subseteq^{\mathbf{T}} v \rrbracket$ as subobjects of $\mathbf{T} Z \times \mathbf{P} Z \times \mathbf{P} Z$. We calculate their pullbacks along $\id \times \langle a, b \rangle$ to be equal subobjects of $\llbracket x, y \mid \forall z . (z \subseteq^{\mathbf{T}} x \leftrightarrow z \subseteq^{\mathbf{T}} y) \rrbracket$:
\[ 
\begin{split}
& (\id \times \langle a, b \rangle)^* \llbracket w, u, v \mid w \subseteq^{\mathbf{T}} u \rrbracket \\
= & \llbracket w, u, v \mid w \subseteq^{\mathbf{T}} u \wedge \forall z . (z \subseteq^{\mathbf{T}} u \leftrightarrow z \subseteq^{\mathbf{T}} v) \rrbracket \\
= & \llbracket w, u, v \mid w \subseteq^{\mathbf{T}} u \wedge w \subseteq^{\mathbf{T}} v \wedge \forall z . (z \subseteq^{\mathbf{T}} u \leftrightarrow z \subseteq^{\mathbf{T}} v) \rrbracket \\
= & \llbracket w, u, v \mid w \subseteq^{\mathbf{T}} v \wedge \forall z . (z \subseteq^{\mathbf{T}} u \leftrightarrow z \subseteq^{\mathbf{T}} v) \rrbracket \\
= & (\id \times \langle a, b \rangle)^*\llbracket w, u, v \mid w \subseteq^{\mathbf{T}} v \rrbracket 
\end{split}
\]

From inspection of the chain of pullbacks
\[
\begin{tikzcd}[ampersand replacement=\&, column sep=small]
{(\id \times \langle a, b \rangle)^* \llbracket w, u, v \mid w \subseteq^{\mathbf{T}} u \rrbracket} \ar[rr, rightarrowtail] \ar[d, rightarrowtail, "f"'] \&\&
{\llbracket w, u, v \mid w \subseteq^{\mathbf{T}} u \rrbracket} \ar[rr] \ar[d, rightarrowtail] \&\&
{\llbracket w, t \mid w \subseteq^{\mathbf{T}} t \rrbracket} \ar[d, rightarrowtail]
\\
{\llbracket z, x, y \mid \forall z . (z \subseteq^{\mathbf{T}} x \leftrightarrow z \subseteq^{\mathbf{T}} y) \rrbracket} \ar[rr, rightarrowtail, "{\id \times \langle a, b \rangle}"] \ar[rrrr, bend right, "{\id \times a}"'] \&\&
{\mathbf{T} Z \times \mathbf{P} Z \times \mathbf{P} Z} \ar[rr, "{\langle \pi^1, \pi^2 \rangle}" ] \&\& {\mathbf{T} Z \times \mathbf{P} Z} ,
\end{tikzcd}
\]
it is evident that
\[
(\id \times a)^* \big( \llbracket w, t \mid w \subseteq^{\mathbf{T}} t \rrbracket) = (\id \times \langle a, b \rangle)^* \llbracket w, u, v \mid w \subseteq^{\mathbf{T}} u \rrbracket .
\]
Similarly,
\[
(\id \times b)^* \big( \llbracket w, t \mid w \subseteq^{\mathbf{T}} t \rrbracket) = (\id \times \langle a, b \rangle)^* \llbracket w, u, v \mid w \subseteq^{\mathbf{T}} v \rrbracket .
\]
So
\[
(\id \times a)^* \big( \llbracket w, t \mid w \subseteq^{\mathbf{T}} t \rrbracket) = (\id \times b)^* \big( \llbracket w, t \mid w \subseteq^{\mathbf{T}} t \rrbracket) ,
\]
and $f$ represents them as a subobject of $\llbracket z, x, y \mid \forall z . (z \subseteq^{\mathbf{T}} x \leftrightarrow z \subseteq^{\mathbf{T}} y) \rrbracket$.

By Proposition \ref{strenghten def}(\ref{strengthen power}) and uniqueness of $\chi$ in (\ref{PT}), we conclude that $a = b$.
\end{proof}

\begin{cor} \label{ExtCor}
$\mathcal{U} \models \mathrm{Ext}_S$
\end{cor}
\begin{proof}
By Proposition \ref{Extensionality}, 
$$\mathbf{M} \models \forall x : \mathbf{P} U . \forall y : \mathbf{P} U . \big[ \big( \forall z : U . (z \in x \leftrightarrow z \in y) \big) \rightarrow x = y \big].$$ 
Now, by routine categorical semantics,
$$
\begin{array}{rl}
& \mathbf{M} \models \forall x, y : \mathbf{P} U . \big[ \big( \forall z : U . (z \in x \leftrightarrow z \in y) \big) \rightarrow x = y \big] \\
\iff & \mathbf{M} \models \forall  x, y : \mathbf{P} U . \big[ \big( \forall z : U . (z \epsin m_S(x) \leftrightarrow z \epsin m_S(y)) \big) \rightarrow x = y \big] \\
\iff & \mathbf{M} \models \forall x\hspace{1pt}', y\hspace{1pt}' : U . \big[ \big(S(x\hspace{1pt}') \wedge S(y\hspace{1pt}')\big) \rightarrow \\
& \big( ( \forall z : U . (z \epsin x\hspace{1pt}' \leftrightarrow z \epsin y\hspace{1pt}') ) \rightarrow x\hspace{1pt}' = y\hspace{1pt}' \big) \big] . \\
\end{array}
$$
So $\mathcal{U} \models \mathrm{Ext}_S$.
\end{proof}

The only axiom of $\mathrm{INFU}_\mathsf{Set}$ left to validate is $\mathrm{SC}_S$ (i.e. Stratified Comprehension). In order to approach this, we first need to construct some signatures and define stratification for an appropriate internal language:

\begin{dfn} \label{MoreSignatures}
Let $\mathcal{S}^M_{N, \in}$ be the subsignature of $\mathcal{S}^M_M$ containing $\mathcal{S}^M_N$ and the relation symbol $\in_A$, which is identified with $m_{\in_A}$, for each object $A$ in $\mathbf{N}$. 

Stratification in the language of $\mathcal{S}^M_{N, \in}$ is defined analogously as in Definition \ref{DefStrat}. A stratification function $s$, of an $\mathcal{S}^M_{N, \in}$-formula $\phi$, is an assignment of a {\em type} in $\mathbb{N}$ to each term in $\phi$, subject to the following conditions (where $\equiv$ is syntactic equality; $n \in \mathbb{N}$; $u, v, w, w_1, \dots, w_n$ are $\mathcal{S}^M_{N, \in}$-terms in $\phi$; $\theta$ is an atomic subformula of $\phi$; $R$ is a relation symbol in $\mathcal{S}^M_N$ which is not equal to $\in_X$ for any $X$ in $\mathbf{N}$; and $A$ is an object in $\mathbf{N}$):
\begin{enumerate}[(i)]
\item if $u \equiv v(w_1, \dots, w_n)$, then $s(u) = s(w_1) = \dots = s(w_n)$,
\item if $\theta \equiv R(w_1, \dots, w_n)$, then $s(w_1) = \dots = s(w_n)$,
\item if $\theta \equiv (u \in_A w)$, then $s(u) + 1 = s(w)$,
\end{enumerate}

It can easily be seen that every stratifiable formula $\phi$ has a {\em minimal} stratification $s_\phi$, in the sense that for every stratification $s$ of $\phi$ and for every term $t$ in $\phi$, $s_\phi(t) \leq s(t)$. Moreover, the minimal stratification, $s_\phi$, is determined by the restriction of $s_\phi$ to the set of variables in $\phi$.

Let $\mathcal{S}^M_{N, \iota}$ be the subsignature of $\mathcal{S}^M_M$ containing $\mathcal{S}^M_N$ and the function symbol $\iota_A$ for each object $A$ in $\mathbf{N}$.

If $\phi$ is a formula in either of these languages, then $\mathbf{M} \models \phi$ is to be understood as satisfaction in the natural $\mathcal{S}^\mathbf{M}_\mathbf{M}$-structure.
\end{dfn}

We start by verifying a form of comprehension for $\mathcal{S}^\mathbf{M}_\mathbf{N}$-formulae:

\begin{prop}\label{ComprehensionN}
If $\phi(w, y)$ is an $\mathcal{S}^\mathbf{M}_\mathbf{N}$-formula, with context $w : \mathbf{T}Z, y : Y$ for $Z, Y$ in $\mathbf N$, and in which $x$ is not free, then
$$\mathbf{M} \models \forall y : Y . \exists x : \mathbf{P} Z . \forall w : \mathbf{T} Z . (w \subseteq^{\mathbf{T}} x \leftrightarrow \phi(w, y)).$$
\end{prop}
\begin{proof}
This is a familiar property of power objects. Considering this instance of (PT) in $\mathbf{N}$:
\begin{equation}
\begin{tikzcd}[ampersand replacement=\&, column sep=small]
\llbracket w : \mathbf{T} Z, y : Y \mid \phi(w, y)\rrbracket \ar[rr] \ar[d, rightarrowtail] \&\& \subseteq^{\mathbf{T}}_{Z} \ar[d, rightarrowtail] \\
\mathbf{T} Z \times Y \ar[rr, "{\id \times \chi}"] \&\& \mathbf{T} Z \times \mathbf{P} Z
\end{tikzcd}
\end{equation}
This pullback along $\id \times \chi$ can be expressed as $\llbracket w : \mathbf{T} Z, y : Y \mid w \subseteq^{\mathbf{T}}_{Z} \chi(y)\rrbracket$ in $\mathbf{N}$. So since $\mathbf{N}$ is a Heyting subcategory of $\mathbf{M}$,
$$\mathbf{M} \models \forall y : Y . \forall w : \mathbf{T} Z . (w \subseteq^{\mathbf{T}} \chi(y) \leftrightarrow \phi(w, y)) \text{, and}$$
$$\mathbf{M} \models \forall y : Y . \exists x : \mathbf{P}Z . \forall w : \mathbf{T} Z . (w \subseteq^{\mathbf{T}} x \leftrightarrow \phi(w, y)),$$
as desired.
\end{proof}

To obtain stratified comprehension for $\mathcal{S}^\mathbf{M}_{\mathbf{N}, \in}$-formulae, we need to establish certain coherence conditions. The facts that $\mathbf{N}$ is a Heyting subcategory of $\mathbf{M}$ and that $\mathbf{T}$ preserves limits as an endofunctor of $\mathbf{N}$ (see Proposition \ref{strenghten def} (\ref{strengthen T limits})), enable us to prove that certain morphisms constructed in $\mathbf{M}$ also exist in $\mathbf{N}$, as in the lemmata below. This is useful when applying (\ref{PT}), since the relation $R$ is required to be in $\mathbf{N}$ (see Definition \ref{IMLU_CAT}).

\begin{lemma} \label{TandProductsCommute}
Let $n \in \mathbb{N}$.
\begin{enumerate}[{\normalfont (a)}]
\item $\iota_1 : 1 \rightarrow \mathbf{T}1$ is an isomorphism in $\mathbf{N}$. 
\item For any $A, B$ of $\mathbf{N}$, $(\iota_{A} \times \iota_{B}) \circ \iota_{A \times B}^{-1} : \mathbf{T}(A \times B) \xrightarrow{\sim} A \times B \xrightarrow{\sim} \mathbf{T} A \times \mathbf{T} B$ is an isomorphism in $\mathbf{N}$.
\item For any $A_1, \dots, A_n$ in $\mathbf{N}$, 
\begin{align*}
& (\iota_{A_1} \times \dots \times \iota_{A_n}) \circ \iota_{A_1 \times \dots \times A_n}^{-1} : \\
& \mathbf{T}(A_1 \times \dots \times A_n) \xrightarrow{\sim} A_1 \times \dots \times A_n \xrightarrow{\sim} \mathbf{T} A_1 \times \dots \times \mathbf{T} A_n
\end{align*}
is an isomorphism in $\mathbf{N}$.
\end{enumerate}
\end{lemma}
\begin{proof}
\begin{enumerate}
\item Since $\mathbf{T}$ preserves limits, $\mathbf{T}\mathbf{1}$ is terminal in $\mathbf{N}$, and since $\mathbf{N}$ is a Heyting subcategory of $\mathbf{M}$, $\mathbf{T}\mathbf{1}$ is terminal in $\mathbf{M}$ as well. So by the universal property of terminal objects, $\mathbf{1}$ and $\mathbf{T}\mathbf{1}$ are isomorphic in $\mathbf{N}$, and the isomorphisms must be $\iota_\mathbf{1}$ and $\iota^{-1}_\mathbf{1}$.
\item Since $\mathbf{T}$ preserves limits, $\mathbf{T}(A \times B)$ is a product of $\mathbf{T}A$ and $\mathbf{T}B$ in $\mathbf{N}$, and since $\mathbf{N}$ is a Heyting subcategory of $\mathbf{M}$, it is such a product in $\mathbf{M}$ as well. Now note that 
$$\pi_{\mathbf{T} A \times \mathbf{T} B}^1 \circ (\iota_A \times \iota_B) \circ \iota^{-1}_{A \times B} = \iota_A \circ \pi_{A \times B}^1 \circ \iota_{A \times B}^{-1} = \mathbf{T}\pi_{A \times B}^1,$$ 
and similarly for the second projection. The left equality is a basic fact about projection morphisms. The right equality follows from that $\iota$ is a natural isomorphism. This means that 
$$(\iota_A \times \iota_B) \circ \iota^{-1}_{A \times B} : \mathbf{T}(A \times B) \xrightarrow{\sim} \mathbf{T}A \times \mathbf{T}B$$ 
is the unique universal morphism provided by the definition of product. Hence, it is an isomorphism in $\mathbf{N}$. 
\item This follows from the two items above by induction on $n$.
\end{enumerate}
\end{proof}

\begin{lemma} \label{iotaTermConversion}
Let $n \in \mathbb{N}$. If $u : A_1 \times \dots \times A_n \rightarrow B$ is a morphism in $\mathbf{N}$, then there is a morphism $v : \mathbf{T} A_1 \times \dots \times \mathbf{T} A_n \rightarrow \mathbf{T} B$ in $\mathbf{N}$, such that 
$$\iota_B \circ u = v \circ (\iota_{A_1} \times \dots \times \iota_{A_n}).$$
\end{lemma}
\begin{proof}
Since $\iota$ is a natural transformation, 
$$\iota_B \circ u = (\mathbf{T}u) \circ \iota_{A_1 \times \dots \times A_n}.$$
Since $\iota$ is a natural isomorphism, 
$$\iota_{A_1 \times \dots \times A_n} = \iota_{A_1 \times \dots \times A_n} \circ (\iota_{A_1} \times \dots \times \iota_{A_n})^{-1} \circ (\iota_{A_1} \times \dots \times \iota_{A_n}).$$ 
Thus, by letting $v = (\mathbf{T}u) \circ \iota_{A_1 \times \dots \times A_n} \circ (\iota_{A_1} \times \dots \times \iota_{A_n})^{-1}$, the result is obtained from Lemma \ref{TandProductsCommute}. 
\end{proof}

\begin{constr} \label{iotaRelationConstruction}
Let $n \in \mathbb{N}$, and let $m_R : R \rightarrowtail A_1 \times \dots \times A_n$ be a morphism in $\mathbf{N}$ that is monic in $\mathbf{M}$; i.e. $m_R$ is a relation symbol in $\mathcal{S}^\mathbf{M}_\mathbf{N}$. Using the isomorphism obtained in Lemma \ref{TandProductsCommute}, we construct $\hat{\mathbf{T}} m_R : \hat{\mathbf{T}} R \rightarrowtail \mathbf{T} A_1 \times \dots \times \mathbf{T} A_n$ in $\mathbf{N}$ as the pullback of $\mathbf{T}m_R$ along that isomorphism:
\[
\begin{tikzcd}[ampersand replacement=\&, column sep=small]
{\hat{\mathbf{T}} R} \ar[rr, "\sim"] \ar[d, rightarrowtail, "{\hat{\mathbf{T}} m_R}"'] \&\& {\mathbf{T} R}  \ar[d, rightarrowtail, "{\mathbf{T} m_R}"] \\
{\mathbf{T} A_1 \times \dots \times \mathbf{T} A_n} \ar[rr, "\sim"] \&\& {\mathbf{T} (A_1 \times \dots \times A_n)}
\end{tikzcd}
\] 
\end{constr}
Note that the definition of $\hat{\mathbf{T}}m_R$ implicitly depends on the factorization $A_1 \times A_n$ chosen for the co-domain of $m_R$.

\begin{lemma} \label{iotaRelationConversion}
Let $m_R : R \rightarrowtail A_1 \times \dots \times A_n$ be as in Construction \ref{iotaRelationConstruction}.
\begin{align*}
\mathbf{M} \models \forall x_1 : A_1 \dots \forall x_n : A_n . \big( & (\mathbf{T} m_R)(\iota_{A_1 \times \dots \times A_n}(x_1, \dots, x_n)) \leftrightarrow \\ 
& (\hat{\mathbf{T}} m_R)(\iota_{A_1}(x_1), \dots, \iota_{A_n}(x_n)) \big).
\end{align*}
\end{lemma}
\begin{proof}
The subobjects 
\begin{align*}
P = \llbracket x_1 : A_1, \dots, x_n : A_n \mid (\mathbf{T} m_R)(\iota_{A_1 \times \dots \times A_n}(x_1, \dots, x_n)) \rrbracket \\
P' = \llbracket x_1 : A_1, \dots, x_n : A_n \mid (\hat{\mathbf{T}} m_R)(\iota_{A_1}(x_1), \dots, \iota_{A_n}(x_n)) \rrbracket
\end{align*}
of $A_1 \times \dots \times A_n$ are obtained as these pullbacks:
\[
\begin{tikzcd}[ampersand replacement=\&, column sep=small]
P \ar[rrrr, "f"] \ar[d, "g"']  \&\&\&\&  \mathbf{T}R \ar[d, "{\mathbf{T}m_R}"]  \\
{A_1 \times \dots \times A_n} \ar[rrrr, "{\iota_{A_1 \times \dots \times A_n}}"']  \&\&\&\&  {\mathbf{T}(A_1 \times \dots \times A_n)}  
\end{tikzcd}
\] 
\[
\begin{tikzcd}[ampersand replacement=\&, column sep=small]
P' \ar[rrrr, "f\hspace{2pt}'"] \ar[d, "g\hspace{1pt}'"']  \&\&\&\&  {\hat{\mathbf{T}}R} \ar[d, "{\hat{\mathbf{T}}m_R}"]  \\
{A_1 \times \dots \times A_n} \ar[rrrr, "{\iota_{A_1} \times \dots \times \iota_{A_n}}"']  \&\&\&\&  {\mathbf{T}A_1 \times \dots \times \mathbf{T}A_n}  
\end{tikzcd}
\] 
Since the bottom morphisms in both of these pullback-diagrams are isomorphisms, it follows from a basic fact about pullbacks that the top ones, $f$ and $f\hspace{2pt}'$, are also isomorphisms. So $P$ and $P'$, as subobjects of $A_1 \times \dots \times A_n$, are also represented by $\iota^{-1}_{A_1 \times \dots \times A_n} \circ \mathbf{T}m_R : \mathbf{T}R \rightarrow A_1 \times \dots \times A_n$ and $(\iota_{A_1} \times \dots \times \iota_{A_n})^{-1} \circ \hat{\mathbf{T}}m_R : \hat{\mathbf{T}}R \rightarrow A_1 \times \dots \times A_n$, respectively.

Now note that by construction of $\hat{\mathbf{T}}$, this diagram commutes:
\[
\begin{tikzcd}[ampersand replacement=\&, column sep=small]
{\hat{\mathbf{T}} R} \ar[rr, "\sim"] \ar[d, rightarrowtail, "{\hat{\mathbf{T}} m_R}"'] \&\& {\mathbf{T} R}  \ar[d, rightarrowtail, "{\mathbf{T} m_R}"] \\
{\mathbf{T} A_1 \times \dots \times \mathbf{T} A_n} \ar[rr, "\sim"] \ar[rd, "{\sim}"'] \&\& {\mathbf{T} (A_1 \times \dots \times A_n)} \ar[ld, "{\sim}"] \\
\& {A_1 \times \dots \times A_n}
\end{tikzcd}
\] 
Therefore, $\iota^{-1}_{A_1 \times \dots \times A_n} \circ \mathbf{T}m_R$ and $(\iota_{A_1} \times \dots \times \iota_{A_n})^{-1} \circ \hat{\mathbf{T}}m_R$ represent the same subobject of $A_1 \times \dots \times A_n$, as desired
\end{proof}

Let $n \in \mathbb{N}$. We recursively define iterated application of $\mathbf{P}$, $\mathbf{T}$ and $\hat{\mathbf{T}}$ in the usual way, as $\mathbf{P}^0 = \mathbf{id}_\mathbf{N}$, $\mathbf{P}^{k+1} = \mathbf{P} \circ \mathbf{P}^k$ (for $k \in \mathbb{N}$), etc. The iterated application of $\iota$ requires a special definition. We define $\iota^n_A : A \xrightarrow{\sim} \mathbf{T}^n A$ recursively by 
\begin{align*}
\iota^0_A &= \id_A, \\
\iota^{k+1}_A &= \iota_{\mathbf{T}^k A} \circ \iota^k_A : A \xrightarrow{\sim} \mathbf{T}^k A \xrightarrow{\sim} \mathbf{T}^{k+1} A \text{, where } k \in \mathbb{N}.
\end{align*}
Since $\iota$ is a natural isomorphism, we have by induction that $\iota^n : \id \xrightarrow{\sim} \mathbf{T}^n$ also is a natural isomorphism.

\begin{lemma} \label{iotaTypeIncrease}
Let $n, k \in \mathbb{N}$. Let $m_R : R \rightarrowtail A_1 \times \dots A_n$ be as in Construction \ref{iotaRelationConstruction}. 
\begin{align*}
\mathbf{M} \models \forall x_1 : A_1 \dots \forall x_n : A_n . \big( & m_R(x_1, \dots, x_n) \leftrightarrow \\
& (\hat{\mathbf{T}}^k m_R) (\iota^k x_1, \dots, \iota^k x_n)\big).
\end{align*}
\end{lemma}
\begin{proof}
Since $\iota^k : \id_\mathbf{M} \rightarrow \mathbf{T}^k$ is a natural isomorphism, 
\begin{align*}
\mathbf{M} \models \forall x_1 : A_1 \dots \forall x_n : A_n . \big( & m_R(x_1, \dots, x_n) \leftrightarrow \\
& (\mathbf{T}^k m_R) (\iota_{A_1 \times \dots \times A_n}^k (x_1, \dots, x_n))\big).
\end{align*}
By iterating Lemma \ref{iotaRelationConversion}, we obtain by induction that 
\begin{align*}
\mathbf{M} \models \forall x_1 : A_1 \dots \forall x_n : A_n . \big( & (\mathbf{T}^k m_R) (\iota_{A_1 \times \dots \times A_n}^k (x_1, \dots, x_n)) \leftrightarrow \\
& (\hat{\mathbf{T}}^k m_R) (\iota_{A_1}^k x_1, \dots, \iota_{A_n}^k x_n)\big).
\end{align*}
The result now follows by combining the two. 
\end{proof}

We shall now show that any stratified $\mathcal{S}^\mathbf{M}_{\mathbf{N}, \in}$-formula $\phi$ can be converted to an $\mathcal{S}^\mathbf{M}_\mathbf{N}$-formula $\phi^{\subseteq^{\mathbf{T}}}$, which is equivalent to $\phi$ in $\mathbf{M}$, i.e. $\mathbf{M} \models \phi \leftrightarrow \phi^{\subseteq^{\mathbf{T}}}$. 

\begin{constr} \label{StratConstruction}
Let $\phi$ be any stratified $\mathcal{S}^M_{N, \in}$-formula. Let $s_\phi$ be the minimal stratification of $\phi$, and let $\mathrm{max}_\phi$ be the maximum value attained by $s_\phi$.
	\begin{itemize}
	\item Let $\phi^\iota$ be the $\mathcal{S}^\mathbf{M}_{\mathbf{N}, \iota}$-formula obtained from $\phi$ by the construction below. We shall replace each atomic subformula $\theta$ of $\phi$ by another atomic formula which is equivalent to $\theta$ in $\mathbf{M}$. We divide the construction into two cases, depending on whether or not $\theta$ is of the form $\theta \equiv t \in_X t'$, for some $X$ in $\mathbf{N}$ and terms $t, t'$ in $\mathcal{S}^\mathbf{M}_\mathbf{N}$:
		\begin{enumerate}
		\item Suppose that $\theta$ is {\em not} of the form $\theta \equiv t \in_X t'$. Then $\theta$ is equivalent in $\mathbf{M}$ to a formula $m_R(x_1, \dots, x_n)$, where $x_1, \dots, x_n$ are variables, as such a monomorphism $m_R$ can be constructed in $\mathbf{N}$ from the interpretations of the relation-symbol and terms appearing in $\theta$. Note that by stratification, $s_\phi(x_1) = \dots = s_\phi(x_n)$. Let $k = \mathrm{max} - s_\phi(x_1)$. In $\phi$, replace $\theta$ by 
		$$(\hat{\mathbf{T}}^k m_R) (\iota^k x_1, \dots, \iota^k x_n).$$
It follows from Lemma \ref{iotaTypeIncrease} that this formula is equivalent to $\theta$ in $\mathbf{M}$.
		\item Suppose that $\theta \equiv u(x_1, \dots, x_n) \in_A v(y_1, \dots, y_{m})$, where $A$ is an object in $\mathbf{N}$, $u, v$ are terms in $\mathcal{S}^\mathbf{M}_\mathbf{N}$, and $x_1, \dots, x_n, y_1, \dots, y_{m}$ are variables. Note that by stratification, 
		$$s_\phi(u) + 1 = s_\phi(x_i) + 1 = s_\phi(v) = s_\phi(y_j),$$
		for each $1 \leq i \leq n$ and $1 \leq j \leq m$. Let $k_u = \mathrm{max}_\phi - s_\phi(u)$ and $k_v = \mathrm{max}_\phi - s_\phi(v)$, whence $k_u = k_v + 1$. By Proposition \ref{membership as iota subset}, $\theta$ is equivalent in $\mathbf{M}$ to 
		\[(\iota_A \circ u) (x_1, \dots, x_n) \subseteq^\mathbf{T}_A v(y_1, \dots, y_{m}).\] 
		Thus, by Lemma \ref{iotaTypeIncrease}, $\theta$ is equivalent in $\mathbf{M}$ to 
		\[(\iota^{k_u}_A \circ u) (x_1, \dots, x_n) (\hat{\mathbf{T}}^{k_v} \subseteq^{\mathbf{T}}) (\iota^{k_v}_{\mathbf{P}A} \circ v) (y_1, \dots, y_{m}).\]
		Now, by iterated application of Lemma \ref{iotaTermConversion}, there are morphisms $u\hspace{1pt}', v\hspace{1pt}'$ in $\mathbf{N}$, such that $\theta$ is equivalent in $\mathbf{M}$ to
		$$u\hspace{1pt}'(\iota^{k_u}(x_1), \dots, \iota^{k_u}(x_n)) (\hat{\mathbf{T}}^{k_v} \subseteq^{\mathbf{T}}) v\hspace{1pt}'(\iota^{k_v}(y_1), \dots, \iota^{k_v}(y_m)).$$ 
		Replace $\theta$ by this formula.
		\end{enumerate}
	\item Let $\phi^{\subseteq^{\mathbf{T}}}$ be the formula in the language of $\mathcal{S}^\mathbf{M}_\mathbf{N}$ obtained from $\phi^\iota$ as follows.
		\begin{itemize}
		\item Replace each term of the form $\iota^{\mathrm{max}_\phi - s_\phi(x)}(x)$ (where $x$ is a variable of sort $A$) by a fresh variable $x\hspace{1pt}'$ (of sort $\mathbf{T}^{\mathrm{max}_\phi - s_\phi(x)} A$).
		\item Replace each quantifier scope or context declaration $x : A$ by $x\hspace{1pt}' : \mathbf{T}^{\mathrm{max}_\phi - s_\phi(x)} A$.
		\end{itemize}
	\end{itemize}
\end{constr}

By construction $\phi^\iota$ is an $\mathcal{S}^\mathbf{M}_{\mathbf{N}, \iota}$-formula, which is equivalent to $\phi$ in $\mathbf{M}$. Let $x$ be an arbitrary variable in $\phi^\iota$. Note that each variable $x$ in $\phi^\iota$ occurs in a term $\iota^{\mathrm{max}_\phi - s_\phi(x)} x$; and conversely, every occurrence of $\iota$ in $\phi^\iota$ is in such a term $\iota^{\mathrm{max}_\phi - s_\phi(x)} x$, for some variable $x$. Therefore, $\phi^{\subseteq^{\mathbf{T}}}$ is an $\mathcal{S}^\mathbf{M}_\mathbf{N}$-formula. So since $\iota^{\mathrm{max}_\phi - s_\phi(x)} : A \rightarrow \mathbf{T}^{\mathrm{max}_\phi - s_\phi(x)} A$ is an isomorphism, for each variable $x : A$ in $\phi^\iota$, we have that $\phi^{\subseteq^{\mathbf{T}}}$ is equivalent to $\phi^\iota$. We record these findings as a lemma:

\begin{lemma}\label{StratConvert}
If $\phi$ is a stratified $\mathcal{S}^\mathbf{M}_{\mathbf{N}, \in}$-formula, then 
$$\mathbf{M} \models \phi \Leftrightarrow \mathbf{M} \models \phi^\iota \Leftrightarrow \mathbf{M} \models \phi^{\subseteq^{\mathbf{T}}},$$ 
where $\phi^\iota$ is an $\mathcal{S}^\mathbf{M}_{\mathbf{N}, \iota}$-formula and $\phi^{\subseteq^{\mathbf{T}}}$ is an $\mathcal{S}^\mathbf{M}_\mathbf{N}$-formula.
\end{lemma}

\begin{prop}[Stratified Comprehension]\label{ComprehensionStratified}
For every stratified $\mathcal{S}^\mathbf{M}_{\mathbf{N}, \in}$-formula $\phi(z)$, where $z : Z$ for some $Z$ in $\mathbf M$,
$$\mathbf{M} \models \exists x : \mathbf{P} Z . \forall z : Z . (z \in x \leftrightarrow \phi(z)).$$
\end{prop}
\begin{proof}
By Lemma \ref{StratConvert}, we have
\begin{align*}
\mathbf{M} \models \hspace{1pt} & \exists x : \mathbf{P} Z . \forall z : Z . (z \in_Z x \leftrightarrow \phi(z, y)) \\
\iff \mathbf{M} \models \hspace{1pt} & \exists x\hspace{1pt}' : \mathbf{T}^k \mathbf{P} Z . \forall z\hspace{1pt}' : \mathbf{T}^{k+1} Z . (z\hspace{1pt}' (\hat{\mathbf{T}}^k \subseteq^{\mathbf{T}}_Z) x\hspace{1pt}' \leftrightarrow \phi^{\subseteq^{\mathbf{T}}}(z\hspace{1pt}')), \tag{$\dagger$}
\end{align*}
for some $k \in \mathbb{N}$, where $x\hspace{1pt}', z\hspace{1pt}'$ are fresh variables. 

In order to apply Proposition \ref{ComprehensionN}, we need to move the $\mathbf{T}$:s through the $\mathbf{P}$ and transform the $\hat{\mathbf{T}}^k \subseteq^{\mathbf{T}}_Z$ into a $\subseteq^{\mathbf{T}}_{\mathbf{T}^k Z}$. Since $\mu : \mathbf{P}\mathbf{T} \rightarrow \mathbf{T}\mathbf{P}$ is a natural isomorphism,
$$\nu =_\mathrm{df} \mathbf{T}^{k-1}(\mu_Z) \circ \mathbf{T}^{k-2}(\mu_{\mathbf{T}Z}) \circ \dots \circ \mathbf{T}(\mu_{\mathbf{T}^{k-2} Z}) \circ \mu_{\mathbf{T}^{k-1} Z} : \mathbf{P} \mathbf{T}^k Z \xrightarrow{\sim} \mathbf{T}^k \mathbf{P} Z,$$
is an isomorphism making this diagram commute:
\[
\begin{tikzcd}[ampersand replacement=\&, column sep=small]
\subseteq^{\mathbf{T}}_{\mathbf{T}^k Z} \ar[rr, "{\sim}"] \ar[d, rightarrowtail, "{m_{\subseteq^{\mathbf{T}}_{\mathbf{T}^k Z}}}"] \&\& \hat{\mathbf{T}}^k \subseteq^{\mathbf{T}}_Z \ar[d, rightarrowtail, "{\hat{\mathbf{T}}^k m_{\subseteq^{\mathbf{T}}_Z}}"] \\
\mathbf{T}^{k+1} Z \times \mathbf{P} \mathbf{T}^k Z \ar[rr, "{\sim}", "{\id \times \nu}"'] \&\& \mathbf{T}^{k+1} Z \times \mathbf{T}^k \mathbf{P} Z
\end{tikzcd} 
\]
So introducing a fresh variable $x\hspace{1pt}'' : \mathbf{P} \mathbf{T}^k Z$, $(\dagger)$ is equivalent to
$$\mathbf{M} \models \exists x\hspace{1pt}'' : \mathbf{P} \mathbf{T}^k Z . \forall z\hspace{1pt}' : \mathbf{T}^{k+1} Z . (z\hspace{1pt}' \subseteq^{\mathbf{T}}_{\mathbf{T}^k  Z} x\hspace{1pt}'' \leftrightarrow \phi^{\subseteq^{\mathbf{T}}}(z\hspace{1pt}')).$$
By Proposition \ref{ComprehensionN} we are done.
\end{proof}

\begin{cor}
$\mathcal{U} \models \mathrm{SC}_S$
\end{cor}
\begin{proof}
Let $\phi(z)$ be a stratified formula in $\mathcal{L}_\mathsf{Set}$. By Proposition \ref{ComprehensionStratified}, 
\[\mathbf{M} \models \exists x : \mathbf{P} U . \forall z : U . (z \in x \leftrightarrow \phi(z)).
\] 
Now,
$$
\begin{array}{rl}
& \llbracket \exists x : \mathbf{P} U . \forall z : U . (z \in x \leftrightarrow \phi(z)) \rrbracket \\
= & \llbracket \exists x : \mathbf{P} U . \forall z : U . (z \epsin m_S(x) \leftrightarrow \phi(z)) \rrbracket \\
= & \llbracket \exists x\hspace{1pt}' : U . \big(S(x\hspace{1pt}') \wedge \forall z : U . (z \epsin x\hspace{1pt}' \leftrightarrow \phi(z))\big) \rrbracket. \\
\end{array}
$$
So $\mathcal{U} \models \mathrm{SC}_S$.
\end{proof}

\begin{thm}\label{ConCatConSet}
$\mathcal{U} \models \mathrm{(I)NFU}$, and if $U \cong \mathbf{P}U$, then $\mathcal{U} \models \mathrm{(I)NF}$. Thus, each of $\mathrm{(I)NF(U)}_\mathsf{Set}$ is interpretable in $\mathrm{(I)ML(U)}_\mathsf{Cat}$, respectively.
\end{thm}
\begin{proof}
The cases of $\mathrm{(I)NFU}$ are settled by the results above on Sethood, Ordered Pair, Extensionality and Stratified Comprehension. The cases of $\mathrm{(I)NF}$ now follow from Lemma \ref{SforNF}.
\end{proof}

\begin{thm}
These theories are equiconsistent\footnote{But see Remark \ref{remML}.}: 
\begin{align*}
& \mathrm{(I)ML(U)_\mathsf{Class}} \\
& \mathrm{(I)ML(U)_\mathsf{Cat}} \\
& \mathrm{(I)NF(U)_\mathsf{Set}}
\end{align*}

More precisely, $\mathrm{(I)ML(U)_\mathsf{Class}}$ interprets $\mathrm{(I)ML(U)_\mathsf{Cat}}$, which in turn interprets $\mathrm{(I)NF(U)_\mathsf{Set}}$; and a model of $\mathrm{(I)ML(U)_\mathsf{Class}}$ can be constructed from a model of $\mathrm{(I)NF(U)_\mathsf{Set}}$.
\end{thm}
\begin{proof}
Combine Theorem \ref{ToposModelOfML}, Corollary \ref{ConSetConClass}, Theorem \ref{ConClassConCat} and Theorem \ref{ConCatConSet}.
\end{proof}

\begin{cor}\label{ConFewAtoms}
These theories are equiconsistent\footnote{But see Remark \ref{remML}.}:
\begin{align*}
& \mathrm{(I)NF}_\mathsf{Set} \\
& \mathrm{(I)NFU_\mathsf{Set}} + ( |V| = |\mathcal P(V)| )
\end{align*}
\end{cor}
\begin{proof}
Only $\Leftarrow$ is non-trivial. By the proofs above, 
\[
\begin{array}{cl}
& \mathrm{Con}\big(\mathrm{(I)NFU_\mathsf{Set}} + ( |V| = |\mathcal P(V)| )\big) \\
\Rightarrow & \mathrm{Con}\big(\mathrm{(I)MLU_\mathsf{Class}} + ( |V| = |\mathcal P(V)| )\big) \\
\Rightarrow & \mathrm{Con}\big(\mathrm{(I)ML_\mathsf{Cat}}\big) \Rightarrow \mathrm{Con}\big(\mathrm{(I)NF}_\mathsf{Set}\big),
\end{array}
\]
as desired.
\end{proof}

For the classical case, this is known from \cite{Cra00}, while the intuitionistic case appears to be new.

\section{The subtopos of strongly Cantorian objects}\label{subtopos}

\begin{dfn}
An object $X$ in $\mathbf{N}$ is {\em Cantorian} if $X \cong \mathbf{T}X$ in $\mathbf{N}$, and is {\em strongly Cantorian} if $\iota_X : X \xrightarrow{\sim} \mathbf{T}X$ is an isomorphism in $\mathbf{N}$. Define $\mathbf{SCan}_{( \mathbf{M}, \mathbf{N})}$ as the full subcategory of $\mathbf{N}$ on the set of strongly Cantorian objects. I.e. its objects are the strongly Cantorian objects, and its morphism are all the morphisms in $\mathbf{N}$ between such objects. When the subscript ${( \mathbf{M}, \mathbf{N})}$ is clear from the context, we may simply write $\mathbf{SCan}$.
\end{dfn}

\begin{prop}\label{SCan has limits}
$\mathbf{SCan}_{( \mathbf{M}, \mathbf{N})}$ has finite limits.
\end{prop}
\begin{proof}
Let $L$ be a limit in $\mathbf{N}$ of a finite diagram $\mathbf{D} : \mathbf{I} \rightarrow \mathbf{SCan}$. Since $\mathbf{N}$ is a Heyting subcategory of $\mathbf{M}$, $L$ is a limit of $\mathbf{D}$ in $\mathbf{M}$; and since $\mathbf{T}$ preserves limits, $\mathbf{T}L$ is a limit of $\mathbf{T} \circ \mathbf{D}$ in $\mathbf{M}$ and in $\mathbf{N}$. But also, since $\mathbf{D}$ is a diagram in $\mathbf{SCan}$, $\mathbf{T}L$ is a limit of $\mathbf{D}$, and $L$ is a limit of $\mathbf{T} \circ \mathbf{D}$, in $\mathbf{M}$ and in $\mathbf{N}$. So there are unique morphisms in $\mathbf{N}$ back and forth between $L$ and $\mathbf{T}L$ witnessing the universal property of limits. Considering these as morphisms in $\mathbf{M}$ we see that they must be $\iota_L$ and $\iota_L^{-1}$. Now $L$ is a limit in $\mathbf{SCan}$ by fullness.
\end{proof}

Before we can show that $\mathbf{SCan}_{( \mathbf{M}, \mathbf{N})}$ has power objects, we need to establish results showing that $\mathbf{SCan}_{( \mathbf{M}, \mathbf{N})}$ is a ``nice'' subcategory of $\mathbf{N}$.

\begin{cor}\label{SCan pres and refl limits}
The inclusion functor of $\mathbf{SCan}_{( \mathbf{M}, \mathbf{N})}$ into $\mathbf{N}$ preserves and reflects finite limits.
\end{cor}
\begin{proof}
To see that it reflects finite limits, simply repeat the proof of Proposition \ref{SCan has limits}. We proceed to show that it preserves finite limits.

Let $L$ be a limit in $\mathbf{SCan}$ of a finite diagram $\mathbf{D} : \mathbf{I} \rightarrow \mathbf{SCan}$. Let $L'$ be a limit of this diagram in $\mathbf{N}$. By the proof of Proposition \ref{SCan has limits}, $L'$ is also such a limit in $\mathbf{SCan}$, whence $L$ is isomorphic to $L'$ in $\mathbf{SCan}$, and in $\mathbf{N}$. So $L$ is a limit of $\mathbf{D}$ in $\mathbf{N}$ as well.
\end{proof}

\begin{cor}\label{SCan mono iff N mono}
$m : A \rightarrow B$ is monic in $\mathbf{SCan}_{( \mathbf{M}, \mathbf{N})}$ iff it is monic in $\mathbf{N}$.
\end{cor}
\begin{proof}
($\Leftarrow$) follows from that $\mathbf{SCan}$ is a subcategory of $\mathbf{N}$.

($\Rightarrow$) Assume that $m : A \rightarrow B$ is monic in $\mathbf{SCan}$. Then $A$ along with $\id_A : A \rightarrow A$ and $\id_A : A \rightarrow A$ is a pullback of $m$ and $m$ in $\mathbf{SCan}$. By Corollary \ref{SCan pres and refl limits}, this is also a pullback in $\mathbf{N}$, from which it follows that $m$ is monic in $\mathbf{N}$.
\end{proof}

\begin{prop}\label{SCan closed under monos}
If $m : A \rightarrowtail B$ is monic in $\mathbf{N}$ and $B$ is in $\mathbf{SCan}_{( \mathbf{M}, \mathbf{N})}$, then $A$ and $m : A \rightarrowtail B$ are in $\mathbf{SCan}_{( \mathbf{M}, \mathbf{N})}$.
\end{prop}
\begin{proof}
Let $m : A \rightarrowtail B$ be a mono in $\mathbf{N}$ and assume that $\iota_B$ is in $\mathbf{N}$. Let $P$ be this pullback in $\mathbf{N}$, which is also a pullback in $\mathbf{M}$ since $\mathbf{N}$ is a Heyting subcategory:
\[
\begin{tikzcd}[ampersand replacement=\&, column sep=small]
P \ar[rrr, rightarrowtail, "n"] \ar[d, rightarrowtail, "{\langle p, q\rangle}"'] \&\&\& B \ar[d, rightarrowtail, "{\langle \id_B, \iota_B \rangle}"] \\
{A \times \mathbf{T}A} \ar[rrr, rightarrowtail, "{m \times \mathbf{T}m}"'] \&\&\& {B \times \mathbf{T}B}
\end{tikzcd}
\]

We shall now establish that the following square is also a pullback in $\mathbf{M}$:
\[
\begin{tikzcd}[ampersand replacement=\&, column sep=small]
A \ar[rrr, rightarrowtail, "m"] \ar[d, rightarrowtail, "{\langle \id_A, \iota_A \rangle}"'] \&\&\& B \ar[d, rightarrowtail, "{\langle \id_B, \iota_B \rangle}"] \\
{A \times \mathbf{T}A} \ar[rrr, rightarrowtail, "{m \times \mathbf{T}m}"'] \&\&\& {B \times \mathbf{T}B}
\end{tikzcd}
\]
The square commutes since $\iota$ is a natural isomorphism. So since $P$ is a pullback, it suffices to find $f : P \rightarrow A$ such that $\langle p, q\rangle = \langle \id_A, \iota_A \rangle \circ f$ and $n = m \circ f$. Let $f\hspace{2pt}' = p$ and let $f\hspace{2pt}'' = \iota_A^{-1} \circ q$. We shall show that $f\hspace{2pt}' = f\hspace{2pt}''$ and that this is the desired $f$. By commutativity of the former square, $n = m \circ f\hspace{2pt}'$ and $\iota_B \circ n = \mathbf{T}m \circ \iota_A \circ f\hspace{2pt}''$, whence 
$$\iota_B \circ m \circ f\hspace{2pt}' = \mathbf{T}m \circ \iota_A \circ f\hspace{2pt}''.$$
Note that $\iota_B \circ m$ is monic, and since $\iota$ is a natural transformation it is equal to $\mathbf{T}m \circ \iota_A$. Hence, $f\hspace{2pt}' = f\hspace{2pt}''$. Let $f = f\hspace{2pt}' = f\hspace{2pt}''$. We have already seen that $n = m \circ f$. That $\langle p, q\rangle = \langle \id_A, \iota_A \rangle \circ f$ is immediately seen by plugging the definitions of $f\hspace{2pt}'$ and $f\hspace{2pt}''$ in place of $f$. Since $P$ is a pullback, it follows that $f$ is an isomorphism in $\mathbf{M}$ and that the latter square is a pullback.

Since $f = p$, $f$ is in $\mathbf{N}$, and since $\mathbf{N}$ is a conservative subcategory of $\mathbf{M}$, $f$ is an isomorphism in $\mathbf{N}$. Now note that $\iota_A = \iota_A \circ f \circ f^{-1} = q \circ f^{-1}$. Therefore, $\iota_A$ is in $\mathbf{N}$ and $A$ is in $\mathbf{SCan}$. So by fullness, $m : A \rightarrow B$ is in $\mathbf{SCan}$, as desired.
\end{proof}

\begin{prop}\label{SCan has power objects}
$\mathbf{SCan}_{( \mathbf{M}, \mathbf{N})}$ has power objects.
\end{prop}
\begin{proof}
Let $A$ be in $\mathbf{SCan}$. We shall show that $\mathbf{P}A$ along with $(\iota_A^{-1} \times \id_{\mathbf{P}A}) \circ m_{\subseteq^{\mathbf{T}}_A} : \hspace{1pt} \subseteq^{\mathbf{T}}_A \hspace{1pt} \rightarrowtail A \times \mathbf{P}A$ is a power object of $A$ in $\mathbf{SCan}$. In Step 1 we show that $(\iota_A^{-1} \times \id_{\mathbf{P}A}) \circ m_{\subseteq^{\mathbf{T}}_A} : \hspace{1pt} \subseteq^{\mathbf{T}}_A \hspace{1pt} \rightarrowtail A \times \mathbf{P}A$ is in $\mathbf{SCan}$, and in Step 2 we show that it satisfies the power object property.

{\em Step 1}: It actually suffices to show that $\mathbf{P}A$ is in $\mathbf{SCan}$. Because then, by Proposition \ref{SCan has limits} and Corollary \ref{SCan pres and refl limits}, $A \times \mathbf{P}A$ is in $\mathbf{SCan}$ (and is such a product in both $\mathbf{N}$ and $\mathbf{SCan}$), so that by Proposition \ref{SCan closed under monos} and fullness of $\mathbf{SCan}$, $(\iota_A^{-1} \times \id_{\mathbf{P}A}) \circ m_{\subseteq^{\mathbf{T}}_A} : \hspace{1pt} \subseteq^{\mathbf{T}}_A \hspace{1pt} \rightarrowtail A \times \mathbf{P}A$ is in $\mathbf{SCan}$.

By Proposition \ref{strenghten def} (\ref{strengthen power}), $\mathbf{PT}A$ along with $m_{\subseteq^{\mathbf{T}}_{\mathbf{T}A}} : \hspace{1pt} \subseteq^{\mathbf{T}}_{\mathbf{T}A} \hspace{1pt} \rightarrow \mathbf{TT}A \times \mathbf{PT}A$ is a power object of $\mathbf{TT}A$ in $\mathbf{N}$. So $\mathbf{PT}A$ along with $(\iota_{\mathbf{T}A}^{-1} \times \id_{\mathbf{PT}A}) \circ m_{\subseteq^{\mathbf{T}}_{\mathbf{T}A}} : \hspace{1pt} \subseteq^{\mathbf{T}}_{\mathbf{T}A} \hspace{1pt} \rightarrow \mathbf{T}A \times \mathbf{PT}A$ is a power object of $\mathbf{T}A$ in $\mathbf{N}$. Moreover, by Proposition \ref{strenghten def} (\ref{strengthen power}), $\mathbf{P}A$ is a power object of $\mathbf{T}A$ in $\mathbf{N}$. Therefore, using the natural isomorphism $\mu : \mathbf{P}\mathbf{T} \rightarrow \mathbf{T}\mathbf{P}$, we obtain an isomorphism $\alpha : \mathbf{P}A \xrightarrow{\sim} \mathbf{P}\mathbf{T}A \xrightarrow{\sim} \mathbf{T}\mathbf{P}A$ in $\mathbf{N}$. This results in the following two-way pullback in both $\mathbf{M}$ and $\mathbf{N}$:
\[
\begin{tikzcd}[ampersand replacement=\&, column sep=small]
{\subseteq^{\mathbf{T}}} \ar[rr, leftrightarrow, "{\sim}"] \ar[d, rightarrowtail] \&\& {\subseteq^{\mathbf{T}}} \ar[d, rightarrowtail] \\
\mathbf{T} A \times \mathbf{P}A \ar[rr, bend left, "{\sim}", "{\id \times \alpha}"'] \&\& \mathbf{T} A \times \mathbf{T}\mathbf{P} A \ar[ll, bend left, "{\sim}"', "{\id \times \alpha^{-1}}"]
\end{tikzcd}
\]
Since this pullback-square can be filled with $\iota_{\mathbf{P}A}$ in place of $\alpha$, the uniqueness property of the pullback implies that $\iota_{\mathbf{P}A} = \alpha$, whence $\iota_{\mathbf{P}A}$ is in $\mathbf{N}$ and $\mathbf{P}A$ is in $\mathbf{SCan}$.

{\em Step 2}: Let $r : R \rightarrowtail A \times B$ be a mono in $\mathbf{SCan}$. By Corollary \ref{SCan mono iff N mono}, $r$ is also monic in $\mathbf{N}$. So since $\iota_A$ is an isomorphism in $\mathbf{N}$, there is a unique $\chi$ in $\mathbf{N}$ such that this is a pullback in $\mathbf{N}$:
\[
\begin{tikzcd}[ampersand replacement=\&, column sep=small]
R \ar[rrr] \ar[d, rightarrowtail, "{r}"'] \&\&\& \subseteq^{\mathbf{T}}_A \ar[d, rightarrowtail, "{(\iota^{-1}_A \times \id_{\mathbf{P}A}) \circ m_{\subseteq^{\mathbf{T}}_A}}"] \\
A \times B \ar[rrr, "{\id_A \times \chi}"'] \&\&\& A \times \mathbf{P} A 
\end{tikzcd} 
\]
By Step 1, by fullness and by Corollary \ref{SCan pres and refl limits}, it is also a pullback in $\mathbf{SCan}$. To see uniqueness of $\chi$ in $\mathbf{SCan}$, suppose that $\chi\hspace{1pt}'$ were some morphism in $\mathbf{SCan}$ making this a pullback in $\mathbf{SCan}$ (in place of $\chi$). Then by Corollary \ref{SCan pres and refl limits}, it would also make it a pullback in $\mathbf{N}$, whence $\chi = \chi\hspace{1pt}'$.
\end{proof}

\begin{thm}
$\mathbf{SCan}_{( \mathbf{M}, \mathbf{N})}$ is a topos.
\end{thm}
\begin{proof}
A category with finite limits and power objects is a topos.
\end{proof}

\chapter{Where to go from here?}\label{ch where to go}

\section{Category theoretic approach to embeddings between models of set theory}

The results of Chapter \ref{ch emb set theory} suggest that it may be fruitful to organize countable models of set theory into categories, and then start asking natural category theoretic questions. In support of the prospects for this approach, let us here take the opportunity to make a case study out of Corollary \ref{Friedman cor} (a generalization of Friedman's embedding theorem). We shall now work towards formulating this theorem as a statement about the category $\mathbf{KP}^\mathcal{P}$ with countable models of $\mathrm{KP}^\mathcal{P} + \Sigma_1^\mathcal{P} \textnormal{-Separation}$ as objects and with topless rank-initial embeddings as morphisms. 

For every object $\mathcal{S}$ in $\mathbf{KP}^\mathcal{P}$:
\begin{enumerate}
\item Let $\mathcal{S}/\mathbf{KP}^\mathcal{P}$ be the co-slice category with the initial objects removed: This may be viewed as the category with proper topless rank-end-extensions of $\mathcal{S}$ as objects, and with rank-initial embeddings point-wise fixing $\mathcal{S}$ as morphisms.
\item Let $\mathbf{Class}_\mathcal{S}$ be the category of $\mathcal{L}^1$-expansions of $\mathcal{S}$, only with identity-morphisms. 
\item Let $\Sigma_1^\mathcal{P}(\mathcal{S}/\mathbf{KP}^\mathcal{P})$ be this category: The objects are the $\Sigma_1^\mathcal{P}$-fragments of the theories of the objects in $\mathcal{S}/\mathbf{KP}^\mathcal{P}$ with parameters in $\mathcal{S}$, and the morphisms are simply the instances of the inclusion relation between these fragments.
\end{enumerate}

A basic fact about standard systems is that they are fixed under end-extensions. This can now be stated as that the function $(\mathcal{M} \mapsto \mathrm{SSy}_\mathcal{S}(\mathcal{M}))$ expands to a functor from $\mathcal{S}/\mathbf{KP}^\mathcal{P}$ to $\mathbf{Class}_\mathcal{S}$.

Similarly, a basic fact about rank-initial embeddings is that they preserve the truth of $\Sigma_1^\mathcal{P}$-formulae with parameters in the domain. This fact can now be stated as that the function $(\mathcal{M} \rightarrow \mathrm{Th}_{\Sigma_1^\mathcal{P}, \mathcal{S}}(\mathcal{M}))$ expands to a functor from $\mathcal{S}/\mathbf{KP}^\mathcal{P}$ to $\Sigma_1^\mathcal{P}(\mathcal{S}/\mathbf{KP}^\mathcal{P})$.

Now note that the basic forward direction (a) $\Rightarrow$ (b) of Corollary \ref{Friedman cor} is simply the statement that there is a canonical functor $\mathbf{Fwd} : \mathcal{S}/\mathbf{KP}^\mathcal{P} \rightarrow \mathbf{Class}_\mathcal{S} \times \Sigma_1^\mathcal{P}(\mathcal{S}/\mathbf{KP}^\mathcal{P})$ obtained by combining the two functors above.

The more difficult to prove direction (a) $\Leftarrow$ (b) of Corollary \ref{Friedman cor} may be stated as that there is a function $\mathbf{Bwd} : \mathbf{Class}_\mathcal{S} \times \Sigma_1^\mathcal{P}(\mathcal{S}/\mathbf{KP}^\mathcal{P}) \rightarrow \mathcal{S}/\mathbf{KP}^\mathcal{P}$ on objects and morphisms, such that $\mathbf{Fwd} \circ \mathbf{Bwd} = \mathbf{Identity}$ on objects and morphisms, but it is not clear from the statement of Corollary \ref{Friedman cor} whether $\mathbf{Bwd}$ can be obtained as a functor (that is, it is not clear that $\mathbf{Bwd}$ can be chosen so as to preserve composition of morphisms):

\begin{que}
Is there a functor $\mathbf{Bwd} : \mathbf{Class}_\mathcal{S} \times \Sigma_1^\mathcal{P}(\mathcal{S}/\mathbf{KP}^\mathcal{P}) \rightarrow \mathcal{S}/\mathbf{KP}^\mathcal{P}$, such that $\mathbf{Fwd} \circ \mathbf{Bwd} = \mathbf{Identity}$?
\end{que}

Our Theorem \ref{Friedman thm} shows that continuum many embeddings can be obtained in Corollary \ref{Friedman cor} (a) $\Leftarrow$ (b). This gives at least some encouragement to that this question can be answered in the affirmative. But some caution with regard to conjecturing is warranted, since it is not at all clear that maximal {\em quantity} in choices would ultimately enable us to make to make all the choices in an the {\em coherent} fashion required.

Other than the sketched case study above, our Gaifman-style Theorem \ref{Gaifman thm} has already been stated largely in category theoretic language. Similar translations could be made for several statements in Section \ref{Characterizations}: For example, the notion of {\em fixed-point set} of a self-embedding can be generalized to the category theoretic notion of {\em equalizer} of a pair of embeddings.

All in all, it appears that the language of category theory would at least provide a fresh perspective on embeddings between models of set theory (or arithmetic).

\section{Directions for further research on stratified algebraic set theory}\label{where to go NF cat}

$\mathrm{(I)ML(U)}_\mathsf{Cat}$ has been shown, respectively, to interpret $\mathrm{(I)NF(U)}_\mathsf{Set}$, and has conversely been shown to be interpretable in $\mathrm{(I)ML(U)}_\mathsf{Class}$, thus yielding equiconsistency results. Since the axioms of a Heyting category can be obtained from the axioms of topos theory, it is natural to ask:

\begin{que}
Can the axioms of $\mathrm{(I)ML(U)}_\mathsf{Cat}$ be simplified? In particular, is it necessary to include the axioms of Heyting categories or do these follow from the other axioms?
\end{que}

To be able to interpret the set theory in the categorical semantics, this research introduces the axiomatization $\mathrm{(I)ML(U)}_\mathsf{Cat}$ corresponding to predicative $\mathrm{(I)ML(U)}_\mathrm{Class}$. This is analogous to the categories of classes for conventional set theory studied e.g. in \cite{ABSS14}. But it remains to answer:

\begin{que}
How should the speculative theory $\mathrm{(I)NF(U)}_\mathsf{Cat}$ naturally be axiomatized? I.e. what is the natural generalization of $\mathrm{(I)NF(U)}_\mathsf{Set}$ to category theory, analogous to topos theory as the natural generalization of conventional set theory? Moreover, can any category modeling this theory be canonically extended to a model of $\mathrm{(I)ML(U)}_\mathsf{Cat}$, or under what conditions are such extensions possible?
\end{que}

Closely intertwined with this question, is the potential project of generalizing to topos theory the techniques of automorphisms and self-embeddings of non-standard models of set theory. In particular, the endofunctor $\mathbf{T}$ considered in this research should arise from an automorphism or self-embedding of a topos. This would be a natural approach to constructing a rich variety of categories modeling the speculative theory $\mathrm{(I)NF(U)}_\mathsf{Cat}$, many of which would presumably be extensible to models of $\mathrm{(I)ML(U)}_\mathsf{Cat}$.

\end{document}